\documentclass[10pt,a4paper]{article}
\usepackage[utf8]{inputenc}
\usepackage[english]{babel}
\usepackage[all]{xy}
\usepackage{enumitem}
\usepackage[dvipsnames]{xcolor}
\usepackage{graphicx}
\usepackage{stmaryrd}
\usepackage{amsfonts,amstext,amssymb,amsmath,amsthm,pspicture,bigints}
\usepackage{fullpage}
\usepackage{moreverb}
\usepackage{pifont}
\usepackage{pst-all}
\usepackage{tikz}
\usepackage[left=2cm,right=2cm,top=2cm,bottom=2cm]{geometry}
\usepackage{esint}
\usepackage{fourier-orns}
\usepackage{mathrsfs}
\usepackage{enumitem}
\usepackage{ulem}
\usepackage{mathtools}
\usepackage{math}
\mathtoolsset{showonlyrefs}
\usepackage{array}
\usepackage{hyperref}
\usepackage{braket}
\newcommand{\ii }{{\rm i} }
\newcommand{\op }{\textnormal{Op}}
\usepackage{authblk}
\newcolumntype{C}[1]{>{\centering\arraybackslash}b{#1}}
\newcolumntype{R}[1]{>{\raggedleft\arraybackslash}b{#1}}
\newcolumntype{L}[1]{>{\raggedright\arraybackslash}b{#1}}
\newcolumntype{M}[1]{>{\centering}m{#1}}
\newtheorem{theo}{Theorem}[section]
\newtheorem{defin}{Definition}[section]
\newtheorem{lem}{Lemma}[section]
\newtheorem{prop}{Proposition}[section]
\newtheorem{assu}{Assumption}

\newtheorem{remark}{Remark}[section]

\newtheorem{cor}{Corollary}[section]
\usepackage{bbm}
\numberwithin{equation}{section}

\date{}
\title{Quasi-periodic traveling electron layers}
\author{Luca Franzoi \thanks{Dipartimento di Matematica, Universit\'a degli Studi di Roma Tor Vergata, Via della Ricerca Scientifica 1, 00133 Roma, Italy.\\
E-mail address : franzoi@mat.uniroma2.it}\qquad Emeric Roulley\thanks{Dipartimento di Matematica ``Federigo Enriques'', Universit\'a degli Studi di Milano, Via Cesare Saldini 50, 20133 Milano,
Italy.\\
		E-mail address : emeric.roulley@unimi.it}}
\begin{document}
	\maketitle
	\begin{abstract}
		We consider the one dimensional space-periodic Vlasov-Poisson equations and construct, close to symmetric flat velocity strips, small amplitude traveling quasi-periodic electron-layers, namely strip-shaped patches of electrons in the phase space. These solutions are found for most values of the strip area. The proof uses a Nash-Moser construction together with reducibility arguments based on pseudo-differential homogeneous expansions and KAM reductions. Thanks to a suitable linear transformation of the unknowns, we reveal a connection between the electron patch problem and the classical, physically relevant, electronic Euler-Poisson system with cubic pressure law. As a direct, non trivial consequence, we obtain small-amplitude quasi-periodic traveling waves for this model as well. To the best of our knowledge, this work provides the first rigorous example of construction of quasi-periodic solutions for kinetic models. 
	\end{abstract}

\smallskip



\tableofcontents
\section{Introduction}
This work proposes to study 
the emergence of coherent structures in the form of quasi-periodic traveling wave solutions for equations coming from the kinetic theory of charged particles.
More precisely, we consider the simplest model describing the evolution of a collisionless plasma, namely the space-periodic one dimensional Vlasov-Poisson equation. While smooth solutions are generally known to decay towards trivial homogeneous states (Landau damping), lowering the regularity allows to obtain long-lived structures such as the classical BGK waves. In this work we construct quasi-periodic solutions in the electron-layer class for which the solution is a uniform probability distribution of electron density inside a strip-shaped domain in the phase space. In this introduction, we present the model, the type of solutions we are looking for and expose our main result. Then, we contextualize our work among the existing literature and end with the presentation of the key steps in the proof of our new Theorem \ref{thm QP traveling electron layers}.
\subsection{Electron layer solutions for the Vlasov-Poisson equations and main result}
This study investigates some dynamical properties for solutions to the one dimensional Vlasov-Poisson system which writes
	\begin{equation}\label{Vlasov Poisson}
		\begin{cases}
			\partial_{t}f(t,x,v)+v\,\partial_{x}f(t,x,v)-E(t,x)\,\partial_{v}f(t,x,v)=0,\\
			E(t,x)=\partial_{x}\boldsymbol{\varphi}(t,x),\\
			\partial_{x}^2\boldsymbol{\varphi}(t,x)=1-\displaystyle\int_{\mathbb{R}}f(t,x,v){\rm d}v,
		\end{cases}\qquad(t,x,v)\in\mathbb{R}_{+}\times\mathbb{T}\times\mathbb{R},
	\end{equation}
where $\mathbb{T}\triangleq\mathbb{R}/2\pi\mathbb{Z}$ is the periodic box of length $2\pi$. This model describes the evolution of a collisionless neutral plasma composed with free electrons moving in a neutralizing uniform background field of ions: for an introduction on its derivation, we refer the reader to \cite[Chap. 13]{BM02} and \cite{D87}. The unknown quantity $f(t,x,v)$ stands for the density of electrons traveling with speed $v$ at position $x$ and time $t.$ The plasma is supposed to have only one-dimensional properties, which explains the absence of magnetic effects and only the presence of electric ones through the electric field $E$. The latter is related to the electric potential $\boldsymbol{\varphi}$ through the second relation in \eqref{Vlasov Poisson}. The periodicity in the space variable $x\in\T$ imposes the neutrality condition
\begin{equation}\label{neutre cond}
	\fint_{\mathbb{T}}\int_{\mathbb{R}}f(t,x,v){\rm d}x{\rm d}v=1,\qquad\fint_{\mathbb{T}}\triangleq\frac{1}{2\pi}\int_{0}^{2\pi}.
\end{equation}
Note that the problem only involves $\partial_{x}\boldsymbol{\varphi}$, therefore $\boldsymbol{\varphi}$ is well-defined up to a time dependent additive constant that can be chosen so that $\boldsymbol{\varphi}$ has zero space average at any time. Throughout the paper, we shall use the notation
$$\forall \,m\in\mathbb{Z}^*,\quad\partial_{x}^{m}1=0\qquad\textnormal{and}\qquad\forall \,j\in\mathbb{Z}^*,\quad\partial_{x}^{m}\mathbf{e}_j=(\ii j)^m\mathbf{e}_j,\qquad\mathbf{e}_{j}(x)\triangleq e^{\ii jx}.$$
Hence, we get from \eqref{neutre cond}
\begin{equation}\label{inv Lap phi}
	\boldsymbol{\varphi}(t,x)=\partial_{x}^{-2}\bigg(1-\int_{\mathbb{R}}f(t,x,v){\rm d}v\bigg).
\end{equation}
Observe that the equation \eqref{Vlasov Poisson} can be reformulated as an active scalar equation as follows
\begin{equation}\label{active eq}
	\partial_{t}f(t,x,v)+\mathbf{v}(t,x,v)\cdot\nabla_{x,v}f(t,x,v)=0,\qquad\mathbf{v}(t,x,v)\triangleq\begin{pmatrix}
		v\\
		-E(t,x)
	\end{pmatrix},\qquad\nabla_{x,v}\triangleq\begin{pmatrix}
		\partial_x\\
		\partial_v
	\end{pmatrix}.
\end{equation}
We mention that the scalar product above must be understood as computed in the tangent space (isomorphic to $\mathbb{R}^2$) to the cylinder manifold $\mathbb{T}\times\mathbb{R}$ at the point $(x,v).$ We refer the reader to \cite[Sec. 1.1]{R23} for more details. Remark that the velocity field $\mathbf{v}$ is $(x,v)$-solenoidal (that is, divergence-free) and can be expressed as
\begin{equation}\label{vel pot}
	\mathbf{v}(t,x,v)=-\nabla_{x,v}^{\perp}\boldsymbol{\Psi}(t,x,v),\qquad\boldsymbol{\Psi}(t,x,v)\triangleq\frac{v^2}{2}+\boldsymbol{\varphi}(t,x),\qquad\nabla_{x,v}^{\perp}\triangleq\begin{pmatrix}
		-\partial_v\\
		\partial_{x}
	\end{pmatrix}.
\end{equation}
 A Yudovich type result has been obtained by Dziurzynski in \cite[Thm. 2.1.1]{D87}, where he showed that any compactly supported initial datum $f_0\in L^{\infty}(\mathbb{T}\times\mathbb{R})$ satisfying the neutrality condition \eqref{neutre cond} generates a unique global in time weak solution $f\in L^{\infty}\big([0,\infty),L^{\infty}(\mathbb{T}\times\mathbb{R})\big)$, which expresses as follows 
$$f(t,x,v)=f_0\big(X_t^{-1}(x,v)\big),\qquad\begin{cases}
	\partial_{t}X_{t}(x,v)=\mathbf{v}\big(t,X_{t}(x,v)\big),\\
	X_0(x,v)=(x,v).
\end{cases}$$
A special subclass of weak solutions is composed of the so-called \textit{patches of electrons}, given by functions of the form
$$f(t,x,v)=\frac{2\pi}{|\Omega_t|}\mathbf{1}_{\Omega_t},\qquad\Omega_t\triangleq X_t(\Omega_0),$$
where $\Omega_0\subset\mathbb{T}\times\mathbb{R}$ is a bounded initial domain.  Moreover, the divergence-free property of the velocity field $\mathbf{v}$ implies
\begin{equation}\label{area pres int}
    \forall\, t\geqslant 0, \quad |\Omega_t|=|\Omega_0|.
\end{equation}
In particular, if the initial domain is strip-shaped, namely there exist two initial $2\pi$-periodic profiles $v_{\pm}^{0}$ such that
$$\Omega_0 \equiv S_0\triangleq\big\{(x,v)\in\mathbb{T}\times\mathbb{R}\quad\textnormal{s.t.}\quad v_{-}^{0}(x)<v<v_{+}^{0}(x)\big\},$$
then, for any time $t>0$, this structure is preserved and we have
\begin{equation}\label{electron layer shape}
	f(t,x,v)=\frac{2\pi}{|S_t|}\mathbf{1}_{S_t}(x,v),\qquad S_t\triangleq X_{t}(S_0)=\big\{(x,v)\in\mathbb{T}\times\mathbb{R}\quad\textnormal{s.t.}\quad v_{-}(t,x)<v<v_{+}(t,x)\big\},
\end{equation}
with $x\mapsto v_{\pm}(t,x)$ two $2\pi$-periodic profiles. Furthermore, the area-preserving condition \eqref{area pres int} becomes
\begin{equation}\label{area cond0}
	\forall\,t\geqslant0,\quad\int_{0}^{2\pi}\big[v_{+}(t,x)-v_{-}(t,x)\big]{\rm d}x=|S_t|=|S_0|=\int_{0}^{2\pi}\big[v_{+}^{0}(x)-v_{-}^{0}(x)\big]{\rm d}x.
\end{equation}
Such solutions are called \textit{electron layers} and are of interest in this study. A family of trivial stationary electron layers  corresponds to symmetric flat strips parametrized by $a>0$, see \cite[Lem. 1.1]{R23},
	\begin{equation}\label{stationary sol}
		f_a(x,v)=\frac{1}{2a}\mathbf{1}_{S_{\textnormal{\tiny{flat}}}(a)}(x,v),\qquad S_{\textnormal{\tiny{flat}}}(a)\triangleq\mathbb{T}\times\left[-a,a\right].
	\end{equation}
	Let us mention that in this work we restrict the discussion to symmetric stationary states in the form \eqref{stationary sol} in order to simplify the  already complicated study, having only a one parameter dependence through the number $a>0$. However, we believe that the main result of the present work can be extended to non-symmetric configurations in the spirit of \cite{R23}, treating the time-periodic setting. Now, looking for non-trivial electron layer solutions  close to these equilibria, we are led to consider an ansatz in the form \eqref{electron layer shape}, where for all time $t\geqslant0,$ we have $|S_t|=4\pi a$ and the profiles $v_{\pm}(t,\,\cdot\,)$ take the following form, with $\varepsilon \in (0,1)$ a small number,
    \begin{equation}\label{v.pm}
        v_{\pm}(t,x)=\pm a+\varepsilon r_{\pm}(t,x),\qquad \fint_{\mathbb{T}}r_{\pm}(t,x){\rm d}x=0.
    \end{equation}

    \begin{figure}[!h]
		\begin{center}
			\begin{tikzpicture}[scale=0.8]
				\draw[->](-1,0)--(8,0)
				node[below right] {$x$};
				\draw[->](0,-3)--(0,3)
				node[left] {$v$};
				\draw[black,dashed] (0,-2)--(6.88,-2);
				\draw[black,dashed] (0,2)--(6.88,2);
				\draw[black] (6.88,-3)--(6.88,3);
				\node at (-0.4,-2) {$-a$};
				\node at (-0.5,2) {$a$};
				\node at (-0.2,-0.2) {$0$};
				\node at (7.15,-0.2) {$2\pi$};
				\draw[domain=0:6.88,thick, black,samples=500] plot [variable=\t] (\t,{2+0.5*sin(50*pi*\t)});
				\draw[domain=0:6.88,thick, black,samples=500] plot [variable=\t] (\t,{-2+0.5*cos(50*pi*\t)});
				\draw[red] (2.85,0.1)--(2.85,-0.1);
				\node at (2.85,-0.3) {$\textcolor{red}{x_2}$};
				\node at (2.85,1.3) {$\textcolor{red}{\varepsilon r_+(t,x_2)}$};
				\draw[->,red](2.85,2)--(2.85,2.5);
				\draw[red] (5.75,0.1)--(5.75,-0.1);
				\node at (5.75,-0.3) {$\textcolor{red}{x_1}$};
				\node at (5.75,-1.3) {$\textcolor{red}{\varepsilon r_-(t,x_1)}$};
				\draw[->,red](5.75,-2)--(5.75,-2.5);
			\end{tikzpicture}
		\end{center}
		\caption{Representation of electron layers near the symmetric homogeneous solution.}
	\end{figure}
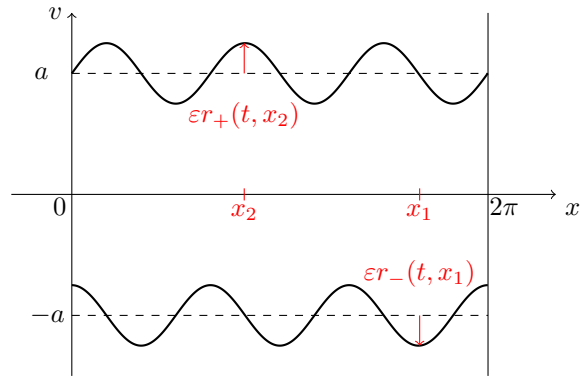
    
	Following the computations carried out in \cite[Sec. 1.2]{R23}, adapting them to our $2\pi$-periodic setting, we obtain that the deformations $r_{\pm}$ solve the following system
	\begin{equation}\label{system rpm}
		\begin{cases}
			\partial_{t}r_{+}(t,x)=-\partial_{x}\Big(\frac{1}{2\varepsilon}\big(a+\varepsilon r_{+}(t,x)\big)^2-\tfrac{1}{2a}(\partial_{xx}^{-1}r_{+})(t,x)+\tfrac{1}{2a}(\partial_{xx}^{-1}r_{-})(t,x)\Big),\\
			\partial_{t}r_{-}(t,x)=-\partial_{x}\Big(\frac{1}{2\varepsilon}\big(-a+\varepsilon r_{-}(t,x)\big)^2-\tfrac{1}{2a}(\partial_{xx}^{-1}r_{+})(t,x)+\tfrac{1}{2a}(\partial_{xx}^{-1}r_{-})(t,x)\Big).
		\end{cases}
	\end{equation}
Note that the system \eqref{system rpm} is Hamiltonian. More precisely, the set of equations \eqref{system rpm} writes
    \begin{equation}\label{firstHAM:intro}
        \partial_tr=\frac{1}{\varepsilon^2}\mathcal{J}\nabla\mathcal{E}(r),\qquad r\triangleq(r_+,r_-),\qquad\mathcal{J}\triangleq\begin{pmatrix}
            -\partial_x & 0\\
            0 & \partial_x
        \end{pmatrix},
    \end{equation}
    where
    $$\mathcal{E}(t)\triangleq\frac{|S_t|}{2\pi}\Big(\mathcal{E}_{\textnormal{\tiny{kin}}}(t)+\mathcal{E}_{\textnormal{\tiny{pot}}}(t)\Big)$$
    is (up to a multiplicative renormalizing constant coefficient) the sum of the kinetic and potential energies $\mathcal{E}_{\textnormal{\tiny{kin}}}$ and $\mathcal{E}_{\textnormal{\tiny{pot}}}.$ For more details, we refer the reader to Proposition \ref{prop Hamiltonian form}.
    The equations \eqref{system rpm} enjoy two important symmetries. First of all, they are time reversible, namely if $(t,x)\mapsto r_{\pm}(t,x)$ is a solution, so is $(t,x)\mapsto r_{\pm}(-t,-x)$. Accordingly, we say that a solution of \eqref{system rpm} is \textit{reversible} if $r_{\pm}(t,x) = r_{\pm}(-t,-x)$. Second, since the vector field is autonomous, the solutions are \textit{invariant by space translations}. At last, we mention that the system \eqref{system rpm} is also invariant under the transformation $(r_+,r_-)\mapsto-(r_-,r_+).$\\
    
	In \cite{R23}, the emergence of both small and large amplitude \textit{E(lectron)-states}, that is uniformly traveling periodic electron layers,  has been proved using local and global bifurcation techniques. In particular, \cite[Thm. 1.1-$(i)$]{R23} gives the emergence of pairs of branches formed with $\mathbf{m}$-symmetric (i.e. $r_{\pm}(t,x+\tfrac{2\pi}{\mathbf{m}})=r_{\pm}(t,x)$) E-states and bifurcating from the symmetric flat strip $S_{\textnormal{\tiny{flat}}}(a)$ at the velocities (renormalized with the $2\pi$-periodic setting)
	$$c_{\mathbf{m}}^{\pm}(a)\triangleq\pm\sqrt{\frac{a^2\mathbf{m}^2+1}{\mathbf{m}^2}}\cdot$$
	The goal of the present paper is to extend this to the quasi-periodic setting. More precisely, we are going to construct \textit{traveling quasi-periodic} waves solutions to \eqref{system rpm}.
    \paragraph{Main result.}
     We first provide the notion of quasi-periodic traveling waves, as introduced in \cite{BFM21}. We also refer the reader to the general presentation in Appendix \ref{sec travQP}.
    \begin{defin}{\bf (Traveling quasi-periodic waves)}\label{def.travQP.intro}
        A function $r:\mathbb{R}\times\mathbb{T}\rightarrow\mathbb{R}^2$ is said to be \textit{traveling quasi-periodic} if there exist $d\in\mathbb{N}^*$, $\vec{\jmath}\in\mathbb{Z}^d$, $\check{r}:\mathbb{T}^{d}\rightarrow\mathbb{R}^2$ and $\omega\in\mathbb{R}^d$ such that
	\begin{equation}\label{non-res QP}
	    r(t,x)=\check{r}(\omega t-\vec{\jmath}x)\qquad\textnormal{with}\qquad\forall \,\ell\in\mathbb{Z}^d,\quad\omega\cdot \ell \neq0.
	\end{equation}
	We call $\omega$ the \textit{frequency vector} and $\vec{\jmath}$ the \textit{velocity vector}.
    \end{defin}

    Motivated by the periodic case \cite{R23}, the goal of this article is to prove the existence of small amplitude quasi-periodic traveling electron layers that are solutions of \eqref{system rpm} and therefore fill invariant tori around the trivial state $f_a$ in \eqref{stationary sol}. The frequency vector $\omega$ will depend on the strip size geometrical parameter $a$, chosen in particular so that the non-resonance property \eqref{non-res QP} holds. Regarding the regularity, we will prove the existence of quasi-periodic traveling electron layers $\check{r}=(\check{r}_{+},\check{r}_{-})$ belonging to some Sobolev space
    \begin{equation}\label{sob.space.d.intro}
        H^s(\T^d,\R^2) \triangleq \Big\{ \check{r}(\,\cdot\,) = \sum_{\ell \in \Z^d} \check{r}_{\ell} e^{\im \ell\cdot (\,\cdot\,)} \in L^2(\T^d,\R^2) \quad \textnormal{s.t.} \quad \| \check{r}\|_{s}<\infty \Big\}, \qquad \| \check{r} \|_{s}^2 \triangleq \sum_{\ell \in\Z^d} \braket{\ell}^{2s}|\check{r}_{\ell}|^2 .
    \end{equation}
    Linearizing the system \eqref{system rpm} at a general state yields
$$\partial_{t}\begin{pmatrix}
	\rho_{+}\\
	\rho_{-}
\end{pmatrix}=\mathcal{J}\mathbf{M}_{\varepsilon r}(a)\begin{pmatrix}
	\rho_{+}\\
	\rho_{-}
\end{pmatrix},\qquad\mathbf{M}_{\varepsilon r}(a)\triangleq\begin{pmatrix}
	v_{+}-\tfrac{1}{2a}\partial_{xx}^{-1} & \tfrac{1}{2a}\partial_{xx}^{-1}\vspace{0.1cm}\\
	\tfrac{1}{2a}\partial_{xx}^{-1} & -v_{-}-\tfrac{1}{2a}\partial_{xx}^{-1}
\end{pmatrix},$$
with $v_{\pm}$ as in \eqref{v.pm}.
Specializing it for the equilibrium state $(r_+,r_-)=(0,0),$ we find that following Fourier multiplier equation
\begin{equation}\label{linear op 0 intro}
	\partial_{t}\begin{pmatrix}
		\rho_{+}\\
		\rho_{-}
	\end{pmatrix}=\mathcal{J}\mathbf{M}_{0}(a)\begin{pmatrix}
		\rho_{+}\\
		\rho_{-}
	\end{pmatrix},\qquad\mathbf{M}_{0}(a)\triangleq\begin{pmatrix}
		a-\tfrac{1}{2a}\partial_{xx}^{-1} & \tfrac{1}{2a}\partial_{xx}^{-1}\vspace{0.1cm}\\
		\tfrac{1}{2a}\partial_{xx}^{-1} & a-\tfrac{1}{2a}\partial_{xx}^{-1}
	\end{pmatrix}.
\end{equation}
Expanding $\rho_{\pm}$ in Fourier series
$$\rho_{\pm}(t,x)=\sum_{j\in\mathbb{Z}^*}\rho_{j}^{\pm}(t)\mathbf{e}_{j}(x),$$
the previous system is equivalent to the following countable set of linear autonomous differential systems
\begin{equation*}
	\partial_{t}\begin{pmatrix}
	\rho_{j}^{+}\\
	\rho_{j}^{-}
\end{pmatrix}=M_{j}(a)\begin{pmatrix}
\rho_{j}^{+}\\
\rho_{j}^{-}
\end{pmatrix},\qquad M_{j}(a)\triangleq\frac{\ii j}{|j|}\begin{pmatrix}
-a|j|-\tfrac{1}{2a|j|} & \tfrac{1}{2a|j|}\\
-\tfrac{1}{2a|j|} & a|j|+\tfrac{1}{2a|j|}
\end{pmatrix}.
\end{equation*}
It is possible to diagonalize for any $j\in\mathbb{Z}^*$ the matrix $M_j(a)$ as follows
\begin{equation}\label{freq intro}
Q_{j}^{-1}(a)M_{j}(a)Q_{j}(a)=-\begin{pmatrix}
		\ii\Omega_{j}(a) & 0\\
		0 & \overline{\ii\Omega_{j}(a)}
	\end{pmatrix},\qquad\Omega_j(a)\triangleq\frac{j}{|j|}\sqrt{a^2j^2+1},    
\end{equation}
where
    \begin{equation}\label{def:Qj-intro}
        Q_{j}(a)\triangleq\frac{1}{\sqrt{1-b_{j}^{2}(a)}}\begin{pmatrix}
	1 & b_{j}(a)\\
	b_{j}(a) & 1
\end{pmatrix},\qquad b_{j}(a)\triangleq\frac{1}{2a|j|\left(\sqrt{a^2j^2+1}+a|j|+\frac{1}{2a|j|}\right)}\in(0,1).
    \end{equation}
    This allows to solve the system \eqref{linear op 0 intro} and find, for almost every values of the parameter $a$, reversible traveling  quasi-periodic solutions. More precisely, given $\mathbb{S}_{\pm}\subset\mathbb{N}^*$ with $\mathbb{S}_+\cap\mathbb{S}_-=\varnothing$, we have the following solutions to \eqref{linear op 0 intro},
    \begin{equation}\label{linear travQP}
        \begin{pmatrix}
		\rho_+\\
		\rho_-
	\end{pmatrix}(t,x)=\sum_{j\in\mathbb{S}_+}\frac{\mathfrak{a}_{j}^+\cos\big(\Omega_{j}(a)t-jx\big)}{\sqrt{1-b_{j}^2(a)}}\begin{pmatrix}
	1\\
	b_{j}(a)
\end{pmatrix}+\sum_{j\in\mathbb{S}_-}\frac{\mathfrak{a}_{j}^-\cos\big(\Omega_{j}(a)t+jx\big)}{\sqrt{1-b_{j}^2(a)}}\begin{pmatrix}
b_{j}(a)\\
1
\end{pmatrix}.
    \end{equation}
    Such solutions are linear superpositions of periodic traveling waves, with the waves in the first sum traveling to the right with speed $\tfrac{1}{j}\Omega_{j}(a)>0$, with $j\in \S_{+}$, and the ones in the second sum traveling to the left with speed $-\tfrac{1}{j}\Omega_{j}(a)$, with $j\in \S_{-}$. The assumption $\S_{+}\cap \S_{-}=\varnothing$ is here to avoid trivial resonances, therefore ensuring that there are no waves traveling in opposite directions with the same phase and generating destructive interferences. In addition, recalling Definition \ref{def.travQP.intro}, the restriction of the geometrical parameter among a full measure set is to impose the non-resonance condition 
    $$\forall\,\ell\in\mathbb{Z}^{|\mathbb{S}_+|+|\mathbb{S}_-|}\setminus\{0\},\quad \Big(\big(\Omega_{j}(a)\big)_{j\in\mathbb{S}_+},\big(-\Omega_{j}(a)\big)_{j\in\mathbb{S}_-}\Big)\cdot\ell\neq0.$$
This is possible since the parameter $a$ modulates the frequencies, see \eqref{freq intro}. For more details, we refer the reader to Lemma \ref{lem linear solutions}.
\smallskip
    
Our main result states that we are able to construct solutions at the nonlinear level preserving this structure. More precisely, we have the following theorem.  
\begin{theo}{\textbf{(Traveling quasi-periodic electron layers)}.}\label{thm QP traveling electron layers}
	Consider the following parameters
	$$0<a_0<a_1,\qquad(d_+,d_-)\in\mathbb{N}^2\setminus\{(0,0)\}.$$
	Fix finite sets of Fourier modes 
    \begin{equation}\label{sites.tang.S}
        \mathbb{S}_+=\big\{j_{1}^{+},\ldots,j_{d_+}^{+}\big\}\subset\mathbb{N}^*,\qquad\mathbb{S}_-=\big\{j_{1}^{-},\ldots,j_{d_-}^{-}\big\}\subset\mathbb{N}^*,\qquad|\mathbb{S}_+|=d_+,\qquad|\mathbb{S}_-|=d_-
    \end{equation}
	such that
    \begin{equation}\label{non.intersect.S}
        \mathbb{S}_+\cap\mathbb{S}_-=\varnothing .
    \end{equation}
	Consider some amplitudes
	$$\big(\mathfrak{a}_{j}^+\big)_{j\in\mathbb{S}_+}\in\mathbb{R}^{d_+},\qquad \big(\mathfrak{a}_{j}^-\big)_{j\in\mathbb{S}_-}\in\mathbb{R}^{d_-}.$$
	There exist $s_0=s_0(d,q_0)>0$ and $\varepsilon_0>0$ such that, for any $\varepsilon\leqslant\varepsilon_0$, there exists a Cantor set $\mathscr{C}(\varepsilon)\subset(a_0,a_1)$
	with 
	$$\lim_{\varepsilon\to 0}|\mathscr{C}(\varepsilon)|=a_1-a_0,$$
	such that, for any $a\in\mathscr{C}(\varepsilon)$, there exists a reversible traveling quasi-periodic electron layer solution to the one dimensional Vlasov-Poisson equation \eqref{Vlasov Poisson} in the form
	$$f(t,x,v)=\tfrac{1}{2a}\mathbf{1}_{S_t}(x,v),\qquad S_t=\big\{(x,v)\in\mathbb{T}\times\mathbb{R}\quad\textnormal{s.t.}\quad-a+\varepsilon r_-(t,x)<v<a+\varepsilon r_+(t,x)\big\},$$
	where
    \begin{equation}\label{solution.thm}
        \begin{pmatrix}
		r_+\\
		r_-
	\end{pmatrix}(t,x)=\sum_{j\in\mathbb{S}_+}\frac{\mathfrak{a}_{j}^+\cos\big(\Omega_{j}^+(a,\varepsilon)t-jx\big)}{\sqrt{1-b_{j}^2(a)}}\begin{pmatrix}
	1\\
	b_{j}(a)
\end{pmatrix}+\sum_{j\in\mathbb{S}_-}\frac{\mathfrak{a}_{j}^-\cos\left(\Omega_{j}^-(a,\varepsilon)t-jx\right)}{\sqrt{1-b_{j}^2(a)}}\begin{pmatrix}
b_{j}(a)\\
1
\end{pmatrix}+\mathtt{p}(t,x)
    \end{equation}
with
$$\forall\, j\in\mathbb{N}^*, \quad b_{j}(a)\triangleq\frac{1}{2a|j|\left(\sqrt{a^2j^2+1}+a|j|+\frac{1}{2a|j|}\right)}\cdot$$
It is associated with the frequency vector
$$\omega(a,\varepsilon)\triangleq\Big(\Omega_{j_1^+}^+(a,\varepsilon),\ldots,\Omega_{j_{d_+}^+}^+(a,\varepsilon),\Omega_{j_1^-}^-(a,\varepsilon),\ldots,\Omega_{j_{d_-}^-}^-(a,\varepsilon)\Big)$$
and the velocity vector
$$\vec{\jmath}\triangleq\big(j_1^+,\ldots,j_{d_+}^+,j_{1}^-,\ldots,j_{d_-}^-\big).$$
In addition, the perturbed frequencies satisfy the following limits
$$\forall\,\kappa\in\{-,+\},\quad\forall \,j\in\mathbb{S}_{\kappa},\quad\lim_{\varepsilon\to0}\Omega_{j}^{\kappa}(a,\varepsilon)=\kappa\sqrt{a^2j^2+1}.$$
Finally, the perturbation $\mathtt{p}$ writes
$$\mathtt{p}(t,x)=\check{\mathtt{p}}\big(\omega(a,\varepsilon)t-\vec{\jmath}x\big),$$
with $\check{\mathtt{p}}\in H^{s_0}(\mathbb{T}^{d_++d_-},\mathbb{R}^2)$ satisfying
$$\forall\,\varphi\in\mathbb{T}^{d_++d_-},\quad\check{\mathtt{p}}(-\varphi)=\check{\mathtt{p}}(\varphi),\qquad\|\check{\mathtt{p}}\|_{H^{s_0}(\mathbb{T}^{d_++d_-},\mathbb{R}^2)}=O(\varepsilon).$$
At last, the solutions are linearly stable.
\end{theo}
\noindent Let us make some important remarks about Theorem \ref{thm QP traveling electron layers}.
	\begin{enumerate}
    \item  We want to emphasize once more that Theorem \ref{thm QP traveling electron layers} (and later Theorem \ref{thm EulerPoisson}) constitutes the first existence result of quasi-periodic solutions for kinetic models. In particular, Theorem \ref{thm QP traveling electron layers} selects initial data giving raise to non-decaying global in time solutions for the Vlasov-Poisson equation \eqref{Vlasov Poisson} with space periodic boundary conditions. This result goes in the opposite direction to the classical Landau damping phenomenon \cite{GI23,MV11}. This latter was mathematically proved first in the seminal work of Mouhot \& Villani \cite{MV11} using ideas from KAM theory. Our Theorem \ref{thm QP traveling electron layers} rather illustrates the KAM theory in its original purpose, namely showing the persistence of invariant tori filled by quasi-periodic solutions for Hamiltonian perturbations of integrable systems. More details are given in Section \ref{sect.ideas}. The solutions described in Theorem \ref{thm QP traveling electron layers} are far more general than the periodic ones found in \cite{R23} and their dynamics is much more complex than the classical BGK waves \cite{BGK57}.
    \item In Section \ref{appendix link Euler-Poisson}, we show that under a suitable linear change of variables, the system \eqref{system rpm} is actually equivalent to the classical Euler-Poisson system with cubic pressure law. Therefore, an immediate corollary of Theorem \ref{thm QP traveling electron layers} is the existence of small amplitude traveling quasi-periodic solution to the Euler-Poisson system with cubic pressure law. A more precise statement is given in Theorem \ref{thm EulerPoisson}.
    \item When working with periodic boundary conditions, the invariance by space translations allows to avoid some otherwise trivial resonances. For instance, in view of \eqref{freq intro}, we have
    $$\forall\, j\in\mathbb{N}^*,\quad\Omega_{j}(a)+\Omega_{-j}(a)=0.$$
    It is possible to remove this resonance by imposing the momentum condition since, with $j,j'\in\Z^*$,
    $$\begin{cases}
        \Omega_{j}(a)+\Omega_{j'}(a)=0,\\
        j-j'=0
    \end{cases} \qquad \Leftrightarrow \qquad\begin{cases}
        \frac{2j}{|j|}\sqrt{a^2j^2+1}=2\Omega_{j}(a) = 0,\\
        j=j'
    \end{cases}\qquad\Leftrightarrow \qquad j=j'=0.$$
    This reaches a contradiction with the fact that $j,j'\neq0.$ The momentum condition $j-j'=0$ naturally appears when working with traveling (quasi-periodic) waves. We refer the reader to Appendix \ref{sec travQP} and Lemma \ref{lem Fourier coeff op}. It is natural in the context of periodic boundary conditions and space translation invariant vector fields. Let us mention that under other types of boundary conditions, such as Dirichlet or specular reflection (see for instance \cite{MW25} in dimension 2), these trivial resonances are not expected to appear, with linear eigenvalues usually simple. However, it is not clear if patch-type solutions might make sense in these settings.

		\item The restriction of the parameter $a$ to the Cantor set $\mathscr{C}(\varepsilon)$ is not only a technical point. The problem involves small divisors and this set of ``good'' parameters is obtained gathering several Diophantine conditions to overcome the small divisors. Nevertheless, we expect that the non-chosen parameters do not generate invariant tori.
        \item  By linear stability we mean that the linearized vector field at the quasi-periodic traveling wave solution \eqref{solution.thm} and acting on quasi-periodic traveling waves has purely imaginary Floquet exponents. Here, the linear stability is a direct consequence of the proof:  it follows as a by-product of the reducibility procedure in Section \ref{sec redu} together with the Berti-Bolle reduction theory in action-angles-normal variables \cite{BB15} . For more details, we refer for instance to \cite[Thm. 1.2]{BBHM18}.
	\end{enumerate}
\subsection{Literature discussion}
Let us contextualize the Theorem \ref{thm QP traveling electron layers} by discussing about the related literature.\\

\textbf{On the Vlasov-Poisson system and related models.}
The Vlasov-Poisson system is the simplest kinetic model allowing to describe the time evolution of a collisionless plasma made of ions and electrons. In addition, we are in the regime where the ions are supposed to have sufficiently inertia to be considered as a uniform background neutralizing field. The well-posedness theory of the Vlasov--Poisson system has developed through a long sequence of progressively weaker results. The first existence theorems for classical solutions were established in low dimensions: in one dimension by Iordanskii \cite{I61} and Cottet \& Raviart \cite{CR84}; in two dimensions by Ukai \& Okabe \cite{UO78}; in the three-dimensional case for small initial data by Bardos \& Degond \cite{BD85}. Further advances for symmetric configurations were obtained in \cite{B77,H81,H82,S87,W80}. See also \cite{R07} for a broader survey of the literature. 
Later, independently by Pfaffelmoser \cite{P92} and Lions \& Perthame \cite{LP91}, was proved global-in-time existence for a large class of smooth initial data. The uniqueness problem has also been addressed in \cite{LP91} and \cite{L06}.

The above mentioned classical results depend on rather restrictive assumptions, typically involving strong integrability properties and the propagation of high moments; see also \cite{GJP00,P14}. Considerable effort has therefore been devoted to extending the theory to physically more natural settings, where only finite energy bounds are available. In this direction, Arsen’ev \cite{A75} (see also \cite{IN79}) proved the existence of global weak solutions assuming bounded initial data and finite kinetic energy, a result later refined in \cite{HH84}, where the boundedness requirement was replaced by a suitable $L^p$-integrability condition.

The difficulty of lowering the regularity assumptions is tied to the nonlinear coupling term $Ef$, which may fail to be locally integrable when $f$ is only integrable. To overcome this obstruction, DiPerna--Lions theory was adapted in \cite{DL88,DL89} through the introduction of renormalized solutions.

A central aspect of the Vlasov--Poisson dynamics is its underlying transport structure: smooth solutions are propagated along the characteristics generated by the phase-space vector field $(x,v)\mapsto(v,-E(t,x))$. At low regularity, however, the existence of a meaningful flow is no longer guaranteed, and the connection between the Eulerian formulation of the equation and the associated Lagrangian dynamics becomes substantially more delicate. This problem has been addressed by Ambrosio, Colombo \& Figalli in \cite{ACF17}, where the authors proved the equivalence between renormalized and Lagrangian weak solutions. As mentioned above, for the one dimensional case in the context of patches of electrons/electron layers -- of interest in the present manuscript -- Dziurzynski's theory applies  \cite{D87} : the flow is well-defined and weak solutions are Lagrangian.\\

In Section \ref{appendix link Euler-Poisson}, we give a correspondence between the system \eqref{system rpm} and the electronic Euler-Poisson system with cubic pressure law. More literature on this model is discussed there.\\

\textbf{Stability of kinetic waves: Landau damping and BGK waves.} The physics literature studying the existence and stability of stationary states or traveling waves for the Vlasov-Poisson equations is now well-developed \cite{AM67,BGK57,BD95,BG49,CC75,GIBFFS88,G70,HRK94,K55,L46,MB00,PA14,SLAS79}.
The system \eqref{Vlasov Poisson} admits a large family of trivial solutions called \textit{homogeneous states} made of time and space independent profiles, namely
$$f(t,x,v)=f_0(v),\qquad\int_{\mathbb{R}}f_0(v){\rm d}v=1.$$
Examples of such equilibria are the Maxwellians $f_0(v)=\tfrac{1}{2\sqrt{\beta}}e^{-\beta (v-v_0)^2}$, with $v_0\in\mathbb{R},\beta>0$, and -- of interest here -- flat electron layers $f_0(v)=\tfrac{1}{b-a}\mathbf{1}_{a<v<b}$, $a<b.$ A natural question concerns the stability/instability of these solutions. An important phenomenon is the so-called Landau damping: discovered by Landau \cite{L46}, this effect shows the decay of the electric field  and the convergence of the electron density   towards a nearby analytic homogeneous  for larger and larger times. The rigorous mathematical proof of this phenomenon was given by Mouhot \& Villani \cite{MV11} in analytic/Gevrey regularity. Their result has later been extended to different settings \cite{GI23,HRSS25,IPWW24} and it inspired later results on nonlinear damping effects for inviscid fluids, as for instance in \cite{BedM15,DM23,IJ20}.

When lowering the regularity, the situation becomes more interesting with the emergence of non-trivial steady states, as pointed out by Lin \& Zeng in \cite{LZ11}.
An important class of these solutions are the {\it BGK waves}, named after Bernstein, Greene \& Kruskal who discovered them \cite{BGK57}.  The first rigorous result of instability for such plasma waves is due to Guo \& Strauss \cite{GS95}. The extension of this stability analysis and other characterizations of BGK waves have been addressed in several other works: we refer the reader, for example, to \cite{BGH26,GL17,GS98,HM24,L01,L05,STZ25}. 
These results reveal that periodic kinetic waves exhibit subtle dynamical behavior and are often spectrally or nonlinearly unstable, emphasizing the need for a careful structural analysis of traveling solutions.  At last, we mention that Bostan \& Poupaud \cite{BP00} proved the existence of periodic mild solutions to \eqref{Vlasov Poisson}.\\



\textbf{Literature on quasi-periodic waves in fluids.}
The present work is actually motivated by a series of articles constructing quasi-periodic (traveling and non) solutions for fluid models. Indeed, electron patches are the kinetic equivalent to the vortex-patches in fluid mechanics. In the last years, a lot of attention has been given in the construction of quasi-periodic vortex patches via KAM methods and Nash-Moser scheme. The first work in this direction is part of the second author's PhD thesis \cite{HR21} where quasi-periodic vortex patches are found near the Rankine circular vortex for the Quasi-Geostrophic Shallow-Water equations. In the same period a similar result was proved for the generalized SQG equation by Hassainia, Hmidi \& Masmoudi \cite{HHM21}. In both articles, they use the internal parameter of the equations to impose non-resonance conditions. The year after, the same techniques were applied to the Euler-$\alpha$ model \cite{R22} and the idea of playing with a geometrical parameter was used to find quasi-periodic vortex patches for the planar Eulerian model near Kirkhoff ellipses \cite{BHM23}, near Rankine vortices in the unit disc \cite{HR21-1} and near annular shaped trivial solutions \cite{HHR23}. All these results were themselves motivated by a series of works on KAM for water-waves. We refer the reader to the seminal papers of Berti \& Montalto \cite{BM18} and Baldi, Berti, Haus \& Montalto \cite{BBHM18} for the construction of quasi-periodic standing water-waves for the gravity-capillary and pure gravity cases, respectively. Then, as part of the first author's PhD thesis, the extension to traveling quasi-periodic water-waves was possible allowing to include the presence of constant vorticity \cite{BFM21,BFM21-1}. Using weak Birkhoff normal form techniques, Feola \& Giuliani could also treat the deep water case \cite{FG20}. These works have recently been extended to the much more delicate 3D case in \cite{FMT25}. Other recent extensions include space quasi-periodic steady solutions (which become quasi-periodic traveling waves in a moving frame) for the Euler equations in the bounded channel \cite{FMM24}, large amplitude traveling waves for the $\beta$-plane approximation of rotating fluids \cite{BFMT25} and their nonlinear stability \cite{FFM26}. At last, we mention some other constructions of time quasi-periodic or almost-periodic solutions to Euler equations on the $d$-dimensional torus \cite{CF13,EPSL23,FM24} that do not rely on KAM arguments. These solutions are obtained as a finite or infinite superposition of rescaled non-interacting traveling waves built from fundamental compactly supported stationary Gravilov blocks.

For the non-autonomous case, in the presence of quasi-periodic forcing term, time quasi-periodic solutions were also constructed for the 3D Euler equations in \cite{BM21} and for 2D Navier-Stokes
equations \cite{FM23} approaching in the zero viscosity limit time quasi-periodic solutions of the 2D Euler equations for all times.\\

In the present manuscript, we follow the ideas of the above mentioned literature (in particular \cite{BM18,BFM21,HHR23}), using the area of the fondamental electron patch domain to avoid resonances and run KAM/Nash-Moser schemes. Also, we use the momentum condition to ensure the traveling quasi-periodicity and remove trivial resonances. For more details, we refer to the next subsection.


\subsection{Ideas of the proof}\label{sect.ideas}
In the remaining part of this introductory section, we present the main steps that lead to the proof of Theorem \ref{thm QP traveling electron layers}, highlighting the principal difficulties that we have to face and solve.\\

Our purpose is to construct a solution living near the linear quasi-periodic traveling solutions \eqref{linear travQP}. It is convenient to first perform a symplectic transformation, chosen to diagonalize the equilibrim linearization
\begin{equation}\label{def:bfQintro}
    \widetilde{r}\triangleq\mathbf{Q}^{-1}r,\qquad\mathbf{Q}\begin{pmatrix}
    \wtr_+\\
    \wtr_-
\end{pmatrix}\triangleq\sum_{j\in\mathbb{Z}^*}Q_j(a)\begin{pmatrix}
    \wtr_j^+\\
    \wtr_j^-
\end{pmatrix}\mathbf{e}_j,
\end{equation}
where the matrices $Q_j(a)$ are defined in \eqref{def:Qj-intro}. One has
\begin{equation}\label{def:bfOmgintro}
    \mathbf{Q}^{-1}\mathcal{J}\mathbf{M}_{0}\mathbf{Q}=-\ii\boldsymbol{\Omega}(a,D),\qquad\boldsymbol{\Omega}(a,D)\triangleq\begin{pmatrix}
    \Omega(a,D) & 0\\
    0 & -\Omega(a,D)
\end{pmatrix},
\end{equation}
where $\bM_{0}$ is as in \eqref{linear op 0 intro} and $\Omega(a,D)$ is the Fourier multiplier associated with the equilibrium frequencies $\big(\Omega_j(a)\big)_{j\in\mathbb{Z}^*}$ in \eqref{freq intro}.
The linear transformation $\mathbf{Q}$ allows to conjugate the full nonlinear system \eqref{firstHAM:intro} into a new Hamiltonian one 
$$\partial_t\widetilde{r}=\frac{1}{\varepsilon^2}\mathcal{J}\nabla H(\widetilde{r}),\qquad H(\widetilde{r})\triangleq\mathcal{E}(\mathbf{Q}\widetilde{r}),$$
which explicitly reads as
\begin{equation}\label{conj:eq.intro}
    \partial_{t} \wtr
=-\ii\boldsymbol{\Omega}(a,D) \wtr +\varepsilon \mathcal{J}\nabla P_{H}(\widetilde{r}),
\end{equation}
where the  Hamiltonian $H$ decomposes into a quadratic part $H_{\mathbf{L}_0}$ and a perturbation $P_H$ at least cubic
\begin{equation}\label{H dec.intro}
	H=H_{\mathbf{L}_0}+\varepsilon P_H,\qquad H_{\mathbf{L}_0}(\rho_+,\rho_-)\triangleq\sum_{j\in\mathbb{N}^*}\frac{\Omega_{j}(a)}{j}\Big(\big|\rho_{j}^+\big|^2+\big|\rho_j^-\big|^2\Big), \qquad P_{H}(\widetilde{r}) = O(\widetilde{r}^{\,3}).
\end{equation}
According to \eqref{H dec.intro}, $H$ is a Hamiltonian quasi-linear perturbation of the integrable system associated with $H_{\mathbf{L}_0}.$ This justifies the use of KAM techniques and is the new starting point of our analysis. The goal is to search for traveling quasi-periodic solutions of the equation \eqref{conj:eq.intro}. The approach that we pursue has been developed in the last decade for Fluid Dynamics PDEs, see for instance \cite{BFM21,BM18,HHR23}.

\paragraph{Action-angle-normal coordinates.} This is the content of Section \ref{sect. action-angles}.
Fix $\mathbb{S}_+$ and $\mathbb{S}_-$ two finite sets of integers as in \eqref{sites.tang.S} corresponding to the excited spatial Fourier modes generating the desired quasi-periodic solutions. We assume that $\S_{+}\cap \S_{-}=\varnothing$, as in \eqref{non.intersect.S}. Consider also some fixed amplitudes $(\mathfrak{a}_j^{\pm})_{j\in\mathbb{S}_{\pm}}$ giving raise to the linear solutions in \eqref{linear travQP}, see Lemma \ref{lem linear solutions}. This implies a decomposition of the functional space 
$$\mathbf{L}_0^2(\mathbb{T})\triangleq L_0^2(\mathbb{T})\times L_0^2(\mathbb{T}),\qquad L_0^2(\mathbb{T})\triangleq\Big\{u\in L^2(\mathbb{T})\quad\textnormal{s.t.}\quad\int_{\mathbb{T}}u(x){\rm d}x=0\Big\}$$
into tangential (localized on the sets $\mathbb{S}_{\pm}$) and normal components 
\begin{equation}\label{aan:intro}
    \widetilde{r}=\sum_{j\in\overline{\mathbb{S}}_+}\widetilde{r}_j^+\begin{pmatrix}
    1\\0
\end{pmatrix}\mathbf{e}_j+\sum_{j\in\overline{\mathbb{S}}_-}\widetilde{r}_j^-\begin{pmatrix}
    0\\1
\end{pmatrix}\mathbf{e}_j+z,\qquad\overline{\mathbb{S}}_{\kappa}\triangleq\mathbb{S}_{\kappa}\cup(-\mathbb{S}_{\kappa}),\qquad z\perp_{\mathbf{L}_0^2(\mathbb{T})}\mathtt{span}\{\mathbf{e}_j,\,\,j\in\overline{\mathbb{S}}_{\pm}\}.
\end{equation}
Then, on the tangential sites, we introduce the action ($I$) and angle ($\vartheta$) variables
\begin{align*}
    I&\triangleq\Big(\big(I_j^+\big)_{j\in\mathbb{S}_+},\big(I_j^-\big)_{j\in\mathbb{S}_-}\Big)\in\mathbb{R}^d,\qquad I_{-j}^{\pm}=I_{j}^{\pm},\\
    \vartheta&\triangleq\Big(\big(\vartheta_j^+\big)_{j\in\mathbb{S}_+},\big(\vartheta_j^-\big)_{j\in\mathbb{S}_-}\Big)\in\mathbb{T}^d,\qquad \vartheta_{-j}^{\pm}=-\vartheta_{j}^{\pm},\\
    d&\triangleq|\mathbb{S}_+|+|\mathbb{S}_-|
    \end{align*}
via a symplectic polar change of coordinates 
\begin{equation}\label{aa:intro}
    \forall\, \kappa\in\{-,+\},\quad\forall\, j\in\mathbb{S}_\kappa,\quad \widetilde{r}_j^{\kappa}=\sqrt{(\mathfrak{a}_j^{\kappa})^2+\tfrac{|j|}{2}I_j^{\kappa}}e^{\ii\theta_j^{\kappa}}.
\end{equation}
This allows to reformulate the problem in terms of embedded tori $\vf \mapsto i(\vf) \triangleq \big(\vartheta(\varphi),I(\varphi),z(\varphi)\big)$, $\varphi\in\mathbb{T}^{d}$, and searching for quasi-periodic solutions $\vf\mapsto r(\vf)=\bA\big(i(\vf)\big)$ (with the map $\mathbf{A}$ obtained combining \eqref{aan:intro} and \eqref{aa:intro}), of  the equation
\begin{equation}\label{conj:eqQP.intro}
    \omega\cdot\pa_{\vf} \wtr
=-\ii\boldsymbol{\Omega}(a,D) \wtr +\varepsilon \mathcal{J}\nabla P_{H}(\widetilde{r}).
\end{equation}
 Here, $\omega\in\R^{d}\setminus\{0\}$ is the unknown frequency vector of the quasi-periodic solution to determine: note also that the operator $\partial_t$ reads as $\omega\cdot\partial_{\varphi}$ with respect to the $(2\pi)^d$-periodic variable $\vf=\omega t$. In particular, we search for solutions that are reversible and traveling. With respect to the coordinates of the embedded torus,
 the reversibility and traveling conditions  translate into
\begin{equation}\label{pappa}
    \begin{aligned}
    \big( \vartheta(-\vf), I(-\vf),z(-\vf,-x)) \big)&=\big(-\vartheta(\varphi),I(\varphi),z(\varphi,x)\big),\\
 \forall\, y\in\mathbb{T},\quad \big(\vartheta(\varphi-\vec{\jmath}y),I(\varphi-\vec{\jmath}y),z(\varphi-\vec{\jmath}y,x)\big)&=\big(\vartheta(\varphi)-\vec{\jmath}y,I(\varphi),z(\varphi,x+y)\big). 
\end{aligned}
\end{equation}

\paragraph{Berti-Bolle approach to Nash-Moser theorem.}
To search for solutions of \eqref{conj:eqQP.intro} as embedded tori $\vf\mapsto i(\vf)$, we now
look for zeroes of the functional
\begin{equation}\label{BB.eq.intro}
    \mathscr{F}_{\varepsilon}(i,\alpha,\omega)=\omega\cdot \partial_{\varphi}i-X_{H_{\varepsilon}^{\alpha}}(i) = 0,
\end{equation}
where $X_{H_{\varepsilon}^\alpha}$ is the Hamiltonian vector field associated with the Hamiltonian $H_{\varepsilon}^\alpha,$ being $H$ expressed in terms of the new variables: its explicit form is given in \eqref{sF.vare.def}.
The degree of freedom $\alpha\in\R^{d}$ is introduced for technical reasons (to impose zero time averages in the inversion of the operator $\omega\cdot\partial_{\varphi}$) and allowing $\omega\in\R^d$ to be a free parameter to modulate: in the end, the variable $\alpha$ is linked back to $\omega$ to determine the final frequencies $\Omega_{j}^{\kappa}(a,\varepsilon)$ in \eqref{solution.thm} and produce an exact solution of \eqref{conj:eqQP.intro}.  Note that the linear solutions \eqref{linear travQP} correspond to the trivial embedding $\varphi\mapsto(\varphi,0,0),$ which is clearly reversible and traveling, in view of \eqref{pappa}.

Trying to invert the linearized operator $d_{(i,\alpha)}\mathscr{F}_{\varepsilon}$ at a given state naturally makes appear small-divisors due to time-space resonances. Therefore, it is hopeless to try a fixed point argument, a classical Newton method, a classical implicit function theorem or bifurcation methods. Indeed, Diophantine conditions (as in the proof of Lemma \ref{lem linear solutions}) imply a loss of a large amount of derivatives in Sobolev spaces for the formal inverse operator. The scheme to construct solutions of \eqref{BB.eq.intro} follows the approach proposed by Berti \& Bolle in \cite{BB15}, which in its essence is a symplectic version of the Lyapunov-Schmidt reduction combined with the nonlinear iteration of the Nash-Moser implicit function theorem. In order to track the traveling property, we implement the momentun preserving version of this theory developed in \cite{BFM21}. Actually, the analysis remaining at the linear level, it is not necessary to conserve the Hamiltonian structure in the Berti-Bolle theory. Here, we follow the simpler approach in \cite{HHR23}, avoiding the introduction of an isotropic torus and rather control the loss of symplecticity. The Nash-Moser iteration (run here in the Sobolev scale)
$$(i_{n+1},\alpha_{n+1})=(i_{n},\alpha_{n})-\Pi_{N_n}T_n\Pi_{N_n}\mathscr{F}_{\varepsilon}(i_n,\alpha_n),$$
is a modified Newton method with regularizing operators (Fourier projections), 
$$\Pi_{N_n}u(\varphi)\triangleq\sum_{\ell\in\mathbb{Z}^d\atop|\ell|\leqslant N_n}u_{\ell}e^{\ii\varphi\cdot\ell},\qquad N_{n}\triangleq N_0^{(\frac{3}{2})^n},\qquad N_0\gg1.$$
In our context, the first guess to start the Nash-Moser scheme is the linear reversible traveling quasi-periodic wave \eqref{linear travQP} that is seen in action-angle-normal variables as the trivial torus embedding $\varphi\mapsto(\varphi,0,0).$ Then, a key point in the Nash-Moser iteration is the construction, at any step of an approximate right inverse $T_n$ of $d_{(i,\alpha)}\mathscr{F}_{\varepsilon}(i_n,\alpha_n),$ that satisfies tame estimates with fixed loss of derivatives. The heart of Berti-Bolle's theory, implemented in Section \ref{sect:BB}, is the following. Given $(i_0,\alpha_0)$ a generic reversible and traveling approximate solution of the iterative scheme, it is possible to find a suitable linear transformation in the action-angle-normal variables that conjugates 
$d_{(i,\alpha)}\mathscr{F}_{\varepsilon}(i_0,\alpha_0)$ to a triangular system in the action-angle-normal variables, up to error terms vanishing at an exact solution. Therefore, one constructs an approximate right inverse to $d_{(i,\alpha)}\mathscr{F}_{\varepsilon}(i_0,\alpha_0),$ with tame estimates, by solving the triangular system. This is done modulo the inversion of the operator $\omega\cdot\partial_{\varphi}$ (there $\alpha$ plays its role to impose zero $\varphi$-averages) and the inversion of the linearized operator in the normal directions denoted $\mathscr{L}_{\omega}.$ This later fact is the core of the analysis and is described more in details in the next praragraph.
Finally, the scheme must guarantee that the final solution is indeed a reversible traveling quasi-periodic wave. This requires the notions of:
\begin{enumerate}[label=\textbullet]
    \item reversible operators, namely sending reversible functions onto anti-reversible functions (or the converse), i.e. 
    \begin{align*}
        r(-\varphi,-x)=-r(\varphi,x)\quad&\Rightarrow\quad (Ar)(-\varphi,-x)=(Ar)(\varphi,x),\\
        r(-\varphi,-x)=r(\varphi,x)\quad&\Rightarrow\quad (Ar)(-\varphi,-x)=-(Ar)(\varphi,x);
    \end{align*}
    \item reversibility preserving operators, namely sending (anti-)reversible functions onto (anti-)reversible functions, i.e. 
   \begin{align*}
        r(-\varphi,-x)=-r(\varphi,x)\quad&\Rightarrow\quad (Ar)(-\varphi,-x)=-(Ar)(\varphi,x),\\
        r(-\varphi,-x)=r(\varphi,x)\quad&\Rightarrow\quad (Ar)(-\varphi,-x)=(Ar)(\varphi,x);
    \end{align*}
    \item momentum preserving operators, namely sending quasi-periodic traveling waves onto quasi-periodic traveling waves, i.e.
    $$\big( \,\forall \, y\in\mathbb{T},\quad r(\varphi-\vec{\jmath}y,x)=r(\varphi,x+y)\Big)\qquad\Rightarrow\qquad \big(\,\forall \, y\in\mathbb{T},\quad (Ar)(\varphi-\vec{\jmath}y,x)=(Ar)(\varphi,x+y)\Big).$$
\end{enumerate}



\paragraph{Reducibility and invertibility of the linearized operator in the normal directions.}
Here we fix a reversible and traveling embedded torus $i_0$ that corresponds to a given approximate solution during the Nash-Moser iteration. As mentioned earlier, according to the Berti-Bolle theory, to construct the approximate right inverse to the full linearized operator at $i_0$, it is sufficient to build a right inverse to the linearized operator restricted to the normal directions. In Proposition \ref{prop Lnormal}, we find that this latter operator writes
 \begin{equation}\label{def:scrlomg-intro}
     \mathscr{L}_{\omega}=\Pi_{\overline{\mathbb{S}}_0}^{\perp}\big(\mathcal{L}^{(0)}-\varepsilon\partial_x\mathcal{R}\big)\Pi_{\overline{\mathbb{S}}_0}^{\perp},
 \end{equation}
 where $\Pi_{\overline{\mathbb{S}}_0}^{\perp}$ is the $\mathbf{L}^2(\mathbb{T})$-orthogonal projection onto the normal directions, $\mathcal{R}$ is a smooth operator and 
 $$\mathcal{L}^{(0)}\triangleq\omega\cdot\partial_{\varphi}\mathbb{I}_2-\tfrac{1}{\varepsilon^2}\di_{\widetilde{r}}\big(\mathcal{J}\nabla H(\widetilde{r})\big),\qquad\mathbb{I}_2\triangleq\textnormal{Id}_{\mathbf{L}^2(\mathbb{T})},\qquad \mathbf{L}^2(\mathbb{T})\triangleq L^2(\mathbb{T})\times L^2(\mathbb{T}).$$
As explained in Proposition \ref{prop L0}, the special structures of the equations and of the transformation $\mathbf{Q}$ in \eqref{def:bfQintro} imply that the operator $\mathcal{L}^{(0)}$ admits the expansion 
\begin{equation}\label{calL0:intro}
    \mathcal{L}^{(0)}=\omega\cdot\partial_{\varphi}\mathbb{I}_{2}+\ii\boldsymbol{\Omega}(a,D)+\varepsilon\partial_{x}\mathbf{F}+\mathbf{R}_0^{(d)}+\mathbf{R}_{0}^{(o)}+\mathbf{S}_{0,M},
\end{equation}
where $\boldsymbol{\Omega}(a,D)$ is as in \eqref{def:bfOmgintro}, $\mathbf{F}$ is a matrix diagonal multiplication operator
$$\mathbf{F}\triangleq\begin{pmatrix}
    f_+ & 0\\
    0 & f_-
\end{pmatrix},$$
the diagonal and anti-diagonal remainders $\mathbf{R}_0^{(d)}$ and $\mathbf{R}_0^{(o)}$ admit the following homogeneous expansions (in the sense of Definition \ref{def.HomExp})
$$\mathbf{R}_0^{(d)}=\begin{pmatrix}
			\displaystyle\sum_{p=1}^{M-1}\mathfrak{r}_{+,0,p}^{(d)}\partial_{x}^{-p} & 0\\
			0 & \displaystyle\sum_{p=1}^{M-1}\mathfrak{r}_{-,0,p}^{(d)}\partial_{x}^{-p}
		\end{pmatrix},\qquad\mathbf{R}_0^{(o)}=\begin{pmatrix}
		0 & \displaystyle\sum_{p=1}^{M-1}\mathfrak{r}_{+,0,p}^{(o)}\partial_{x}^{-p}\\
		\displaystyle\sum_{p=1}^{M-1}\mathfrak{r}_{-,0,p}^{(o)}\partial_{x}^{-p} & 0
	\end{pmatrix},$$
with $\mathfrak{r}_{+,0,p}^{(d)},\mathfrak{r}_{-,0,p}^{(d)},\mathfrak{r}_{+,0,p}^{(o)},\mathfrak{r}_{-,0,p}^{(o)}$ are $\varepsilon$-small functions of $(a,\omega;\varphi,x)$, and $\mathbf{S}_{0,M}$ is a pseudo-differential operator of order $-M$ and size $\varepsilon.$ Thus, the operator $\mathcal{L}^{(0)}$ is composed of a diagonal transport component with variable coefficients and a matrix pseudo-differential operator of order $-1.$ Our goal is to find linear, continuous, reversibility and momentum preserving transformations that conjugate $\mathscr{L}_{\omega}$ to a diagonal Fourier multiplier matrix operator. Nowadays, this procedure is referred to as \textit{reducibility} and is carried out through a sequence of precise transformations that progressively eliminate the various components in orders and size, with the goal to extract the normal form of the starting operator. 

A standard strategy suggests to start with the reduction the first order transport dynamics via (symplectic) diffeomorphisms, as it contributes in \eqref{calL0:intro} with the highest order of derivatives. Such diffeomorphisms preserve the particular homogeneous expansion form of the diagonal remainder $\mathbf{R}_0^{(d)}$ up to even smoothier terms: this latter fact is referred to as Egorov's Theorem. However, the anti-diagonal terms are not smoothing at any order and do not enjoy nice integral representations. Therefore, one cannot implement the reducibility in the hybrid off-diagonal/isotropic operatorial topology as in \cite{HHR23}. In addition, these terms loose the pseudo-differential structure after the diagonal transport reduction. Hence, one must perform a preliminary block-decoupling process. We now describe the main steps of this reduction scheme that leads to the inversion of $\sL_{\omega}$:
\begin{enumerate}[label=\textbullet]
    \item \textit{Block-decoupling up to smoothing operators:} This is the content of Section \ref{sect.block.anti}. The purpose is to eliminate the anti-diagonal part up to smoothing operators. As recalled in Proposition \ref{lem compo commu hom exp}, the class of operators having homogeneous expansion up to smoothing operators is stable under composition. Therefore, due to the particular form of $\mathcal{L}^{(0)}$ in \eqref{calL0:intro}, we are able to execute the reduction in a rather explicit way. More precisely, we can find a real, reversibility and momentum preserving transformation in the form
    $$\boldsymbol{\Phi}_{M-1}=\big(\mathbb{I}_2+\mathbf{X}_{1}\big)\circ\ldots\circ \big(\mathbb{I}_2+\mathbf{X}_{M-1}\big),\qquad\mathbf{X}_{0}=0,\qquad\mathbf{X}_{m}=\begin{pmatrix}
		0 & \mathfrak{p}_{+,m}\partial_{x}^{-m-1}\\
		\mathfrak{p}_{-,m}\partial_{x}^{-m-1} & 0
	\end{pmatrix},$$
    with $\mathfrak{p}_{+,m},\mathfrak{p}_{-,m}$ well-chosen functions such that
    $$\begin{array}{ccccccccccccc}
        \mathbf{R}_0^{(o)} & \overset{\mathbb{I}_2+\mathbf{X}_{1}}{\longrightarrow} &\mathbf{R}_1^{(o)} & \overset{\mathbb{I}_2+\mathbf{X}_{2}}{\longrightarrow} & \ldots & \overset{\mathbb{I}_2+\mathbf{X}_{m}}{\longrightarrow} & \mathbf{R}_m^{(o)} & \overset{\mathbb{I}_2+\mathbf{X}_{m+1}}{\longrightarrow} & \ldots & \overset{\mathbb{I}_2+\mathbf{X}_{M-2}}{\longrightarrow} &  \mathbf{R}_{M-2}^{(o)} & \overset{\mathbb{I}_2+\mathbf{X}_{M-1}}{\longrightarrow} &\mathbf{R}_{M-1}^{(o)}
    \end{array},$$
    with
    $$\mathbf{R}_m^{(o)}=\begin{pmatrix}
		0 & \displaystyle\sum_{p=m+1}^{M-1}\mathfrak{r}_{+,m,p}^{(o)}\partial_{x}^{-p}\\
		\displaystyle\sum_{p=m+1}^{M-1}\mathfrak{r}_{-,m,p}^{(o)}\partial_{x}^{-p} & 0
	\end{pmatrix},\qquad \mathfrak{r}_{\pm,m,p}^{(o)}=O(\varepsilon)$$
    so that in the end
    $$\boldsymbol{\Phi}_{M-1}^{-1}\mathcal{L}^{(0)}\boldsymbol{\Phi}_{M-1}=\omega\cdot\partial_{\varphi}\mathbb{I}_{2}+\ii\boldsymbol{\Omega}(a,D)+\varepsilon\partial_{x}\mathbf{F}+\mathbf{R}_{[-1]}^{(d)}+\mathbf{S}_{M},$$
    where $\mathbf{R}_{[-1]}^{(d)}$ is a purely diagonal remainders admitting the following homogeneous expansion
$$\mathbf{R}_{[-1]}^{(d)}=\begin{pmatrix}
			\displaystyle\sum_{p=1}^{M-1}\mathfrak{r}_{+,p}^{(d)}\partial_{x}^{-p} & 0\\
			0 & \displaystyle\sum_{p=1}^{M-1}\mathfrak{r}_{-,p}^{(d)}\partial_{x}^{-p}
		\end{pmatrix},$$
with $\mathfrak{r}_{+,p}^{(d)},\mathfrak{r}_{-,p}^{(d)}$ being $\varepsilon$-small functions of $(a,\omega;\varphi,x)$ and $\mathbf{S}_{M}$ is a pseudo-differential operator of order $-M$ and size $\varepsilon.$ 
    \item \textit{Transport reduction and Egorov argument:}
    The next step is to reduce to constant coefficients the first order (transport) component of $\mathcal{L}^{(0)}.$ This is now a classical KAM-type reduction procedure through toroidal diffeomorphisms that has been developed in \cite{FGMP19} for the scalar case with Lipshitz regularity in the parameters, and in \cite{BM21} for the 3-dimensional linear transport with higher regularity on the parameters. Symplectic variants associated with the anti-symmetric tensor $\partial_x$ have later been implemented in \cite{HR21} for scalar Hamiltonian vector fields and in \cite{HHR23} for matricial situation. Let us emphasize that here, similarly to \cite{HHR23}, we have a purely $2\times 2$ real system with two different transport velocities. Therefore, one must reduce each component separately and handle later mixed resonances. The result is given in Proposition \ref{prop trans red} and summarizes as follows. We can construct symplectic reversibility and momentum preserving diffeomorphisms of the torus in the form
    $$\mathscr{B}_{\pm}\rho(a,\omega;\varphi,x)=\big(1+\partial_{x}\beta_{\pm}(a,\omega;\varphi,x)\big)\rho\big(a,\omega;\varphi,x+\beta_{\pm}(a,\omega;\varphi,x)\big)$$
    continuous over $H^s$, and two constant coefficients $\mathtt{c}_{\pm}\triangleq\mathtt{c}_{\pm}(a,\omega;i_0)$ such that, in restriction to the set
    \begin{equation}\label{Cset:trans-intro}
        \mathcal{O}_{\textnormal{\tiny{T}},\infty}^{\gamma,\tau_1}(i_0)\triangleq\bigcap_{\kappa\in\{-,+\}}\bigcap_{{(\ell,j)\in\mathbb{Z}^d\times\mathbb{Z}\setminus\{(0,0)\}\atop\vec{\jmath}\cdot\ell+j=0}}\bigg\{(a,\omega)\in\mathcal{O}\quad\textnormal{s.t.}\quad\big|\omega\cdot\ell+j\mathtt{c}_{\kappa}(a,\omega;i_0)\big|>\frac{\gamma^{\upsilon}}{\langle\ell\rangle^{\tau_1}}\bigg\},
    \end{equation}
    with $\cO\subset \R^{d+1}$ a reference set as in \eqref{cO.ref.set}, and $\tau_1>d,$ $\upsilon\in(0,1),$ we have
    $$\mathscr{B}_{\pm}^{-1}\Big(\omega\cdot\partial_{\varphi}+\partial_{x}\Big(\big(\pm a+f_{\pm}(r)\big)\cdot\Big)\Big)\mathscr{B}_{\pm}=\omega\cdot\partial_{\varphi}+\mathtt{c}_{\pm}\partial_{x}.$$
    The non-resonance conditions \eqref{Cset:trans-intro} naturally appear to handle small divisors when solving homological equations during the transport reduction. Note that in \eqref{Cset:trans-intro} we have also imposed the momentum condition $\vec{\jmath}\cdot\ell+j=0$ in order to conserve the traveling property along the reduction scheme, in the spirit of \cite{BFM21}.
    
    Next, we study the effect of this conjugation on the lower order terms. For this, we make appeal to Egorov's Theorem stating that the conjugation by symplectic diffeomorphisms of the torus preserves the pseudo-differential operator class. Actually, in our situation, we use a more quantitative version (Proposition \ref{Egorov thm}) due to Berti, Kappeler \& Montalto \cite{BKM21} regarding the pseudo-differential homogeneous expansion class. More precisely, for any $m\in\mathbb{Z}$ and any $N\in\mathbb{N},$
    $$\mathscr{B}^{-1}(\mathfrak{p}\partial_{x}^m)\mathscr{B}=\sum_{k=0}^{N}\mathfrak{p}_{m-k}\partial_{x}^{m-k}+R_N,$$
    where $R_N$ is a $(m-N-1)$-smoothing pseudo-differential operator enjoying good tame estimates. Hence, considering the transformation
    $$\pmb{\mathscr{B}}\triangleq\begin{pmatrix}
		\mathscr{B}_+ & 0\\
		0 & \mathscr{B}_-
	\end{pmatrix},$$
	in Proposition \ref{prop trans Ego}, we find that in restriction to the Borel set 
    $\mathcal{O}_{\textnormal{\tiny{T}},\infty}^{\gamma,\tau_1}(i_0)$ in \eqref{Cset:trans-intro}, the following identity holds
	$$\pmb{\mathscr{B}}^{-1}\boldsymbol{\Phi}_{M-1}^{-1}\mathcal{L}^{(0)}\boldsymbol{\Phi}_{M-1}\pmb{\mathscr{B}}=\omega\cdot\partial_{\varphi}\mathbb{I}_2+\ii\boldsymbol{\Omega}(a,D)+\begin{pmatrix}
		(\mathtt{c}_+-a)\partial_{x} & 0\\
		0 & (\mathtt{c}_-+a)\partial_{x}
	\end{pmatrix}+\mathcal{R}_{[-1]}^{(d)}+\mathcal{S}_{M},$$
    where
    $$\mathcal{R}_{[-1]}^{(d)}\triangleq\begin{pmatrix}
		\displaystyle\sum_{p=1}^{M-1}\mathfrak{d}_{+,p}^{(d)}\,\partial_{x}^{-p} & 0\\
		0 & \displaystyle\sum_{p=1}^{M-1}\mathfrak{d}_{-,p}^{(d)}\,\partial_{x}^{-p}
	\end{pmatrix},\qquad \mathfrak{d}_{-,p}^{(d)}=O(\varepsilon\gamma^{-1})$$
    and $\mathcal{S}_{M}$ is a small $M$-smoothing operator that does not belong to the pseudo-differential class anymore but enjoys nice tame estimates.
    \item \textit{Projection onto normal modes and KAM diagonalization:}
    We can easily transpose the previous partial reduction concerning $\mathcal{L}^{(0)}$ (the full linearized operator of the original equations) onto $\mathscr{L}_{\omega}$ through \eqref{def:scrlomg-intro} by considering the following real, reversibility and momentum preserving projected transformation 
    $$\mathscr{M}_{\perp}\triangleq\Pi_{\overline{\mathbb{S}}_0}^{\perp}\boldsymbol{\Phi}_{M-1}\pmb{\mathscr{B}}\Pi_{\overline{\mathbb{S}}_0}^{\perp}.$$
    In Proposition \ref{lemma.sL.beforeKAM}, we show that the following reduction holds for the parameters inside 
    $\mathcal{O}_{\textnormal{\tiny{T}},\infty}^{\gamma,\tau_1}(i_0)$,
\begin{equation*}
\mathscr{M}_{\perp}^{-1}\mathscr{L}_{\omega}\mathscr{M}_{\perp}=\mathscr{L}^{[0]},\qquad\mathscr{L}^{[0]}\triangleq\omega\cdot\partial_{\varphi}\Pi_{\overline{\mathbb{S}}_0}^{\perp}+\im \,\mathscr{D}^{[0]}+\mathscr{R}^{[0]},
\end{equation*}
where
$$\mathscr{D}^{[0]}\triangleq\begin{pmatrix}
    \mathscr{D}_+^{[0]} & 0\\
    0 & \mathscr{D}_-^{[0]}
\end{pmatrix},\qquad\mathscr{D}_{\kappa}^{[0]}\triangleq\underset{j\in\mathbb{Z}^*\setminus\overline{\mathbb{S}}_{\kappa}}{\mathtt{diag}}\Big(\mu_{j,\kappa}^{[0]}\Big),\qquad \mu_{j,\kappa}^{[0]}(a,\omega)\triangleq\kappa\Omega_j(a)+j(\mathtt{c}_\kappa(a,\omega)-\kappa a)$$
and $\mathscr{R}^{[0]}$ is a $\mathcal{D}^q$-($-1$)-modulo tame operator in the sense of Definition \ref{modulo.tame.def}.

The structure of $\mathscr{L}^{[0]}$ is well-adapted for the implementation of a KAM reducibility scheme. Classically, for the water-wave equations, this latter is done exploiting the complex Hamiltonian reformulation of the equations, see \cite{BBHM18,BFM21,BFM21-1,BM18}. However, here such a reformulation seems hard to find out since we deal with a two-component real system. Therefore, one has to modify the KAM scheme. The first occurrence of such modified reduction is performed in \cite[Prop. 7.1]{HHR23} for constructing KAM invariant tori near annular vortex patches. In \cite{HHR23}, the authors work with a mixed matrix topology which is well-adapted to the fact that the anti-diagonal terms are smoothing at any order. In our situation, this is not the case and we prefer using the more adapted (modulo)-tame operator theory from \cite{BM18}, recalled in Section \ref{sec:tameOP}. To the best of our knowledge the only example of two-component reduction within this functional framework is due to Fu-Qu-Wu \cite{FQW25} for the two-component intermediate long-wave equation. Nevertheless they only need to run the scheme in the Lipshitz regularity. Here we perform it in the more challenging $C^{q}$-regularity. Using the $\mathcal{D}^{q}$-($-1$)-modulo tame formalism we can run a KAM reduction scheme at the operator level allowing to fully diagonalize the operator $\mathscr{L}^{[0]}.$ More precisely, as stated in Propositions \ref{prop.KAM.iter} and \ref{prop.KAM.conv}, we can find an invertible transformation $\Psi_{\infty}$ continuous over $H^s$, reversibility and momentum preserving, so that, under the following \textit{second-order Melnikov non-resonance conditions}
\begin{align}
		\mathcal{O}_{\mathtt{KAM},\infty}^{2\gamma,\tau}(i_0)& \triangleq
        \bigcap_{\kappa\in\{-,+\}}\bigcap_{(\ell,j,j')\in\mathbb{Z}^d\times(\mathbb{Z}\setminus\overline{\mathbb{S}}_{\kappa})^2\atop{\vec{\jmath}\cdot \ell+j-j'=0}}  \bigg\lbrace(a,\omega)\in\mathcal{O} \,\textnormal{ s.t. }\big|\omega\cdot \ell+\mu_{j,\kappa}^{[\infty]}(a,\omega;i_0)-\mu_{j',\kappa}^{[\infty]}(a,\omega;i_0)\big|>\frac{2\gamma\langle j-j'\rangle}{\langle \ell\rangle^{\tau}}\bigg\rbrace\\
		& \cap \bigcap_{(\ell,j,j')\in\mathbb{Z}^d\times(\mathbb{Z}\setminus\overline{\mathbb{S}}_{+})\times(\mathbb{Z}\setminus\overline{\mathbb{S}}_{-})\atop{\vec{\jmath}\cdot \ell+j-j'=0}}  \bigg\lbrace(a,\omega)\in\mathcal{O}\,\textnormal{ s.t. }\big|\omega\cdot \ell+\mu_{j,+}^{[\infty]}(a,\omega;i_0)-\mu_{j',-}^{[\infty]}(a,\omega;i_0)\big|>\frac{2\gamma \braket{j,j'}}{\langle \ell\rangle^{\tau}}\bigg\rbrace,
	\end{align}
    the following identity holds
    $$\Psi_{\infty}^{-1}\mathscr{L}^{[0]}\Psi_{\infty}=\omega\cdot\partial_{\varphi}\Pi_{\overline{\mathbb{S}}_0}^{\perp}+\ii\,\mathscr{D}_{\infty},\qquad\mathscr{D}_{\infty}\triangleq\begin{pmatrix}
			\mathscr{D}_{+}^{[\infty]} & 0\\
			0 & \mathscr{D}_{-}^{[\infty]}
		\end{pmatrix},$$
        where, for any $\kappa\in\{-,+\},$
        $$\mathscr{D}_{\kappa}^{[\infty]}\triangleq\underset{j\in\mathbb{Z}^*\setminus\overline{\mathbb{S}}_{\kappa}}{\mathtt{diag}}\Big(\mu_{j,\kappa}^{[\infty]}\Big),\qquad\mu_{j,\kappa}^{[\infty]}(a,\omega;i_0)\triangleq\mu_{j,\kappa}^{[0]}(a,\omega;i_0)+\mathtt{r}_{j,\kappa}^{[\infty]}(a,\omega;i_0),\qquad\mathtt{r}_{j,\kappa}^{[\infty]}=O\Big(\frac{\varepsilon\gamma^{-1}}{|j|}\Big).$$
    \item \textit{Inversion:} Proposition \ref{prop.almost.inv} will finally state that in restriction to the set 
    $$\mathcal{O}_{\textnormal{\tiny{T}},\infty}^{\gamma,\tau_1}(i_0)\cap\mathcal{O}_{\mathtt{KAM},\infty}^{2\gamma,\tau}(i_0)\cap\mathcal{O}_{\mathrm{\tiny{Inv}},\infty}^{\gamma,\tau}(i_0),$$
    with
    \begin{equation*}
            \mathcal{O}_{\mathrm{\tiny{Inv}},\infty}^{\gamma,\tau}(i_0)\triangleq\bigcap_{\kappa\in\{-,+\}}\bigcap_{(\ell,j)\in\mathbb{Z}^{d}\times(\mathbb{Z}^*\setminus\overline{\mathbb{S}}_{\kappa})\atop{\vec{\jmath}\cdot \ell+j=0}}\bigg\{(a,\omega)\in\mathcal{O}\quad\textnormal{s.t.}\quad\big|\omega\cdot \ell+\mu_{j,\kappa}^{[\infty]}(a,\omega;i_0)\big|>\frac{\gamma\braket{j}}{\braket{\ell}^{\tau}}\bigg\},
        \end{equation*}
        we find a full right inverse to $\mathscr{L}_{\omega}$ enjoying tame estimates.
\end{enumerate}

\paragraph{Measure estimates and momentum preserving degenerate KAM theory.} The Nash-Moser scheme produces a final state $(a,\omega)\mapsto\big(i_{\infty}(a,\omega),\alpha_{\infty}(a,\omega)\big)$ solution to $\mathscr{F}_{\varepsilon}(i,\alpha;a,\omega)=0$ provided that we restrict the parameters $(a,\omega)$ to a Borel set $\mathtt{G}_{\infty}$ constructed as the intersection of all the Cantor sets encountered along the different schemes of the multiple reductions. A solution to the original problem is obtained by constructing a frequency curve
$a\mapsto\omega(a,\varepsilon)$ solution to the implicit equation that corresponds to a solution of the original problem for
$$\alpha_{\infty}\big(a,\omega(a,\varepsilon)\big)=\Big(\big(-\Omega_{j}(a)\big)_{j\in\mathbb{S}_+},\big(\Omega_{j}(a)\big)_{j\in\mathbb{S}_-}\Big).$$
Therefore, one finally obtains a desired traveling quasi-periodic solution for any value $a$ taken inside
$$\mathscr{C}(\varepsilon)\triangleq\Big\{a\in(a_0,a_1)\quad\textnormal{s.t.}\quad\big(a,\omega(a,\varepsilon)\big)\in\mathtt{G}_{\infty}\Big\}.$$
In Proposition \ref{prop.measure}, we show that the measure of the set $\mathscr{C}(\varepsilon)$ satisfies
$$(a_1-a_0)-C\varepsilon^{\delta}\leqslant|\mathscr{C}(\varepsilon)|\leqslant a_1-a_0$$
for some $\delta>0,$ thus proving that the set $\mathscr{C}(\varepsilon)$ is asymptotically of full measure in $[a_0,a_1]$ as $\varepsilon\to0.$ This is proved by using classical measure theory (R\"ussmann's Theorem \cite{R01}) requiring transversality, i.e. lower bound on maximal derivative with respect to the geometrical parameter $a$ of the non-resonance conditions appearing in the various Borel sets. The transversality bounds are obtained by perturbative arguments from the equilibrium transversality. These latter are proved in Lemma \ref{Russmann equilibre} and follow combining the next two ingredients:
\begin{enumerate}[label=\textbullet]
    \item the non-degeneracy of the equilibrium frequency vector (and variants), that is the function 
$$a\mapsto\Big(\big(\Omega_{j}(a)\big)_{j\in\mathbb{S}_+},\big(-\Omega_{j}(a)\big)_{j\in\mathbb{S}_-}\Big)$$
is not contained in any hyperplane. This is proved in Lemma \ref{lem non-deg}.
\item the momentum condition like $\vec{\jmath}\cdot\ell+j-j'=0$ in the spirit of \cite{BFM21} to handle trivial resonances as explained in Remark 3 after Theorem \ref{thm QP traveling electron layers}. This simplifies the approach in \cite{HHR23} where they impose really weak quasi-resonance conditions to handle the two different transport velocities interactions.
\end{enumerate}

\medskip

\noindent\textbf{Acknowledgements:} 

\smallskip

Both authors L.F. and E.R. are supported by the ERC STARTING GRANT 2021 ``Hamiltonian Dynamics, Normal Forms and Water Waves'' (HamDyWWa), Project Number: 101039762. They are also supported by INdAM - GNAMPA Project ``(In)-stability in Fluid Mechanics'', CUP E53C25002010001. Part of this project was carried out while the author E.R. was supported by PRIN 2020XB3EFL, "Hamiltonian and Dispersive PDEs". 

\smallskip

Similarly to \cite{R23}, the author E.R. would like to thank Frédéric Rousset who made him discover the notion of patches of electrons during a conference in Lyon. The present paper is actually the original idea he got at this conference and \cite{R23} was only a natural preliminary step in this direction. He also would like to thank Changzhen Sun for useful discussions about the electronic Euler-Poisson system.

\section{Link with the electronic Euler-Poisson system}\label{appendix link Euler-Poisson}
The purpose of this small section, independant from the rest of the paper, is to highlight the relation between our system for the evolution of the patch boundaries and the classical electronic Euler-Poisson system given by
\begin{equation}\label{Euler-Poisson system}
	\begin{cases}
		\partial_{t}\rho+\partial_{x}\big((1+\rho) u\big)=0,\\
		\partial_{t}\big((1+\rho)u\big)+\partial_{x}\big((1+\rho)u^2\big)+\partial_{x}P(\rho)=(1+\rho)\partial_{x}\phi,\\
		\partial_{xx}^2\phi=\rho.
	\end{cases}
\end{equation}
The system \eqref{Euler-Poisson system} describes the dynamics of electrons, coupled to a
background density of ions through a self-consistent electric field. In such a model, one neglects
the ions motion, to account for the time-scale separation induced by the small mass ratio between
electrons and ions. Here $1+\rho$, $u$ and $\partial_{x}\phi$ represent the electron density, the electron velocity and the self-consistent electric field, respectively. The thermal pressure of electrons $P(\rho)$ is often assumed to follow an adiabatic polytropic $\gamma$ law
\begin{equation}\label{polytopic law}
    P(\rho)=T(1+\rho)^{\gamma},\qquad T\in\mathbb{R}_+,\qquad\gamma\geqslant1.
\end{equation}
The number $T$ is the global temperature of the plasma. According to Jackson \cite[p. 337]{J62}, in order to be physically relevant, the exponent $\gamma$ must related to the number $\fm$ of degrees of freedom of the system through the relation $\gamma=\frac{\fm+2}{\fm}$. In dimension one, we have $\fm=1$ and it corresponds to taking $\gamma=3$, namely having cubic pressure law. In \cite{GHZ17}, Guo, Han and Zhang proved the global well-posedness and abscence of shocks for all times for smooth small amplitude solutions for the system \eqref{Euler-Poisson system} with cubic pressure law. For extensions to higher dimensions, we refer the reader to \cite{G98,IP13,LW14}. In general, the local well-posedness theory can be carried out assuming the more general condition $P'(1)\geqslant0.$ It is worth emphasizing that the homogeneous equilibria are, at most, spectrally marginally stable. Consequently, the analyses developed in the aforementioned works rely heavily on dispersive mechanisms and normal form techniques. Hence, the arguments are rather sensitive to the precise structure of the pressure law.\\

In contrast, the literature concerning the existence and stability of traveling waves for Euler-Poisson systems is quite poor and the situation behaves better in the two-species or ionic cases \cite{BK19,BK22,CDMS95,CDMS96,GHHR26,HNS03,HS02,R26}. For the electronic Euler-Poisson, the only mathematical results seem to be \cite{NRS23} and \cite{R23}.\\

Let us now explicit the link between the electronic Euler-Poisson system \eqref{Euler-Poisson system} and the electron patch system \eqref{system rpm}. We shall see that the electron density $\rho$ and electron velocity $u$ are obtained from a linear combination of the deformations $r_+,r_-$ of the patch boundaries. More precisely, define 
\begin{equation}
    \begin{cases}
        \rho \triangleq \frac{\varepsilon}{2a}(r_+ - r_-), \\
        u \triangleq \frac{\varepsilon}{2}(r_+ + r_-), 
    \end{cases} \Leftrightarrow \quad \begin{cases}
        r_{+} = \frac{1}{\varepsilon} (u + a \rho), \\
        r_{-} = \frac{1}{\varepsilon} (u - a \rho).
    \end{cases}
\end{equation}
Then, from \eqref{system rpm}, we get
\begin{align*}
	\partial_{t}\rho&=\frac{\varepsilon}{2a}(\partial_tr_+-\partial_tr_-)\\
	&=-\frac{1}{4a}\partial_{x}\big((a+\varepsilon r_+)^2-(-a+\varepsilon r_-)^2\big)\\
	&=-\frac{1}{4a}\partial_{x}\big((2a+\varepsilon r_+-\varepsilon r_-)(r_++ r_-)\big)\\
	&=-\partial_{x}\big((1+\rho)u\big),
\end{align*}
which corresponds to the first equation in \eqref{Euler-Poisson system}. Besides, \eqref{system rpm} also implies
\begin{align*}
	\partial_{t}u&=\frac{\varepsilon}{2}(\partial_{t}r_++\partial_{t}r_-)\\
	&=-\frac{1}{4}\partial_{x}\Big((a+\varepsilon r_+)^2+(-a+\varepsilon r_-)^2\Big)+\frac{\varepsilon}{2a}\partial_{x}^{-1}(r_+-r_-)\\
	&=-a^2\partial_{x}\rho-\tfrac{1}{2}\partial_{x}\big(a^2\rho^2+u^2\big)+\partial_{x}^{-1}\rho.
\end{align*}
Therefore,
\begin{align*}
	\partial_{t}\big((1+\rho)u\big)&=u\partial_{t}\rho+(1+\rho)\partial_{t}u\\
	&=-u\partial_{x}\big((1+\rho)u\big)+(1+\rho)\Big[-a^2\partial_{x}\rho-\tfrac{1}{2}\partial_{x}\big(a^2\rho^2+u^2\big)+\partial_{x}^{-1}\rho\Big]\\
    & = -u \pa_{x} \big( (1+\rho)u \big) - (1+\rho) u \pa_{x}u - a^2 (1+\rho)\pa_{x} - a^2 (1+\rho)\rho \pa_{x}\rho + (1+\rho)\pa_{x}^{-1}\rho.
\end{align*}
Also,
$$ u \pa_{x} \big( (1+\rho)u \big)  = \pa_{x} \big( (1+\rho)u^2 \big) - (1+\rho) u \pa_{x}u$$
and
$$(1+\rho) \pa_{x}\rho + (1+\rho)\rho\pa_{x}\rho = (1+\rho)^2 \pa_{x}\rho = \pa_{x}\big( \tfrac13 (1+\rho)^3\big).$$
Combining the foregoing calculations yields
\begin{align*}
    \partial_{t}\big((1+\rho)u\big)+\partial_{x}\big((1+\rho)u^2\big)-(1+\rho)\partial_{x}^{-1}\rho&=-\tfrac{a^2}{3}\partial_{x}\big((1+\rho)^3\big).
\end{align*}
This corresponds to the second (and third) equation in \eqref{Euler-Poisson system} with a pressure term in the power law form
$$P(\rho)=\frac{a^2}{3}(1+\rho)^3,$$
that is \eqref{polytopic law} with
$$T=\frac{a^2}{3}\qquad\textnormal{and}\qquad\gamma=3.$$
This relation has a clear physical meaning, since $a\gg1$ corresponds to a high-velocity regime for the particles, which implies strong thermal agitation and therefore a high temperature $T\gg1.$\\

Thus, as a consequence of Theorem \ref{thm QP traveling electron layers}, we obtain the following result.
\begin{theo}
    {\textbf{(Traveling quasi-periodic solutions for the electronic Euler-Poisson system with cubic pressure term)}}\label{thm EulerPoisson}\\
	Consider the following parameters
	$$0<a_0<a_1,\qquad(d_+,d_-)\in\mathbb{N}^2\setminus\{(0,0)\}.$$
	Fix finite sets of Fourier modes 
	$$\mathbb{S}_+=\big\{j_{1}^{+},\ldots,j_{d_+}^{+}\big\}\subset\mathbb{N}^*,\qquad\mathbb{S}_-=\big\{j_{1}^{-},\ldots,j_{d_-}^{-}\big\}\subset\mathbb{N}^*,\qquad|\mathbb{S}_+|=d_+,\qquad|\mathbb{S}_-|=d_-$$
	such that
	$$\mathbb{S}_+\cap\mathbb{S}_-=\varnothing.$$
	Consider some amplitudes
	$$\left(\mathfrak{a}_{j}^+\right)_{j\in\mathbb{S}_+}\in\mathbb{R}^{d_+},\qquad \left(\mathfrak{a}_{j}^-\right)_{j\in\mathbb{S}_-}\in\mathbb{R}^{d_-}.$$
	There exist $s_0=s_0(d,q_0)>0$ and $\varepsilon_0>0$ such that for any $\varepsilon\leqslant\varepsilon_0$, there exists a Cantor set $\mathscr{C}(\varepsilon)\subset(a_0,a_1)$
	with 
	$$\lim_{\varepsilon\to 0}|\mathscr{C}(\varepsilon)|=a_1-a_0$$
	such that for any $a\in\mathscr{C}(\varepsilon)$ there exists a traveling quasi-periodic solution to the electronic Euler-Poisson system with cubic pressure term
    $$\begin{cases}
		\partial_{t}\rho+\partial_{x}\big((1+\rho) u\big)=0,\\
		\partial_{t}\big((1+\rho)u\big)+\partial_{x}\big((1+\rho)u^2\big)+\frac{a^2}{3}\partial_{x}\big((1+\rho)^3\big)=(1+\rho)\partial_{x}\phi,\\
		\partial_{xx}^2\phi=\rho
	\end{cases}$$
    taking the form
	\begin{align*}
	\begin{pmatrix}
		\rho\\
		u
	\end{pmatrix}(t,x)&=\frac{\varepsilon}{2}\sum_{j\in\mathbb{S}_+}\mathfrak{a}_{j}^+\cos\big(\Omega_{j}^+(a,\varepsilon)t-jx\big)\begin{pmatrix}
	\sqrt{\frac{1-b_j(a)}{a^2\big(1+b_{j}(a)\big)}}\\
	\sqrt{\frac{1+b_j(a)}{1-b_{j}(a)}}
\end{pmatrix}\\
&\quad+\frac{\varepsilon}{2}\sum_{j\in\mathbb{S}_-}\mathfrak{a}_{j}^-\cos\left(\Omega_{j}^-(a,\varepsilon)t-jx\right)\begin{pmatrix}
-\sqrt{\frac{1-b_j(a)}{a^2\big(1+b_{j}(a)\big)}}\\
	\sqrt{\frac{1+b_j(a)}{1-b_{j}(a)}}
\end{pmatrix}+\mathtt{q}(t,x)
	\end{align*}
with
$$\forall j\in\mathbb{N}^*,\quad b_{j}(a)\triangleq\frac{1}{2a|j|\left(\sqrt{a^2j^2+1}+a|j|+\frac{1}{2a|j|}\right)}\cdot$$
It is associated with the frequency vector
$$\omega(a,\varepsilon)\triangleq\Big(\Omega_{j_1^+}^+(a,\varepsilon),\ldots,\Omega_{j_d^+}^+(a,\varepsilon),\Omega_{j_1^-}^-(a,\varepsilon),\ldots,\Omega_{j_{d_-}^-}^-(a,\varepsilon)\Big)$$
and the velocity vector
$$\vec{\jmath}\triangleq\big(j_1^+,\ldots,j_{d_+}^+,j_{1}^-,\ldots,j_{d_-}^-\big).$$
In addition, the perturbed frequencies satisfy the following limits
$$\forall\kappa\in\{-,+\},\quad\forall j\in\mathbb{S}_{\kappa},\quad\lim_{\varepsilon\to0}\Omega_{j}^{\kappa}(a,\varepsilon)=\kappa\sqrt{a^2j^2+1}.$$
Finally, the perturbation $\mathtt{q}$ writes
$$\mathtt{q}(t,x)=\check{\mathtt{q}}\big(\omega(a,\varepsilon)t-\vec{\jmath}x\big),$$
with $\check{\mathtt{q}}\in H^{s_0}(\mathbb{T}^{d_++d_-},\mathbb{R}^2)$ satisfying
$$\check{\mathtt{q}}(-\varphi)=\check{\mathtt{q}}(\varphi),\qquad\|\check{\mathtt{q}}\|_{H^{s_0}(\mathbb{T}^{d_++d_-},\mathbb{R}^2)}=O(\varepsilon^2).$$
\end{theo}
\begin{remark}
    Let us make the following remarks about the Theorem \ref{thm EulerPoisson}.
    \begin{enumerate}
        \item The Hamiltonian structure of system \eqref{Euler-Poisson system} is well-known. However, as mentioned in \cite{NRS23}, where they construct in particular traveling waves in the spatial case of the real line,
        \begin{quote}
            \footnotesize
            "Equilibria or relative equilibria - such as traveling waves - of the electronic Euler-Poisson system, even those of small amplitude, seem to have received no attention at all. The main reason is surely that none of the traveling waves whose existence requires the co-existence of two distinct equilibria at the same phase speed may exist for this system. Among plane waves this excludes both solitary waves and kinks."
        \end{quote}
        Therefore, finding traveling quasi-periodic solutions to \eqref{Euler-Poisson system} was a highly non-trivial task. By leveraging the equivalence with the simpler system \eqref{system rpm}, we have successfully constructed such solutions in the case of a cubic pressure law.
        \item Denoting $\mathtt{p}=(\mathtt{p}_1,\mathtt{p}_2)^{\top}$ the perturbation term in Theorem \ref{thm QP traveling electron layers}, then
        $$\mathtt{q}=\frac{\varepsilon}{2}\begin{pmatrix}
            a^{-1}(\mathtt{p}_1-\mathtt{p}_2)\\
            \mathtt{p}_1+\mathtt{p}_2
        \end{pmatrix}.$$
        \item The equivalence between the systems \eqref{Euler-Poisson system} and \eqref{system rpm} can be used to state that the results in \cite{R23} have a counterpart being the construction traveling periodic waves for the electronic Euler-Poisson system, recovering (with a simpler approach) the ones found in \cite{NRS23} in the particular case of cubic pressure law. This analogy can also be exploited in the converse way. Indeed, in \cite{NRS23} the authors showed that the traveling waves of the electronic Euler-Poisson system (with general pressure law that covers our cubic one) are spectrally unstable. This spectral instability is neither modulational nor co-periodic. Hence, we can deduce that the E-states (i.e. traveling periodic electron layers) constructed in \cite{R23} are spectrally unstable.
    \end{enumerate}
\end{remark}

\section{Hamiltonian system and its linearization}\label{sec QPS VP}
Throughout the paper, we consider the geometrical parameter $a$ introduced in \eqref{stationary sol} ranging in the interval $(a_0,a_1)$, with
\begin{equation}\label{a01 bfm}
	0<a_0<a_1.
\end{equation}
For $\varepsilon>0$ a small parameter, we take an electron layer ansatz in the form
\begin{equation}\label{ansatz.electron.layer}
    f(t,x,v)=\frac{2\pi}{|S_t|}\mathbf{1}_{S_t}(x,v),\qquad S_t=\big\{(x,v)\in\mathbb{T}\times\mathbb{R}\quad\textnormal{s.t.}\quad v_-(t,x)<v<v_+(t,x)\big\},
\end{equation}
with
\begin{equation*}
    |S_t|=4\pi a,\qquad v_{\pm}(t,x)=\pm a+\varepsilon r_{\pm}(t,x),\qquad\fint_{\mathbb{T}}r_{\pm}(t,x){\rm d}x=0.
\end{equation*} 
This solution lives close to the symmetric flat strip $S_{\textnormal{\tiny{flat}}}(a)$, as in \eqref{stationary sol}, and has the same area. The deformation functions $r_{\pm}$ satisfy the system \eqref{system rpm}. The purpose of this section is to expose the Hamiltonian nature of this system and then study its linearization around the stationary solution.

\subsection{Transport Hamiltonian system for the deformations of the flat strip} 

We shall now highlight the Hamiltonian nature of the system \eqref{system rpm}. This structure is important for the construction of quasi-periodic solution, specifically in the tangential analysis in Section \ref{sect. action-angles}. The corresponding Hamiltonian involves the kinetic and potential energies, is reversible and invariant under translations. The result reads as follows.
\begin{prop}\label{prop Hamiltonian form}
	Consider the following functional \begin{equation}\label{definition Hamiltonien}
		\mathcal{E}(t)\triangleq\frac{|S_t|}{2\pi}\Big(\mathcal{E}_{\textnormal{\tiny{kin}}}(t)+\mathcal{E}_{\textnormal{\tiny{pot}}}(t)\Big),
	\end{equation}
	where $\mathcal{E}_{\textnormal{\tiny{kin}}}(t)$ and $\mathcal{E}_{\textnormal{\tiny{kin}}}(t)$ are respectively the kinetic and the electrostatic potential energies defined by
    \begin{equation}
        \mathcal{E}_{\textnormal{\tiny{kin}}}(t)\triangleq\fint_{\mathbb{T}}\int_{\mathbb{R}}\frac{v^2}{2}f(t,x,v){\rm d}v{\rm d}x,\qquad\mathcal{E}_{\textnormal{\tiny{pot}}}(t)\triangleq-\frac{1}{2}\fint_{\mathbb{T}}\left(1-\int_{\mathbb{R}}f(t,x,v){\rm d}v\right)\boldsymbol{\varphi}(t,x){\rm d}x. \label{defin kin pot}
    \end{equation}
	Then, the system \eqref{system rpm} is Hamiltonian, namely it can be written as
	\begin{equation}\label{HAM VP}
		\partial_{t}r=\frac{1}{\varepsilon^2}\mathcal{J}\nabla\mathcal{E}(r),\qquad r\triangleq(r_{+},r_-),\qquad\mathcal{J}\triangleq\begin{pmatrix}
			-\partial_{x} & 0\\
			0 & \partial_{x}
		\end{pmatrix},
	\end{equation}
	where $\nabla\triangleq(\nabla_{r_+},\nabla_{r_-})$ denotes the $\mathbf{L}^2(\mathbb{T})$-gradient, with $\bL^2(\T)\triangleq L^2(\T)\times L^2(\T)$, associated with the scalar product
    \begin{equation}
        \big\langle (f_+,f_-),(g_+,g_-)\big\rangle_{\mathbf{L}^2(\mathbb{T})}\triangleq\fint_{\mathbb{T}}\big(f_+(x)g_+(x)+f_-(x)g_-(x)\big){\rm d}x.
    \end{equation}
\end{prop} 
\begin{remark}
	The quantities $\mathcal{E}_{\textnormal{\tiny{kin}}}$ and $\mathcal{E}_{\textnormal{\tiny{pot}}}$ are easily found in the physical literature when $(x,v)\in\mathbb{R}^3\times\mathbb{R}^3$ but can be adapted to our one-dimensional/periodic constraint and our choice of sign for the electrical potential. Using that renormalization by the area $|S_t|=4\pi a$ is constant in time, it is therefore a classical fact that $\mathcal{E}(t)$ in \eqref{definition Hamiltonien} is a prime integral of \eqref{Vlasov Poisson}.
\end{remark}
\begin{proof}
By \eqref{ansatz.electron.layer}, the kinetic energy $\cE_{\rm kin}$ in \eqref{defin kin pot} writes
	\begin{equation}\label{new:Ekin}
	    \mathcal{E}_{\textnormal{\tiny{kin}}}(r)(t)=\frac{1}{2a}\fint_{\mathbb{T}}\bigg(\int_{-a+\varepsilon r_-(t,x)}^{a+\varepsilon r_+(t,x)}\frac{v^2}{2}{\rm d}v\bigg){\rm d}x.
	\end{equation}
	Differentiating it with respect to $r_{\pm}$ in the direction $h_{\pm},$ we get
    \begin{align}
        \left\langle\nabla_{r_{\pm}}\mathcal{E}_{\textnormal{\tiny{kin}}}(r)(t),h_{-}(t)\right\rangle_{L^2(\mathbb{T})}&=\pm\frac{\varepsilon}{2a}\fint_{\mathbb{T}}\frac{\big(\pm a+\varepsilon r_{\pm}(t,x)\big)^2}{2}h_{\pm}(t,x){\rm d}x.
    \end{align}
	It implies that
	\begin{equation}\label{grad kin}
		\nabla_{r_+}\mathcal{E}_{\textnormal{\tiny{kin}}}(r)(t,x)=\frac{\varepsilon\big(a+\varepsilon r_+(t,x)\big)^2}{4a},\qquad\nabla_{r_-}\mathcal{E}_{\textnormal{\tiny{kin}}}(r)(t,x)=-\frac{\varepsilon\big(-a+\varepsilon r_-(t,x)\big)^2}{4a}\cdot
	\end{equation}
	Our next step is to study the potential energy $\cE_{\rm pot}$ in \eqref{defin kin pot}. By \eqref{ansatz.electron.layer}, we note that
	\begin{align*}
		1-\int_{\mathbb{R}}f(t,x,v){\rm d}v&=1-\frac{1}{2a}\int_{-a+\varepsilon r_-(t,x)}^{a+\varepsilon r_+(t,x)}{\rm d}v=-\frac{\varepsilon}{2a}\big(r_+(t,x)-r_-(t,x)\big).
	\end{align*}
	Therefore, using the third equation in \eqref{Vlasov Poisson}, we write
	\begin{equation}\label{new:Epot}
        \mathcal{E}_{\textnormal{\tiny{pot}}}(r)(t)=-\frac{\varepsilon^2}{8a^2}\fint_{\mathbb{T}}\big(r_+(t,x)-r_-(t,x)\big)\partial_{xx}^{-1}\big(r_+(t,x)-r_-(t,x)\big){\rm d}x.
	\end{equation}
	Differentiating with respect to $r_{\pm}$ in the direction $h_{\pm}$ and integrating by parts give
	\begin{align*}
		\left\langle\nabla_{r_{\pm}}\mathcal{E}_{\textnormal{\tiny{pot}}}(r)(t),h_{\pm}(t)\right\rangle_{L^2(\mathbb{T})}&=\mp\frac{\varepsilon^2}{8a^2}\fint_{\mathbb{T}}h_{\pm}(t,x)\partial_{xx}^{-1}\big(r_+(t,x)-r_-(t,x)\big){\rm d}x\\
		&\quad\mp\frac{\varepsilon^2}{8a^2}\fint_{\mathbb{T}}\big(r_+(t,x)-r_-(t,x)\big)\partial_{xx}^{-1}h_{\pm}(t,x){\rm d}x\\
		&=\mp\frac{\varepsilon^2}{4a^2}\fint_{\mathbb{T}}h_{\pm}(t,x)\partial_{xx}^{-1}\big(r_+(t,x)-r_-(t,x)\big){\rm d}x,
	\end{align*}
	which implies in turn
	\begin{equation}\label{grad pot}
		\nabla_{r_{\pm}}\mathcal{E}_{\textnormal{\tiny{pot}}}(r)(t,x)=\mp\tfrac{\varepsilon^2}{4a^2}\partial_{xx}^{-1}\big(r_+(t,x)-r_-(t,x)\big).
	\end{equation}
	Combining \eqref{definition Hamiltonien}, \eqref{grad kin}, \eqref{grad pot} and recalling that $|S_t|=4\pi a,$ we infer
	\begin{align*}
		\nabla_{r_+}\mathcal{E}(r)(t,x)&=\tfrac{\varepsilon}{2}\big(a+\varepsilon r_+(t,x)\big)^2-\tfrac{\varepsilon^2}{2a}\partial_{xx}^{-1}\big(r_+(t,x)-r_-(t,x)\big),\\
		\nabla_{r_-}\mathcal{E}(r)(t,x)&=-\left(\tfrac{\varepsilon}{2}\big(-a+\varepsilon r_-(t,x)\big)^2-\tfrac{\varepsilon^2}{2a}\partial_{xx}^{-1}\big(r_+(t,x)-r_-(t,x)\big)\right).
	\end{align*}
	Comparing with \eqref{system rpm}, we get \eqref{HAM VP}. This concludes the proof of Proposition \ref{prop Hamiltonian form}.
\end{proof}

The natural phase space associated with the Hamiltonian equation \eqref{HAM VP} is the zero-mean $\mathbf{L}_0^2(\mathbb{T})$ space defined by
\begin{equation}\label{def:bbL20}
    \mathbf{L}_0^2(\mathbb{T})\triangleq L_0^2(\mathbb{T})\times L_0^2(\mathbb{T}),\qquad L_{0}^{2}(\mathbb{T})\triangleq\left\{f\in L^2(\mathbb{T})\quad\textnormal{s.t.}\quad\fint_{\mathbb{T}}f(x){\rm d}x=0\right\}.
\end{equation} 
It is endowed with the symplectic form $\mathscr{W}$ given by 
\begin{equation}\label{symplectic.form}
    \begin{aligned}
    \mathscr{W}\big((r_+,r_-),(h_+,h_-)\big)& \triangleq\big\langle\mathcal{J}^{-1}(r_+,r_-)\,,\,(h_+,h_-)\big\rangle_{\mathbf{L}^2(\mathbb{T})} \\
    &=\fint_{\mathbb{T}}\Big[\partial_{x}^{-1}r_-(x)h_-(x)-\partial_{x}^{-1}r_+(x)h_+(x)\Big]{\rm d}x.
\end{aligned}
\end{equation}
The vector field $X_{\mathcal{E}}$ is defined through
$$d\mathcal{E}(r_+,r_-)[h_+,h_-]=\mathscr{W}\big(X_{\mathcal{E}}(r_+,r_-),(h_+,h_-)\big),\qquad\textnormal{i.e.}\qquad X_{\mathcal{E}}(r_+,r_-)=\mathcal{J}\nabla\mathcal{E}(r_+,r_-).$$
Now consider the involution $\mathscr{S}$ and the translation operator $\mathscr{T}_{y}$ (for a fixed $y\in\mathbb{T}$) given by
\begin{equation}\label{inv trans}
	\big(\mathscr{S}(r_+,r_-)\big)(x)\triangleq(r_+,r_-)(-x),\qquad\big(\mathscr{T}_{y}(r_+,r_-)\big)(x)\triangleq(r_+,r_-)(x+y).
\end{equation}
Then, one easily checks, using changes of variables, that
\begin{equation}\label{rev.inv.cE}
    X_{\mathcal{E}}\circ\mathscr{S}=-\mathscr{S}\circ X_{\mathcal{E}},\qquad\textnormal{i.e.}\qquad\mathcal{E}\circ\mathscr{S}=\mathcal{E},
\end{equation}
and
\begin{equation}\label{trans.inv.cE}
    X_{\mathcal{E}}\circ\mathscr{T}_y=\mathscr{T}_y\circ X_{\mathcal{E}},\qquad\textnormal{i.e.}\qquad\mathcal{E}\circ\mathscr{T}_y=\mathcal{E},
\end{equation}
which means that the Hamiltonian $\mathcal{E}$ is reversible and invariant under space translations, respectively.
\begin{remark}
   The system \eqref{HAM VP} has also an other symmetry. Defining the involution
$$\mathscr{I}(r_+,r_-)\triangleq-(r_-,r_+),$$
then it is clear from \eqref{definition Hamiltonien}, \eqref{new:Ekin} and \eqref{new:Epot} that
$$\mathcal{E}\circ\mathscr{I}=\mathcal{E}.$$
However, this last symmetry will not be used along this study.
\end{remark}

\subsection{Linear dynamics}\label{sect.linear.dyn}
The purpose of this subsection is to give the expression of the linearized operator at a general state and then to study the associated solutions for the linearized system at the equilibrium state $(r_+,r_-)=(0,0).$

Linearizing \eqref{system rpm} at a generic state $r=(r_+,r_-)$ in the direction $(\rho_+,\rho_-)$ yields
\begin{equation}\label{linear at general state}
    \partial_{t}\begin{pmatrix}
	\rho_{+}\\
	\rho_{-}
\end{pmatrix}=\mathcal{J}\mathbf{M}_{\varepsilon r}(a)\begin{pmatrix}
	\rho_{+}\\
	\rho_{-}
\end{pmatrix},\qquad\mathbf{M}_{\varepsilon r}(a)\triangleq\begin{pmatrix}
	v_{+}-\tfrac{1}{2a}\partial_{xx}^{-1} & \tfrac{1}{2a}\partial_{xx}^{-1}\vspace{0.1cm}\\
	\tfrac{1}{2a}\partial_{xx}^{-1} & -v_{-}-\tfrac{1}{2a}\partial_{xx}^{-1}
\end{pmatrix}, \qquad  v_{\pm}\triangleq \pm a+\varepsilon r_{\pm},
\end{equation}
with $\cJ$  as in \eqref{HAM VP}. At the equilibrium $(r_+,r_-)=(0,0)$, this becomes
\begin{equation}\label{lin eq QP}
	\partial_{t}\begin{pmatrix}
		\rho_{+}\\
		\rho_{-}
	\end{pmatrix}=\mathcal{J}\mathbf{M}_{0}(a)\begin{pmatrix}
		\rho_{+}\\
		\rho_{-}
	\end{pmatrix},\qquad\mathbf{M}_{0}(a)\triangleq\begin{pmatrix}
		a-\tfrac{1}{2a}\partial_{xx}^{-1} & \tfrac{1}{2a}\partial_{xx}^{-1}\vspace{0.1cm}\\
		\tfrac{1}{2a}\partial_{xx}^{-1} & a-\tfrac{1}{2a}\partial_{xx}^{-1}
	\end{pmatrix}.
\end{equation}
Expanding $\rho_{\pm}$ in Fourier series
$$\rho_{\pm}(t,x)=\sum_{j\in\mathbb{Z}^*}\rho_{j}^{\pm}(t)\mathbf{e}_{j}(x),$$
the system \eqref{lin eq QP} is equivalent to the following countably many linear autonomous differential systems
\begin{equation}\label{lin diff syst}
	\partial_{t}\begin{pmatrix}
	\rho_{j}^{+}\\
	\rho_{j}^{-}
\end{pmatrix}=M_{j}(a)\begin{pmatrix}
\rho_{j}^{+}\\
\rho_{j}^{-}
\end{pmatrix},\qquad M_{j}(a)\triangleq\frac{\ii j}{|j|}A_{j}(a), \qquad A_{j}(a) \triangleq  \begin{pmatrix}
-a|j|-\tfrac{1}{2a|j|} & \tfrac{1}{2a|j|}\\
-\tfrac{1}{2a|j|} & a|j|+\tfrac{1}{2a|j|}
\end{pmatrix}.
\end{equation}
The characteristic polynomial of the matrix $A_{j}(a)$ writes
\begin{align*}
	\chi_{A_j(a)}&=\Big(X+\big(a|j|+\tfrac{1}{2a|j|}\big)\Big)\Big(X-\big(a|j|+\tfrac{1}{2a|j|}\big)\Big)+\tfrac{1}{4a^2|j|^2}\\
	&=X^{2}-a^2|j|^2-1\\
	&=\Big(X-\sqrt{a^{2}|j|^{2}+1}\Big)\Big(X+\sqrt{a^{2}|j|^{2}+1}\Big).
\end{align*}
Hence, the matrix $M_{j}(a)$ is diagonalizable with simple pure imaginary eigenvalues $\pm\ii\Omega_{j}(a)$ defined by
\begin{equation}\label{eigenvalues VP}
	\Omega_{j}(a)\triangleq\tfrac{j}{|j|}\sqrt{a^{2}|j|^{2}+1}.
\end{equation}
In the sequel, refering the reader to the general discussion on pseudo-differential operators in Appendix \ref{appendix pseudo}, we shall consider the following corresponding Fourier multiplier $\Omega(a,D)$ acting on the space $L^{2}(\mathbb{T})$ associated with the following symbol
\begin{equation}\label{symbol eigenvalues VP}
	\Omega(a,\xi)\triangleq\chi(\xi)\frac{\xi}{|\xi|}\sqrt{a^2\xi^2+1},\qquad\forall j\in\mathbb{Z}^*,\quad\Omega(a,j)=\Omega_{j}(a),
\end{equation}
where $\chi\in C^{\infty}(\R,\R)$ is an even, positive $C^{\infty}$ cut-off function such that
\begin{equation}
    \chi(\xi) = 0 \ \textnormal{ if } \ |\xi|\leqslant\tfrac13 \,, \quad \chi(\xi) = 1 \ \textnormal{ if } \ |\xi|\geqslant \tfrac23 \,, \quad \pa_{\xi}\chi(\xi) >0 \ \ \forall\,\xi\in (\tfrac13,\tfrac23) .\label{def chi}
\end{equation}
Before going further into the analysis of the linearized operator, let us first discuss some basic properties of the equilibrium eigenvalues \eqref{eigenvalues VP} and their associated Fourier multiplier $\Omega(a,D).$
\begin{lem}\label{lem prop eig VP}
	For any $a\in (a_0,a_1)$, the frequencies $\Omega_{j}(a)$ in \eqref{eigenvalues VP} and their associated operator $\Omega(a,D)$ in \eqref{symbol eigenvalues VP} enjoy the following properties:
	\begin{enumerate}[label=(\roman*)]
		\item For any $\xi\in\mathbb{R},$ we have $\Omega(a,-\xi)=-\Omega(a,\xi)$;
		\item The sequence $\big(\Omega_{j}(a)\big)_{j\in\mathbb{N}^*}$ is strictly increasing;
		\item For any $j\in\mathbb{Z},$ we have
		$|\Omega_{j}(a)|\geqslant a_0|j|$;
		\item Let $M\in\mathbb{N}$. We denote
        \begin{equation}\label{frakNM.def}
            \mathfrak{N}_M\triangleq\max\big\{p\in\mathbb{N}\quad\textnormal{s.t.}\quad1-2p>-M\big\}=\begin{cases}
            \big\lfloor\tfrac{M+1}{2}\big\rfloor, & \textnormal{if }M\equiv0\,\,[2],\\
            \big\lfloor\tfrac{M+1}{2}\big\rfloor-1, & \textnormal{if }M\equiv1\,\,[2].
        \end{cases}
        \end{equation}
        Then, we have that $\Omega(a,D)\in OPS^{1}$ (recall Definition \ref{defin OPS}) with the following decomposition
		\begin{equation}\label{hom exp omgD}
			\ii\Omega(a,D)=a\partial_{x}+\sum_{p=1}^{\mathfrak{N}_M}\alpha_p\,a^{1-2p}\partial_{x}^{1-2p}+S_{M}(a,D),\qquad \alpha_{p}\triangleq\frac{(-1)^{p}}{p!}\prod_{m=0}^{p-1}(\tfrac{1}{2}-m),
		\end{equation}
	where $S_M(a,D)$ is a reversible Fourier multiplier in $OPS^{-M}$.
	\end{enumerate} 
\end{lem}
\begin{proof}
	The properties \textit{(i)}, \textit{(ii)} and \textit{(iii)} are readily obtained from \eqref{eigenvalues VP}-\eqref{symbol eigenvalues VP}.\\
	We prove now item \textit{(iv)}.
	According to \eqref{symbol eigenvalues VP}, we write for any $|\xi|\geqslant\tfrac13$
	\begin{equation}\label{split eig}
		\Omega(a,\xi)=\chi(\xi)\Big[a\xi+a\xi\left(\sqrt{1+\tfrac{1}{a^2\xi^2}}-1\right)\Big]\triangleq \chi(\xi)\Big(a\xi+\mathtt{r}(a,\xi)\Big).
	\end{equation}
	Note that it is sufficient to restrict the discussion to $|\xi|\geqslant\tfrac13$ because the cut of function $\chi$ vanishes below. Clearly $-\mathtt{r}(a,-\xi)=\mathtt{r}(a,\xi)>0.$ 
    In particular, the remainder $\mathtt{r}(a,\xi)$ is explicitly given, for a fixed $ M\in \mathbb{N}^*$, by
    \begin{equation}\label{r.(a.xi)}
        \mathtt{r}(a,\xi)=\sum_{p=1}^{\mathfrak{N}_M}\frac{c_{p}}{a^{2p-1}\xi^{2p-1}}+ \mathtt{s}_{M}(a,\xi), 
    \end{equation}
    where
    $$c_p\triangleq\frac{1}{p!}\prod_{m=0}^{p-1}\left(\tfrac{1}{2}-m\right)$$
    and
    \begin{equation}\label{def sM}
        \begin{aligned}
            \mathtt{s}_{M}(a,\xi) &\triangleq\frac{1}{\mathfrak{N}_M!(a\xi)^{2\mathfrak{N}_M}} \int_{0}^{1} (1-\tau)^{\mathfrak{N}_M} \Big( \partial_{z}^{\mathfrak{N}_M+1} \sqrt{1+z} \Big)|_{z=\frac{\tau}{(a\xi)^2}} {\rm d}\tau\\
            &=\frac{1}{\mathfrak{N}_M!(a\xi)^{2\mathfrak{N}_M}} \prod_{m=0}^{\mathfrak{N}_M}(\tfrac{1}{2}-m)\int_{0}^{1} (1-\tau)^{\mathfrak{N}_M} \left(1+\frac{\tau}{a^2\xi^2}\right)^{-\tfrac{1}{2}-\mathfrak{N}_M} {\rm d}\tau\\
            &=a\xi\left(\mathfrak{N}_M+1\right) c_{\mathfrak{N}_M+1}\int_{0}^{1}(1-\tau)^{\mathfrak{N}_M}\left(1+\tau a^2\xi^2\right)^{-\frac{1}{2}-\mathfrak{N}_M}d\tau.
        \end{aligned}
    \end{equation}
	Therefore,
	\begin{align}
    \ii\Omega(a,D)&=\op\big(\chi(\xi)a\ii\xi\big)+\op\left(\chi(\xi)\sum_{p=1}^{\mathfrak{N}_M}\frac{(-1)^pc_{p}}{a^{2p-1}(\ii\xi)^{2p-1}}\right)+\op\big(\chi(\xi)\ii\mathtt{s}_{M}(a,D)\big)\nonumber\\
		&=a\partial_{x}+\sum_{p=1}^{\mathfrak{N}_M}(-1)^{p}c_{p}a^{1-2p}\partial_{x}^{1-2p}+S_{M}(a,D),\qquad S_{M}(a,D)\triangleq\op\big(\chi(\xi)\ii\mathtt{s}_{M}(a,\xi)\big).\label{expand up to smooth omg}
	\end{align}
Clearly, with the last expression in \eqref{def sM}, we get $\mathtt{s}_{M}(a,-\xi)=-\mathtt{s}_{M}(a,\xi)$, which implies the reversibility of $S_{M}(a,D)$ by virtue of Proposition \ref{properties OPS}-6. Now, let us prove that $S_M(a,D)$ is indeed of order $-M.$
	Note that, from the last expression in \eqref{def sM}, one readily has that for any $(k,\ell)\in\mathbb{N}^2$,
	$$\sup_{|\xi|\geqslant\frac{1}{3}\atop a\in[a_0,a_1]}\langle\xi\rangle^{\ell+M}\left|\partial_{a}^{k}\partial_{\xi}^{\ell}\mathtt{s}_{M}(a,\xi)\right|\leqslant C(a_0,k,\ell)<\infty.$$
    Since $\chi$ is constant on $\{\xi\in\mathbb{R}:|\xi|\geqslant \tfrac23 \},$ then the Leibniz rule combined with the previous estimate allows to conclude the proof of the claim.
\end{proof}

Consider the symplectic (and thus invertible) transfer matrix $Q_{j}(a)$ defined by
\begin{equation}\label{def:Qj}
    Q_{j}(a)\triangleq\frac{1}{\sqrt{1-b_{j}^{2}(a)}}\begin{pmatrix}
	1 & b_{j}(a)\\
	b_{j}(a) & 1
\end{pmatrix},\qquad b_{j}(a)\triangleq\frac{1}{2a|j|\left(\sqrt{a^2j^2+1}+a|j|+\frac{1}{2a|j|}\right)}\in(0,1).
\end{equation}
Note that $Q_{j}(a)\in SL_{2}(\mathbb{R})=Sp_2(\mathbb{R}).$ In particular it is invertible with inverse given by
\begin{equation}\label{inverse Qj}
    Q_{j}^{-1}(a)=\frac{1}{\sqrt{1-b_{j}^{2}(a)}}\begin{pmatrix}
	1 & -b_{j}(a)\\
	-b_{j}(a) & 1
\end{pmatrix}.
\end{equation}
Moreover, one checks that
\begin{equation}\label{diago Mj}
	\forall j\in\mathbb{Z}^*,\quad Q_{j}^{-1}(a)M_{j}(a)Q_{j}(a)=-\begin{pmatrix}
		\ii\Omega_{j}(a) & 0\\
		0 & \overline{\ii\Omega_{j}(a)}
	\end{pmatrix}.
\end{equation}
Hence, one solves the linear differential systems \eqref{lin diff syst} and obtains the following structure of solutions to \eqref{lin eq QP}
\begin{align*}
	\begin{pmatrix}
		\rho_{+}\\
		\rho_{-}
	\end{pmatrix}(t,x)&=\sum_{j\in\mathbb{Z}^*}A_je^{\ii(jx-\Omega_{j}(a)t)}\begin{pmatrix}
	1\\
	b_{j}(a)
\end{pmatrix}+B_je^{\ii(jx+\Omega_{j}(a)t)}\begin{pmatrix}
b_{j}(a)\\
1
\end{pmatrix},
\end{align*}
where $A_{j}=\overline{A_{-j}}$ and $B_j=\overline{B_{-j}}.$ These solutions are either periodic, quasi-periodic or almost-periodic. In particular, the reversible solutions, namely solutions satisfying $\rho_{\pm}(t,x) = \rho_{\pm}(-t,-x)$, have the form
\begin{align}
   \begin{pmatrix}
	\rho_+\\
	\rho_{-}
\end{pmatrix}(t,x)=\sum_{j\in\mathbb{N}^*}\wtA_j\cos\big(\Omega_{j}(a)t-jx\big)\begin{pmatrix}
1\\
b_{j}(a)
\end{pmatrix}+\wtB_j\cos\big(\Omega_{j}(a)t+jx\big)\begin{pmatrix}
b_{j}(a)\\
1
\end{pmatrix}, \label{comp sol lin rev}
\end{align}
with $\wtA_j,\wtB_j\in\mathbb{R}.$

\subsection{Degenerate KAM theory and linear solutions}\label{sec transversal}
Our next task is to check the so called \textit{Rüssmann transversality conditions} for the equilibrium frequency vectors. We refer to Lemma \ref{Russmann equilibre} for a complete statement. Such conditions are required for measuring some sets of parameters in the sequel.\\

We fix $(d_+,d_-)\in\mathbb{N}^2\setminus\{(0,0)\}$ and consider two finite sets of distinct elements
\begin{equation}\label{defin S+S-}
	\mathbb{S}_+\triangleq\big\{j_{1}^{+},\ldots,j_{d_+}^{+}\big\}\subset\mathbb{N}^*,\qquad\mathbb{S}_-\triangleq\big\{j_{1}^{-},\ldots,j_{d_-}^{-}\big\}\subset\mathbb{N}^*,
\end{equation}
with
\begin{equation}\label{inter S+S-}
	|\mathbb{S}_+|=d_+,\qquad|\mathbb{S}_-|=d_-,\qquad\mathbb{S}_+\cap\mathbb{S}_-=\varnothing.
\end{equation}
In what follows, we also use the following notations
\begin{equation}\label{bbS}
	d\triangleq d_++d_-,\qquad\mathbb{S}\triangleq\mathbb{S}_+\sqcup\mathbb{S}_-\triangleq\big\{j_{1},\ldots,j_{d}\big\}\subset\mathbb{N}^*,\qquad\mathbb{S}_0\triangleq\mathbb{S}\cup\{0\}.
\end{equation}
For later purposes, we shall also introduce the symmetric tangential set
	\begin{equation}\label{symmetrized tangential}
	    \forall\,\kappa\in\{-,+\},\quad\overline{\mathbb{S}}_{\kappa}\triangleq\mathbb{S}_{\kappa}\sqcup(-\mathbb{S}_{\kappa}),\qquad\overline{\mathbb{S}}\triangleq\overline{\mathbb{S}}_{+}\sqcup\overline{\mathbb{S}}_-,\qquad\overline{\mathbb{S}}_0\triangleq\mathbb{S}_0\sqcup(-\mathbb{S}_0).
	\end{equation}
Then, we define the \textit{equilibrium frequency vector} (also called \textit{unperturbed frequency vector} later on)
\begin{equation}\label{equilibrium vector freq}
    \omega_{\textnormal{Eq}}(a)\triangleq\Big(\Omega_{j_1^+}(a),\ldots,\Omega_{j_{d_+}^+}(a),-\Omega_{j_1^-}(a),\ldots,-\Omega_{j_{d_-}^-}(a)\Big)\in\mathbb{R}^d.
\end{equation}

\begin{lem}[Non-degeneracy]\label{lem non-deg}
	The following vector functions 
    \begin{align}
        a & \mapsto\omega_{\textnormal{Eq}}(a) \in \R^{d}, \qquad a\mapsto\omega_{\textnormal{Eq}}(a)-\vec{\jmath}\,a\in \R^{d}, \qquad a\mapsto\big(\omega_{\textnormal{Eq}}(a),a\big)\in \R^{d+1}, \\
        a & \mapsto \big( \omega_{\rm Eq}(a), \Omega_{j}(a) \big) \in \R^{d+1} , \quad j\in\Z^*\setminus\overline{\S} , \\
        a & \mapsto \big( \omega_{\rm Eq}(a), \Omega_{j}(a), \Omega_{j'}(a) \big) \in \R^{d+2} , \quad j,j'\in\Z^*\setminus\overline{\S} , \quad |j|\neq |j'| ,
    \end{align}
    are non-degenerate on $[a_0,a_1]$, that is
	\begin{align}
		\forall\, c\in\mathbb{R}^d,&\qquad\Big(\forall \,a\in[a_0,a_1],\quad \big\langle c,\omega_{\textnormal{Eq}}(a)\big\rangle_{\mathbb{R}^d}=0\Big)\quad\Rightarrow\quad c=0,\\
       {\forall\, c\in\mathbb{R}^d},& { \qquad\Big(\forall\, a\in[a_0,a_1],\quad \big\langle c,\omega_{\textnormal{Eq}}(a)-\vec{\jmath}\,a \big\rangle_{\mathbb{R}^d}=0\Big)\quad\Rightarrow\quad c=0,} \\
		\forall \,c\in\mathbb{R}^{d+1},&\qquad\Big(\forall\, a\in[a_0,a_1],\quad \big\langle c,\big(\omega_{\textnormal{Eq}}(a),a\big)\big\rangle_{\mathbb{R}^{d+1}}=0\Big)\quad\Rightarrow\quad c=0, \\
        \forall \,c\in\mathbb{R}^{d+1},&\qquad\Big(\forall \, a\in[a_0,a_1],\quad \big\langle c,\big(\omega_{\textnormal{Eq}}(a),\Omega_{j}(a)\big)\big\rangle_{\mathbb{R}^{d+1}}=0\Big)\quad\Rightarrow\quad c=0 , \\
        \forall \,c\in\mathbb{R}^{d+2},&\qquad\Big(\forall \,a\in[a_0,a_1],\quad \big\langle c,\big(\omega_{\textnormal{Eq}}(a),\Omega_{j}(a),\Omega_{j'}(a)\big)\big\rangle_{\mathbb{R}^{d+2}}=0\Big)\quad\Rightarrow\quad c=0.
	\end{align}
\end{lem}
\begin{proof}
    We start by proving that $a  \mapsto\omega_{\textnormal{Eq}}(a)$ is non-degenerate.
	We decompose $c\in\mathbb{R}^d$ into $c=(c_{+},c_{-})\in\mathbb{R}^{d_+}\times\mathbb{R}^{d_-}$, then, up to replacing $c_{-}$ by $-c_{-}$, we see that we can assume that $d_-=0$ (i.e. $d_+=d$) and prove that the vector function $a\mapsto\big(\Omega_{j}(a)\big)_{j\in\mathbb{S}_+}$ is non-degenerate on $[a_0,a_1]$ in $\mathbb{R}^d$. Until the end of the proof, we simplify the notation into $j_{k}^+\triangleq j_k$ for $k\in\llbracket 1,d\rrbracket.$ Assume that there exists $(c_1,\ldots,c_d)\in\mathbb{R}^{d}$ such that
	$$\forall \,a\in[a_0,a_1],\quad\sum_{k=1}^{d}c_{k}\sqrt{a^2j_{k}^{2}+1}=0.$$
	This equation is equivalent to
	$$\forall\, x\in\big[a_0^2,a_1^2\big],\quad\sum_{k=1}^{d}c_{k}\sqrt{xj_{k}^{2}+1}=0.$$
	By analytic continuation principle, we extend the previous relation to the range $x\in\left(-\tfrac{1}{ j_{d}^2},\infty\right).$ For any $n\in\llbracket 0,d-1\rrbracket,$ differentiating the previous relation $n$ times with respect to $x$ and evaluating at $x=0$ yields 
	$$\forall\, n\in\llbracket 0,d-1\rrbracket,\quad \sum_{k=1}^{d}c_kj_{k}^{2n}=0.$$
	This is a Vandermonde invertible system since from \eqref{inter S+S-} all the $j_k^2$ are distinct. Hence, $c_1=\ldots=c_d=0,$ which achieves the proof of the first point. 
    
    {The second and the third vector functions are obtained by similar techniques, noting that $\partial_{\alpha}^n a = 0 $ for any $n\geqslant 2$, whereas the last two cases follow from the first, by extending $\omega_{\rm Eq}(a)$ with the extra components $(\Omega_{j}(a))$ and $(\Omega_{j}(a),\Omega_{j'}(a))$ and arguing as before. This ends the proof of Lemma \ref{lem non-deg}.}  
\end{proof}

From the non-degeneracy of the vectors in Lemma \ref{lem non-deg}, we deduce the following momentum preserving Rüssmann transversality conditions.
\begin{lem}[Equilibrium transversality]\label{Russmann equilibre}
    Let $\S_{+}$, $\S_{-}$, $(d_{+},d_{-})$ as in \eqref{defin S+S-}, \eqref{inter S+S-} and \eqref{bbS}. We define the velocity vector $\vec{\jmath}\in \Z^{d}$ as
    \begin{align}\label{velocity vector}
    \vec{\jmath}&\triangleq\big(j_{1}^+,\ldots,j_{d_+}^+,j_{1}^-,\ldots,j_{d_-}^-\big)\in\mathbb{Z}^{d}.
	\end{align}
	There exist $\rho_0>0$ and $q_0\in\mathbb{N}$ such that the following hold true.
	\begin{enumerate}[label=(\roman*)]
		\item For all $\ell\in\mathbb{Z}^d\setminus\{0\}$,
		$$\inf_{a\in[a_0,a_1]}\max_{k\in\llbracket 0,q_0\rrbracket}\Big|\partial_{a}^{k}\omega_{\textnormal{Eq}}(a)\cdot\ell\Big|\geqslant\rho_0\langle\ell\rangle.$$
        \item For all $\ell \in\mathbb{Z}^{d}$, 
		$$\inf_{a\in[a_0,a_1]}\max_{k\in\llbracket 0,q_0\rrbracket}\Big|\partial_{a}^{k}\big(\omega_{\textnormal{Eq}}(a) - \vec{\jmath}\,a \big)\cdot \ell \Big|\geqslant\rho_0\langle\ell\rangle.$$
		\item Let $\kappa\in\{-,+\}.$ For all $(\ell,j)\in\mathbb{Z}^{d}\times(\mathbb{Z}^*\setminus\overline{\mathbb{S}}_{\kappa})$ with $\vec{\jmath}\cdot\ell+j=0,$
		$$\inf_{a\in[a_0,a_1]}\max_{k\in\llbracket 0,q_0\rrbracket}\Big|\partial_{a}^{k}\big(\omega_{\textnormal{Eq}}(a)\cdot\ell+\kappa\Omega_{j}(a)\big)\Big|\geqslant\rho_0\langle\ell\rangle.$$
		\item Let $\kappa\in\{-,+\}.$ For all $(\ell,j,j')\in\mathbb{Z}^{d}\times(\mathbb{Z}^*\setminus\overline{\mathbb{S}}_{\kappa})^2$ with $(\ell,j,j')\neq(0,j,j)$ and $\vec{\jmath}\cdot\ell+j-j'=0,$
		$$\inf_{a\in[a_0,a_1]}\max_{k\in\llbracket 0,q_0\rrbracket}\Big|\partial_{a}^{k}\big(\omega_{\textnormal{Eq}}(a)\cdot\ell+\kappa\Omega_{j}(a)-\kappa\Omega_{j'}(a)\big)\Big|\geqslant\rho_0\langle\ell\rangle.$$
		\item For all $(\ell,j,j')\in\mathbb{Z}^{d}\times(\mathbb{Z}^*\setminus\overline{\mathbb{S}}_+)\times(\mathbb{Z}^*\setminus\overline{\mathbb{S}}_-)$ with $\vec{\jmath}\cdot\ell+j-j'=0,$
		$$\inf_{a\in[a_0,a_1]}\max_{k\in\llbracket 0,q_0\rrbracket}\Big|\partial_{a}^{k}\big(\omega_{\textnormal{Eq}}(a)\cdot\ell+\Omega_{j}(a)-\Omega_{j'}(a)\big)\Big|\geqslant\rho_0\langle\ell\rangle.$$
	\end{enumerate}
\end{lem}
\begin{proof}
	\textit{(i)} We argue by contradiction. Hence, we assume the existence of sequences $(a_m)_{m\in \N}\in[a_0,a_1]^{\mathbb{N}}$ and $(\ell_m)_{m\in\N}\in(\mathbb{Z}^d\setminus\{0\})^{\mathbb{N}}$ such that
	\begin{equation}\label{seqrus-1}
		\forall \,k\in\mathbb{N},\quad\forall\, m\geqslant k,\quad  \Big|\partial_{a}^{k}\Big(\omega_{\textnormal{Eq}}(a) \cdot\tfrac{\ell_m}{|\ell_m|} \Big)\Big|_{a=a_m}\Big|<\frac{1}{1+m}.
	\end{equation}
	By compactness, using that $\ell_{m}\neq 0$, up to subsequences we have the following convergences
    \begin{equation}\label{CV a l/l blue}	\lim_{m\to\infty}a_m=\overline{a}\in[a_0,a_1],\quad\textnormal{and}\quad\lim_{m\to\infty}\frac{\ell_m}{|\ell_m|}=\overline{\ell}\in\mathbb{R}^*
	\end{equation}
	Therefore, taking $m\to\infty$ in \eqref{seqrus-1}, we get
	$$\forall\, k\in\mathbb{N},\quad\partial_{a}^{k}\omega_{\textnormal{Eq}}(\overline{a})\cdot\overline{\ell}=0.$$
	We deduce that the analytic mapping $a\mapsto\omega_{\textnormal{Eq}}(a)\cdot\overline{\ell}$ is identically zero, which is in contradiction with Lemma \ref{lem non-deg} since $\overline{\ell}\neq 0$.\\
    \textit{(ii)} As in the previous point, we argue by contradiction, assuming the existence of sequences $(a_m)_{m\in\N}\in[a_0,a_1]^{\mathbb{N}}$ and $(\ell_m)_{m\in\N}\in(\mathbb{Z}^d\setminus\{0\})^{\mathbb{N}}$ 
    such that
	\begin{equation}\label{seqrus-2-blue}
		\forall\, k\in\mathbb{N},\quad\forall \,m\geqslant k,\quad
        \Big|\partial_{a}^{k}\Big(\big(\omega_{\textnormal{Eq}}(a) -\vec{\jmath}\,a \big)\cdot\tfrac{\ell_m}{|\ell_m|} \Big)\Big|_{a=a_m}\Big|
        <\frac{1}{1+m}\cdot
	\end{equation}
	Once again, a compactness argument implies, up to subsequences, the same convergences as in \eqref{CV a l/l blue}.
	Passing to the limit $m\to\infty$ in \eqref{seqrus-2-blue}, we obtain
	$$\forall \,k\in\mathbb{N},\quad\partial_{a}^{k}\Big(\big(\omega_{\textnormal{Eq}}(a)-\vec{\jmath}\,a\big)\cdot\overline{\ell}\Big)\Big|_{a=\overline{a}}=0.$$
	Hence, the analytic mapping $a\mapsto\big(\omega_{\textnormal{Eq}}(a) -\vec{\jmath}\,a \big)\cdot\overline{\ell}$ is identically zero, which contradicts Lemma \ref{lem non-deg} since $\overline{\ell}\neq0$.\\
	\textit{(iii)} First, if $\ell=0$, then, by the momentum condition $\vec{\jmath}\cdot \ell + j=0$, we have $j=0$, which is excluded.
    Therefore, we restrict the discussion to the case $\ell\in\mathbb{Z}^d\setminus\{0\}.$ In addition, by triangle inequality and Lemma \ref{lem prop eig VP}-\textit{(iii)}, we get for any $a\in[a_0,a_1],$
	\begin{align*}
		\big|\omega_{\textnormal{Eq}}(a)\cdot\ell+\kappa\Omega_{j}(a)\big|&\geqslant|\Omega_{j}(a)|-\big|\omega_{\textnormal{Eq}}(a)\cdot\ell\big|\\
		&\geqslant a_0 |j|-C|\ell|\\
		&\geqslant|\ell|,
	\end{align*}
	provided that $|j|\geqslant C_1|\ell|$ with $C_1\triangleq \tfrac{C+1}{a_0}.$ Therefore, we shall restrict our discussion to the indices satisfying
    \begin{equation}\label{restr.transv.iii}
    \ell\in\mathbb{Z}^{d}\setminus\{0\},\qquad j\in\mathbb{Z}^*\setminus\overline{\mathbb{S}}_{\kappa},\qquad |j|<C_1|\ell|.
    \end{equation}
	As in the first point, we argue by contradiction, assuming the existence of sequences $(a_m)_{m\in\N}\in[a_0,a_1]^{\mathbb{N}}$, $(\ell_m)_{m\in\N}\in(\mathbb{Z}^d\setminus\{0\})^{\mathbb{N}}$ and $(j_m)_{m\in\N}\in(\mathbb{Z}^*\setminus\overline{\mathbb{S}}_{\kappa})^{\mathbb{N}}$ such that
	\begin{equation}\label{momentum m iii}
	    |j_m|<C_1|\ell_m|,\qquad\vec{\jmath}\cdot \ell_m+j_m=0
	\end{equation}
	and
	\begin{equation}\label{seqrus-3}
		\forall\, k\in\mathbb{N},\quad\forall \,m\geqslant k,\quad\Big|\partial_{a}^{k}\Big(\omega_{\textnormal{Eq}}(a)\cdot\tfrac{\ell_m}{|\ell_m|}+\kappa\tfrac{\Omega_{j_m}(a)}{|\ell_m|}\Big)\Big|_{a=a_m}\Big|<\frac{1}{1+m}\cdot
	\end{equation}
	Once again, using also \eqref{momentum m iii}, a compactness argument implies that, up to subsequences,
	\begin{equation}\label{CV a l/l j/l}
		\lim_{m\to\infty}a_m=\overline{a}\in[a_0,a_1],\qquad\lim_{m\to\infty}\frac{\ell_m}{|\ell_m|}=\overline{\ell}\in\mathbb{R}^*\qquad\textnormal{and}\qquad\lim_{m\to\infty}\frac{j_m}{|\ell_m|}=\overline{d}\in\mathbb{R}.
	\end{equation} We distinguish two cases.\\
	\texttt{Case 1: $(\ell_m)_{m\in\N}$ is bounded.} By \eqref{momentum m iii}, also the sequence $(j_{m})_{m\in\N}$ is bounded. Therefore, by compactness, up to subsequences, we have that both sequences of integers satisfy
	$$\ell_m=\ell_\infty\in\mathbb{Z}^{d}\setminus\{0\},\qquad\textnormal{and}\qquad j_m=j_\infty\in\mathbb{Z}^*\setminus\overline{\mathbb{S}}_{\kappa} \qquad \textnormal{definitely for any} \quad m\gg 1 \quad \textnormal{large enough}.$$
	Hence, taking the limit $m\to\infty$ in \eqref{seqrus-3} and \eqref{momentum m iii}, we obtain
    \begin{equation} \label{momentum lim iii}
        \forall \,k\in\mathbb{N},\quad\partial_{a}^{k}\Big(\omega_{\textnormal{Eq}}(a)\cdot\ell_\infty+\kappa\Omega_{j_\infty}(a)\Big)\Big|_{a=\overline{a}}=0\qquad\textnormal{and}\qquad\vec{\jmath}\cdot\ell_{\infty}+j_{\infty}=0.
    \end{equation}
	Consequently, the analytic mapping $a\mapsto\omega_{\textnormal{Eq}}(a)\cdot\ell_{\infty}+\kappa\Omega_{j_\infty}(a)$ is identically zero, namely
	\begin{equation}\label{identically zero}
		\forall\, a\in[a_0,a_1],\quad\omega_{\textnormal{\tiny{Eq}}}(a)\cdot\ell_{\infty}+\kappa\Omega_{j_\infty}(a)=0.
	\end{equation}
	Now, we distinguish two subcases.
	\begin{enumerate}[label=\textbullet]
		\item \texttt{Subcase 1.1: $j_\infty\not\in\overline{\mathbb{S}}_{-\kappa}.$} In this case $j_{\infty}\in \Z^{*}\setminus\overline{\S}$ and we reach a contradiction with Lemma \ref{lem non-deg} since $(\ell_{\infty},\kappa)\neq 0$.
		\item \texttt{Subcase 1.2: $j_\infty\in\overline{\mathbb{S}}_{-\kappa}.$} There exists $k_0\in\llbracket 1,d_{-\kappa}\rrbracket$ such that $j_{\infty}=\mathtt{sgn}(j_{\infty})j_{k_0}^{-\kappa},$ where $\mathtt{sgn}$ is the sign function. We denote 
		\begin{equation}\label{notation l infty}
			\ell_{\infty}=(\ell_{\infty}^{[+,1]},\ldots,\ell_{\infty}^{[+,d_+]},\ell_{\infty}^{[-,1]},\ldots,\ell_{\infty}^{[-,d_-]}).
		\end{equation}
        Then, recalling \eqref{equilibrium vector freq}, the identity in \eqref{identically zero} reads
        \begin{equation} \label{indentically expanded} 
              \begin{aligned}
        0&= \omega_{\textnormal{\tiny{Eq}}}(a)\cdot\ell_{\infty}+\kappa\Omega_{j_\infty}(a)\\
            & =\sum_{k\in \llbracket1,d_{\kappa}\rrbracket} \ell_{\infty}^{[\kappa,k]}\Omega_{j_{k}^{\kappa}}(a)  + \sum_{k\in \llbracket1,d_{-\kappa}\rrbracket\setminus\{k_0\}} \ell_{\infty}^{[-\kappa,k]}\Omega_{j_{k}^{-\kappa}}(a)  + \big( -\kappa\,\ell_{\infty}^{[-\kappa,k_0]}+\kappa\,\mathtt{sgn}(j_{\infty})\big)\Omega_{j_{k_{0}}^{-\kappa}}(a).
        \end{aligned}
        \end{equation}
		Combining \eqref{identically zero}, \eqref{indentically expanded} and Lemma \ref{lem non-deg}, we obtain
		$$\forall\, k\in\llbracket 1, d_{\kappa}\rrbracket,\quad\ell_{\infty}^{[\kappa,k]}=0\qquad\textnormal{and}\qquad\forall\, k\,\in\llbracket 1,d_{-\kappa}\rrbracket\setminus\{k_0\},\quad\ell_{\infty}^{[-\kappa,k]}=0$$
		and
		$$-\kappa\,\ell_{\infty}^{[-\kappa,k_0]}+\kappa\,\mathtt{sgn}(j_{\infty})=0,\qquad\textnormal{i.e.}\qquad\ell_{\infty}^{[-\kappa,k_0]}=\mathtt{sgn}(j_{\infty}).$$
		Inserting this into the momentum condition \eqref{momentum lim iii}, we get
		$$0=\vec{\jmath}\cdot\ell_{\infty}+j_{\infty}=j_{k_0}^{-\kappa}\ell_{\infty}^{[-\kappa,k_0]}+\mathtt{sgn}(j_{\infty})j_{k_{0}}^{-\kappa}=2j_{\infty}.$$
		This is a contradiction with $j_{\infty}\in\mathbb{Z}^*.$
	\end{enumerate}
	\texttt{Case 2: $(\ell_m)_{m\in\N}$ is unbounded.} Without loss of generality, up to subsequences we assume 
	$$\lim_{m\to\infty}|\ell_m|=\infty.$$
	According to the Lemma \ref{lem prop eig VP}-\textit{(iv)}, we write
	\begin{equation}\label{decompo Omj}
		\Omega_{j}(a)=a j+\mathtt{r}_{j}(a),
	\end{equation}
	with $\mathtt{r}_{j}(a)$ satisfying  $\mathtt{r}_{-j}(a)=-\mathtt{r}_{j}(a)$ and
	\begin{equation}\label{e-ttrj}
		\forall\, k\in\mathbb{N},\quad\sup_{a\in[a_0,a_1]\atop
			j\in\mathbb{Z}^*}|j||\partial_{a}^k\mathtt{r}_{j}(a)|\leqslant C(a_0,a_1,k).
	\end{equation}
	Then, we write
	\begin{equation*}
	    \frac{\Omega_{j_m}(a)}{|\ell_m|}=\frac{a j_m}{|\ell_m|}+\frac{\mathtt{r}_{j_m}(a)}{|\ell_m|}\cdot
	\end{equation*}
	Hence, whatever the fact that $(j_m)_{m\in\N}$ is bounded or not, by \eqref{e-ttrj} we have
	$$\forall \, k\in\mathbb{N},\quad\lim_{m\to\infty}\frac{\partial_{a}^{k}\mathtt{r}_{j_m}(\overline{a})}{|\ell_m|}=0.$$
	Thus, taking the limit $m\to\infty$ in \eqref{seqrus-3} and recalling \eqref{CV a l/l j/l}, we obtain
	$$\forall \,k\in\mathbb{N},\quad\partial_{a}^{k}\Big(\omega_{\textnormal{Eq}}(a)\cdot\overline{\ell}+\overline{d}\kappa\, a\Big)\Big|_{a=\overline{a}}=0.$$
    Consequently, the analytic mapping $a\mapsto\omega_{\textnormal{Eq}}(a)\cdot\overline{\ell}+\overline{d}\kappa \,a$ is identically zero, namely
	\begin{equation*}
		\forall\, a\in[a_0,a_1],\quad\omega_{\textnormal{\tiny{Eq}}}(a)\cdot\overline{\ell}+\overline{d}\kappa \,a = 0,
	\end{equation*}
	which also leads to a contradiction with Lemma \ref{lem non-deg} since $(\overline{\ell},\overline{d})\neq 0$.\\
	\textit{(iv)} 
    First, if $\ell=0$, then, by the momentum condition $\vec{\jmath}\cdot \ell + j-j'=0$, we have $j=j'$, which is excluded. Therefore, we restrict the discussion to $\ell \in \Z^{d}\setminus\{0\}$.
    By triangle inequality and \eqref{decompo Omj}-\eqref{e-ttrj}, we have
	\begin{align*}
		\big|\omega_{\textnormal{Eq}}(a)\cdot\ell+\kappa\Omega_{j}(a)-\kappa\Omega_{j'}(a)\big|&\geqslant a|j-j'|-|\mathtt{r}_{j}(a)|-|\mathtt{r}_{j'}(a)|-\big|\omega_{\textnormal{Eq}}(a)\cdot\ell\big|\\
		&\geqslant a_0|j-j'|-C|\ell|\\
		&\geqslant|\ell|,
	\end{align*}
	provided that $|j-j'|\geqslant C_2|\ell|$ with $C_2\triangleq\tfrac{C+1}{a_0}.$ Therefore, we shall restrict our discussion to the indices satisfying
	$$\ell\in\mathbb{Z}^d\setminus\{0\},\qquad j,j'\in\mathbb{Z}^*\setminus\overline{\mathbb{S}}_{\kappa},\qquad |j-j'|<C_2|\ell|.$$
	As in the first point, we argue by contradiction assuming the existence of sequences $(a_m)_{m\in\N}\in[a_0,a_1]^{\mathbb{N}}$, $(\ell_m)_{m\in\N}\in(\mathbb{Z}^d\setminus\{0\})^{\mathbb{N}}$ and $(j_m)_{m\in\N},(j'_m)_{m\in\N}\in(\mathbb{Z}^*\setminus\overline{\mathbb{S}}_{\kappa})^{\mathbb{N}}$ such that
    \begin{equation}\label{momemtum m iv}
        |j_m-j'_m|<C_2|\ell_m|, \qquad \vec{\jmath}\cdot \ell_{m} + j_{m}-j_{m}'= 0 , \qquad j_m \neq j'_m ,
    \end{equation}
	and
	\begin{equation}\label{seqrus-4}
		\forall\, k\in\mathbb{N},\quad\forall\, m\geqslant k,\quad\Big|\partial_{a}^{k}\Big(\omega_{\textnormal{Eq}}(a)\cdot\tfrac{\ell_m}{|\ell_{m}|}+\kappa\,\tfrac{\Omega_{j_m}(a)-\Omega_{j'_m}(a)}{|\ell_{m}|}\Big)\Big|_{a=a_m}\Big|<\frac{1}{1+m}\cdot
	\end{equation}
	By compactness, up to subsequences, we have the following convergences
	\begin{equation}\label{compact iv}
	    \lim_{m\to\infty}a_m=\overline{a}\in[a_0,a_1],\qquad\lim_{m\to\infty}\frac{\ell_m}{|\ell_m|}=\overline{\ell}\in\mathbb{R}^*\qquad\textnormal{and}\qquad\lim_{m\to\infty}\frac{j_m-j'_m}{|\ell_m|}=\overline{d}\in\mathbb{R}.
	\end{equation}
    We distinguish four cases.\\
	\texttt{Case 1: $(\ell_m)_{m\in\N}$, $(j_m)_{m\in\N}$ and $(j'_m)_{m\in\N}$ are bounded.}  Therefore, by compactness, up to subsequences, we have that these sequences of integers satisfy
	$$\ell_m=\ell_\infty\in\mathbb{Z}^{d}\setminus\{0\},\qquad j_m=j_\infty\in\mathbb{Z}^*\setminus\overline{\mathbb{S}}_{\kappa}\qquad\textnormal{and}\qquad j'_m=j'_\infty\in\mathbb{Z}^*\setminus\overline{\mathbb{S}}_{\kappa}$$
    definitely for any $m\gg 1$ large enough. In particular, by \eqref{momemtum m iv}, we have that
    \begin{equation}\label{jinfty neq jinftyprime}
	 \vec{\jmath}\cdot \ell_{\infty} + j_{\infty} - j'_{\infty} = 0\qquad\textnormal{and}\qquad j_{\infty}\neq j_{\infty}'.
\end{equation}
    Passing to the limit $m\to\infty$ in \eqref{seqrus-4}, we infer
	$$\forall\, k\in\mathbb{N},\quad\partial_{a}^{k}\Big(\omega_{\textnormal{Eq}}(a)\cdot\ell_\infty+\kappa\Omega_{j_\infty}(a)-\kappa\Omega_{j'_\infty}(a)\Big)\Big|_{a=\overline{a}}=0.$$
	Hence, the analytic mapping $a\mapsto\omega_{\textnormal{Eq}}(a)\cdot\ell_\infty+\kappa\Omega_{j_\infty}(a)-\kappa\Omega_{j'_\infty}(a)$ is identically zero, namely
	\begin{equation}\label{identically zero2}
		\forall \,a\in[a_0,a_1],\quad\omega_{\textnormal{\tiny{Eq}}}(a)\cdot\ell_{\infty}+\kappa\Omega_{j_\infty}(a)-\kappa\Omega_{j_{\infty}'}(a)=0.
	\end{equation}
	 We further need to seperate three cases.
	\begin{enumerate}[label=\textbullet]
		\item \texttt{Subcase 1.1: $j_{\infty},j_{\infty}'\in\mathbb{Z}^*\setminus\overline{\mathbb{S}}_{-\kappa}.$} In this case $j_{\infty},j_{\infty}'\in \Z^{*}\setminus\overline{\S}.$ If $|j_{\infty}|\neq|j_{\infty}'|$, then we immediately reach a contradiction with Lemma \ref{lem non-deg}. Otherwise, in view of \eqref{jinfty neq jinftyprime}, we have $j_{\infty}=-j_{\infty}'$ and \eqref{identically zero2} becomes
        $$\forall\,a\in[a_0,a_1],\quad\omega_{\textnormal{\tiny{Eq}}}(a)\cdot \ell_{\infty}+2\kappa\Omega_{j_{\infty}}(a)=0,$$
        which, combined with the fact that $j_{\infty}\in\mathbb{Z}^*\setminus\overline{\mathbb{S}},$ leads also to a contradiction with Lemma \ref{lem non-deg} since $(\ell_{\infty},2\kappa)\neq 0$.
		\item \texttt{Subcase 1.2: $j_{\infty},j_{\infty}'\in\overline{\mathbb{S}}_{-\kappa}.$} There exists $(k_0,k_0')\in\llbracket 1,d_{-\kappa}\rrbracket^2$ such that $j_{\infty}=\mathtt{sgn}(j_{\infty})j_{k_0}^{-\kappa}$ and $j_{\infty}'=\mathtt{sgn}(j_{\infty}')j_{k_0'}^{-\kappa}$. We denote $\ell_{\infty}$ as in \eqref{notation l infty}.  Then, recalling \eqref{equilibrium vector freq}, the identity in \eqref{identically zero2} reads
        \begin{equation} \label{indentically expanded2} 
              \begin{aligned}
        0&= \omega_{\textnormal{\tiny{Eq}}}(a)\cdot\ell_{\infty}+\kappa\Omega_{j_\infty}(a)-\kappa \Omega_{j_{\infty}'}(a)\\
            & =\sum_{k\in \llbracket1,d_{\kappa}\rrbracket} \ell_{\infty}^{[\kappa,k]}\Omega_{j_{k}^{\kappa}}(a)  + \sum_{k\in \llbracket1,d_{-\kappa}\rrbracket\setminus\{k_0,k_0'\}} \ell_{\infty}^{[-\kappa,k]}\Omega_{j_{k}^{-\kappa}}(a)  \\
            & \quad + \big( -\kappa\,\ell_{\infty}^{[-\kappa,k_0]}+\kappa\,\mathtt{sgn}(j_{\infty})\big)\Omega_{j_{k_{0}}^{-\kappa}}(a) +  \big( -\kappa\,\ell_{\infty}^{[-\kappa,k_0']}-\kappa\,\mathtt{sgn}(j_{\infty}')\big)\Omega_{j_{k_{0}'}^{-\kappa}}(a).
        \end{aligned}
        \end{equation}
        Combining \eqref{identically zero2}, \eqref{indentically expanded2} and Lemma \ref{lem non-deg}, we obtain
		$$\forall\, k\in\llbracket 1, d_{\kappa}\rrbracket,\quad\ell_{\infty}^{[\kappa,k]}=0\qquad\textnormal{and}\qquad\forall\, k\in\llbracket 1,d_{-\kappa}\rrbracket\setminus\{k_0,k_0'\},\quad\ell_{\infty}^{[-\kappa,k_0]}=0$$
		as well as
		$$-\kappa\ell_{\infty}^{[-\kappa,k_0]}+\kappa\mathtt{sgn}(j_{\infty})=0\qquad\textnormal{and}\qquad -\kappa\ell_{\infty}^{[-\kappa,k_0']}-\kappa\mathtt{sgn}(j_{\infty}')=0.$$
		i.e.
		$$\ell_{\infty}^{[-\kappa,k_0]}=\mathtt{sgn}(j_{\infty})\qquad\textnormal{and}\qquad \ell_{\infty}^{[-\kappa,k_0']}=-\mathtt{sgn}(j_{\infty}').$$
		
		Inserting this into the momentum condition \eqref{jinfty neq jinftyprime}, we get
		\begin{align*}
			0&=\vec{\jmath}\cdot\ell_{\infty}+j_{\infty}-j_{\infty}'\\
			&=j_{k_0}^{-\kappa}\ell_{\infty}^{[-\kappa,k_0]}+j_{k_0'}^{-\kappa}\ell_{\infty}^{[-\kappa,k_0']}+\mathtt{sgn}(j_{\infty})j_{k_{0}}^{-\kappa}-\mathtt{sgn}(j_{\infty}')j_{k_{0}'}^{-\kappa}\\
			&=2(j_{\infty}-j_{\infty}').
		\end{align*}
		This is a contradiction with the second identity in \eqref{jinfty neq jinftyprime}.
        	\item \texttt{Subcase 1.3: ($j_{\infty}\in\overline{\mathbb{S}}_{-\kappa}$ and $j_{\infty}'\in\mathbb{Z}^*\setminus\overline{\mathbb{S}}_{-\kappa}$) or ($j_{\infty}\in\mathbb{Z}^*\setminus\overline{\mathbb{S}}_{-\kappa}$ and $j_{\infty}'\in\overline{\mathbb{S}}_{-\kappa}$).} This case is a combination of the previous subcases, and its analysis is therefore omitted.
	\end{enumerate}
    \texttt{Case 2: $(\ell_{m})_{m\in\N}$ is bounded, with $(j_m)_{m\in\N}$ bounded and $(j'_m)_{m\in\N}$ unbounded (or viceversa).} This case is actually excluded, as it violates the momentum condition \eqref{momemtum m iv} in the limit $m\to\infty$.\\
	\texttt{Case 3: $(\ell_{m})_{m\in\N}$ is bounded and $(j_m)_{m\in\N}$, $(j'_m)_{m\in\N}$ are unbounded.} Without loss of generality, up to subsequences we have $$\lim_{m\to\infty}j_m=\infty\qquad\textnormal{and}\qquad\lim_{m\to\infty}j'_m=\infty,$$
    Using the splitting \eqref{decompo Omj}, we write
    \begin{equation}\label{split-diff-omg}
        \frac{\Omega_{j_{m}}(a)-\Omega_{j_{m}'}(a)}{|\ell_{m}|} = \frac{a(j_{m}-j_{m}')}{|\ell_{m}|} + \frac{\tr_{j_{m}}(a)}{|\ell_{m}|} - \frac{\tr_{j_{m}'}(a)}{|\ell_{m}|},
    \end{equation}
	According to \eqref{e-ttrj}, we have
	$$\forall\, k\in\mathbb{N},\quad\lim_{m\to\infty}\partial_{a}^{k}\mathtt{r}_{j_m}(\overline{a})=\lim_{m\to\infty}\partial_{a}^{k}\mathtt{r}_{j'_m}(\overline{a})=0.$$
	Thus, taking $m\to\infty$ in \eqref{seqrus-4} and using \eqref{compact iv} imply
	$$\forall\, k\in\mathbb{N},\quad\partial_{a}^{k}\Big(\omega_{\textnormal{Eq}}(a)\cdot\overline{\ell}+\overline{d}\kappa\,a\Big)\Big|_{a=\overline{a}}=0,$$
    Therefore, the analytic function $a\mapsto\omega_{\textnormal{\tiny{Eq}}}(a)\cdot\overline{\ell}+\overline{d}\kappa\,a$ with $(\overline{\ell},\overline{d})\neq0$ vanishes identically, which provides a contradiction as previously.\\
    \texttt{Case 4: $(\ell_m)_{m\in\N}$ is unbounded.} In this case, we have the convergence
	$$\lim_{m\to\infty}|\ell_m|=\infty.$$
    Recall from \eqref{e-ttrj}, that the sequence $(\tr_j(a))_{j\in\mathbb{Z}^*}$ and its derivatives are uniformly bounded. Therefore, using one more time the splitting \eqref{split-diff-omg}, passing to the limit $m\to\infty$ in \eqref{seqrus-4} and using \eqref{compact iv}, we deduce
	$$\forall\, k\in\mathbb{N},\quad\partial_{a}^{k}\Big(\omega_{\textnormal{Eq}}(a)\cdot \overline{\ell}+\overline{d}\kappa\, a\Big)\Big|_{a=\overline{a}}=0,$$
	Therefore, the analytic function $a\mapsto\omega_{\textnormal{\tiny{Eq}}}(a)\cdot\overline{\ell}+\overline{d}\kappa\, a$ with $(\overline{\ell},\overline{d})\neq0$ vanishes identically, which gives a contradiction with Lemma \ref{lem non-deg}.\\
	\textit{(v)} This point is proved similarly to the previous one so we omit it. This concludes the proof of Lemma \ref{Russmann equilibre}.
\end{proof}
We conclude this section with the following lemma stating the existence of traveling quasi-periodic solutions to the linear system \eqref{lin eq QP} obtained exciting only a finite number of Fourier modes.
\begin{lem}\label{lem linear solutions}
	\textbf{(Traveling quasi-periodic solutions to the equilibrium linearized system)}\\
    Let $\S_{+}$, $\S_{-}$ as in \eqref{defin S+S-}, \eqref{inter S+S-}.
	Then, there exists a set $\mathscr{C}_{\textnormal{\tiny{Eq}}}\subset(a_0,a_1)$ of full measure in $(a_0,a_1)$ such that, for any $a\in\mathscr{C}_{\textnormal{\tiny{Eq}}}$, the system \eqref{lin eq QP} admits the following reversible traveling quasi-periodic solutions
    \begin{equation}\label{statement sol lin}
    \begin{pmatrix}
		\rho_+\\
		\rho_{-}
	\end{pmatrix}(t,x)=\sum_{j\in\mathbb{S}_+}A_j\cos\big(\Omega_{j}(a)t-jx\big)\begin{pmatrix}
		1\\
		b_{j}(a)
	\end{pmatrix}+\sum_{j\in\mathbb{S}_-}B_j\cos\big(\Omega_{j}(a)t+jx\big)\begin{pmatrix}
		b_{j}(a)\\
		1
	\end{pmatrix},    
    \end{equation}
	with $(A_j)_{j\in\mathbb{S}_+},(B_j)_{j\in\mathbb{S}_-}\subset\mathbb{R}.$ Such solutions are associated with the frequency vector $\omega_{\textnormal{\tiny{Eq}}}(a)$  as in \eqref{equilibrium vector freq} and velocity vector $\vec{\jmath}$ defined in \eqref{velocity vector}.
\end{lem}
\begin{proof}
The form of the solution \eqref{statement sol lin} has been computed in \eqref{comp sol lin rev} at the end of Section \ref{sect.linear.dyn}. We are left to prove the measure estimate for the non-resonance conditions leading to quasi-periodicity.
	Let us consider $d_+$ and $d_-$ as in \eqref{inter S+S-}. Fix $\tau\triangleq (d_++d_-)q_0$ and define
	$$\mathscr{C}_{\textnormal{\tiny{Eq}}}\triangleq\bigcup_{\gamma\in(0,1)}\mathscr{C}_{\textnormal{\tiny{Eq}}}^{\gamma},\qquad\mathscr{C}_{\textnormal{\tiny{Eq}}}^{\gamma}\triangleq\bigcap_{\ell\in\mathbb{Z}^{d_++d_-}\setminus\{0\}}\Big\{a\in(a_0,a_1)\quad\textnormal{s.t.}\quad|\omega_{\textnormal{\tiny{Eq}}}(a)\cdot\ell|>\frac{\gamma}{\langle\ell\rangle^{\tau}}\Big\}.$$
	Then, using \cite[Lem. 3.6]{HR21} (which is a modified version of the classical Rüssmann Lemma \cite[Thm. 17.1]{R01}) together with Lemma \ref{Russmann equilibre}-\textit{(i)}, we obtain for any $\gamma\in(0,1),$
	\begin{align*}
		\big|(a_0,a_1)\setminus\mathscr{C}_{\textnormal{\tiny{Eq}}}^{\gamma}\big|&\leqslant\sum_{\ell\in\mathbb{Z}^{d_++d_-}\setminus\{0\}}\left|\Big\{a\in(a_0,a_1)\quad\textnormal{s.t.}\quad|\omega_{\textnormal{\tiny{Eq}}}(a)\cdot\ell|\leqslant\frac{\gamma}{\langle\ell\rangle^{\tau}}\Big\}\right|\\
		&\lesssim\gamma^{\frac{1}{q_0}}\sum_{\ell\in\mathbb{Z}^{d_++d_-}\setminus\{0\}}\langle\ell\rangle^{-\frac{\tau+1}{q_0}}\\
		&\lesssim\gamma^{\frac{1}{q_0}}.
	\end{align*}
The convergence of the previous series is due to our choice of $\tau.$ Hence,
$$a_1-a_0-C\gamma^{\frac{1}{q_0}}\leqslant|\mathscr{C}_{\textnormal{\tiny{Eq}}}^{\gamma}|\leqslant|\mathscr{C}_{\textnormal{\tiny{Eq}}}|\leqslant a_1-a_0.$$
Letting $\gamma\to0$, we get that $\mathscr{C}_{\textnormal{\tiny{Eq}}}$ is of full measure in $(a_0,a_1).$ Finally, observe that any choice of $a\in\mathscr{C}_{\textnormal{\tiny{Eq}}}$ implies that $\omega_{\textnormal{\tiny{Eq}}}(a)$ is non-resonant. This achieves the proof of Lemma \ref{lem linear solutions}.
\end{proof}

\section{Reformulation with embedded tori}\label{sect. action-angles}
As we have seen in Lemma \ref{lem linear solutions}, there exist traveling quasi-periodic wave solutions at the linear level. Our aim  now is to construct traveling quasi-periodic waves nearby such linear solutions for the full nonlinear system. Refering to the proof of Lemma \ref{lem linear solutions}, small divisors/resonances appear. Consequently, one must employ a Nash-Moser scheme to construct the desired solutions. The aim of the present section is to introduce a suitable functional adapted to the implementation of the implicit function Nash-Moser theorem in Section \ref{sect.nonlin.sol}. 
We achieve this in two steps: first we transform the Hamiltonian system through a suitable change of unknowns in order to put the equilibrium equation in normal form; then we introduce suitable action-angle-normal symplectic coordinates allowing to reformulate our problem in terms of embedded tori.

\medskip

\noindent {\bf Normal form of the unperturbed linear term.}
Note that the system \eqref{HAM VP} can be seen as a quasilinear perturbation of the equilibrium linear system \eqref{lin eq QP}. Indeed, we can write
\begin{equation}\label{XPE}
    \partial_{t}\begin{pmatrix}
	r_+\\
	r_-
\end{pmatrix} = \frac{1}{\varepsilon^{2}} \cJ \nabla \cE(r) =\mathcal{J}\mathbf{M}_0\begin{pmatrix}
	r_+\\
	r_-
\end{pmatrix}+\varepsilon X_{P_{\mathcal{E}}}(r),\qquad X_{P_{\mathcal{E}}}(r)\triangleq\begin{pmatrix}
-r_+\partial_{x}r_+\\
-r_-\partial_{x}r_-
\end{pmatrix}.
\end{equation}
Let us emphasize that the particular structure of the equations implies that the vector field $X_{P_{\mathcal{E}}}(r)$ is independant of $\varepsilon.$
Our first step is to perfom a symplectic change the unknown so that the new linearized system is in normal form. We define the matrix operator
\begin{equation}\label{bbQ-1}
    \mathbf{Q}\begin{pmatrix}
	\rho_+\\
	\rho_{-}
\end{pmatrix}=\sum_{j\in\mathbb{Z}^*}Q_{j}(a)\begin{pmatrix}
\rho_j^+\\
\rho_{j}^-
\end{pmatrix}\mathbf{e}_{j}, \qquad \textnormal{with inverse} \qquad  \mathbf{Q}^{-1}\begin{pmatrix}
\rho_+\\
\rho_{-}
\end{pmatrix}=\sum_{j\in\mathbb{Z}^*}Q_{j}^{-1}(a)\begin{pmatrix}
\rho_j^+\\
\rho_{j}^-
\end{pmatrix}\mathbf{e}_{j},
\end{equation}
where $Q_j(a)$ and $Q_{j}^{-1}(a)$ are the diagonalizing matrices introduced in \eqref{def:Qj} and \eqref{inverse Qj}. From \eqref{diago Mj}, one has
\begin{equation}\label{bold.Omega}
    \mathbf{Q}^{-1}\mathcal{J}\mathbf{M}_{0}\mathbf{Q}=-\ii\begin{pmatrix}
	\Omega(a,D) & 0\\
	0 & -\Omega(a,D)
\end{pmatrix}\triangleq-\ii\boldsymbol{\Omega}(a,D),\qquad\forall j\in\mathbb{Z}^*,\quad\ii\Omega(a,D)\mathbf{e}_{j}\triangleq\ii\Omega_{j}(a)\mathbf{e}_{j}, 
\end{equation}
with $\Omega(a,D)$ as in Lemma \ref{lem prop eig VP} (see also \eqref{symbol eigenvalues VP}) and $\Omega_{j}(a)$ as in \eqref{eigenvalues VP}.
Let us consider the new unknown $\widetilde{r}\triangleq \mathbf{Q}^{-1}r.$ Then the Hamiltonian system \eqref{HAM VP} becomes
\begin{equation}\label{new:HAMDyn}
    \partial_{t}\widetilde{r}=\tfrac{1}{\varepsilon^2}\mathcal{J}\nabla H(\widetilde{r}),\qquad H(\widetilde{r})\triangleq\mathcal{E}(\mathbf{Q}\widetilde{r}).
\end{equation}
We mention that the linearized equation at the equilibrium state $(\tilde{r}_+,\tilde{r}_-)=(0,0)$ for the new Hamiltonian system is
\begin{equation}\label{new lin system}
	\partial_{t}\begin{pmatrix}
		\rho_+\\
		\rho_-
	\end{pmatrix}=\mathcal{J}\mathbf{L}_0\begin{pmatrix}
		\rho_+\\
		\rho_-
	\end{pmatrix}=\mathcal{J}\nabla H_{\mathbf{L}_0}(\rho_+,\rho_-),
\end{equation}
where
$$\mathbf{L}_{0}\begin{pmatrix}
	\rho_+\\
	\rho_-
\end{pmatrix}\triangleq\mathbf{Q}^{-1}\mathbf{M}_0\mathbf{Q}\begin{pmatrix}
	\rho_+\\
	\rho_-
\end{pmatrix}=\sum_{j\in\mathbb{Z}^*}\frac{\Omega_{j}(a)}{\ii j}\begin{pmatrix}
	\rho_j^+\\
	\rho_j^-
\end{pmatrix}$$
and
\begin{equation}\label{equilibr HAM}
	H_{\mathbf{L}_0}(\rho_+,\rho_-)\triangleq\tfrac{1}{2}\left\langle\mathbf{L}_0\begin{pmatrix}
		\rho_+\\
		\rho_-
	\end{pmatrix}\,,\,\begin{pmatrix}
		\rho_+\\
		\rho_-
	\end{pmatrix}\right\rangle=\sum_{j\in\mathbb{N}^*}\frac{\Omega_{j}(a)}{j}\Big(\big|\rho_{j}^+\big|^2+\big|\rho_j^-\big|^2\Big).
\end{equation}
Note that the Hamiltonian $H_{\mathbf{L}_0}$ is in normal form and, by Lemma \ref{lem linear solutions}, the traveling quasi-periodic reversible solutions to \eqref{new lin system} take the following form
$$\begin{pmatrix}
	\rho_+\\
	\rho_-
\end{pmatrix}(t,x)=\sum_{j\in\mathbb{S}_+}A_j\cos\big(\Omega_{j}(a)-jx\big)\begin{pmatrix}
	1\\
	0
\end{pmatrix}+\sum_{j\in\mathbb{S}_-}B_j\cos\big(\Omega_{j}(a)+jx\big)\begin{pmatrix}
0\\
1
\end{pmatrix},$$
with $\mathbb{S}_+,\mathbb{S}_-$ as in \eqref{defin S+S-}-\eqref{inter S+S-}, for some coefficients $A_j,B_j\in\mathbb{R}.$ 
By \eqref{XPE}, \eqref{bold.Omega}, \eqref{new lin system},  the new Hamiltonian system \eqref{new:HAMDyn} associated with the Hamiltonian $H(\wtr)$ writes
\begin{equation}\label{conj:eq}
    \partial_{t}\begin{pmatrix}
	\widetilde{r}_+\\
	\widetilde{r}_-
\end{pmatrix}=\mathcal{J}\mathbf{L}_{0}\begin{pmatrix}
\widetilde{r}_+\\
\widetilde{r}_-
\end{pmatrix}+\varepsilon X_{P_{H}}(\widetilde{r}),\qquad\mathcal{J}\mathbf{L}_{0}=\mathbf{Q}^{-1}\mathcal{J}\mathbf{M}_{0}\mathbf{Q}=-\ii\boldsymbol{\Omega}(a,D),\qquad X_{P_{H}}(\widetilde{r})=\mathbf{Q}^{-1}X_{P_{\mathcal{E}}}(\mathbf{Q}\widetilde{r}),
\end{equation}
where the  Hamiltonian $H$ reads
\begin{equation}\label{H dec}
	H=H_{\mathbf{L}_0}+\varepsilon P_H, \qquad P_{H}(\wtr)\triangleq P_{\cE}(\bQ \wtr).
\end{equation}

\medskip

\noindent {\bf Action-angle coordinates and embedded tori.}
Next, we split the phase space into tangential and normal subspaces. We recall that the symmetrized tangential sets are introduced in \eqref{symmetrized tangential}. Then, we have the following decomposition of the space $\mathbf{L}_0^2(\mathbb{T})$ in \eqref{def:bbL20}
\begin{equation}\label{split phsp}
	\mathbf{L}_{0}^{2}(\mathbb{T})=\mathbf{H}_{\overline{\mathbb{S}}}\overset{\perp}{\oplus}\mathbf{H}_{\overline{\mathbb{S}}_0}^{\perp}
\end{equation}
where $\bH_{\overline{\S}}$ is the finite dimensional \textit{tangential subspace}
\begin{equation}
    \mathbf{H}_{\overline{\mathbb{S}}}\triangleq\left\lbrace v=\sum_{j\in\overline{\mathbb{S}}_+}v_{j}^{+}\begin{pmatrix}
	1\\
	0
\end{pmatrix}\mathbf{e}_{j}+\sum_{j\in\overline{\mathbb{S}}_-}v_{j}^{-}\begin{pmatrix}
0\\
1
\end{pmatrix}\mathbf{e}_{j},\quad v_{j}^{\pm}\in\mathbb{C},\quad\overline{v_{-j}^{\pm}}=v_{j}^{\pm}\right\rbrace , \label{H.tang.space}
\end{equation}
and $\mathbf{H}_{\overline{\mathbb{S}}_0}^{\perp}$ is the \textit{normal subspace}
\begin{align}
    \mathbf{H}_{\overline{\mathbb{S}}_0}^{\perp}\triangleq\left\lbrace z=\sum_{j\in\mathbb{Z}^*\setminus\overline{\mathbb{S}}_+}z_{j}^{+}\begin{pmatrix}
	1\\
	0
\end{pmatrix}\mathbf{e}_{j}+\sum_{j\in\mathbb{Z}^*\setminus\overline{\mathbb{S}}_-}z_{j}^{-}\begin{pmatrix}
0\\
1
\end{pmatrix}\mathbf{e}_{j},\quad z_{j}^{\pm}\in\mathbb{C},\quad\overline{z_{-j}^{\pm}}=z_{j}^{\pm}\right\rbrace . \label{H.normal.space}
\end{align}
The subspace $\mathbf{H}_{\overline{\mathbb{S}}_0}^{\perp}$ is the $\mathbf{L}_{0}^{2}$-orthogonal of the subspace $\mathbf{H}_{\overline{\mathbb{S}}}$, and any $r\in\mathbf{L}_{0}^2$ can be decomposed into $r=v+z$ with $(v,z)\in\mathbf{H}_{\overline{\mathbb{S}}}\times\mathbf{H}_{\overline{\mathbb{S}}_0}^{\perp}.$ The tangential and normal variables $v$ and $z$ are obtained by considering the orthogonal pojections $\Pi_{\overline{\mathbb{S}}}$ and $\Pi_{\overline{\mathbb{S}}_0}^{\perp}$ associated with the decomposition \eqref{split phsp}, explicitly given by
$$\Pi_{\overline{\mathbb{S}}}=\begin{pmatrix}
	\Pi_{+} & 0\\
	0 & \Pi_{-}
\end{pmatrix},\qquad\Pi_{\overline{\mathbb{S}}_0}^{\perp}=\begin{pmatrix}
\Pi_{+}^{\perp} & 0\\
0 & \Pi_{-}^{\perp}
\end{pmatrix},\qquad\Pi_{\pm}r\triangleq\sum_{j\in\mathbb{S}_{\pm}}r_{j}\mathbf{e}_j\qquad\Pi_{\pm}^{\perp}\triangleq\textnormal{Id}-\Pi_{\pm}.$$
On the tangential subspace $\mathbf{H}_{\overline{\mathbb{S}}}$ we introduce the action-angle variables $(I,\vartheta)\in\mathbb{R}^d\times\mathbb{T}^d$ as follows : fix some amplitudes $(\mathfrak{a}_{j}^{\pm})_{j\in\overline{\mathbb{S}}_{\pm}}\in\mathbb{R}^{d_{\pm}}$ with
$$\forall j\in\overline{\mathbb{S}}_{\pm},\quad\mathfrak{a}_{-j}^{\pm}=\mathfrak{a}_{j}^{\pm}.$$
Then, write the Fourier coefficients of any function $v\in\mathbf{H}_{\overline{\mathbb{S}}}$ in symplectic polar coordinates
$$\forall j\in\overline{\mathbb{S}}_{\pm},\quad v_{j}^{\pm}\triangleq\sqrt{(\mathfrak{a}_{j}^{\pm})^{2}+\tfrac{|j|}{2}I_{j}^{\pm}}\,e^{\ii\vartheta_{j}^{\pm}}, \qquad  \textnormal{with} \qquad I_{-j}^{\pm}=I_j^{\pm},\qquad\vartheta_{-j}^{\pm}=-\vartheta_{j}^{\pm}.  $$
The \textit{actions} and \textit{angles variables} are then given as the vectors
$$I\triangleq\Big(\big(I_j^+\big)_{j\in\mathbb{S}_+},\big(I_j^-\big)_{j\in\mathbb{S}_-}\Big)\in\mathbb{R}^d,\qquad\vartheta\triangleq\Big(\big(\vartheta_j^+\big)_{j\in\mathbb{S}_+},\big(\vartheta_j^-\big)_{j\in\mathbb{S}_-}\Big)\in\mathbb{T}^d.$$
The above definitions allow to consider the following transformation
\begin{equation}\label{transfo aan-state}
	\mathbf{A}:\begin{array}[t]{rcl}
		\mathbb{T}^d\times\mathbb{R}^d\times\mathbf{H}_{\overline{\mathbb{S}}_0}^{\perp} & \rightarrow & \mathbf{L}_{0}^2(\mathbb{T})\\
		(\vartheta,I,z) & \mapsto & v(\vartheta,I)+z,
	\end{array}
\end{equation}
where
$$v(\vartheta,I)\triangleq\sum_{j\in\overline{\mathbb{S}}_{+}}\begin{pmatrix}
\sqrt{(\mathfrak{a}_{j}^{+})^{2}+\tfrac{|j|}{2}I_{j}^+}\,e^{\ii\vartheta_{j}^+}\\
0
\end{pmatrix}\mathbf{e}_{j}+\sum_{j\in\overline{\mathbb{S}}_-}\begin{pmatrix}
0\\
\sqrt{(\mathfrak{a}_{j}^{-})^{2}+\tfrac{|j|}{2}I_{j}^-}\,e^{\ii\vartheta_{j}^-}
\end{pmatrix}\mathbf{e}_{j}.$$
Straightforward calculations lead to
$$\forall\, j\in\overline{\mathbb{S}}_{\pm},\quad\begin{cases}
	\di v_{j}^{\pm}=\frac{|j|e^{\ii\vartheta_j^{\pm}}}{4\sqrt{(\mathfrak{a}_j^{\pm})^2+\frac{|j|}{2}I_j^{\pm}}}\,\di I_j^{\pm}+\ii e^{\ii\vartheta_{j}^{\pm}}\sqrt{(\mathfrak{a}_j^{\pm})^2+\tfrac{|j|}{2}I_j^{\pm}}\,\di\vartheta_{j}^{\pm},\vspace{0.2cm}\\
	\di v_{-j}^{\pm}=\frac{|j|e^{-\ii\vartheta_j^{\pm}}}{4\sqrt{(\mathfrak{a}_j^{\pm})^2+\frac{|j|}{2}I_j^{\pm}}}\,\di I_j^{\pm}-\ii e^{-\ii\vartheta_{j}^{\pm}}\sqrt{(\mathfrak{a}_j^{\pm})^2+\tfrac{|j|}{2}I_j^{\pm}}\,\di \vartheta_{j}^{\pm}.
\end{cases}$$
Therefore,
$$\forall \, j\in\overline{\mathbb{S}},\quad \di v_{j}^{\pm}\wedge \di v_{-j}^{\pm}=\tfrac{\ii|j|}{2} \di \vartheta_j^{\pm}\wedge \di I_j^{\pm}.$$
As a consequence, the symplectic 2-form in \eqref{symplectic.form} reads
\begin{align*}
	\mathscr{W}&=\sum_{j\in\overline{\mathbb{S}}_-}\tfrac{1}{\ii j}\di v_{j}^{-}\wedge \di v_{-j}^{-}-\sum_{j\in\overline{\mathbb{S}}_+}\tfrac{1}{\ii j}\di v_{j}^{+}\wedge \di v_{-j}^{+}+\mathscr{W}|_{\mathbf{H}_{\overline{\mathbb{S}}_0}^{\perp}}\\
	&=\sum_{j\in\mathbb{S}_-}\di \vartheta_{j}^-\wedge \di I_j^--\sum_{j\in\mathbb{S}_+}\di \vartheta_{j}^+\wedge \di I_j^++\mathscr{W}|_{\mathbf{H}_{\overline{\mathbb{S}}_0}^{\perp}}.
\end{align*}
The Poisson structure associated with $\mathscr{W}$ is defined through
\begin{equation}
    \{F,G\}\triangleq\mathscr{W}(X_{F},X_G)=\langle\nabla F,\mathbf{J}\nabla G\rangle, \qquad \textnormal{where} \qquad \langle(\vartheta,I,z),(\vartheta',I',z')\rangle\triangleq\vartheta\cdot\vartheta'+I\cdot I'+\langle z,z'\rangle_{\mathbf{L}_0^2(\mathbb{T})},
\end{equation}
and the unbounded transformation $\mathbf{J}$ is given by
\begin{equation}\label{poisson.aan}
    \mathbf{J}:(\vartheta,I,z)\mapsto\big(\mathtt{J}I,-\mathtt{J}\vartheta,\mathcal{J}z\big),\qquad\mathtt{J}\triangleq\begin{pmatrix}
	-\mathtt{I}_{d_+} & 0\\
	0 & \mathtt{I}_{d_-}
\end{pmatrix},
\end{equation}
with $\mathtt{I}_{d_{\pm}}$ denoting the identity matrix of size $d_{\pm}.$ Any application
\begin{equation}\label{emd.torus.coord}
    \begin{aligned}
        \mathbb{T}^{d} & \rightarrow & \mathbb{T}^d\times\mathbb{R}^d\times\mathbf{H}_{\overline{\mathbb{S}}_0}^{\perp}\\
	\varphi & \mapsto & \big(\vartheta(\varphi),I(\varphi),z(\varphi,\cdot)\big)
    \end{aligned}
\end{equation}
is called an \textit{embedded torus} and can be identified with a function $\varphi\in\mathbb{T}^d\mapsto r(\varphi,\cdot)\in \mathbf{L}_{0}^{2}$ through the transformation $\mathbf{A}.$ Let us emphasize that in this new unknowns, each trivial embedding $\varphi\mapsto(\varphi,0,0)$ corresponds to one of the equilibrium solutions in Lemma \ref{lem linear solutions}. In the new variables, according to \eqref{H dec} and \eqref{equilibr HAM}, we obtain a new Hamiltonian system associated with
\begin{equation}\label{Ham in aan}
    H_{\varepsilon}\triangleq H\circ\mathbf{A}=-\big(\mathtt{J}\omega_{\textnormal{Eq}}(a)\big)\cdot I+\tfrac{1}{2}\big\langle\mathbf{L}_0 z,z\big\rangle_{\mathbf{L}_0^{2}(\mathbb{T})}+\varepsilon\mathcal{P},\qquad\mathcal{P}\triangleq P_{H}\circ\mathbf{A}.
\end{equation}
This corresponds to the vector field
$$X_{H_{\varepsilon}}\triangleq\big(\mathtt{J}\partial_{I}H_{\varepsilon},-\mathtt{J}\partial_{\vartheta}H_{\varepsilon},\Pi_{\overline{\mathbb{S}}_0}^{\perp}\nabla_zH_{\varepsilon}\big).$$
Actually, in the sequel, it is more convenient to replace the frequency vector by $\alpha\in\mathbb{R}^d$ and to consider the following modified Hamiltonian
\begin{equation}\label{Ham.aan.mod}
    H_{\varepsilon}^{\alpha}\triangleq\alpha\cdot I+\tfrac{1}{2}\big\langle\mathbf{L}_{0}z,z\big\rangle_{\mathbf{L}_0^{2}(\mathbb{T})}+\varepsilon\mathcal{P},
\end{equation}
and the nonlinear operator
\begin{align}
    \mathscr{F}_{\varepsilon}(i,\alpha,a,\omega)\triangleq\omega\cdot \partial_{\varphi}i(\varphi)-X_{H_{\varepsilon}^{\alpha}}\big(i(\varphi)\big)=\begin{pmatrix}
	\omega\cdot\partial_{\varphi}\vartheta(\varphi)-\alpha-\varepsilon\mathtt{J}\partial_I\mathcal{P}\big(i(\varphi)\big)\\
	\omega\cdot\partial_{\varphi}I(\varphi)+\varepsilon\mathtt{J}\partial_{\vartheta}\mathcal{P}\big(i(\varphi)\big)\\
	\omega\cdot\partial_{\varphi}z(\varphi)-\mathcal{J}\mathbf{L}_0z(\varphi)-\varepsilon\Pi_{\overline{\mathbb{S}}_0}^{\perp}\mathcal{J}\nabla_{z}\mathcal{P}\big(i(\varphi)\big)
\end{pmatrix}. \label{sF.vare.def}
\end{align}
If $ \mathscr{F}_{\varepsilon}(i,\alpha)=0$, then the embedding $\vf\mapsto i(\vf)$ is an invariant torus for the Hamiltonian vector field $X_{H_{\varepsilon}^\alpha}$ filled with quasi-periodic solutions with frequency $\omega$.
Observe that
\begin{equation}\label{itriv}
    \mathscr{F}_{0}\big(i_{\textnormal{\tiny{triv}}},-\mathtt{J}\omega_{\textnormal{Eq}}(a),a,\omega_{\textnormal{Eq}}(a)\big)=0,\qquad i_{\textnormal{\tiny{triv}}}(\varphi)\triangleq(\varphi,0,0),
\end{equation}
which is just a reformulation of Lemma \ref{lem linear solutions}. Our purpose now becomes to find invariant tori $i$ near the trivial solution $i_{\textnormal{\tiny{triv}}}$. Therefore, we define the following quantities
\begin{equation}\label{varia.Theta}
    \Theta(\varphi)\triangleq\vartheta(\varphi)-\varphi\qquad\textnormal{and}\qquad\mathfrak{I}\triangleq(\Theta,I,z)=i-i_{\textnormal{\tiny{triv}}}.
\end{equation}
In Section \ref{sect.nonlin.sol}, we shall apply the Nash-Moser implicit function theorem  in order to find a non-trivial zeros
$\big(i(a,\omega),\alpha(a,\omega)\big)=\big(i_{\textnormal{\tiny{triv}}} +\fI(a,\omega),\alpha(a,\omega)\big)$ of the functional $\mathscr{F}_{\varepsilon}$ in \eqref{sF.vare.def}.\\

\medskip

\noindent {\bf Reversible and traveling tori.}
We are looking for $r$ a reversible traveling wave, namely satisfying $r(\vf,x)=r(-\vf,-x)$ and $r(\varphi-\vec{\jmath}y,x)=r(\varphi,x+y)$ (see \eqref{traveling.w.app}). In the language of embedded tori, we have the following definition.
\begin{defin}\label{rev trav torus}
We define
\begin{equation}\label{inv trav aan}
    \mathfrak{S}:(\vartheta,I,z)\mapsto(-\vartheta,I,\mathscr{S}z),\qquad\mathfrak{T}_{y}:(\vartheta,I,z)\mapsto(\vartheta-\vec{\jmath}y,I,\mathscr{T}_yz),
\end{equation}
with $\mathscr{S}$ and $\mathscr{T}_{y}$ the involution and translations defined in \eqref{inv trans}.
    A torus $i=(\vartheta,I,z):\mathbb{T}^{d}\to \T^d\times \R^d\times\bH_{\overline{S}_{0}}^\perp$ is called:
    \begin{enumerate}[label=\textbullet]
        \item \textit{reversible} if, for any $\vf\in\T^d$, $\mathfrak{S}i(\vf)=i(-\vf)$;
        \item \textit{anti-reversible} if, for any $\vf\in\T^d$, $-\mathfrak{S}i(\vf)=i(-\vf)$;
        \item \textit{traveling} if, for any $\vf\in\T^d$ and $y\in\R$, $\mathfrak{T}_{y}(i)(\varphi)=i(\varphi-\vec{\jmath}y)$.
    \end{enumerate} 
\end{defin}

By \eqref{Ham.aan.mod}, \eqref{Ham in aan}, \eqref{new:HAMDyn}, \eqref{rev.inv.cE}, \eqref{trans.inv.cE} and Definition \ref{rev trav torus}, we note that the Hamiltonians $H_{\varepsilon}^{\alpha}$ are invariant under the involution $\mathfrak{S}$ and the translations $\mathfrak{T}_{y}$ (for any $y\in\mathbb{T}$), namely
$$H_{\varepsilon}^{\alpha}\circ\mathfrak{S}=H_{\varepsilon}^{\alpha}=H_{\varepsilon}^{\alpha}\circ\mathfrak{T}_{y}.$$
As we look for $\vartheta(\vf)$ of the form $\vartheta(\vf)=\vf+\Theta(\vf)$ with a $(2\pi)^d$-periodic function $\Theta:\R^d\to\R^d$, $\vf\mapsto \Theta(\vf)$ (see \eqref{varia.Theta}), in order to justify the momentum preserving property of the constructed approximate inverse in Section \ref{sect.approx.inverse}, we consider the \textit{traveling variation operator}
$$\forall\, y\in\mathbb{T},\quad\mathfrak{DT}_y\begin{pmatrix}
    \Theta\\
    I\\
    z
\end{pmatrix}\triangleq\begin{pmatrix}
    \Theta\\
    I\\
    \mathscr{T}_yz
\end{pmatrix}.$$
\begin{defin}\label{defi:travivaria}
    We say that an element $\vf\mapsto g(\vf)\in\mathbb{R}^d\times\mathbb{R}^d\times\mathbf{H}_{\overline{\mathbb{S}}_0}^{\perp}$ is a \textit{traveling wave variation} if and only if
$$\forall \, y\in\mathbb{T},\quad\forall\,\varphi\in\mathbb{T}^d,\quad\mathfrak{DT}_yg(\varphi)=g(\varphi-\vec{\jmath}y).$$
\end{defin}
Recall the following result from \cite[Lem. 3.34]{BFM21}.
\begin{lem}\label{lem:MP-torus}
We have the following:
    \begin{enumerate}[label=(\roman*)]
        \item According to the Definition \ref{defin:OP-sym}, a linear operator $A(\varphi)$ on $\mathbb{R}^d\times\mathbb{R}^d\times\mathbf{H}_{\overline{\mathbb{S}}_0}^{\perp}$ is momentum preserving if and only if
    $$\forall \, y\in\mathbb{T},\quad\forall\,\varphi\in\mathbb{T}^d,\quad A(\varphi-\vec{\jmath}y)\circ\mathfrak{DT}_y=\mathfrak{DT}_y\circ A(\varphi).$$
    \item If $g$ is a traveling wave variation and $A(\varphi)$ is a momentum preserving operator, then $\varphi\mapsto A(\varphi)g(\varphi)$ is a traveling wave variation.
    \end{enumerate}
\end{lem}

\section{Tame estimates and approximate right inverse}\label{sect.approx.inverse}
In this section, we provide the material to run the Nash-Moser analysis in Section \ref{sect.nonlin.sol} in order to solve (recall \eqref{sF.vare.def}, \eqref{Ham.aan.mod})
\begin{equation}
     \mathscr{F}_{\varepsilon}(i,\alpha,a,\omega)=\omega\cdot \partial_{\varphi}i-X_{H_{\varepsilon}^{\alpha}}(i) = 0 .
\end{equation}
Here start topological considerations and we refer the reader to the functional framework introduced in the Appendix \ref{appendix-gene}. We define $\mathcal{O}$ the initial set of parameters -- where $(a,\omega)$ are taken -- as follows
\begin{equation}\label{cO.ref.set}
    \cO \triangleq (a_0,a_1) \times B_{R_0}(0), \qquad 0<a_0 <a_1,
\end{equation}
where $B_{R_0}(0)$ is the open ball in $\R^d$ such that, for some $R_0>0$,
\begin{equation}\label{omega.Eq.inball}
    \omega_{\rm Eq} ([a_0,a_1]) \subset B_{R_0/2}(0) .
\end{equation}
We also fix $(d,\gamma,q,s_0,S)$ as in \eqref{d gm S s0}. Let us insist on the fact that one needs a $C^q$-regularity with respect to the parameters $(a,\omega)$ in order to apply the transversality analysis of Section \ref{sec transversal} required later for estimating the measure of the final Cantor set in Section \ref{subsect.conclusion}. In what follows, we assume that
\begin{equation}\label{ansatz.fI.small}
\varepsilon\leqslant\varepsilon_0\qquad\textnormal{and}\qquad\|\mathfrak{I}\|_{s_0+\sigma_{\rm max}}^{q,\gamma,\mathcal{O}}\leqslant1,
\end{equation}
where $\sigma_{\rm max}>0$ is a fixed loss of derivatives depending on $d,q$, and where
\begin{equation}
\|\mathfrak{I}\|_{s}^{q,\gamma,\mathcal{O}}\triangleq\|\Theta\|_{s}^{q,\gamma,\mathcal{O}}+\|I\|_{s}^{q,\gamma,\mathcal{O}}+\|z\|_{s}^{q,\gamma,\mathcal{O}}    
\end{equation}
is the norm associated with the periodic component $\mathfrak{I}$ as in \eqref{varia.Theta}.

\subsection{Tame estimates for the functional}
To run the Nash-Moser scheme, it is essential to get tame estimates on the functional, its linearized operator and second order differential in a neighbourhood of the trivial solution and its approximate right inverse. Tame estimates are linear in high regularity norm with potential loss of derivatives. They are the key point to close the induction argument. Here we provide the tame estimates for these quantities except the approximate inverse, which needs a special analysis performed in the next sections.

\begin{lem}\label{lemma.Xp.est}
    Let $\fI$ in \eqref{varia.Theta} satisfy \eqref{ansatz.fI.small}. Then, for any $s\geqslant s_0$ and for all $\whi\triangleq(\wh\vartheta,\whI,\whw)$, it holds
    \begin{align}
    \|X_{\mathcal{P}}\|_{s}^{q,\gamma,\mathcal{O}}&\lesssim_{s,q,d}1+\|\mathfrak{I}\|_{s+1}^{q,\gamma,\mathcal{O}},\\
    \|\di_iX_{\mathcal{P}}[\hat{i}]\|_{s}^{q,\gamma,\mathcal{O}}&\lesssim_{s,q,d}\|\hat{i}\|_{s+1}^{q,\gamma,\mathcal{O}}+\|\mathfrak{I}\|_{s+1}^{q,\gamma,\mathcal{O}}\|\hat{i}\|_{s_0+1}^{q,\gamma,\mathcal{O}},\\
    \|\di_i^2X_{\mathcal{P}}[\hat{i},\hat{i}]\|_{s}^{q,\gamma,\mathcal{O}}&\lesssim_{s,q,d}\|\hat{i}\|_{s+1}^{q,\gamma,\mathcal{O}}\|\hat{i}\|_{s_0+1}^{q,\gamma,\mathcal{O}}+\|\mathfrak{I}\|_{s+1}^{q,\gamma,\mathcal{O}}\left(\|\hat{i}\|_{s_0+1}^{q,\gamma,\mathcal{O}}\right)^2.
\end{align}
\end{lem}
\begin{proof}
    From the explicit expression of the vector field $X_{P_{\mathcal{E}}}$ in \eqref{XPE}, one readily obtains
$$\di_{r}X_{P_{\mathcal{E}}}(r)[\rho]=\begin{pmatrix}
	-\rho_+\partial_{x}r_+-r_+\partial_{x}\rho_+\\
	-\rho_-\partial_{x}r_--r_-\partial_{x}\rho_-
\end{pmatrix},\qquad \di_{r}^{2}X_{P_{\mathcal{E}}}(r)[\rho_{1},\rho_2]=\begin{pmatrix}
-\rho_+^{1}\partial_{x}\rho_+^2-\rho_+^2\partial_{x}\rho_+^1\\
-\rho_-^1\partial_{x}\rho_-^2-\rho_-^2\partial_{x}\rho_-^1
\end{pmatrix}.$$
As a consequence, by Lemma \ref{lem functions}-$(i)$, we have the following estimates for any $s\geqslant s_0$
\begin{align}
	\|X_{P_{\mathcal{E}}}(r)\|_{s}^{q,\gamma,\mathcal{O}}&\lesssim_{s,q,d}\|r\|_{s_0+1}^{q,\gamma,\mathcal{O}}\|r\|_{s+1}^{q,\gamma,\mathcal{O}},\label{Taylor0 XPE}\\
	\|\di_rX_{P_{\mathcal{E}}}(r)[\rho]\|_{s}^{q,\gamma,\mathcal{O}}&\lesssim_{s,q,d}\|r\|_{s_0+1}^{q,\gamma,\mathcal{O}}\|\rho\|_{s+1}^{q,\gamma,\mathcal{O}}+\|r\|_{s+1}^{q,\gamma,\mathcal{O}}\|\rho\|_{s_0+1}^{q,\gamma,\mathcal{O}},\label{Taylor1 XPE}\\
	\|\di_r^{2}X_{P_{\mathcal{E}}}(r)[\rho_1,\rho_2]\|_{s}^{q,\gamma,\mathcal{O}}&\lesssim_{s,q,d}\|\rho_1\|_{s_0+1}^{q,\gamma,\mathcal{O}}\|\rho_2\|_{s+1}^{q,\gamma,\mathcal{O}}+\|\rho_1\|_{s+1}^{q,\gamma,\mathcal{O}}\|\rho_2\|_{s_0+1}^{q,\gamma,\mathcal{O}}.\label{Taylor2 XPE}
\end{align}
Combining \eqref{Taylor0 XPE}, \eqref{Taylor1 XPE} and \eqref{Taylor2 XPE} together with the continuity of $\mathbf{Q}^{\pm1}$ on $H^{s}(\mathbb{T})\times H^s(\mathbb{T}),$ we get for any $s\geqslant s_0,$
\begin{align}
	\|X_{P_{H}}(\widetilde{r})\|_{s}^{q,\gamma,\mathcal{O}}&\lesssim_{s,q,d}\|\widetilde{r}\|_{s_0+1}^{q,\gamma,\mathcal{O}}\|\widetilde{r}\|_{s+1}^{q,\gamma,\mathcal{O}},\label{Taylor0 XPH}\\
	\|\di_{\widetilde{r}}X_{P_{H}}(\widetilde{r})[\rho]\|_{s}^{q,\gamma,\mathcal{O}}&\lesssim_{s,q,d}\|\widetilde{r}\|_{s_0+1}^{q,\gamma,\mathcal{O}}\|\rho\|_{s+1}^{q,\gamma,\mathcal{O}}+\|\widetilde{r}\|_{s+1}^{q,\gamma,\mathcal{O}}\|\rho\|_{s_0+1}^{q,\gamma,\mathcal{O}},\label{Taylor1 XPH}\\
	\|\di_{\widetilde{r}}^{2}X_{P_{H}}(\widetilde{r})[\rho_1,\rho_2]\|_{s}^{q,\gamma,\mathcal{O}}&\lesssim_{s,q,d}\|\rho_1\|_{s_0+1}^{q,\gamma,\mathcal{O}}\|\rho_2\|_{s+1}^{q,\gamma,\mathcal{O}}+\|\rho_1\|_{s+1}^{q,\gamma,\mathcal{O}}\|\rho_2\|_{s_0+1}^{q,\gamma,\mathcal{O}}.\label{Taylor2 XPH}
\end{align}
Besides, observe that
$$\forall\,\alpha,\beta\in\mathbb{N}^d,\quad\|\partial_{\vartheta}^{\alpha}\partial_{I}^{\beta}v(\vartheta,I)\|_{s}^{q,\gamma,\mathcal{O}}\lesssim 1+\|\mathfrak{I}\|_{s}^{q,\gamma,\mathcal{O}}.$$
Then, in view of the definition of the transformation $\mathbf{A}$ in \eqref{transfo aan-state} and the definition of $\mathcal{P}$ in \eqref{Ham in aan}, we get from \eqref{Taylor0 XPH}, \eqref{Taylor1 XPH} and \eqref{Taylor2 XPH} the desired estimates.
\end{proof}

\subsection{Approximate right inverse}\label{sect:BB}
	We shall construct an approximate right inverse of the corresponding linearized operator at any given reversible and quasi-periodic traveling state $(i_0,\alpha_0)$ near the flat torus 
	\begin{equation}\label{Linearized-op-F-DC}
		\di_{(i,\alpha)}\mathscr{F}_{\varepsilon}(i_0,\alpha_0)=\omega\cdot\partial_\varphi i_0  - \di_i X_{H_\varepsilon^{\alpha_0}} ( i_0 )-\Big(\begin{smallmatrix}
			\mathtt{J}\widehat\alpha  \\
			0
			\\
			0  
		\end{smallmatrix}\Big).
	\end{equation}
	To achieve this goal, we make appeal to the Berti-Bolle approach \cite{BB15} where they approximately decouple the linearized equations through a triangular system in the action-angle-normal variables. Here, we shall implement a slightly simplified version of this strategy adapted to the vectorial case and developed in \cite[Section 6]{HHR23}; we work directly with the original torus $i_0=(\vartheta_0,I_0,z_0)$ instead of introducing an intermediate isotropic one as in \cite{BB15}. In addition, in our case, we shall also follow the traveling property preservation in the spirit of \cite{BFM21,BFM21-1} in order to guarantee the construction of a traveling quasi-periodic solution at the end of the scheme.\\
	 
	Let us consider $G_0$ the diffeomorphism $G_0:\mathbb{T}^d\times\mathbb{R}^d\times\mathbf{H}_{\overline{\mathbb{S}}_0}^{\perp}\rightarrow\mathbb{T}^d\times\mathbb{R}^d\times\mathbf{H}_{\overline{\mathbb{S}}_0}^{\perp}$ defined by
	\begin{equation}\label{defi:G0}
		\begin{pmatrix}
			\vartheta\\
			I\\
			z
		\end{pmatrix}\triangleq G_0\begin{pmatrix}
			\phi \\
			\mathtt{y} \\
			w
		\end{pmatrix} \triangleq 
		\begin{pmatrix}
			\vartheta_0(\phi)\\
			I_0 (\phi) + \mathtt{L}_1(\phi)\mathtt{y} + \mathtt{L}_2(\phi) w\\
			z_0(\phi) + w
		\end{pmatrix}, 
	\end{equation}
	where
	\begin{equation*}
		\mathtt{L}_1(\phi)\triangleq  \mathtt{J}[\partial_\varphi \vartheta_0(\phi)]^{-\top},\qquad \mathtt{L}_2(\phi) \triangleq \mathtt{J}[(\partial_\vartheta \widetilde{z}_0)(\vartheta_0(\phi))]^\top \mathcal{J}^{-1},\qquad \widetilde{z}_0 (\vartheta) \triangleq z_0 \big(\vartheta_0^{-1} (\vartheta)\big).
	\end{equation*}
	The tensors  $\mathcal{J}$ and $\tJ$ have been introduced in \eqref{HAM VP} and \eqref{poisson.aan}, respectively. In the new coordinates, the torus $i_0$ is trivial (see \eqref{itriv}), that is
	$$G_0\big(i_{\textnormal{\tiny{triv}}}(\varphi)\big)=G_0(\varphi,0,0)=i_0(\varphi).$$
    Straightforward calculations similar to \cite[Lem. 6.3]{BFM21} show that the transformation $G_0$ is invariant under the transformations $\mathfrak{S}$ and $\mathfrak{T}_y$ as in \eqref{rev trav torus}, namely
    \begin{equation}\label{G0-revMP}
        G_0\circ\mathfrak{S}=\mathfrak{S}\circ G_0\qquad\textnormal{and}\qquad\forall \, y\in\mathbb{T},\quad G_0\circ\mathfrak{T}_y=\mathfrak{T}_y\circ G_0.
    \end{equation}
    In what follows, we denote 
        \begin{equation}\label{tilde.G0}
            \widetilde{G}_0 (\phi,y,w,\alpha)\triangleq\big(G_0(\phi,y,w),\alpha\big) 
        \end{equation}
		the diffeomorphism acting as the identity on the $\alpha$-component. We also quantify how an embedded torus $i_0(\mathbb{T})$ is approximately invariant for the Hamiltonian vector field $X_{H_\varepsilon^{\alpha_0}}$ in terms of the ``error function''     
		\begin{equation}\label{def Z}
			Z\triangleq(Z_1,Z_2,Z_3)\triangleq\mathscr{F}_{\varepsilon}(i_0,\alpha_0)=\omega\cdot\partial_\varphi i_0-X_{H^{\alpha_0}_\varepsilon}(i_0).
		\end{equation}
	 We now conjugate the linear operator $\di_{(i,\alpha)} \mathscr{F}_{\varepsilon} (i_0,{\alpha}_0)$ in \eqref{Linearized-op-F-DC} with the linear change of variables given by
	\begin{equation}\label{Differential G0}
		D G_0\big( i_{\textnormal{\tiny{triv}}}(\vf) \big)
		\begin{pmatrix}
			\widehat\phi\\
			\widehat y \\
			\widehat w
		\end{pmatrix} 
		=
		\begin{pmatrix}
			\partial_\varphi \vartheta_0(\varphi) & 0 & 0 \\
			\partial_\varphi I_0(\varphi) & \mathtt{L}_1(\varphi) & 
			\mathtt{L}_2(\varphi)\\
			\partial_\varphi z_0(\varphi) & 0 & \textnormal{Id}
		\end{pmatrix}
		\begin{pmatrix}
			\widehat\phi\\
			\widehat y\\
			\widehat w
		\end{pmatrix}.
	\end{equation}
	 Although the transformation $G_0$ introduced in \eqref{defi:G0} is not symplectic and the
nonlinear Hamiltonian structure is no longer preserved, the conjugation of the linearized operator
via the linear change of variables $DG_{0}(i_{\rm triv})$ leads to a triangular system with small errors of size $Z$ as in \eqref{def Z}. Yet we emphasize that the Hamiltonian structure
of the original system is of paramount importance, namely in the construction of the transformation $G_0$ leading to the final triangular system.
      More precisely, we have the following result.
	\begin{prop}\label{Prop:Conj}
    The following properties hold true.
    \begin{enumerate}[label=(\roman*)]
        \item The map $DG_0(i_{\textnormal{\tiny{triv}}})$ is reversibility and momentum preserving in the sense of Lemma \ref{lem:MP-torus}.
        \item The following conjugation formula holds:
		\begin{align}\label{Id-conj}
			& [D G_0(i_{\textnormal{\tiny{triv}}})]^{-1}\di_{(i,\alpha)} \mathscr{F}_{\varepsilon} (i_0,\alpha_0) D\widetilde G_0(i_{\textnormal{\tiny{triv}}})
			[\widehat \phi, \widehat y, \widehat w, \widehat \alpha ]
			= \mathbb{D} [\widehat \phi, \widehat y, \widehat w, \widehat \alpha ]+\mathbb{E} [\widehat \phi, \widehat y, \widehat w],
		\end{align}
		where:
		\begin{itemize}
			\item the operator $ \mathbb{D}$ has the triangular form
			\begin{align}
				\mathbb{D} [\widehat \phi, \widehat y, \widehat w, \widehat \alpha ]\triangleq
				\begin{pmatrix}
					\omega\cdot\partial_\varphi  \widehat\phi-\big[K_{20}(\varphi) \widehat y+K_{11}^\top(\varphi) \widehat w+\mathtt{L}_1^\top (\varphi)\widehat \alpha\big]
					\\
					\omega\cdot\partial_\varphi  \widehat y+\mathcal{B}(\varphi) \widehat \alpha \\
					\mathscr{L}_{\omega}\widehat w-\mathcal{J}\big[K_{11}(\varphi) \widehat y+\mathtt{L}_2^{\top}(\varphi) \widehat\alpha\big] 
				\end{pmatrix}, \label{DD.def}
			\end{align}
			where $\mathcal{B}(\varphi)$ and $K_{20}(\varphi) $ are $d \times d$ real matrices given by 
			\begin{align}
				\mathcal{B}(\varphi)&\triangleq\mathtt{L}_1^{-1}(\varphi)\partial_\varphi I_0(\varphi)\mathtt{L}_1^\top (\varphi)+[\partial_\varphi z_0(\varphi)]^{\top}\mathtt{L}_2^\top (\varphi),\\
				K_{20}(\varphi)&\triangleq\varepsilon\mathtt{L}_1^\top(\varphi)(\pa_{II}\mathcal{P})\big(i_0(\varphi)\big)\mathtt{L}_1(\varphi),
			\end{align}
            $\mathscr{L}_{\omega}$ is the linearized operator in the normal directions
            \begin{equation}\label{def hat L}
		\mathscr{L}_{\omega}\triangleq \Pi_{\overline{\mathbb{S}}_0}^\bot \big(\omega\cdot \partial_\varphi-\mathcal{J}K_{02}(\varphi) \big)\Pi_{\overline{\mathbb{S}}_0}^\bot
	\end{equation}
			with $K_{02}(\varphi)$ the linear self-adjoint operator of $  \mathbf{H}_{\overline{\mathbb{S}}_0}^{\perp}$ given by 
			\begin{align}\label{def K02}
				K_{02}(\varphi)&\triangleq(\partial_{z}\nabla_z K_\varepsilon^{\alpha_0})\big(i_0(\varphi)\big)+\varepsilon\mathtt{L}_2^\top(\varphi)(\partial_{II}\mathcal{P})\big(i_0(\varphi)\big)\mathtt{L}_2(\varphi)\\
                &\quad+\varepsilon \mathtt{L}_2^\top(\varphi)(\partial_{zI}\mathcal{P})\big(i_0(\varphi)\big)+\varepsilon(\partial_I\nabla_z \mathcal{P})\big(i_0(\varphi)\big)\mathtt{L}_2(\varphi),
			\end{align}
			and where $K_{11}(\varphi)  \in {\mathcal L}({\mathbb R}^d,   \mathbf{H}_{\overline{\mathbb{S}}_0}^{\perp})$ is given by
			\begin{align*}
				K_{11}(\varphi)&\triangleq\varepsilon\mathtt{L}_2^\top(\varphi)(\partial_{II}\mathcal{P})\big(i_0(\varphi)\big)\mathtt{L}_1(\varphi)+\varepsilon(\partial_I\nabla_z\mathcal{P})\big(i_0(\varphi)\big)\mathtt{L}_1(\varphi);
			\end{align*}
			\item the remainder $\mathbb{E} $ is given by
			\begin{align}\label{def:bbE}
				\mathbb{E}[\widehat \phi, \widehat y, \widehat w] &\triangleq
				[D G_0(i_{\textnormal{\tiny{triv}}})]^{-1}    \partial_\varphi Z(\varphi) \widehat\phi + 
                \begin{pmatrix}
                    0 
					\\
					\mathcal{A}(\varphi)\big[K_{20}(\varphi) \widehat y+K_{11}^\top(\varphi) \widehat w\big]-R_{10}(\varphi) \widehat y -R_{01}(\varphi) \widehat w 
					\\
					0  
                \end{pmatrix},\qquad
			\end{align}
			where $\mathcal{A}(\varphi)$ and   $R_{10}(\varphi) $ are $d \times d$ real matrices defined as 
			\begin{align*}
				\mathcal{A}(\varphi)& \triangleq[\partial_\varphi \vartheta_0(\varphi)]^\top{\mathtt J} \partial_\varphi I_0(\varphi)-[\partial_\varphi I_0(\varphi)]^\top {\mathtt J}\partial_\varphi \vartheta_0(\varphi)  -[\partial_\varphi z_0(\varphi)]^{\top} \mathcal{J}^{-1} \partial_\varphi z_0(\varphi),
				\\
				R_{10}(\varphi)&\triangleq  [\partial_\varphi Z_1(\varphi)]^{\top}  [\partial_\varphi \vartheta_0(\varphi)]^{-\top}, 
			\end{align*}
			and where $R_{01}(\varphi)\in {\mathcal L}( \mathbf{H}_{\overline{\mathbb{S}}_0}^{\perp},{\mathbb R}^d)$ is defined by
			\begin{align*}
				R_{01}(\varphi)&\triangleq [\partial_\varphi Z_1(\varphi)]^{\top} [(\partial_\vartheta \widetilde{z}_0)(\vartheta_0(\varphi))]^\top \mathcal{J}^{-1}- [\partial_\varphi  Z_3(\varphi)]^{\top} \mathcal{J}^{-1}.
			\end{align*}
		\end{itemize}
    \end{enumerate}
	\end{prop}
    \begin{proof}
        The proof of item $(ii)$ and the fact that $DG_0$ is reversibility preserving is given in \cite[Prop. 6.1]{HHR23}. The momentum preserving property for $DG_0$ in item $(i)$ follows by differentiating the second equation in \eqref{G0-revMP} and Lemma \ref{lem:MP-torus}-$(i)$.
    \end{proof}

    Let us emphasize that the error term $\mathbb{E}$ vanishes at an exact solution. This is clear from its definition in \eqref{def:bbE} and some algebraic manipulations on the matrix $\mathcal{A}(\varphi)$ (see \cite[Lem. 6.2]{HHR23}). Therefore, we are left to invert the triangular system associated with the operator $\mathbb{D}.$ For this, we need to invert the operators $\omega\cdot\partial_{\varphi}$ and $\mathscr{L}_{\omega}.$ We introduce the inverse operator $(\omega\cdot\partial_{\varphi})_{\rm ext}^{-1}$ defined as follows (recalling the definition of the cuf-off function $\chi(\,\cdot\,)$ in \eqref{def chi})
        \begin{equation}\label{ext.omegadivf}
          \forall\,h = \sum_{\ell\in\Z^d\setminus\{0\}} h_{\ell} \be_{\ell} , \qquad   (\omega\cdot\pa_{\vf})_{\rm ext}^{-1} h \triangleq \sum_{\ell\in\Z^{d}\setminus\{0\}} \frac{\chi((\omega\cdot\ell)\gamma^{-1}|\ell|^{\tau_1})}{\im(\omega\cdot\ell) } h_{\ell} \be_{\ell}.
        \end{equation}
    It is an exact inverse of $\omega\cdot\partial_{\varphi}$ on the Diophantine set
    \begin{equation}\label{setDC}
        \mathtt{DC}(\gamma,\tau_1)\triangleq\bigcap_{\ell\in\mathbb{Z}^d\setminus\{0\}}\Big\{ \omega\in \R^d \quad\textnormal{s.t.}\quad|\omega\cdot\ell|>\frac{\gamma}{|\ell|^{\tau_1}}\Big\},\qquad\tau_1>d-1.
    \end{equation}
    The construction of a right inverse to $\mathscr{L}_{\omega}$ is the purpose of Section \ref{sec redu}, see Proposition \ref{prop.almost.inv}. We state here the corresponding result in the following assumption.
    \begin{assu}\label{ARI}
         {\bf (Full right inverse for $\sL_{\omega}$):} Let $(\gamma,q,d,\tau_{1},\tau,s_0,s_h,\mu_2)$ as in \eqref{d gm S s0}, \eqref{setDC}, \eqref{parametres:transport} and \eqref{some.constants}. There exists $\sigma_5\triangleq \sigma_5(\tau_1,\tau,q,d)>0$ such that, if \eqref{ansatz.fI.small} holds with $s_h+\sigma_5\leqslant s_0+\sigma_{\rm max},$ then there exists a linear operator $\widehat{\mathtt{T}}$  defined in $\mathcal{O}$ such that, for any  traveling wave  function $h$, we have that $\wh\tT h$  is a traveling wave function satisfying the estimate
		\begin{equation}
			\forall \, s\in\,[ s_0, S],\quad\|\widehat{\mathtt{T}}\rho\|_{s}^{q,\gamma ,\mathcal{O}}\lesssim_{s,q,d}\gamma^{-1}\big(\|\rho\|_{s+\sigma_5}^{q,\gamma ,\mathcal{O}}+\|\mathfrak{I}_{0}\|_{s+\sigma_5}^{q,\gamma ,\mathcal{O}}\|\rho\|_{s_{0}+\sigma_5}^{q,\gamma,\mathcal{O}}\big).
		\end{equation}
        In addition, if $h$ is anti-reversible, then $\wh\tT h$ is reversible.
		Moreover, in a Cantor set
		$\mathtt{G}\triangleq \mathtt{G}(\gamma,\tau_{1},\tau,i_{0})\subset(a_0,a_1)\times\mathtt {DC} (\gamma, \tau_1),$ we have
        $$\sL_{\omega}\wh\tT=\Pi_{\overline{\mathbb{S}}_0}^{\perp}.$$
    \end{assu}
	The following proposition  shows that the principal term  $\mathbb{D}$ in \eqref{Id-conj} of Proposition \ref{Prop:Conj}-$(ii)$ has an exact inverse. In the following, we denote
    \begin{equation}
        \|(\phi,y,w,\alpha) \|_{s}^{q,\gamma,\cO}\triangleq \max \{ \|(\phi,y,w) \|_{s}^{q,\gamma,\cO}, |\alpha|^{q,\gamma,\cO}\}.
    \end{equation}
	\begin{prop}\label{prop:decomp-lin}
		Let $(\gamma,d,q,\tau_1,\tau,s_0)$ as in \eqref{d gm S s0}, \eqref{setDC} and \eqref{some.constants}. There exists $\sigma_6\triangleq \sigma_6(\tau_1,\tau,q,d) >0$   such that if \eqref{ansatz.fI.small} holds with $s_h+\sigma_6\leqslant s_0+\sigma_{\rm max}$, then there exists an operator $[{\mathbb D}]_{\textnormal{ext}}^{-1}$, well-defined on the whole set of parameters $\mathcal{O},$ such that, for any $ g \triangleq (g_1, g_2, g_3) $ anti-reversible traveling wave variation, namely satisfying (see also Definition \ref{rev trav torus} and Definition \ref{defi:travivaria})
		\begin{equation}\label{symmetry g1 g2 g3}
			g_1(\varphi) = g_1(- \varphi),\quad g_2(\varphi) = - g_2(- \varphi),\quad g_3(\varphi) = - (\mathscr{S} g_3)(\varphi),\qquad\forall\,y\in\R, \quad (\mathscr{T}_{y}g_3)(\varphi)=g_3(\varphi-\vec{\jmath}y), 
		\end{equation}
		the function 
		$ [{\mathbb D}]_{\textnormal{ext}}^{-1} g =(\wh\phi,\why,\whw,\wh\alpha) $ 
		satisfies the estimate
		\begin{equation} \label{esti.Dn.ext.inv}
		\forall\, s \in [s_0,S], \quad	\| [{\mathbb D}]_{\textnormal{ext}}^{-1}g \|_{s}^{q,\gamma,\mathcal{O}}
			\lesssim_{s,q,d} \gamma^{-1} \big( \| g \|_{s + \sigma_6}^{q,\gamma,\mathcal{O}}
			+  \| {\mathfrak I}_0  \|_{s + \sigma_6}^{q,\gamma,\mathcal{O}}
			\| g \|_{s_0 + \sigma_6}^{q,\gamma,\mathcal{O}}  \big)
		\end{equation}
		and for all $(a,\omega) \in \mathtt{G}$ one has
        \begin{equation}\label{Dn.ext.id}
            {\mathbb D} [{\mathbb D}]_{\textnormal{ext}}^{-1} =\textnormal{Id}.
        \end{equation}
        Moreover, the first three components $(\wh\phi,\why,\whw)$ of $ [{\mathbb D}]_{\textnormal{ext}}^{-1} g$ form a reversible traveling wave variation, according to Definition \ref{rev trav torus}.
	\end{prop}
    \begin{proof}
        The proof of the existence of $\big([{\mathbb D}]_{\textnormal{ext}}^{-1}\big)$ and of \eqref{esti.Dn.ext.inv}, \eqref{Dn.ext.id} is given in \cite[Prop. 6.3]{HHM21}. We are left to prove the claim that $(\wh\phi,\why,\whw)$ is a reversible traveling wave variation. We have that $ [{\mathbb D}]_{\textnormal{ext}}^{-1} g  $ is the solution of the linear system, for any $(a,\omega) \in \tG$,
        \begin{equation}\label{triang.syst}
            \mathbb{D}[\wh\phi,\why,\whw,\wh\alpha] = g
        \end{equation}
        with $\mathbb{D}$ as in \eqref{DD.def} and $g=(g_1,g_2,g_3)$ satisfying \eqref{symmetry g1 g2 g3}. This is a triangular system.  By Assumption \ref{ARI}, the operator $\sL_{\omega}$ has a right inverse that is reversible and momentum preserving. Therefore, one solves the third equation in \eqref{triang.syst}. As for the first and the second equations, we point out that the unknown $\wh\alpha$ is indeed used here to remove the average in $\vf$ from the first equation in \eqref{triang.syst}. Consequently, these equations can be solved by using the operator $[\omega\cdot\pa_{\vf}]_{\rm ext}^{-1}$ in \eqref{ext.omegadivf}, which preserves the reversible structure.
        Then, following the proof of \cite[Prop. 6.3]{HHM21} and arguing as in \cite[Prop. 6.5]{BFM21}, one concludes that $(\wh\phi,\why,\whw)$ is indeed a reversible traveling wave variation.
    \end{proof}
	In view of \eqref{Id-conj} in Proposition \ref{Prop:Conj}-$(ii)$ , we obtain the following decomposition 
	\begin{equation*}
		\begin{aligned}
			\di_{(i,\alpha)}\mathscr{F}_{\varepsilon}(i_{0},\alpha_{0})
			&=DG_{0}(i_{\textnormal{\tiny{triv}}}) \, {\mathbb{D}}\, [D\widetilde{G}_{0}(i_{\textnormal{\tiny{triv}}})]^{-1}+ DG_{0}(i_{\textnormal{\tiny{triv}}}) \, {\mathbb E} \, [D\widetilde{G}_{0}(i_{\textnormal{\tiny{triv}}})]^{-1}.
		\end{aligned}
	\end{equation*}
	Applying the operator 
	\begin{equation}\label{def inverse T} 
		{\rm T}_0 \triangleq {\rm T}_0(i_0) \triangleq D { \widetilde G}_0(i_{\textnormal{\tiny{triv}}})\, [{\mathbb D}]_{\textnormal{ext}}^{-1}\,[D G_0(i_{\textnormal{\tiny{triv}}})]^{-1} 
	\end{equation} 
	to the right of the last identity, for all $(a,\omega)\in\mathtt{G},$
	\begin{align*}
	\di_{(i,\alpha)}\mathscr{F}_{\varepsilon}(i_{0},\alpha_{0}) {\rm T}_{0}-\textnormal{Id}=\mathfrak{E},\qquad\mathfrak{E}\triangleq DG_{0}(i_{\textnormal{\tiny{triv}}}) \,  {\mathbb E} \, [D\widetilde{G}_{0}(i_{\textnormal{\tiny{triv}}})]^{-1}{\rm T}_{0}.
			\end{align*}
	Consequently, the operator ${\rm T}_0$ is an approximate right inverse for $\di_{(i,\alpha)} \mathscr{F}_{\varepsilon}(i_0,\alpha_0)$. More precisely, we have the following result.
	
	\begin{theo}  \label{theo appr inv}
		{\bf (Approximate inverse)}
		Let $(\gamma,q,d,\tau_{1},\tau,s_{0},\mu_2,s_h,S)$ satisfy \eqref{d gm S s0}, \eqref{setDC}, \eqref{parametres:transport} and \eqref{some.constants}. There exists $ { \overline\sigma}= { \overline\sigma}(\tau_1,\tau,d,q)>0$   such that, if \eqref{ansatz.fI.small} holds with $s_h+\overline{\sigma}\leqslant s_0+\sigma_{\rm max}$,
		then, for any smooth $ g = (g_1, g_2, g_3) $ anti-reversible traveling wave variation, that is satisfying \eqref{symmetry g1 g2 g3}, 
		the operator $ {\rm T}_0 $ defined in \eqref{def inverse T} satisfies 
		\begin{equation}\label{tame T0}
			\forall\, s\in [s_0,S],\quad \| {\rm T}_0 g\|_{s}^{q,\gamma,\mathcal{O}}\lesssim_{s,q,d}\gamma^{-1}\big(\|g\|_{s+{\overline\sigma}}^{q,\gamma,\mathcal{O}}+\|\mathfrak{I}_{0}\|_{s+{\overline\sigma}}^{q,\gamma,\mathcal{O}}\|g\|_{s_{0}+\overline{\sigma}}^{q,\gamma,\mathcal{O}}\big).
		\end{equation}
        Moreover, the first three components of ${\rm T}_{0}g$ form a reversible traveling wave variation, according to Definition \ref{rev trav torus} and Definition \ref{defi:travivaria}.
		Finally, ${\rm T}_0$ is an approximate right inverse of $\di_{(i,\alpha)} 
		\mathscr{F}_{\varepsilon}(i_0,\alpha_0)$ on the Cantor set $\mathtt{G}$, that is for all $(a,\omega)\in\mathtt{G}$ one has
		\begin{equation}\label{splitting of approximate inverse}
			\di_{(i,\alpha)}\mathscr{F}_{\varepsilon}(i_0,\alpha_0){\rm T}_0-{\rm Id}= \mathfrak{E},
		\end{equation}
		where the operator $\mathfrak{E}$ is defined in the whole set $\mathcal{O}$ and enjoy the following estimates
		\begin{align}
			\|\mathfrak{E}\rho\|_{s_0}^{q,\gamma,\mathcal{O}} & \lesssim_{q,d}  \gamma^{-1 } \| \mathscr{F}_{\varepsilon}(i_0,\alpha_0)\|_{s_0+\overline\sigma}^{q,\gamma,\mathcal{O}} \|\rho\|_{s_0 + \overline\sigma}^{q,\gamma,\mathcal{O}}.\label{calE1}
		\end{align}
	\end{theo}
    \begin{proof}
        The proof of \eqref{tame T0}, \eqref{splitting of approximate inverse} and \eqref{calE1} follows as in \cite[Theorem 5.1]{HR21} and we omit it. The claim that the first three components of ${\rm T}_{0}g$ form a reversible traveling wave variation follows from the definition of ${\rm T}_{0}$ in \eqref{def inverse T}, \eqref{tilde.G0}, Proposition \ref{Prop:Conj}-$(i)$ and Proposition \ref{prop:decomp-lin}.
    \end{proof}

\section{Reducibility and inversion of the linearized operator}\label{sec redu}
For this section, we consider again the parameters $(d,\gamma,q,s_0,S)$ fixed as in \eqref{d gm S s0}. Let $i_0=(\vartheta_0,I_0,z_0)$ be a fixed traveling reversible torus, i.e. satisfying \eqref{rev trav torus}, and with the additional smallness ansatz \eqref{ansatz.fI.small}. Our aim here is to study the linearized operator in the normal directions taking the form (recall \eqref{def hat L})
\begin{equation}\label{Lperp}
 	\mathscr{L}_{\omega}=\mathscr{L}_{\omega}(i_0)=\Pi_{\overline{\mathbb{S}}_0}^{\perp}\big(\omega\cdot\partial_\varphi-\mathcal{J}K_{02}(\varphi)\big)\Pi_{\overline{\mathbb{S}}_0}^{\perp}.
\end{equation}
The ultimate goal is to prove Proposition \ref{prop.almost.inv} in order to obtain the existence  of a full right inverse with tame estimates of the operator in \eqref{Lperp} as previously stated in Assumption \ref{ARI}.
We shall achieve this by a reducibility procedure consisting of a series of reversibility and momentum preserving, invertible transformations that conjugate $\mathscr{L}_{\omega}$ to a reversible Fourier multiplier. We refer to Definition \ref{defin:OP-sym} for the definitions of real, reversible, reversibility and momentum preserving linear operators. We emphasize that the particular structure of our equations allows to run this reducibility procedure in the class of operator with homogeneous expansions (see Definition \ref{def.HomExp}), fundamental for applying the Egorov argument of Proposition \ref{prop:Ego}. For a concise presentation of the main reducibility steps, we refer the reader to the introductive Section \ref{sect.ideas}.\\

\noindent\textbf{Notation:} Given two different states $i_1$ and $i_2$ satisfying \eqref{ansatz.fI.small} and a function $f:\mathbb{T}^d\times\mathbb{R}^d\times\mathbf{H}_{\overline{\mathbb{S}}_0}^{\perp}\to X$ (with $X$ a set), we denote
$$\Delta_{12} f\triangleq f(i_2)-f(i_1).$$
The periodic components are also denoted $\displaystyle\mathfrak{I}_{k}(\varphi)\triangleq i_{k}(\varphi)-(\varphi,0,0),$ $k\in\{1,2\}.$


\subsection{Structure of the linearized operator in the normal directions}
In this section we link the linear operator $\mathscr{L}_{\omega}$ in \eqref{Lperp}, which acts only on the normal directions, to the one arising from the linearization at any approximate solution of the complete Hamiltonian system associated with $H$ (see \eqref{new:HAMDyn}). Furthermore, the latter operator will be suitably expanded in Proposition \ref{prop L0} for the purposes of the reducibility scheme in the forthcoming sections.\\

The proof of the following proposition can be adapted from \cite[Prop. 6.1]{HR21}. Let us mention that the traveling property can also be easily tracked (for a similar argument, see for instance \cite[Lemma 7.1]{BFM21}).
\begin{prop}\label{prop Lnormal}
    The following holds:
\begin{enumerate}[label=(\roman*)]
	\item  The  torus $i_0$ is associated with a state $\widetilde{r}=(\widetilde{r}_+,\widetilde{r}_-)$ in the form
	$$\widetilde{r}(\varphi)=\mathbf{A}\big(i_0(\varphi)\big),$$
	where the application $\mathbf{A}$ is given in \eqref{transfo aan-state}. The  reversibility and traveling properties of $i_0$ read
	$$\widetilde{r}(-\varphi,-x)=\widetilde{r}(\varphi,x)\qquad\textnormal{and}\qquad\forall\, y\in\mathbb{T},\quad \widetilde{r}(\varphi,x+y)=\widetilde{r}(\varphi-\vec{\jmath}\,y,x).$$
	In addition, one has the following estimates, for any $s\geqslant s_0,$	
	\begin{align}
		\|\widetilde{r}\|_{s}^{q,\gamma,\mathcal{O}}&\lesssim_{s,q,d}1+\|\mathfrak{I}_{0}\|_{s}^{q,\gamma,\mathcal{O}},\label{esti r I0}\\
		\|\Delta_{12}\widetilde{r}\|_{s}&\lesssim_{s,q,d}\|\Delta_{12}i\|_{s}+\| \Delta_{12}i\|_{s_0}\max_{\ell\in\{1,2\}}\|\mathfrak{I}_{\ell}\|_{s}.\label{esti r I0d}
	\end{align}
\item The linearized operator $\mathscr{L}_{\omega}$ in \eqref{Lperp} admits the following expression
	\begin{equation}\label{sL_omega}
		\mathscr{L}_{\omega}=\Pi_{\overline{\mathbb{S}}_0}^{\perp}\big(\mathcal{L}^{(0)}-\varepsilon\partial_{x}\mathcal{R}\big)\Pi_{\overline{\mathbb{S}}_0}^{\perp}\,,\qquad\mathcal{R}\triangleq \begin{pmatrix} 
			T_{\mathbb{K}_{+}^{(d)}} & T_{\mathbb{K}_{+}^{(o)}}\\
			T_{\mathbb{K}_{-}^{(o)}} & T_{\mathbb{K}_{-}^{(d)}}
		\end{pmatrix},
	\end{equation}
	where $\mathcal{L}^{(0)}$ is the operator defined by
    \begin{equation}\label{def:calL0}
        \mathcal{L}^{(0)}\triangleq\omega\cdot\partial_{\varphi}\mathbb{I}_2-\tfrac{1}{\varepsilon^2}\di_{\widetilde{r}}\big(\mathcal{J}\nabla H(\widetilde{r})\big),\qquad\mathbb{I}_2\triangleq\begin{pmatrix}
            \textnormal{Id} & 0\\
            0 & \textnormal{Id}
        \end{pmatrix}, \qquad \textnormal{Id}\triangleq\textnormal{Id}_{L^2(\mathbb{T})},
    \end{equation}
   whereas, for any $(\kappa,\nu)\in\{-,+\}\times\{d,o\},$ the operator $T_{\mathbb{K}_{\kappa}^{(\nu)}}$ is an integral  operator in the form 
$$T_{\mathbb{K}_{\kappa}^{(\nu)}}[\rho](a;\varphi,x)=\int_{\mathbb{T}}\rho(a;\varphi,x')\mathbb{K}_{\kappa}^{(\nu)}(a;\varphi,x,x'){\rm d}x',$$ 
whose kernel
$\mathbb{K}_{\kappa}^{(\nu)}$ is momentum and reversibility preserving, that is
\begin{align*}
	\mathbb{K}_{\kappa}^{(\nu)}(a;-\varphi,-x,-x')&=\mathbb{K}_{\kappa}^{(\nu)}(a;\varphi,x,x'),\\
	\forall \, y\in\mathbb{T},\quad\mathbb{K}_{\kappa}^{(\nu)}(a;\varphi-\vec{\jmath}y,x,x')&=\mathbb{K}_{\kappa}^{(\nu)}(a;\varphi,x+y,x'+y)
\end{align*}
and satisfies the following estimates for any $s\geqslant s_0,$
 	\begin{align}
		\|\mathbb{K}_{\kappa}^{(\nu)}\|_{s}^{q,\gamma,\mathcal{O}}&\lesssim_{s,q,d} 1+\|\mathfrak{I}_{0}\|_{s+2}^{q,\gamma,\mathcal{O}},\label{e-Kkappanu}\\
		\|\Delta_{12}\mathbb{K}_{\kappa}^{(\nu)}\|_{s}&\lesssim_{s,q,d}\|\Delta_{12}i\|_{s+2}+\|\Delta_{12}i\|_{s_0+2}\max_{\ell\in\{1,2\}}\|\mathfrak{I}_{\ell}\|_{s+2}.\label{e-Kkappanu diff}
	\end{align}
\end{enumerate}
\end{prop}

\begin{remark}
    The operator $\cL^{(0)}$ in \eqref{def:calL0} is defined on a dense subspace of $\bL^2$, even though, by the Hamiltonian nature of the nonlinear system, the linearized operator preserves the average in space. We operate in this way because in the end we are interested in proving the existence of a right inverse to the operator $\sL_{\omega}$ in \eqref{sL_omega}, which acts trivially on the average in space.
\end{remark}

\smallskip

Our next goal is to give a new expression of the linear operator $\mathcal{L}^{(0)}$ in \eqref{def:calL0}, more suitable for our following analysis. In particular, we highlight its homogeneous expansion up to smoothing pseudo-differential operators.

\smallskip

The linear operator $\mathbf{Q}$ in \eqref{bbQ-1} admits the following pseudo-differential formulation
$$\mathbf{Q}=\mathbb{I}_{2,0}+\begin{pmatrix}
	\mathtt{K}_+(a,D) & \mathtt{K}_{-}(a,D)\\
	\mathtt{K}_{-}(a,D) & \mathtt{K}_+(a,D)
\end{pmatrix},\qquad\mathbb{I}_{2,0}=\begin{pmatrix}
\textnormal{Id}_{L_0^2(\mathbb{T})} & 0\\
0 & \textnormal{Id}_{L_0^2(\mathbb{T})}
\end{pmatrix}, \qquad $$
where $\textnormal{Id}_{L_0^2(\mathbb{T})}\triangleq \textnormal{Op}\big(\chi(\xi)\big)$ denotes the identity operator of $L_{0}^{2}(\mathbb{T})$ and the Fourier multipliers $\mathtt{K}_{\pm}(a,D)$ are associated with the following symbol
\begin{align}
	\forall\,\xi\in\mathbb{R},\quad\mathtt{K}_{+}(a,\xi)&\triangleq\chi(\xi)\frac{1-\sqrt{1-b^2(a,\xi)}}{\sqrt{1-b^{2}(a,\xi)}},\qquad\mathtt{K}_{-}(a,\xi)\triangleq\chi(\xi)\frac{b(a,\xi)}{\sqrt{1-b^{2}(a,\xi)}}\label{def symb ttKn}\\
	\forall\,|\xi|\geqslant\tfrac{1}{2},\quad b(a,\xi)&\triangleq\frac{1}{2a|\xi|\left(\sqrt{a^2\xi^2+1}+a|\xi|+\frac{1}{2a|\xi|}\right)}\in(0,1)\label{def symb b}.
\end{align}
We recall that the cut-off function $\chi$ is defined in \eqref{def chi}. Actually, we have the following more precise (but still rough) bound
\begin{equation}\label{unif bns symb b}
	\sup_{|\xi|\geqslant\frac{1}{2}\atop a\in[a_0,a_1]}b(a,\xi)\leqslant\frac{2}{2+a_0^2}<1.
\end{equation}
Note that $b(a,\xi)$ can also be written
$$\forall\,|\xi|\geqslant\tfrac{1}{2},\quad b(a,\xi)=1-2a\xi\mathtt{r}(a,\xi),$$
where $\mathtt{r}(a,\xi)$ is the symbol introduced in \eqref{split eig}. Then, proceeding as in the proof of Lemma \ref{lem prop eig VP}-$(iv)$, we obtain the following estimates
\begin{equation}\label{e-symb b}
	\forall\,(k,\ell)\in\mathbb{N}^2,\quad\sup_{|\xi|\geqslant\frac{1}{2}\atop a\in[a_0,a_1]}\langle\xi\rangle^{\ell+2}\left|\partial_{a}^{k}\partial_{\xi}^{\ell}b(a,\xi)\right|\leqslant C(a_0,k,\ell).
\end{equation}
This implies in particular that $\mathtt{K}_{\pm}(a,D)\in OPS^{-2}.$ In addition, since $\tK_{\pm}(a,\xi)\in\mathbb{R}$ and is independent of $(\varphi,x)$, then, according to Proposition \ref{properties OPS}-$(vi)$, the operators $\tK_{\pm}(a,D)$ are reversibility and momentum preserving.
By \eqref{bbQ-1} and \eqref{inverse Qj}, we also have
$$\mathbf{Q}^{-1}=\mathbb{I}_{2,0}+\begin{pmatrix}
	\mathtt{K}_+(a,D) & -\mathtt{K}_{-}(a,D)\\
	-\mathtt{K}_{-}(a,D) & \mathtt{K}_+(a,D)
\end{pmatrix}.$$
Recalling the notations \eqref{linear at general state}-\eqref{lin eq QP}, linearizing \eqref{new:HAMDyn} gives
\begin{equation}\label{new.linH}
    \begin{aligned}
	\tfrac{1}{\varepsilon^2}\di_{\widetilde{r}}\big(\mathcal{J}\nabla H(\widetilde{r})\big)[\rho]&=\tfrac{1}{\varepsilon^2}\mathbf{Q}^{-1}\di_{r}\big(\mathcal{J}\nabla\mathcal{E}\big)(\mathbf{Q}\widetilde{r})[\mathbf{Q}\rho]\\
	&=\mathbf{Q}^{-1}\mathcal{J}\mathbf{M}_{\varepsilon\mathbf{Q}\widetilde{r}}\mathbf{Q}\rho\\
	&=\mathbf{Q}^{-1}\mathcal{J}\mathbf{M}_{0}\mathbf{Q}\rho+\mathbf{Q}^{-1}\mathcal{J}\big(\mathbf{M}_{\varepsilon\mathbf{Q}\widetilde{r}}-\mathbf{M}_{0}\big)\mathbf{Q}\rho.
\end{aligned}
\end{equation}
Besides, we have
\begin{align*}
	\mathbf{M}_{\varepsilon\mathbf{Q}\widetilde{r}}-\mathbf{M}_{0}&=\begin{pmatrix}
		\varepsilon r_{+} & 0\\
		0 & -\varepsilon r_{-}
	\end{pmatrix}\\
&=\varepsilon\begin{pmatrix}
		\widetilde{r}_++\mathtt{K}_{+}(a,D)\widetilde{r}_++\mathtt{K}_{-}(a,D)\widetilde{r}_- & 0\\
		0 & -\widetilde{r}_--\mathtt{K}_{-}(a,D)\widetilde{r}_+-\mathtt{K}_{+}(a,D)\widetilde{r}_-
	\end{pmatrix}\\
&\triangleq\varepsilon\begin{pmatrix}
		f_{+}(\widetilde{r}) & 0\\
		0 & -f_{-}(\widetilde{r})
	\end{pmatrix}.
\end{align*}
Note that
$$f_{-}(\widetilde{r}_+,\widetilde{r}_-)=f_+(\widetilde{r}_-,\widetilde{r}_+).$$
The following estimates hold true for any $s\in[s_0,S]$,
\begin{align}
	\|f_{\pm}(\widetilde{r})\|_{s}^{q,\gamma,\mathcal{O}}&\lesssim_{s,q,d}\|\widetilde{r}\|_{s}^{q,\gamma,\mathcal{O}},\label{e-fk}\\
	\|\Delta_{12}f_{\pm}\|_{s}&\lesssim_{s,q,d}\|\Delta_{12}\widetilde{r}\|_{s}.\label{e12-fk}
\end{align}
Moreover, if $\widetilde{r}(-\varphi,-x)=\widetilde{r}(\varphi,x)=\widetilde{r}(\varphi-\vec{\jmath}x)$ (i.e. $\widetilde{r}$ is a reversible traveling wave), then
\begin{align}
	f_{\pm}(\widetilde{r})(a;-\varphi,-x)&=f_{\pm}(\widetilde{r})(a;\varphi,x),\label{rev fpm}\\
	\forall \,y\in\mathbb{T},\quad f_{\pm}(\widetilde{r})(a;\varphi-\vec{\jmath}y,x)&=f_{\pm}(\widetilde{r})(a;\varphi,x+y).\label{trav fpm}
\end{align}
Then,
\begin{equation}\label{conj:lineps}
    \mathbf{Q}^{-1}\mathcal{J}\big(\mathbf{M}_{\varepsilon\mathbf{Q}\widetilde{r}}-\mathbf{M}_{0}\big)\mathbf{Q}=-\varepsilon\partial_{x}\mathbf{F}-\varepsilon\widetilde{\mathbf{R}},
\end{equation}
where
\begin{equation}\label{def bfF and bfRt}
	\mathbf{F}\triangleq\begin{pmatrix}
		f_+(\widetilde{r})\cdot & 0\\
		0 & f_-(\widetilde{r})\cdot
	\end{pmatrix},\qquad\widetilde{\mathbf{R}}=\begin{pmatrix}
		\widetilde{R}_{+,0}^{(d)} & \widetilde{R}_{+,0}^{(o)}\vspace{0.2cm}\\
		\widetilde{R}_{-,0}^{(o)} & \widetilde{R}_{-,0}^{(d)}
	\end{pmatrix},
\end{equation}
with
\begin{equation}\label{Rtld}
	\begin{aligned}
		\widetilde{R}_{+,0}^{(d)}&\triangleq\partial_{x}\Big(\mathtt{K}_+(a,D)f_{+}(\widetilde{r})+f_{+}(\widetilde{r})\mathtt{K}_{+}(a,D)+\mathtt{K}_{+}(a,D)f_{+}(\widetilde{r})\mathtt{K}_{+}(a,D)-\mathtt{K}_{-}(a,D)f_{-}(\widetilde{r})\mathtt{K}_{-}(a,D)\Big),\\
		\widetilde{R}_{-,0}^{(d)}&\triangleq\partial_{x}\Big(\mathtt{K}_{+}(a,D)f_{-}(\widetilde{r})+f_{-}(\widetilde{r})\mathtt{K}_{+}(a,D)+\mathtt{K}_{+}(a,D)f_{-}(\widetilde{r})\mathtt{K}_{+}(a,D)-\mathtt{K}_{-}(a,D)f_{+}(\widetilde{r})\mathtt{K}_{-}(a,D)\Big),\\
		\widetilde{R}_{+,0}^{(o)}&\triangleq\partial_{x}\Big(f_{+}(\widetilde{r})\mathtt{K}_{-}(a,D)-\mathtt{K}_{-}(a,D)f_{-}(\widetilde{r})-\mathtt{K}_{-}(a,D)f_{-}(\widetilde{r})\mathtt{K}_{+}(a,D)+\mathtt{K}_{+}(a,D)f_{+}(\widetilde{r})\mathtt{K}_{-}(a,D)\Big),\\
		\widetilde{R}_{-,0}^{(o)}&\triangleq\partial_{x}\Big(f_{-}(\widetilde{r})\mathtt{K}_{-}(a,D)-\mathtt{K}_{-}(a,D)f_{+}(\widetilde{r})-\mathtt{K}_{-}(a,D)f_{+}(\widetilde{r})\mathtt{K}_{+}(a,D)+\mathtt{K}_{+}(a,D)f_{-}(\widetilde{r})\mathtt{K}_{-}(a,D)\Big).
	\end{aligned}
\end{equation}
Combining \eqref{new.linH} and \eqref{conj:lineps}, recalling also the definition of $\boldsymbol{\Omega}(a,D)$ in \eqref{bold.Omega}, yields
$$\tfrac{1}{\varepsilon^2}\di_{\widetilde{r}}\big(\mathcal{J}\nabla H(\widetilde{r})\big)=-\ii\boldsymbol{\Omega}(a,D)-\varepsilon\partial_{x}\mathbf{F}-\varepsilon\widetilde{\mathbf{R}}.$$
Therefore, the linearized operator $\mathcal{L}^{(0)}$ in \eqref{def:calL0} at a general state writes
\begin{equation}\label{def cal L0}
	\mathcal{L}^{(0)}=\omega\cdot\partial_{\varphi}\mathbb{I}_2-\tfrac{1}{\varepsilon^2}\di_{\widetilde{r}}\big(\mathcal{J}\nabla H(\widetilde{r})\big)=\omega\cdot\partial_{\varphi}\mathbb{I}_{2}+\ii\boldsymbol{\Omega}(a,D)+\varepsilon\partial_{x}\mathbf{F}+\varepsilon\widetilde{\mathbf{R}}.
\end{equation}
Actually, the previous expression is not well-adapted for our further reduction scheme. In the next proposition, we provide a more suitable expression, where the remainder is decomposed into a diagonal/off-diagonal homogeneous expansion up to smoothing pseudo-differential operators (see Definition \ref{def.HomExp}).

\begin{prop}\label{prop L0}
	For any $M\in\mathbb{N}^*,$ the linearized operator  $\mathcal{L}^{(0)}$ defined in \eqref{def cal L0} admits the decomposition
	$$\mathcal{L}^{(0)}=\omega\cdot\partial_{\varphi}\mathbb{I}_2+\ii\boldsymbol{\Omega}(a,D)+\varepsilon\partial_{x}\mathbf{F}+\mathbf{R}_{0}^{(d)}+\mathbf{R}_{0}^{(o)}+\mathbf{S}_{0,M},$$
	with the following properties:
	\begin{enumerate}[label=(\roman*)]
        \item The matrix pseudo-differential operator $\boldsymbol{\Omega}(a,D)$ is defined in \eqref{bold.Omega};
        \item The multiplication matrix operator $\bF$ is defined in \eqref{def bfF and bfRt}, with entries satisfying \eqref{rev fpm}-\eqref{trav fpm} and the estimates \eqref{e-fk}-\eqref{e12-fk};
		\item The operators $\mathbf{R}_{0}^{(d)}$ and $\mathbf{R}_{0}^{(o)}$ admit an homogeneous expansion of degree $-1$ up to order $1-M$ (see Definition  \ref{def.HomExp}-$(ii)$) in the form
		$$\mathbf{R}_0^{(d)}=\begin{pmatrix}
			\displaystyle\sum_{p=1}^{M-1}\mathfrak{r}_{+,0,p}^{(d)}\partial_{x}^{-p} & 0\\
			0 & \displaystyle\sum_{p=1}^{M-1}\mathfrak{r}_{-,0,p}^{(d)}\partial_{x}^{-p}
		\end{pmatrix},\qquad\mathbf{R}_0^{(o)}=\begin{pmatrix}
		0 & \displaystyle\sum_{p=1}^{M-1}\mathfrak{r}_{+,0,p}^{(o)}\partial_{x}^{-p}\\
		\displaystyle\sum_{p=1}^{M-1}\mathfrak{r}_{-,0,p}^{(o)}\partial_{x}^{-p} & 0
	\end{pmatrix},$$
with for any $(\kappa,\nu,p)\in\{-,+\}\times\{d,o\}\times\llbracket 1,M-1\rrbracket,$ the function $\mathfrak{r}_{\kappa,0,p}^{(\nu)}=\mathfrak{r}_{\kappa,0,p}^{(\nu)}(a;\varphi,x)\in\mathbb{R}$ satisfying
\begin{align}
	\mathfrak{r}_{\kappa,0,p}^{(\nu)}(a;-\varphi,-x)&=(-1)^{p+1}\mathfrak{r}_{\kappa,0,p}^{(\nu)}(a;\varphi,x),\label{rev r0}\\
	\forall \,y\in\mathbb{T},\quad\mathfrak{r}_{\kappa,0,p}^{(\nu)}(a;\varphi-\vec{\jmath}y,x)&=\mathfrak{r}_{\kappa,0,p}^{(\nu)}(a;\varphi,x+y).\label{trav r0}
\end{align}
In particular, the operators $\mathbf{R}_{0}^{(d)}$ and $\mathbf{R}_{0}^{(o)}$ are real, reversible and momentum preserving. Moreover, for any $s\in[s_0,S]$,
\begin{align}
	\|\mathfrak{r}_{\kappa,0,p}^{(\nu)}\|_{s}^{q,\gamma,\mathcal{O}}&\lesssim_{s,q,d,M}\varepsilon\|\widetilde{r}\|_{s+\sigma_0(M)}^{q,\gamma,\mathcal{O}},\\
	\|\Delta_{12}\mathfrak{r}_{\kappa,0,p}^{(\nu)}\|_{s}&\lesssim_{s,q,d,M}\varepsilon\|\Delta_{12}\widetilde{r}\|_{s+\sigma_0(M)}
\end{align}
for some $\sigma_0(M)>0$; 
\item The operator $\mathbf{S}_{0,M}$ is real, reversible, momentum preserving and in $OPS^{-M}$. It satisfies the following estimates for any $s\in[s_0,S]$ and $\alpha\in\mathbb{N},$
\begin{align}
	\|\mathbf{S}_{0,M}\|_{-M,s,\alpha}^{q,\gamma,\mathcal{O}}&\lesssim_{s,q,d,M,\alpha}\varepsilon \|\widetilde{r}\|_{s+\sigma_0(\alpha,M)}^{q,\gamma,\mathcal{O}},\label{e-S0M}\\
	\|\Delta_{12}\mathbf{S}_{0,M}\|_{-M,s,\alpha}&\lesssim_{s,q,d,M,\alpha}\varepsilon\|\Delta_{12}\widetilde{r}\|_{s+\sigma_0(\alpha,M)}\label{e12-S0M}
\end{align}
for some $\sigma_0(\alpha,M)>0.$
	\end{enumerate}
\end{prop}
\begin{proof}
	The proof consists in a reformulation of the remainder term $\varepsilon\widetilde{\mathbf{R}}$ given by \eqref{def bfF and bfRt}-\eqref{Rtld}. Prooceeding as in the proof of Lemma \ref{lem prop eig VP}-\textit{(iv)}, we  expand the symbol \eqref{def symb ttKn} using Taylor's formulas and obtain, for any $M\in\mathbb{N}^*$ and $\kappa\in\{-,+\},$ a decomposition in the form
	\begin{equation}\label{decomp Kkp}
		\mathtt{K}_{\kappa}(a,D)=\sum_{p=1}^{\widetilde{\mathfrak{N}}_M}c_{\kappa,p}(a)\partial_{x}^{-2p}+S_{\kappa,M}(a,D),
	\end{equation}
	where 
    $$\widetilde{\mathfrak{N}}_M\triangleq\max\big\{p\in\mathbb{N}\quad \textnormal{s.t.}\ \, -2p>-M\big\}=\begin{cases}
        \big\lfloor\tfrac{M}{2}\big\rfloor-1, & \textnormal{if }M\equiv0\,\,[2],\\
        \big\lfloor\tfrac{M}{2}\big\rfloor, & \textnormal{if }M\equiv1\,\,[2],
    \end{cases}$$
    each $c_{\kappa,p}(a)$ is a constant depending on $\kappa,p$ and smooth (rational) in $a$, and where the operator $S_{\kappa,M}(a,D)$ is a Fourier multiplier in $OPS^{-M}.$ With these decompositions in hand, one applies Proposition \ref{properties OPS}-$(i)$ and
	Proposition \ref{lem compo commu hom exp}-$(i)$ to \eqref{Rtld} and writes, for any $\kappa\in\{-,+\}$ and $\nu\in\{d,o\},$
	$$\varepsilon \widetilde{R}_{\kappa,0}^{(\nu)}=\sum_{p=1}^{M-1}\mathfrak{r}_{\kappa,0,p}^{(\nu)}\partial_{x}^{-p}+S_{\kappa,0,M}^{(\nu)},$$
	where the $\mathfrak{r}_{\kappa,0,p}^{(\nu)}$ are (modulo multiplicative constants) finite sum and products of $f_{+}(\widetilde{r}),f_-(\widetilde{r})$ and their (space) derivatives up to a certain finite order depending on $M.$ Therefore, there exists some $\sigma_0(M)>0$ such that, using in particular \eqref{e-fk}-\eqref{e12-fk} and \eqref{esti r I0},
	\begin{align*}
		\|\mathfrak{r}_{\kappa,0,p}^{(\nu)}\|_{s}^{q,\gamma,\mathcal{O}}&\lesssim_{s,q,d,M}\varepsilon \max_{\kappa'\in\{-,+\}}\|f_{\kappa'}(\widetilde{r})\|_{s+\sigma_0(M)}^{q,\gamma,\mathcal{O}}\lesssim_{s,q,d,M}\varepsilon \|\widetilde{r}\|_{s+\sigma_0(M)}^{q,\gamma,\mathcal{O}},\\
		\|\Delta_{12}\mathfrak{r}_{\kappa,0,p}^{(\nu)}\|_{s}&\lesssim_{s,q,d,M}\varepsilon \max_{\kappa'\in\{-,+\}}\|\Delta_{12}f_{\kappa'}\|_{s+\sigma_0(M)}\lesssim_{s,q,d,M}\varepsilon \|\Delta_{12}\widetilde{r}\|_{s+\sigma_0(M)}.
	\end{align*}
The traveling and reversibility properties \eqref{rev r0}-\eqref{trav r0} are obtained after tedious computations from \eqref{rev fpm}-\eqref{trav fpm}, \eqref{Rtld} and \eqref{decomp Kkp}, Proposition \ref{lem compo commu hom exp}-$(i)$. By applying Remark \ref{compo:RevMom}, one readily has from \eqref{Rtld}, \eqref{rev fpm}, \eqref{trav fpm} and the fact that $\mathtt{K}_{\pm}(a,D)$ are reversibility and momentum preserving, that the remainders $\widetilde{R}_{\kappa,0}^{(\nu)}$ is reversible and momentum preserving. By difference with what preceeds, so is $S_{\kappa,0,M}^{(\nu)}.$ In addition, from \eqref{e-rem-hom-comp}, we obtain for some $\sigma_0(\alpha,M)>0$,
	\begin{align}
		\|S_{\kappa,0,M}^{(\nu)}\|_{-M,s,\alpha}^{q,\gamma,\mathcal{O}}&\lesssim_{s,q,d,M,\alpha}\varepsilon \|\widetilde{r}\|_{s+\sigma_0(\alpha,M)}^{q,\gamma,\mathcal{O}},\label{e-Skkp}\\
		\|\Delta_{12}S_{\kappa,0,M}^{(\nu)}\|_{-M,s,\alpha}&\lesssim_{s,q,d,M,\alpha}\varepsilon \|\Delta_{12}\widetilde{r}\|_{s+\sigma_0(\alpha,M)}.\label{e12-Skkp}
	\end{align}
Then, we define
	$$\mathbf{S}_{0,M}\triangleq\begin{pmatrix}
		S_{+,0,M}^{(d)} &  S_{+,0,M}^{(o)}\vspace{0.2cm}\\
		 S_{-,0,M}^{(o)} &  S_{-,0,M}^{(d)}
	\end{pmatrix}.$$
The estimates \eqref{e-Skkp}-\eqref{e12-Skkp} give \eqref{e-S0M}-\eqref{e12-S0M}. This ends the proof of Proposition \ref{prop L0}.
\end{proof}

\subsection{Block-decoupling up to sufficiently smoothing remainders}\label{sect.block.anti}
In this section, we transform the linearized operator $\mathcal{L}^{(0)}$ described in Proposition \ref{prop L0} through an iterative procedure in order to make the off-diagonal operator $\bR_{0}^{(o)}$ being smoothing up to a large enough order $M\in\mathbb{N}^*$ that may be chosen later, see \eqref{choice M}.
\begin{prop}\label{prop.blockanti}
	Let $M\in\mathbb{N}^*$. There exists $\varepsilon_{0}>0$ small enough such that, for  any $\varepsilon\in (0,\varepsilon_{0})$, there exist real, reversibility and momentum preserving matrix operators $\mathbf{X}_0,\mathbf{X}_1,\ldots,\mathbf{X}_{M-1}$ in the form
	\begin{equation}\label{def:Xm}
	    \mathbf{X}_{0}=0,\qquad\forall\, m\in\llbracket1,M-1\rrbracket,\quad\mathbf{X}_{m}=\begin{pmatrix}
		0 & \mathfrak{p}_{+,m}\partial_{x}^{-m-1}\\
		\mathfrak{p}_{-,m}\partial_{x}^{-m-1} & 0
	\end{pmatrix},
	\end{equation}
with, for any $\kappa\in\{-,+\}$ and $m\in\llbracket 1,M-1\rrbracket,$ the function $\mathfrak{p}_{\kappa,m}\triangleq\mathfrak{p}_{\kappa,m}(a,\omega;\varphi,x)\in\mathbb{R}$ satisfying
\begin{align}
	\mathfrak{p}_{\kappa,m}(a,\omega;-\varphi,-x)&=(-1)^{m+1}\mathfrak{p}_{\kappa,m}(a,\omega;\varphi,x),\\
	\forall \, y\in\mathbb{T},\quad\mathfrak{p}_{\kappa,m}(a,\omega;\varphi-\vec{\jmath}y,x)&=\mathfrak{p}_{\kappa,m}(a,\omega;\varphi,x+y)
\end{align}
and the estimates, for  some $\sigma_{m}(M)\triangleq\sigma_m(s_0,q,d,M)>0$ and for any $s \in [s_0,S]$,
\begin{align}
\|\mathfrak{p}_{\kappa,m}\|_{s}^{q,\gamma,\mathcal{O}}&\lesssim_{s,q,d,M}\varepsilon\|\widetilde{r}\|_{s+\sigma_{m}(M)}^{q,\gamma,\mathcal{O}},\label{e-pm}\\
\|\Delta_{12}\mathfrak{p}_{\kappa,m}\|_{s}&\lesssim_{s,q,d,M}\varepsilon\Big(\|\Delta_{12}\widetilde{r}\|_{s+\sigma_{m}(M)}+\|\Delta_{12}\widetilde{r}\|_{s_0+\sigma_{m}(M)}\max_{k\in\{1,2\}}\|\widetilde{r}_k\|_{s+\sigma_{m}(M)}\Big),\label{e12-pm}
\end{align}
such that the following holds. Denoting
\begin{equation}\label{compoPhim}
\boldsymbol{\Phi}_0\triangleq\mathbb{I}_2,\qquad \forall\,m\in\llbracket 1,M-1\rrbracket, \quad  \boldsymbol{\Phi}_{m}\triangleq \big(\mathbb{I}_2+\mathbf{X}_{1}\big)\circ\ldots\circ \big(\mathbb{I}_2+\mathbf{X}_{m}\big),
\end{equation}
we have, for any $m\in\llbracket1,M-1\rrbracket$,
\begin{equation}\label{calLM}
	\mathcal{L}^{(m)}\triangleq\boldsymbol{\Phi}_{m}^{-1}\mathcal{L}^{(0)}\boldsymbol{\Phi}_{m}=\omega\cdot\partial_{\varphi}\mathbb{I}_{2}+\ii\boldsymbol{\Omega}(a,D)+\varepsilon\partial_{x}\mathbf{F}+\mathbf{R}_m^{(d)}+\mathbf{R}_{m}^{(o)}+\mathbf{S}_{m,M},
\end{equation}
with the following properties: 
\begin{enumerate}[label=(\roman*)]
	\item The operator $\mathbf{R}_{m}^{(d)}$ admits an homogeneous expansion of degree $-1$ up to order $1-M$ in the form
    \begin{equation}\label{bR.m.diag}
        \mathbf{R}_m^{(d)}=\begin{pmatrix}
		\displaystyle\sum_{p=1}^{M-1}\mathfrak{r}_{+,m,p}^{(d)}\partial_{x}^{-p} & 0\\
		0 & \displaystyle\sum_{p=1}^{M-1}\mathfrak{r}_{-,m,p}^{(d)}\partial_{x}^{-p}
	\end{pmatrix},
    \end{equation}
with for any $\kappa\in\{-,+\}$ and $p\in\llbracket 1,M-1\rrbracket,$ the function $\mathfrak{r}_{\kappa,m,p}^{(d)}\triangleq\mathfrak{r}_{\kappa,m,p}^{(d)}(a,\omega;\varphi,x)\in\mathbb{R}$ satisfying
\begin{align}
	\mathfrak{r}_{\kappa,m,p}^{(d)}(a,\omega;-\varphi,-x)&=(-1)^{p+1}\mathfrak{r}_{\kappa,m,p}^{(d)}(a,\omega;\varphi,x), \label{rev rm diag} \\
	\forall \, y\in\mathbb{T},\quad\mathfrak{r}_{\kappa,m,p}^{(d)}(a,\omega;\varphi-\vec{\jmath}y,x)&=\mathfrak{r}_{\kappa,m,p}^{(d)}(a,\omega;\varphi,x+y). \label{trav rm diag}
\end{align}
In particular, the operator $\mathbf{R}_m^{(d)}$ is real, reversible and momentum preserving. Moreover, for any $s\in[s_0,S],$
\begin{align}
	\|\mathfrak{r}_{\kappa,m,p}^{(d)}\|_{s}^{q,\gamma,\mathcal{O}}&\lesssim_{s,q,d,M}\varepsilon\|\widetilde{r}\|_{s+\sigma_{m}(M)}^{q,\gamma,\mathcal{O}}, \label{e-rm-diag} \\
    \|\Delta_{12}\mathfrak{r}_{\kappa,m,p}^{(d)}\|_{s}&\lesssim_{s,q,d,M}\varepsilon\big(\|\Delta_{12}\widetilde{r}\|_{s+\sigma_{m}(M)}+\|\Delta_{12}\widetilde{r}\|_{s_0+\sigma_{m}(M)}\max_{k\in\{1,2\}}\|\widetilde{r}_k\|_{s+\sigma_{m}(M)}\big).\label{e12-rm-diag}
\end{align}

\item The operator $\mathbf{R}_{m}^{(o)}$ admits an homogeneous expansion of degree $-m-1$ up to order $1-M$ in the form
\begin{equation}\label{bR.m.off}
    \mathbf{R}_m^{(o)}=\begin{pmatrix}
	0 & \displaystyle\sum_{p=m+1}^{M-1}\mathfrak{r}_{+,m,p}^{(o)}\partial_{x}^{-p}\\
	\displaystyle\sum_{p=m+1}^{M-1}\mathfrak{r}_{-,m,p}^{(o)}\partial_{x}^{-p} & 0
\end{pmatrix},
\end{equation}
with for any $\kappa\in\{-,+\}$ and $p\in\llbracket m+1,M-1\rrbracket,$ the function $\mathfrak{r}_{\kappa,m,p}^{(o)}\triangleq\mathfrak{r}_{\kappa,m,p}^{(o)}(a,\omega;\varphi,x)\in\mathbb{R}$ satisfying
\begin{align}
	\mathfrak{r}_{\kappa,m,p}^{(o)}(a,\omega;-\varphi,-x)&=(-1)^{p+1}\mathfrak{r}_{\kappa,m,p}^{(o)}(a,\omega;\varphi,x),\label{rev rm}\\
	\forall \, y\in\mathbb{T},\quad\mathfrak{r}_{\kappa,m,p}^{(o)}(a,\omega;\varphi-\vec{\jmath}y,x)&=\mathfrak{r}_{\kappa,m,p}^{(o)}(a,\omega;\varphi,x+y).\label{trav rm}
\end{align}
In particular, the operator $\mathbf{R}_m^{(o)}$ is real, reversible and momentum preserving. Moreover, for any $s\in[s_0,S],$
\begin{align}
	\|\mathfrak{r}_{\kappa,m,p}^{(o)}\|_{s}^{q,\gamma,\mathcal{O}}&\lesssim_{s,q,d,M}\varepsilon\|\widetilde{r}\|_{s+\sigma_{m}(M)}^{q,\gamma,\mathcal{O}},\label{e-rm}\\
	\|\Delta_{12}\mathfrak{r}_{\kappa,m,p}^{(o)}\|_{s}&\lesssim_{s,q,d,M}\varepsilon\big(\|\Delta_{12}\widetilde{r}\|_{s+\sigma_{m}(M)}+\|\Delta_{12}\widetilde{r}\|_{s_0+\sigma_{m}(M)}\max_{k\in\{1,2\}}\|\widetilde{r}_k\|_{s+\sigma_{m}(M)}\big).\label{e12-rm}
\end{align}
\item The matrix operator $\mathbf{S}_{m,M}$ is real, reversible, momentum preserving and in $OPS^{-M}$. It satisfies the following estimates, for any $s\in[s_0,S]$ and $\alpha\in\mathbb{N},$
\begin{align}
	\|\mathbf{S}_{m,M}\|_{-M,s,\alpha}^{q,\gamma,\mathcal{O}}&\lesssim_{s,q,d,M,\alpha}\varepsilon\|\widetilde{r}\|_{s+\sigma_{m}(\alpha,M)}^{q,\gamma,\mathcal{O}}, \label{e.SmM} \\
	\|\Delta_{12}\mathbf{S}_{m,M}\|_{-M,s,\alpha}&\lesssim_{s,q,d,M,\alpha}\varepsilon\big(\|\Delta_{12}\widetilde{r}\|_{s+\sigma_{m}(\alpha,M)}+\|\Delta_{12}\widetilde{r}\|_{s_0+\sigma_{m}(\alpha,M)}\max_{k\in\{1,2\}}\|\widetilde{r}_k\|_{s+\sigma_{m}(\alpha,M)}\big)\,\,\,\, \label{e12.SmM}
\end{align}
for some $\sigma_m(\alpha,M)\triangleq\sigma_m(s_0,q,d,\alpha,M)>0.$
\end{enumerate} 

\end{prop}
\begin{proof}
	We proceed by induction on $m\in\llbracket0,M-1\rrbracket$. The case $m=0$ follows from Proposition \ref{prop L0}. Assume that the step $m\in\llbracket0,M-2\rrbracket$ is already settled and we have to prove the claims at the step $m+1.$ Consider an a priori unknown invertible transformation $\phi_{m+1}$ in the form
	\begin{equation}\label{form.phimplus1}
	    \Phi_{m+1}\triangleq\mathbb{I}_{2}+\mathbf{X}_{m+1},\qquad\mathbf{X}_{m+1}\triangleq\begin{pmatrix}
		0 & \mathfrak{p}_{+,m+1}\partial_{x}^{-m-2}\\
		\mathfrak{p}_{-,m+1}\partial_{x}^{-m-2} & 0
	\end{pmatrix}.
	\end{equation}
	Then, the following identity holds
	\begin{equation}\label{compo Lm phim+1}
		\begin{aligned}
			\mathcal{L}^{(m)}\Phi_{m+1}&=\Phi_{m+1}\big(\omega\cdot\partial_{\varphi}\mathbb{I}_{2}+\ii\boldsymbol{\Omega}(a,D)+\varepsilon\partial_{x}\mathbf{F}+\mathbf{R}_m^{(d)}\big)\\
			&\quad+\big[\omega\cdot\partial_{\varphi}\mathbb{I}_{2}+\ii\boldsymbol{\Omega}(a,D)+\varepsilon\partial_{x}\mathbf{F}\,,\,\mathbf{X}_{m+1}\big]+\mathbf{R}_m^{(o)}\\
			&\quad+\big[\mathbf{R}_m^{(d)}\,,\,\mathbf{X}_{m+1}\big]+\mathbf{R}_m^{(o)}\mathbf{X}_{m+1}+\mathbf{S}_{m,M}\Phi_{m+1}.
		\end{aligned}
	\end{equation}
	Our purpose is to properly choose the operator $\mathbf{X}_{m+1}$ to get rid of the anti-diagonal terms of order $-m-1$. Such terms appear in the second line in \eqref{compo Lm phim+1}. One has
	\begin{align*}
		&\big[\omega\cdot\partial_{\varphi}\mathbb{I}_2+\ii\boldsymbol{\Omega}(a,D)+\varepsilon\partial_{x}\mathbf{F},\mathbf{X}_{m+1}\big]=\begin{pmatrix}
			0 & A_{+,m+1}\\
			A_{-,m+1} & 0
		\end{pmatrix},
	\end{align*}
where for any $\kappa\in\{-,+\},$
\begin{equation}\label{Ak-m+1}
	\begin{aligned}
		A_{\kappa,m+1}&\triangleq[\omega\cdot\partial_{\varphi},\mathfrak{p}_{\kappa,m+1}\partial_{x}^{-m-2}]+\kappa[\ii\Omega(a,D),\mathfrak{p}_{\kappa,m+1}\partial_{x}^{-m-2}]_{\rm a}\\
		&\quad+\varepsilon\Big(\partial_{x}\, f_{\kappa}\mathfrak{p}_{\kappa,m+1}\partial_{x}^{-m-2}-\mathfrak{p}_{\kappa,m+1}\partial_{x}^{-m-1}f_{-\kappa}\Big);
	\end{aligned}
\end{equation}
we recall the notations $[A,B]$ and $[A,B]_{\rm a}$ to respectively designate the commutator and anti-commutator between operators $A$ and $B$, namely
$$[A,B]\triangleq AB-BA,   \qquad [A,B]_{\rm a} \triangleq AB+BA.$$
We fix $\kappa\in\{-,+\}$ and we separately analyze the contributions to $A_{\kappa,m+1}$ in \eqref{Ak-m+1}.
\\[1mm]
\noindent $\blacktriangleright$ {\sc Contribution from $[\omega\cdot\partial_{\varphi},\mathfrak{p}_{\kappa,m+1}\partial_{x}^{-m-2}]$.}
A direct computation shows that
\begin{align}
	[\omega\cdot\partial_{\varphi},\mathfrak{p}_{\kappa,m+1}\partial_{x}^{-m-2}]=(\omega\cdot\partial_{\varphi}\mathfrak{p}_{\kappa,m+1})\partial_{x}^{-m-2}. \label{pasqua1}
\end{align}
\noindent $\blacktriangleright$ {\sc Contribution from $[\ii\Omega(a,D),\mathfrak{p}_{\kappa,m+1}\partial_{x}^{-m-2}]_{\rm a}$.}
Next, by the expansion in \eqref{hom exp omgD} (recalling also the notation in \eqref{frakNM.def}), we write
\begin{align}
	[\ii\Omega(a,D),\mathfrak{p}_{\kappa,m+1}\partial_{x}^{-m-2}]_{\rm a}&=[a\partial_{x},\mathfrak{p}_{\kappa,m+1}\partial_{x}^{-m-2}]_{\rm a}+\sum_{p=1}^{\mathfrak{N}_M}\alpha_{p}a^{1-2p}[\partial_{x}^{1-2p},\mathfrak{p}_{\kappa,m+1}\partial_{x}^{-m-2}]_{\rm a}\\
	&\quad+[S_{M}(a,D),\mathfrak{p}_{\kappa,m+1}\partial_{x}^{-m-2}]_{\rm a}. \label{pasqua2}
\end{align}
%
One readily has
\begin{align}
	[a\partial_{x},\mathfrak{p}_{\kappa,m+1}\partial_{x}^{-m-2}]_{\rm a}=2a\mathfrak{p}_{\kappa,m+1}\partial_{x}^{-m-1}+a(\partial_{x}\mathfrak{p}_{\kappa,m+1})\partial_{x}^{-m-2}. \label{pasqua2.1}
\end{align}
Now, from Proposition \ref{lem compo commu hom exp}-$(iii)$, we have
\begin{align}
	&\sum_{p=1}^{\mathfrak{N}_M}\alpha_{p}a^{1-2p}[\partial_{x}^{1-2p},\mathfrak{p}_{\kappa,m+1}\partial_{x}^{-m-2}]_{\rm a}\\
	&=\sum_{p=1}^{\mathfrak{N}_M}\alpha_{p}a^{1-2p}\bigg(2\mathfrak{p}_{\kappa,m+1}\partial_{x}^{-m-1-2p}+\sum_{p_2=1}^{M-m-2p-2}C_{p_2,1-2p}(\partial_{x}^{p_2}\mathfrak{p}_{\kappa,m+1})\partial_{x}^{-m-1-2p-p_2}+\mathfrak{R}_{\kappa,M,p}\bigg), \label{pasqua2.1.1}
\end{align}
where, for any $s\in[s_0,S]$ and any $\alpha\in\mathbb{N},$
\begin{align*}
	\max_{p\in\llbracket 1,\mathfrak{N}_M\rrbracket}\|\mathfrak{R}_{\kappa,M,p}\|_{-M,s,\alpha}^{q,\gamma,\mathcal{O}}&\lesssim_{s,q,d,M,\alpha}\|\mathfrak{p}_{\kappa,m+1}\|_{s+\sigma_{m+1}(\alpha,M)}^{q,\gamma,\mathcal{O}},\\
	\max_{p\in\llbracket 1,\mathfrak{N}_M\rrbracket}\|\Delta_{12}\mathfrak{R}_{\kappa,M,p}\|_{-M,s,\alpha}&\lesssim_{s,q,d,M,\alpha}\|\Delta_{12}\mathfrak{p}_{\kappa,m+1}\|_{s+\sigma_{m+1}(\alpha,M)},
\end{align*}
for some $\sigma_{m+1}(\alpha,M)>\sigma_{m}(\alpha,M).$ We mention that in what follows, we may keep writing $\sigma_{m+1}(\alpha,M)$ to denote a loss a derivatives at step $m+1$ for the smoothing remainders. Hence, its value may change (increase) from one line to another but it still remains finite and independent of $s$. 
Collecting \eqref{pasqua2.1}, \eqref{pasqua2.1.1} into \eqref{pasqua2}, we compactly write
\begin{align}\label{pasqua2.1.2}
    [\ii\Omega(a,D),\mathfrak{p}_{\kappa,m+1}\partial_{x}^{-m-2}]_{\rm a}&=2a\mathfrak{p}_{\kappa,m+1}\partial_{x}^{-m-1}+\sum_{p=m+2}^{M-1}\hat{\mathfrak{p}}_{\kappa,m+1,p}\partial_{x}^{-p}+\mathfrak{S}_{\kappa,M,1}
\end{align}
with $\hat{\mathfrak{p}}_{\kappa,m+1,p}$ related to $\mathfrak{p}_{\kappa,m+1}$ and its derivatives up to a finite order depending on $M$, and with $\mathfrak{S}_{\kappa,M,1}$  related to  $\mathfrak{R}_{\kappa,M,p}$, $p\in\llbracket1,\mathfrak{N}_{M}\rrbracket$, and $[S_{M}(a,D),\mathfrak{p}_{\kappa,m+1}\partial_{x}^{-m-2}]_{\rm a}$. They satisfy
\begin{align*}
	\max_{p\in\llbracket m+2,M-1\rrbracket}\|\hat{\mathfrak{p}}_{\kappa,m+1,p}\|_{s}^{q,\gamma,\mathcal{O}}&\lesssim_{s,q,d,M}\|\mathfrak{p}_{\kappa,m+1}\|_{s+\sigma_{m+1}(M)}^{q,\gamma,\mathcal{O}},\\
	\max_{p\in\llbracket m+2,M-1\rrbracket}\|\Delta_{12}\hat{\mathfrak{p}}_{\kappa,m+1,p}\|_{s}&\lesssim_{s,q,d,M}\|\Delta_{12}\mathfrak{p}_{\kappa,m+1}\|_{s+\sigma_{m+1}(M)},
\end{align*}
with $\sigma_{m+1}(M)>\sigma_{m}(M)$ which may also vary from line to line as $\sigma_{m+1}(\alpha,M)$. The remainder also satisfies 
\begin{align*}
	\|\mathfrak{S}_{\kappa,M,1}\|_{-M,s,\alpha}^{q,\gamma,\mathcal{O}}&\lesssim_{s,q,d,M,\alpha}\|\mathfrak{p}_{\kappa,m+1}\|_{s+\sigma_{m+1}(\alpha,M)}^{q,\gamma,\mathcal{O}},\\
	\|\Delta_{12}\mathfrak{S}_{\kappa,M,1}\|_{-M,s,\alpha}&\lesssim_{s,q,d,M,\alpha}\|\Delta_{12}\mathfrak{p}_{\kappa,m+1}\|_{s+\sigma_{m+1}(\alpha,M)}.
\end{align*}
\noindent $\blacktriangleright$ {\sc Contribution from $\partial_{x}\,f_{\kappa}\mathfrak{p}_{\kappa,m+1}\partial_{x}^{-m-2}-\mathfrak{p}_{\kappa,m+1}\partial_{x}^{-m-1}f_{-\kappa}$.}
We now analyze the last term in \eqref{Ak-m+1}. Note that
\begin{align}
	\partial_{x}\,f_{\kappa}\mathfrak{p}_{\kappa,m+1}\partial_{x}^{-m-2}-\mathfrak{p}_{\kappa,m+1}\partial_{x}^{-m-1}f_{-\kappa}&=\big(\partial_{x}(f_{\kappa}\mathfrak{p}_{\kappa,m+1})\big)\partial_{x}^{-m-2}+f_{\kappa}\mathfrak{p}_{\kappa,m+1}\partial_{x}^{-m-1}\\
	&\quad-\mathfrak{p}_{\kappa,m+1}f_{-\kappa}\partial_{x}^{-m-1}-\mathfrak{p}_{\kappa,m+1}[\partial_{x}^{-m-1},f_{-\kappa}].\label{pasqua3}
\end{align}
Using Proposition \ref{lem compo commu hom exp}-$(i)$, we write
\begin{align}
	&\partial_{x}\, f_{\kappa}\mathfrak{p}_{\kappa,m+1}\partial_{x}^{-m-2}-\mathfrak{p}_{\kappa,m+1}\partial_{x}^{-m-1}f_{-\kappa}\\
	&=\mathfrak{p}_{\kappa,m+1}\big(f_{\kappa}-f_{-\kappa}\big)\partial_{x}^{-m-1}+\sum_{p=m+2}^{M-1}\widetilde{\mathfrak{p}}_{\kappa,m+1,p}\partial_{x}^{-p}+\mathfrak{S}_{\kappa,M,2}, \label{pasqua3.1}
\end{align}
where $\widetilde{\mathfrak{p}}_{\kappa,m+1,p}$ is (up to multiplicative constants) sum and products of $f_{+},f_-,\mathfrak{p}_{\kappa,m+1}$ and their derivatives up to a finite order depending on $M.$ Therefore, we have
\begin{align*}
	\max_{p\in\llbracket m+2,M-1\rrbracket}\|\widetilde{\mathfrak{p}}_{\kappa,m+1,p}\|_{s}^{q,\gamma,\mathcal{O}}&\lesssim_{s,q,d,M}\max_{\kappa'\in\{-,+\}}\|f_{\kappa'}\|_{s+\sigma_{m+1}(M)}^{q,\gamma,\mathcal{O}}+\|\mathfrak{p}_{k,m+1}\|_{s+\sigma_{m+1}(M)}^{q,\gamma,\mathcal{O}},\\
	\max_{p\in\llbracket m+2,M-1\rrbracket}\|\Delta_{12}\widetilde{\mathfrak{p}}_{\kappa,m+1,p}\|_{s}&\lesssim_{s,q,d,M}\max_{\kappa'\in\{-,+\}}\|\Delta_{12}f_{\kappa'}\|_{s+\sigma_{m+1}(M)}+\|\Delta_{12}\mathfrak{p}_{\kappa,m+1}\|_{s+\sigma_{m+1}(M)},
\end{align*}
and
\begin{align*}
	\|\mathfrak{S}_{\kappa,M,2}\|_{-M,s,\alpha}^{q,\gamma,\mathcal{O}}&\lesssim_{s,q,d,M,\alpha}\|\mathfrak{p}_{\kappa,m+1}\|_{s+\sigma_{m+1}(\alpha,M)}^{q,\gamma,\mathcal{O}},\\
	\|\Delta_{12}\mathfrak{S}_{\kappa,M,2}\|_{-M,s,\alpha}&\lesssim_{s,q,d,M,\alpha}\|\Delta_{12}\mathfrak{p}_{\kappa,m+1}\|_{s+\sigma_{m+1}(\alpha,M)}.
\end{align*}
\noindent $\blacktriangleright$ {\sc Expansion of $A_{\kappa,m+1}$.}
Combining the expansions \eqref{pasqua1}, \eqref{pasqua2.1.2} and \eqref{pasqua3.1} into \eqref{Ak-m+1}, we end up with
\begin{align}
    A_{\kappa,m+1}& =\big(2a\kappa+\varepsilon(f_{\kappa}-f_{-\kappa})\big)\mathfrak{p}_{\kappa,m+1}\partial_{x}^{-m-1} \\
    & + (\omega\cdot\partial_{\varphi}\mathfrak{p}_{\kappa,m+1})\partial_{x}^{-m-2} +\sum_{p=m+2}^{M-1}(\kappa\hat{\mathfrak{p}}_{\kappa,m+1,p}+\varepsilon\widetilde{\mathfrak{p}}_{\kappa,m+1,p})\partial_{x}^{-p}+\kappa\mathfrak{S}_{\kappa,M,1}+\varepsilon\mathfrak{S}_{\kappa,M,2}.
\end{align}
In order to get rid of the terms of order $\partial_{x}^{-m-1}$ in \eqref{compo Lm phim+1}, with $\bR_{m}^{(o)}$ as in \eqref{bR.m.off}, we make the following choice 
\begin{align}\label{choice frak p}
	\mathfrak{p}_{\kappa,m+1}(a,\omega;\varphi,x)\triangleq-\frac{\mathfrak{r}_{\kappa,m,m+1}^{(o)}(a,\omega;\varphi,x)}{2a\kappa+\varepsilon\big(f_{\kappa}(\widetilde{r})(a,\omega;\varphi,x)-f_{-\kappa}(\widetilde{r})(a,\omega;\varphi,x)\big)}\cdot
\end{align}
By the  induction hypotheses \eqref{rev rm}-\eqref{trav rm} together with \eqref{rev fpm}-\eqref{trav fpm}, we get
\begin{align}
	\mathfrak{p}_{\kappa,m+1}(a,\omega;-\varphi,-x)&=(-1)^{m}\mathfrak{p}_{\kappa,m+1}(a,\omega;\varphi,x), \label{rev pm+1} \\
	\forall\, y\in\mathbb{T},\quad\mathfrak{p}_{\kappa,m+1}(a,\omega;\varphi-\vec{\jmath}y,x)&=\mathfrak{p}_{\kappa,m+1}(a,\omega;\varphi,x+y). \label{trav pm+1}
\end{align}
This implies, by virtue of Lemma \ref{lem:revmomHE}, the reversibility and momentum preserving properties of $\mathbf{X}_{m+1}$ in \eqref{form.phimplus1}.  Using product and composition laws of Lemma \ref{lem functions}, \eqref{e-fk}-\eqref{e12-fk} and \eqref{e-rm}-\eqref{e12-rm} together with \eqref{esti r I0} and \eqref{ansatz.fI.small} (with $\sigma_{\rm max}>\sigma_m(M)$), we get
\begin{align}
	\|\mathfrak{p}_{\kappa,m+1}\|_{s}^{q,\gamma,\mathcal{O}}&\lesssim_{s,q,d,M}\varepsilon\|\widetilde{r}\|_{s+\sigma_{m}(M)}^{q,\gamma,\mathcal{O}},\label{estim:frakpmplus1}\\
	\|\Delta_{12}\mathfrak{p}_{\kappa,m+1}\|_{s}&\lesssim_{s,q,d,M}\varepsilon\big(\|\Delta_{12}\widetilde{r}\|_{s+\sigma_{m}(M)}+\|\Delta_{12}\widetilde{r}\|_{s_0+\sigma_{m}(M)}\max_{\ell\in\{1,2\}}\|\widetilde{r}_{\ell}\|_{s+\sigma_{m}(M)}\big).\label{estim12:frakpmplus1}
\end{align}
With the choice \eqref{choice frak p}, we infer
\begin{equation}\label{whats left bB}
  \big[\omega\cdot\partial_{\varphi}\mathbb{I}_{2}+\ii\boldsymbol{\Omega}(a,D)+\varepsilon\partial_{x}\mathbf{F}\,,\,\mathbf{X}_{m+1}\big]+\mathbf{R}_m^{(o)}  = \bB_{m+1}^{(o)} \triangleq \begin{pmatrix}
	0 & B_{+,m+1}\\
	B_{-,m+1} & 0
\end{pmatrix},
\end{equation}
where, for any $\kappa\in\{+,-\},$ the operator $B_{\kappa,m+1}$ is in $OPS^{-m-2}$ and writes
\begin{equation}
    \begin{aligned}
        B_{\kappa,m+1}\triangleq (\omega\cdot\partial_{\varphi}\mathfrak{p}_{\kappa,m+1})\partial_{x}^{-m-2} +\sum_{p=m+2}^{M-1}(\mathfrak{r}_{\kappa,m,p}^{(o)}\partial_{x}^{-p} +\kappa\hat{\mathfrak{p}}_{\kappa,m+1,p}+\varepsilon\widetilde{\mathfrak{p}}_{\kappa,m+1,p})\partial_{x}^{-p}+\kappa\mathfrak{S}_{\kappa,M,1}+\varepsilon\mathfrak{S}_{\kappa,M,2}.
    \end{aligned}
\end{equation}
According to the explicit expressions \eqref{form.phimplus1} and \eqref{choice frak p}, then using \eqref{normpseudo:matrix}, Proposition \ref{properties OPS}-$(i)$-$(iii)$-$(iv)$ together with \eqref{estim:frakpmplus1}, \eqref{esti r I0} and \eqref{ansatz.fI.small}, we infer
$$\|\mathbf{X}_{m+1}\|_{-m-2,s_0,0}^{q,\gamma,\mathcal{O}}\lesssim_{q,d}\max_{\kappa\in\{-,+\}}\|\mathfrak{p}_{\kappa,m+1}\|_{s_0}^{q,\gamma,\mathcal{O}}\lesssim_{q,d}\varepsilon.$$
Thus, the map $\Phi_{m+1}=\mathbb{I}_2+\mathbf{X}_{m+1}$ is  also invertible by the Neumann series argument, satisfying
\begin{equation}\label{phi.m+1 d and o}
    \begin{aligned}
    \Phi_{m+1}^{-1} & = \big( \mathbb{I}_2+\mathbf{X}_{m+1} \big)^{-1} = \big( \mathbb{I}_2 - \mathbf{X}_{m+1} \big)  \big( \mathbb{I}_2 -\mathbf{X}_{m+1}^2 \big)^{-1} \\
    & = \big( \mathbb{I}_2 -\mathbf{X}_{m+1}^2 \big)^{-1} - \bX_{m+1}\big( \mathbb{I}_2 -\mathbf{X}_{m+1}^2 \big)^{-1} \,,
\end{aligned}
\end{equation}
where $\big( \mathbb{I}_2 -\mathbf{X}_{m+1}^2 \big)^{-1}$ is also well defined by means of Neumann series.
We set 
\begin{equation}\label{bold.Phi.update}
    \boldsymbol{\Phi}_{m+1}\triangleq\boldsymbol{\Phi}_{m}\circ\Phi_{m+1}=\boldsymbol{\Phi}_{m}\circ\big(\mathbb{I}_2+\mathbf{X}_{m+1}\big),\qquad\mathcal{L}^{(m)}\triangleq\boldsymbol{\Phi}_{m}^{-1}\mathcal{L}^{(0)}\boldsymbol{\Phi}_{m}.
\end{equation}
Therefore, by \eqref{compo Lm phim+1}, \eqref{whats left bB} and \eqref{bold.Phi.update}, we obtain
\begin{align}
	\mathcal{L}^{(m+1)}&\triangleq \boldsymbol{\Phi}_{m+1}^{-1} \cL^{(0)}\boldsymbol{\Phi}_{m+1} =   \Phi_{m+1}^{-1}\mathcal{L}^{(m)}\Phi_{m+1}\\
    & = \omega\cdot\partial_{\varphi}\mathbb{I}_{2}+\ii\boldsymbol{\Omega}(a,D)+\varepsilon\partial_{x}\mathbf{F} + \bR_{m}^{(d)}  + \Phi_{m+1}^{-1} \big( \bB_{m+1}^{(o)} + [\bR_{m}^{(d)},\bX_{m}] +  \bR_{m}^{(o)} \bX_{m+1} + \bS_{m,M} \Phi_{m+1} \big) \\
	&=\omega\cdot\partial_{\varphi}\mathbb{I}_{2}+\ii\boldsymbol{\Omega}(a,D)+\varepsilon\partial_{x}\mathbf{F}+\mathbf{R}_{m+1}^{(d)}+\mathbf{R}_{m+1}^{(o)}+\mathbf{S}_{m+1,M},
\end{align}
where, by using \eqref{phi.m+1 d and o} and \eqref{form.phimplus1} to split between the diagonal and off-diagonal contributions, and by expanding in homogeneinity up to order $1-M$, the operators $\mathbf{R}_{m+1}^{(d)},$ $\mathbf{R}_{m+1}^{(o)}$ and $\mathbf{S}_{m+1,M}$ are defined through the following relations
\begin{align}
    \bR_{m+1}^{(d)} + \wt\bS_{m+1,M}^{(d)} & \triangleq  \bR_{m}^{(d)}+ \big( \mathbb{I}_2 -\mathbf{X}_{m+1}^2 \big)^{-1} \bR_{m}^{(o)} \bX_{m+1} - \bX_{m+1}\big( \mathbb{I}_2 -\mathbf{X}_{m+1}^2 \big)^{-1} \big(  \bB_{m+1}^{(o)} + [\bR_{m}^{(d)},\bX_{m}] \big) , \label{impl.Rm+1 diag} \\
    \bR_{m+1}^{(o)} + \wt\bS_{m+1,M}^{(o)} &  \triangleq  \big( \mathbb{I}_2 -\mathbf{X}_{m+1}^2 \big)^{-1} \big( \bB_{m+1}^{(o)} + [\bR_{m}^{(d)},\bX_{m}] \big) - \bX_{m+1}\big( \mathbb{I}_2 -\mathbf{X}_{m+1}^2 \big)^{-1} \bR_{m}^{(o)} \bX_{m+1}, \label{impl.Rm+1 off} \\
     \bS_{m+1,M} & \triangleq  \wt\bS_{m+1,M}^{(d)} +  \wt\bS_{m+1,M}^{(o)} + \Phi_{m+1}^{-1} \bS_{m,M} \Phi_{m+1}, \label{impl.bSm+1}
\end{align}
with $ \bR_{m+1}^{(d)}$ and $ \bR_{m+1}^{(o)}$ satisfying the homogeneous expansions up to order $1-M$ in \eqref{bR.m.diag} and \eqref{bR.m.off} at the step $m+1$, respectively.
The symmetry relations \eqref{rev rm diag}-\eqref{trav rm diag}, \eqref{rev rm}-\eqref{trav rm} at the step $m+1$ and $\bS_{m+1,M}$ being real, reversible and momentum preserving follow by \eqref{impl.Rm+1 diag}-\eqref{impl.bSm+1}, \eqref{whats left bB}, \eqref{rev pm+1}-\eqref{trav pm+1}, Proposition \ref{prop L0}-$(i)$,$(ii)$, Remark \ref{compo:RevMom}, Lemma \ref{lem:revmomHE}, Proposition \ref{lem compo commu hom exp} and the induction hypotheses on $\bR_{m}^{(d)}$, $\bR_{m}^{(o)}$, $\bS_{m,M}$. The estimates \eqref{e-rm-diag}-\eqref{e12-rm-diag}, \eqref{e-rm}-\eqref{e12-rm-diag} and \eqref{e.SmM}-\eqref{e12.SmM} at the step $m+1$ follow by \eqref{impl.Rm+1 diag}-\eqref{impl.bSm+1}, \eqref{whats left bB}, \eqref{estim:frakpmplus1}-\eqref{estim12:frakpmplus1} and the induction hypotheses on $\bR_{m}^{(d)}$, $\bR_{m}^{(o)}$, $\bS_{m,M}$. This concludes the proof of Proposition \ref{prop.blockanti}.
\end{proof}
As a consequence of the previous proposition, we have the following result.
\begin{prop}\label{prop.blockanti.final}
	Let $M\in\mathbb{N}^*$. There exists a real, reversibility and momentum preserving, invertible matrix operator $\boldsymbol{\Phi}_{M-1}$ satisfying the tame estimates for any $s\in[s_0,S],$
    \begin{align}
        \|\boldsymbol{\Phi}_{M-1}^{\pm1}-\mathbb{I}_2\|_{-2,s,0}^{q,\gamma,\mathcal{O}}&\lesssim_{s,q,d,M}\varepsilon\big(1+\|\mathfrak{I}_0\|_{s+\sigma(M)}^{q,\gamma,\mathcal{O}}\big),\label{e-PhiM}
        \\
        \|\Delta_{12}\boldsymbol{\Phi}_{M-1}\|_{-2,s,0}&\lesssim_{s,q,d,M}\varepsilon\big(\|\Delta_{12}i\|_{s+\sigma(M)}+\|\mathfrak{I}_0\|_{s+\sigma(M)}\|\Delta_{12}i\|_{s_0+\sigma(M)}\big)\label{e12-PhiM}
    \end{align}
    for some $\sigma(M)>0$. In addition, denoting $\boldsymbol{\Phi}_{M-1}^{\star}$ the $\mathbf{L}^2(\mathbb{T})$-adjoint of the transformation $\boldsymbol{\Phi}_{M-1},$ we have that $\boldsymbol{\Phi}_{M-1}^{\star}$ is a real, reversibility and momentum preserving operator satisfying the following estimates for any $s\in[s_0,S],$
    \begin{align}
        \|(\boldsymbol{\Phi}_{M-1}^{\star})^{\pm1}-\mathbb{I}_2\|_{-2,s,0}^{q,\gamma,\mathcal{O}}&\lesssim_{s,q,d,M}\varepsilon\big(1+\|\mathfrak{I}_0\|_{s+\sigma(M)}^{q,\gamma,\mathcal{O}}\big),\label{e-PhiMstar}
        \\
        \|\Delta_{12}(\boldsymbol{\Phi}_{M-1}^{\star})^{\pm1}\|_{-2,s,0}&\lesssim_{s,q,d,M}\varepsilon\big(\|\Delta_{12}i\|_{s+\sigma(M)}+\|\mathfrak{I}_0\|_{s+\sigma(M)}\|\Delta_{12}i\|_{s_0+\sigma(M)}\big).\label{e12-PhiMstar}
    \end{align}
    Moreover, we have
	\begin{equation}\label{calLM2}
		\mathcal{L}^{(M-1)}\triangleq\boldsymbol{\Phi}_{M-1}^{-1}\mathcal{L}^{(0)}\boldsymbol{\Phi}_{M-1}=\omega\cdot\partial_{\varphi}\mathbb{I}_2+\ii\boldsymbol{\Omega}(a,D)+\varepsilon\partial_{x}\mathbf{F}+\mathbf{R}_{[-1]}^{(d)}+\mathbf{S}_{M},
	\end{equation}
where $\mathbf{R}_{[-1]}^{(d)}$ and $\mathbf{S}_{M}$ are real, reversible and momentum preserving remainders. The operator $\mathbf{R}_{[-1]}^{(d)}$ admits an homogeneous expansion of degree $-1$ up to order $1-M$ in the form
\begin{equation}\label{bfRd}
    \mathbf{R}_{[-1]}^{(d)}=\begin{pmatrix}
	\displaystyle\sum_{p=1}^{M-1}\mathfrak{r}_{+,p}^{(d)}\partial_{x}^{-p} & 0\\
	0 & \displaystyle\sum_{p=1}^{M-1}\mathfrak{r}_{-,p}^{(d)}\partial_{x}^{-p}
\end{pmatrix},
\end{equation}
with for any $\kappa\in\{-,+\}$ and $p\in\llbracket 1,M-1\rrbracket,$ the function $\mathfrak{r}_{\kappa,p}^{(d)}\triangleq\mathfrak{r}_{\kappa,p}^{(d)}(a,\omega;\varphi,x)\in\mathbb{R}$ satisfying
\begin{align}
	\mathfrak{r}_{\kappa,p}^{(d)}(a,\omega;-\varphi,-x)&=(-1)^{p+1}\mathfrak{r}_{\kappa,p}^{(d)}(a,\omega;\varphi,x), \label{rev rm diag final} \\
	\forall \, y\in\mathbb{T},\quad\mathfrak{r}_{\kappa,p}^{(d)}(a,\omega;\varphi-\vec{\jmath}y,x)&=\mathfrak{r}_{\kappa,p}^{(d)}(a,\omega;\varphi,x+y). \label{trav rm diag final}
\end{align}
Moreover, for any $\kappa\in\{-,+\}$ and $s\in[s_0,S],$
\begin{align}
	\max_{p\in\llbracket 1,M-1\rrbracket}\|\mathfrak{r}_{\kappa,p}^{(d)}\|_{s}^{q,\gamma,\mathcal{O}}&\lesssim_{s,q,d,M}\varepsilon \big( 1+\|\mathfrak{I}_0\|_{s+\sigma(M)}^{q,\gamma,\mathcal{O}}\big),\label{estim frakrkp}\\
	\max_{p\in\llbracket 1,M-1\rrbracket}\|\Delta_{12}\mathfrak{r}_{\kappa,p}^{(d)}\|_{s}&\lesssim_{s,q,d,M}\varepsilon\big(\|\Delta_{12}i\|_{s+\sigma(M)}+\|\Delta_{12}i\|_{s_0}\max_{\ell\in\{1,2\}}\|\mathfrak{I}_{\ell}\|_{s+\sigma(M)}\big)\label{estim frakrkp12}
\end{align}
for some $\sigma(M)\triangleq\sigma(s_0,q,d,M)>0.$
The matrix operator $\mathbf{S}_{M}$ is in $OPS^{-M}$ and satisfies the following estimates for any $s\in[s_0,S]$ and $\alpha\in\mathbb{N},$
\begin{align}
	\|\mathbf{S}_{M}\|_{-M,s,\alpha}^{q,\gamma,\mathcal{O}}&\lesssim_{s,q,d,M,\alpha}\varepsilon\big(1+\|\mathfrak{I}_0\|_{s+\sigma(\alpha,M)}^{q,\gamma,\mathcal{O}}\big), \label{est.bSM} \\
	\|\Delta_{12}\mathbf{S}_{M}\|_{-M,s,\alpha}&\lesssim_{s,q,d,M,\alpha}\varepsilon\big(\|\Delta_{12}i\|_{s+\sigma(\alpha,M)}+\|\Delta_{12}i\|_{s_0}\max_{\ell\in\{1,2\}}\|\mathfrak{I}_{\ell}\|_{s+\sigma(\alpha,M)}\big) \label{est12.bSM}
\end{align}
for some $\sigma(\alpha,M)\triangleq\sigma(s_0,q,d,\alpha,M)>0.$
\end{prop}
\begin{proof}
    The claims follows directly from Proposition \ref{prop.blockanti}, noting that, for $m=M-1$, the operator in \eqref{calLM} coincides with the one in \eqref{calLM2}, having $\bR_{M-1}^{(o)}\equiv 0$ and setting $\bR_{[-1]}^{(d)}\triangleq\bR^{(d)}_{M-1}$ and $\bS_{M}\triangleq \bS_{M-1,M}$. In view of \eqref{compoPhim}, the estimate \eqref{e-PhiM} for $\boldsymbol{\Phi}_{M-1}$ is a consequence of Proposition \ref{properties OPS}-$(i)$-$(iii)$-$(iv)$, \eqref{def:Xm}, \eqref{e-pm}, \eqref{esti r I0} and \eqref{ansatz.fI.small} provided that $\sigma_{\rm max}$ is chosen large enough. The analogous estimate for $\boldsymbol{\Phi}_{M-1}^{-1}$ follows with a similar argument making appeal to Neumann series. In addition, we note that, from \eqref{compoPhim}, we have
    $$\Delta_{12}\boldsymbol{\Phi}_{M-1}=\sum_{m=1}^{M-1}\big(\mathbb{I}_2+\mathbf{X}_{1}(i_1)\big)\circ\ldots\circ\big(\mathbb{I}_2+\mathbf{X}_{m-1}(i_1)\big)\circ\Delta_{12}\mathbf{X}_m\circ\big(\mathbb{I}_2+\mathbf{X}_{m+1}(i_2)\big)\circ\ldots\circ\big(\mathbb{I}_2+\mathbf{X}_{M-1}(i_2)\big).$$
    And one sees from \eqref{def:Xm} that 
    $$\Delta_{12}\mathbf{X}_m=\begin{pmatrix}
		0 & (\Delta_{12}\mathfrak{p}_{+,m})\partial_{x}^{-m-1}\\
		(\Delta_{12}\mathfrak{p}_{-,m})\partial_{x}^{-m-1} & 0
	\end{pmatrix}.$$
    Therefore, the estimate \eqref{e12-PhiM} follows from Proposition \ref{properties OPS}-$(i)$-$(iii)$-$(iv)$, \eqref{e-pm}, \eqref{e12-pm}, \eqref{esti r I0}, \eqref{esti r I0d} and \eqref{ansatz.fI.small} provided that $\sigma_{\rm max}$ is chosen large enough. Let us now turn to the estimates for the adjoint. In view of \eqref{compoPhim}, one has
    $$\boldsymbol{\Phi}_{M-1}^{\star}=\big(\mathbb{I}_2+\mathbf{X}_{M-1}^{\star}\big)\circ\ldots\circ \big(\mathbb{I}_2+\mathbf{X}_{1}^{\star}\big).$$
    A simple computation using the explicit expression \eqref{def:Xm} shows that for any $m\in\llbracket1,M-1\rrbracket,$
    $$\mathbf{X}_m^{\star}=\begin{pmatrix}
        0 & (-1)^{m+1}\partial_x^{-m-1}\mathfrak{p}_{+,m}\\
        (-1)^{m+1}\partial_x^{-m-1}\mathfrak{p}_{-,m} & 0
    \end{pmatrix}.$$
   From this explicit expression, one can invoke Proposition \ref{properties OPS}-$(i)$-$(iii)$-$(iv)$, \eqref{e-pm} and \eqref{esti r I0} to get
   $$\|\mathbf{X}_m^
   {\star}\|_{-m-1,s,0}^{q,\gamma,\mathcal{O}}\lesssim_{s,q,d}\max_{\kappa\in\{-,+\}}\|\mathfrak{p}_{\kappa,m}\|_{s}^{q,\gamma,\mathcal{O}}\lesssim_{q,d}\varepsilon\big(1+\|\mathfrak{I}_0\|_{s+\sigma(M)}^{q,\gamma,\mathcal{O}}\big).$$
   The estimates \eqref{e-PhiMstar} and \eqref{e12-PhiMstar} for $\boldsymbol{\Phi}_{M-1}^{\star}$ follow immediately proceeding as before for $\boldsymbol{\Phi}_{M-1}$, with the estimates for the inverse adjoint obtained via Neumann series.  Finally, the estimates \eqref{estim frakrkp}-\eqref{estim frakrkp12} and \eqref{est.bSM}-\eqref{est12.bSM} follow by \eqref{e-rm-diag}-\eqref{e12-rm-diag} (setting $\fr_{\kappa,p}^{(d)}\triangleq \fr_{\kappa,M-1,p}^{(d)}$ for any $\kappa\in\{-,+\}$, $p=\llbracket1,M-1\rrbracket$), \eqref{e.SmM}-\eqref{e12.SmM} and Proposition \ref{prop Lnormal}-$(i)$. This achieves the proof of Proposition \ref{prop.blockanti.final}.
\end{proof}

\subsection{Transport reduction and Egorov argument}\label{sect.transport}
In this section, we straighten the first order vector field $\omega\cdot\pa_{\vf}\mathbb{I}_{2}+\big(\begin{smallmatrix}
			1 & 0\\
			0 & -1
		\end{smallmatrix}\big)a\partial_x +\varepsilon \pa_{x}\bF$  in the operator $\mathcal{L}^{(M-1)}$ described in \eqref{calLM2}  to a constant coefficients one, up to an arbitrarily small operator. We achieve this by conjugating the operator $\cL^{(M-1)}$ with a diffeomorphism of the torus, which is determined in the same spirit as in \cite[Prop. 6.2]{HR21}. The conjugation of the terms expanded in homogeneity up to order $M-1$ is computed via Egorov Theorem in Proposition \ref{Egorov thm}. We remark that new $-M$-smoothing remainders do not belong anymore to the class of pseudo-differential operators, but rather to a more general class of operators of $\mathcal{D}^{q}$-tame operators presented in Section \ref{sec:tameOP} and first introduced in \cite{BM18}. 

\smallskip
        
        First, we present the reduction of the transport operator.
\begin{prop}\label{prop trans red}
	Fix also $\upsilon\in(0,\tfrac{1}{q+2}].$ For any $(\mu_2,s_1,s_h,\sigma_1)$ satisfying
	\begin{equation}\label{parametres:transport}
	    \begin{array}{ll}
	    \mu_2\geqslant4\tau_1q+6\tau_1+3, & s_1\geqslant s_0,\\
	     s_h\geqslant\max(s_1,\tfrac{3}{2}\mu_2+s_0+\tau_1q+\tau_1+3),\qquad\qquad & \sigma_1\triangleq s_0+8\tau_1q+12\tau_1+12,
	\end{array}
	\end{equation}
	there exists $\varepsilon_0>0$ such that, if (recalling \eqref{scale.KAM})
	\begin{equation}\label{sml+bnd trans}
	    \varepsilon\gamma^{-1}N_0^{\mu_2}\leqslant\varepsilon_0
	\end{equation}
    and if \eqref{ansatz.fI.small} holds with $s_0+\sigma_{\rm max}>s_h+\sigma_1,$
	then the following hold true: 
\begin{enumerate}[label=(\roman*)]
	\item There exist $\mathtt{c}_{\pm}\in W^{q,\gamma,\infty}(\mathcal{O},\mathbb{C})$ satisfying, for any $\kappa\in\{-,+\},$
	\begin{align}
		\big\|(a,\omega)\mapsto\mathtt{c}_{\kappa}(a,\omega;i_0)-\kappa a\big\|^{q,\gamma,\mathcal{O}}&\lesssim_{q,d}\varepsilon,\label{estimazione ttc}\\
		\|\Delta_{12}\mathtt{c}_{\kappa}\|&\lesssim_d\varepsilon\|\Delta_{12}i\|_{s_0+\sigma_1};\label{estimazione ttc12}
	\end{align}
	\item There exist $\beta_{\pm}\in W^{q,\gamma,\infty}(\mathcal{O},H^{s}(\T))$ and $\wh\beta_{\pm}\in W^{q,\gamma,\infty}(\mathcal{O},H^{s}(\T)) $  related to each other by the relation
    \begin{equation}
        x \mapsto z \triangleq x +\beta_{\pm}(a,\omega;\vf,x) \qquad \Leftrightarrow \qquad z \mapsto x \triangleq z +  \wh\beta_{\pm}(a,\omega;\vf,z)
    \end{equation}
    satisfying
    \begin{align}
		\beta_{\kappa}(a,\omega;-\varphi,-x)&=-\beta_{\kappa}(a,\omega;\varphi,x),\label{symmet beta}\\
		\forall \, y\in\mathbb{T},\quad\beta_{\kappa}(a,\omega;\varphi-\vec{\jmath}y,x)&=\beta_{\kappa}(a,\omega;\varphi,x+y)\label{traveling beta}
	\end{align}
    and
	\begin{align}
		\widehat{\beta}_{\kappa}(a,\omega;-\varphi,-z)&=-\widehat{\beta}_{\kappa}(a,\omega;\varphi,z),\label{symmet betah}\\
		\forall \, y\in\mathbb{T},\quad\widehat{\beta}_{\kappa}(a,\omega;\varphi-\vec{\jmath}y,z)&=\widehat{\beta}_{\kappa}(a,\omega;\varphi,x+z).\label{traveling betah}
	\end{align}
Moreover, the following estimates hold, for any $s \in [s_0,S]$,
	    \begin{align}
	        \|\widehat{\beta}_{\kappa}\|_{s}^{q,\gamma,\mathcal{O}}\lesssim\|\beta_{\kappa}\|_{s}^{q,\gamma,\mathcal{O}}&\lesssim_{s,q,d}\varepsilon\gamma^{-1}\big(1+\|\mathfrak{I}_0\|_{s+\sigma_1}^{q,\gamma,\mathcal{O}}\big),\label{estimazione beta}\\
            \|\Delta_{12}\beta\|_{s_1}+\|\Delta_{12}\widehat{\beta}\|_{s_1}&\lesssim_{s_1,q,d}\varepsilon\gamma^{-1}\|\Delta_{12}i\|_{s_1+\sigma_1};\label{estimazione beta12}
            \end{align}
    \item   For any $\kappa\in\{-,+\},$ the transformation $\mathscr{B}_{\kappa}$, defined as
	$$\mathscr{B}_{\kappa}\rho\triangleq\big(1+\partial_{x}\beta_{\kappa}\big)(\mathcal{B}_{\kappa}\rho),\qquad(\mathcal{B}_{\kappa}\rho)(a,\omega;\varphi,x)\triangleq\rho\big(a,\omega;\varphi,x+\beta_{\kappa}(a,\omega;\varphi,x)\big)$$
	is invertible, with inverse given by
	$$\mathscr{B}_{\kappa}^{-1}\rho=\big(1+\partial_{x}\widehat{\beta}_{\kappa}\big)(\mathcal{B}_{\kappa}^{-1}\rho),\qquad(\mathcal{B}_{\kappa}^{-1}\rho)(a,\omega;\varphi,x)\triangleq\rho\big(a,\omega;\varphi,x+\widehat{\beta}_{\kappa}(a,\omega;\varphi,x)\big).$$
    Moreover, the maps $\mathcal{B}_{\kappa}^{\pm1},\sB_{\kappa}^{\pm 1}$ are reversibility and momentum preserving, and satisfy, for any $s\in[s_0,S]$
            \begin{align}
            \|\mathcal{B}_{\kappa}^{\pm1}\rho\|_{s}^{q,\gamma,\mathcal{O}}+\|\mathscr{B}_{\kappa}^{\pm1}\rho\|_{s}^{q,\gamma,\mathcal{O}}&\lesssim_{s,q,d}\|\rho\|_{s}^{q,\gamma,\mathcal{O}}+\|\beta_{\kappa}\|_{s+\sigma_1}^{q,\gamma,\mathcal{O}}\|\rho\|_{s_0}^{q,\gamma,\mathcal{O}},\label{eB}\\
		\|(\mathcal{B}_{\kappa}^{\pm1}-\textnormal{Id})\rho\|_{s}^{q,\gamma,\mathcal{O}}+\|(\mathscr{B}_{\kappa}^{\pm1}-\textnormal{Id})\rho\|_{s}^{q,\gamma,\mathcal{O}}&\lesssim_{s,q,d}\|\beta_{\kappa}\|_{s_0+\sigma_1}^{q,\gamma,\mathcal{O}}\|\rho\|_{s+1}^{q,\gamma,\mathcal{O}}+\|\beta_{\kappa}\|_{s+\sigma_1}^{q,\gamma,\mathcal{O}}\|\rho\|_{s_0}^{q,\gamma,\mathcal{O}},\label{B-Id}\\
        \|(\Delta_{12}\mathcal{B}_{\kappa}^{\pm1})\rho\|_{s_1}+\|(\Delta_{12}\mathscr{B}_{\kappa}^{\pm1})\rho\|_{s_1}&\lesssim_{s_1,q,d}\varepsilon\gamma^{-1}\|\Delta_{12}i\|_{s_1+\sigma_1}\|\rho\|_{s_1+1};\label{B12}
	    \end{align}
    \item 
	For any $(a,\omega)\in\cO$ belonging to the Borel set
    \begin{equation}\label{transport.nonres}
        \mathcal{O}_{\textnormal{\tiny{T}},\infty}^{\gamma,\tau_1}(i_0)\triangleq\bigcap_{\kappa\in\{-,+\}}\bigcap_{{(\ell,j)\in\mathbb{Z}^d\times\mathbb{Z}\setminus\{(0,0)\}\atop\vec{\jmath}\cdot\ell+j=0}}\bigg\{(a,\omega)\in\mathcal{O}\quad\textnormal{s.t.}\quad\big|\omega\cdot\ell+j\mathtt{c}_{\kappa}(a,\omega;i_0)\big|>\frac{\gamma^{\upsilon}}{\langle\ell\rangle^{\tau_1}}\bigg\},
    \end{equation}
	the following identity holds, for any $\kappa \in \{-,+\}$,
    \begin{equation}\label{redu.trasporto.op}
        \mathscr{B}_{\kappa}^{-1}\Big(\omega\cdot\partial_{\varphi}+\partial_{x}\big(\big[\kappa a+\varepsilon f_{\kappa}(\widetilde{r})\big]\cdot\big)\Big)\mathscr{B}_{\kappa}=\omega\cdot\partial_{\varphi}+\mathtt{c}_{\kappa}\partial_{x}.
    \end{equation}
\end{enumerate}
\end{prop}
\begin{proof}
     We apply \cite[Prop. 6.2]{HR21} to obtain the existence of $\mathtt{c}_{\pm}$, $\beta_{\pm}$  satisfying the estimates \eqref{estimazione ttc}, \eqref{estimazione ttc12}, \eqref{estimazione beta}, \eqref{estimazione beta12}, the symmetry properties \eqref{symmet beta}-\eqref{symmet betah}, and, for any $(a,\omega)\in\mathcal{O}_{\textnormal{\tiny{T}},\infty}^{\gamma,\tau_1}(i_0)$ as in \eqref{transport.nonres} and for any $\kappa\in\{-,+\},$ the conjugation \eqref{redu.trasporto.op} holds.
	The only difference with \cite[Prop. 6.2]{HR21} is the momentum preserving properties \eqref{traveling beta}, \eqref{traveling betah} for $\beta_{\kappa}$ and $\wh\beta_{\kappa}$. These conditions are easily obtained  by the initial traveling wave property \eqref{trav fpm} and a slight modification of the iterative proof in the given reference to provide a quasi-periodic traveling wave at each step. Indeed, along the scheme we use the preservation of the momentum condition $\vec{j}\cdot\ell+j=0$ inside the Cantor sets while solving the homological equation at each step, see \eqref{transport.nonres} for the final Cantor set. The space-time projectors are replaced by the traveling projectors defined in \eqref{proj.traveling}. Note that the momentum condition also allows to modify the non-resonance lower bound in the Cantor set with respect to the given reference. This approach has been first developed in \cite[Appendix A]{BFM21-1}.
\end{proof}

We are now ready to conjugate the full operator $\cL^{(M-1)}$ in \eqref{calLM2}.

\begin{prop}\label{prop trans Ego}
    Let $M\in\N^*$ and $(\mu_2,s_1)$ as in \eqref{parametres:transport} and assume the smallness condition \eqref{sml+bnd trans} of Proposition \ref{prop trans red}. We define the matrix transformation
	\begin{equation}\label{def:matB}
	    \pmb{\mathscr{B}}\triangleq\begin{pmatrix}
		\mathscr{B}_+ & 0\\
		0 & \mathscr{B}_-
	\end{pmatrix}.
	\end{equation}
	Then, in restriction to the Borel set $\mathcal{O}_{\textnormal{\tiny{T}},\infty}^{\gamma,\tau_1}(i_0)$ in \eqref{transport.nonres}, the following identity holds
	\begin{equation}\label{frakL0}
    \mathfrak{L}^{(0)}\triangleq\pmb{\mathscr{B}}^{-1}\mathcal{L}^{(M-1)}\pmb{\mathscr{B}}=\omega\cdot\partial_{\varphi}\mathbb{I}_2+\ii\boldsymbol{\Omega}(a,D)+\begin{pmatrix}
		(\mathtt{c}_+-a)\partial_{x} & 0\\
		0 & (\mathtt{c}_-+a)\partial_{x}
	\end{pmatrix}+\mathcal{R}_{[-1]}^{(d)}+\mathcal{S}_{M},
	\end{equation}
    where:
    \begin{enumerate}[label=(\roman*)]
        \item The constants $\tc_{\pm}$ are as in Proposition \ref{prop trans red}-$(i)$;
        \item The remainder $\mathcal{R}_{[-1]}^{(d)}$ is real, reversible and momentum preserving. It admits an homogeneous expansion of degree $-1$ up to order $1-M$ in the form
	$$\mathcal{R}_{[-1]}^{(d)}\triangleq\begin{pmatrix}
		\displaystyle\sum_{p=1}^{M-1}\mathfrak{d}_{+,p}^{(d)}\,\partial_{x}^{-p} & 0\\
		0 & \displaystyle\sum_{p=1}^{M-1}\mathfrak{d}_{-,p}^{(d)}\,\partial_{x}^{-p}
	\end{pmatrix},$$
	where, for any $\kappa\in\{-,+\}$ and any $p\in\llbracket1,M-1\rrbracket,$ the function $\mathfrak{d}_{\kappa,p}^{(d)}\triangleq\mathfrak{d}_{\kappa,p}^{(d)}(a,\omega;\varphi,x)$ satisfies the symmetry properties
    \begin{align}
        \mathfrak{d}_{\kappa,p}^{(d)}(a,\omega;-\varphi,-x)&=(-1)^{p+1}\mathfrak{d}_{\kappa,p}^{(d)}(a,\omega;\varphi,x),\label{rev delta}\\
	\forall \, y\in\mathbb{T},\quad\mathfrak{d}_{\kappa,p}^{(d)}(a,\omega;\varphi-\vec{\jmath}y,x)&=\mathfrak{d}_{\kappa,p}^{(d)}(a,\omega;\varphi,x+y),\label{trav delta}
    \end{align}
    and the following estimates, for any $s\in[s_0,S],$
	\begin{align}
		\|\mathfrak{d}_{\kappa,p}^{(d)}\|_{s}^{q,\gamma,\mathcal{O}}&\lesssim_{s,q,d,M}\varepsilon\gamma^{-1}\big(1+\|\mathfrak{I}_0\|_{s+\widetilde{\sigma}_0(M)}^{q,\gamma,\mathcal{O}}\big),\label{estimazione frakd}\\
		\|\Delta_{12}\mathfrak{d}_{\kappa,p}^{(d)}\|_{s_1}&\lesssim_{s,q,d,M}\varepsilon\gamma^{-1}\|\Delta_{12}i\|_{s_1+\widetilde{\sigma}_0(M)}\label{estimazione frakd12}
	\end{align}
	for some $\widetilde{\sigma}_0(M)>0;$
	\item The operator $\mathcal{S}_{M}$ is real, reversible and momentum preserving. Moreover,  for any $l\in \mathbb{N}^d$, with $|l|\leqslant l_0$ and any $n_1,n_2\in\mathbb{R}$ with $n_1+n_2\leqslant M - (q+l_0) - 1$, the operator $\langle D\rangle^{n_1}\partial_{\varphi}^{l}\mathcal{S}_{M}\langle D\rangle^{n_2}$ is $\mathcal{D}^{q}$-tame with tame constants satisfying, for any $s\in[s_0,S]$,
	\begin{align}
		\mathfrak{M}_{\langle D\rangle^{n_1}\partial_{\varphi}^{l}\mathcal{S}_{M}\langle D\rangle^{n_2}}(s)&\lesssim_{s,q,d,M}\varepsilon\gamma^{-1}\big(1+\|\mathfrak{I}_0\|_{s+\widetilde{\sigma}_{0}(M,l_0)}^{q,\gamma,\mathcal{O}}\big),\label{estimazione calSM}\\
		\|\langle  D\rangle^{n_1}\pa_{\vf}^{l}\Delta_{12}\mathcal{S}_{M}\langle D\rangle^{n_2}\|_{\mathcal{L}(\mathbf{H}^{s_1})}&\lesssim_{s,q,d,M}\varepsilon\gamma^{-1}\|\Delta_{12}i\|_{s_1+\widetilde{\sigma}_{0}(M,l_0)}\label{estimazione calSM12}
	\end{align}
    for some $\widetilde{\sigma}_{0}(M,l_0)>0.$
    \end{enumerate}
\end{prop}

\begin{proof}
	According to \eqref{calLM2} in Proposition \ref{prop.blockanti.final}, \eqref{redu.trasporto.op} in Proposition \ref{prop trans red}-$(iv)$ and \eqref{def:matB}, we have
    \begin{equation}\label{conj:LM}
        \begin{aligned}
        \pmb{\mathscr{B}}^{-1}\mathcal{L}^{(M-1)}\pmb{\mathscr{B}}&=\omega\cdot\partial_{\varphi}\mathbb{I}_2+\begin{pmatrix}
            \mathtt{c}_+\partial_{x} & 0\\
            0 & \mathtt{c}_-\partial_{x}\end{pmatrix}+\pmb{\mathscr{B}}^{-1}\bigg[\ii\boldsymbol{\Omega}(a,D)-\begin{pmatrix}
                a\partial_x & 0\\
                0 & -a\partial_{x}
                \end{pmatrix}\bigg]\pmb{\mathscr{B}}\\
                &\quad+\pmb{\mathscr{B}}^{-1}\mathbf{R}_{[-1]}^{(d)}\pmb{\mathscr{B}}+\pmb{\mathscr{B}}^{-1}\mathbf{S}_{M}\pmb{\mathscr{B}}.
    \end{aligned}
    \end{equation}
    Our purpose now is to study the conjugation effect on the lower order terms. 
    \\[1mm]
    $\blacktriangleright$ {\sc Action on the negative orders of the equilibrium Fourier multiplier:} Observe, in view of \eqref{bold.Omega} and \eqref{expand up to smooth omg}, that
    \begin{equation}\label{decomp iboldOmg}
        \ii\boldsymbol{\Omega}(a,D)-\begin{pmatrix}
                a\partial_x & 0\\
                0 & -a\partial_{x}
                \end{pmatrix}=\ii\boldsymbol{\Omega}_{\textnormal{neg},M+1}(a,D)+\widetilde{\mathbf{S}}_{M+1}(a,D),
    \end{equation}
    where
    $$\ii\boldsymbol{\Omega}_{\textnormal{neg},M+1}(a,D)\triangleq\begin{pmatrix}
                    \ii\Omega_{\textnormal{neg},M+1}(a,D) & 0\\
                    0 & -\ii\Omega_{\textnormal{neg},M+1}(a,D)
        \end{pmatrix},$$
    with, recalling the notations \eqref{frakNM.def}-\eqref{hom exp omgD},
    \begin{equation}\label{iomg-St}
        \ii\Omega_{\textnormal{neg},M+1}(a,D)\triangleq\sum_{p=1}^{\mathfrak{N}_{M+1}}\alpha_{p}a^{1-2p}\partial_{x}^{1-2p}\qquad\textnormal{and}\qquad\widetilde{\mathbf{S}}_{M+1}(a,D)\triangleq\begin{pmatrix}
        S_{M+1}(a,D) & 0\\
        0 & -S_{M+1}(a,D)
    \end{pmatrix}.
    \end{equation}
    Applying Egorov's Theorem (Proposition \ref{Egorov thm}), we get for any $\kappa\in\{-,+\},$ and $p\in\{1,...,\mathfrak{N}_{M+1}\},$
	\begin{equation}\label{conjug:dx}
	    \mathscr{B}_{\kappa}^{-1}\circ\partial_{x}^{1-2p}\circ\mathscr{B}_{\kappa}=\sum_{k=0}^{M-2p}\mathfrak{p}_{1-2p-k}^{[p,\kappa]}\partial_x^{1-2p-k}+\mathfrak{R}_{M,p,\kappa}^{[1]},
	\end{equation}
	with by virtue of Lemma \ref{lem:revmomHE} and Proposition \ref{prop trans red}-$(iii)$, for any $k\in\llbracket0.M-2p\rrbracket,$ the functions $\mathfrak{p}_{1-2p-k}^{[p,\kappa]}\triangleq\mathfrak{p}_{1-2p-k}^{[p,\kappa]}(a,\omega;\varphi,x)$ satisfying
    \begin{align}
	\mathfrak{p}_{1-2p-k}^{[p,\kappa]}(a,\omega;-\varphi,-x)&=(-1)^{k}\mathfrak{p}_{1-2p-k}^{[p,\kappa]}(a,\omega;\varphi,x), \label{rev fpk} \\
	\forall \, y\in\mathbb{T},\quad\mathfrak{p}_{1-2p-k}^{[p,\kappa]}(a,\omega;\varphi-\vec{\jmath}y,x)&=\mathfrak{p}_{1-2p-k}^{[p,\kappa]}(a,\omega;\varphi,x+y), \label{trav fpk}
\end{align}
and the operator $\mathfrak{R}_{M,p,\kappa}^{[1]}$ is real, reversible and momentum preserving.
In addition, according to \eqref{Egorov estim} and \eqref{estimazione beta}, for any $s\in[s_0,S]$ (here and in the sequel, the value of $\widetilde{\sigma}_0(M)>0$ may increase from line to line),
	\begin{equation}\label{efrp1}
	    \big\|\mathfrak{p}_{1-2p}^{[p,\kappa]}-1\big\|_{s}^{q,\gamma,\mathcal{O}}+\sup_{k\in\llbracket1,M-2p\rrbracket}\big\|\mathfrak{p}_{1-2p-k}^{[p,\kappa]}\big\|_{s}^{q,\gamma,\mathcal{O}}\lesssim_{s,q,d,M}\varepsilon\gamma^{-1}\big(1+\|\mathfrak{I}_0\|_{s+\widetilde{\sigma}_0(M)}^{q,\gamma,\mathcal{O}}\big)
	\end{equation}
	and, by \eqref{Egorov estim12} and \eqref{estimazione beta12}
    \begin{equation}\label{e12frp1}
        \max_{k\in\llbracket 0,M-2p\rrbracket}\big\|\Delta_{12}\mathfrak{p}_{1-2p-k}^{[p,\kappa]}\big\|_{s_1}\lesssim_{s_1,q,d,M}\varepsilon\gamma^{-1}\|\Delta_{12}i\|_{s_1+\widetilde{\sigma}_0(M)}.
    \end{equation}
    Moreover, from \eqref{Egorov Dq tame}-\eqref{Egorov 12R} and \eqref{estimazione beta}-\eqref{estimazione beta12}, one has for any $l\in \mathbb{N}^d$ with $|l|\leqslant l_0$ and $n_1,n_2\in\mathbb{N}$ with $n_1+n_2\leqslant M-(q+l_0)+2p-1$,
	\begin{align}
		\mathfrak{M}_{\langle D\rangle^{n_1}\partial_{\varphi}^{l}\mathfrak{R}_{M,p,\kappa}^{[1]}\langle D\rangle^{n_2}}(s)\lesssim_{s,q,d,M}\varepsilon\gamma^{-1}\big(1+\|\mathfrak{I}_0\|_{s+\widetilde{\sigma}_0(M,l_0)}^{q,\gamma,\mathcal{O}}\big),\label{eRR1}\\
		\|\langle D\rangle^{n_1}\partial_{\varphi}^{l}\Delta_{12}\mathfrak{R}_{M,p,\kappa}^{[1]}\langle D\rangle^{n_2}\|_{\mathcal{L}(\mathbf{H}^{s_1})}\lesssim_{s_1,q,d,M}\|\Delta_{12}i\|_{s_1+\widetilde{\sigma}_0(M,l_0)}\label{e12RR1}
	\end{align}
    for some $\sigma(M,l_0)>0.$ Combining \eqref{iomg-St} and \eqref{conjug:dx}, we obtain
    \begin{equation}\label{conj:iomglow}
        \begin{aligned}
        \mathscr{B}_{\kappa}^{-1}\ii\Omega_{\textnormal{neg},M+1}(a,D)\mathscr{B}_{\kappa}&=\sum_{p=1}^{\mathfrak{N}_{M+1}}\sum_{k=0}^{M-2p}\alpha_pa^{1-2p}\mathfrak{p}_{1-2p-k}^{[p,\kappa]}\partial_{x}^{1-2p-k}+\sum_{p=1}^{\mathfrak{N}_{M+1}}\alpha_pa^{1-2p}\mathfrak{R}_{M,p,\kappa}^{[1]}\\
        &=\ii\Omega_{\textnormal{neg},M+1}(a,D)+\sum_{p=1}^{\mathfrak{N}_{M+1}}\alpha_pa^{1-2p}\big(\mathfrak{p}_{1-2p}^{[p,\kappa]}-1\big)\partial_{x}^{1-2p}\\
        &\quad+\sum_{p=1}^{\mathfrak{N}_{M+1}}\sum_{k=1}^{M-2p}\alpha_pa^{1-2p}\mathfrak{p}_{1-2p-k}^{[p,\kappa]}\partial_{x}^{1-2p-k}+\sum_{p=1}^{\mathfrak{N}_{M+1}}\alpha_pa^{1-2p}\mathfrak{R}_{M,p,\kappa}^{[1]}.
    \end{aligned}
    \end{equation}
	$\blacktriangleright$ {\sc Action on the diagonal remainder term:} Using \eqref{def:matB} and \eqref{bfRd}, we write
	\begin{equation}\label{conj.pmbRd}
	    \pmb{\mathscr{B}}^{-1}\mathbf{R}_{[-1]}^{(d)}\pmb{\mathscr{B}}=\begin{pmatrix}
		\displaystyle\sum_{p=1}^{M-1}\mathscr{B}_{+}^{-1}\mathfrak{r}_{+,p}\partial_{x}^{-p}\mathscr{B}_{+} & 0\\
		0 & \displaystyle\sum_{p=1}^{M-1}\mathscr{B}_{-}^{-1}\mathfrak{r}_{-,p}\partial_{x}^{-p}\mathscr{B}_{-}
	\end{pmatrix}.
	\end{equation}
Employing one more time Egorov's Theorem in Proposition \ref{Egorov thm}, we get for any $\kappa\in\{-,+\}$ and any $p\in\llbracket 1,M-1\rrbracket,$
\begin{equation}\label{conj:rdx}
    \mathscr{B}_{\kappa}^{-1}\mathfrak{r}_{\kappa,p}\partial_{x}^{-p}\mathscr{B}_{\kappa}=\sum_{p_2=0}^{M-1-p}\widetilde{\mathfrak{r}}_{\kappa,p,p_2}\partial_{x}^{-p-p_2}+\mathfrak{R}_{M,p,\kappa}^{[2]},
\end{equation}
where, according to \eqref{Egorov estim},  \eqref{estimazione beta}, \eqref{estim frakrkp} and \eqref{Egorov estim12}, \eqref{estimazione beta12}, \eqref{estim frakrkp12}, we have for any $s\in[s_0,S],$
\begin{align}
	\|\widetilde{\mathfrak{r}}_{k,p,0}-\mathfrak{r}_{k,p}\|_{s}^{q,\gamma,\mathcal{O}}+\max_{p_2\in\llbracket 1,M-1-p\rrbracket}\|\widetilde{\mathfrak{r}}_{k,p,p_2}\|_{s}^{q,\gamma,\mathcal{O}}&\lesssim_{s,q,d,M}\varepsilon\gamma^{-1}\big(1+\|\mathfrak{I}_0\|_{s+\widetilde{\sigma}_0(M)}^{q,\gamma,\mathcal{O}}\big),\label{estim tildefrakrkp}\\
	\max_{p_2\in\llbracket0,M-1-p\rrbracket}\|\Delta_{12}\widetilde{\mathfrak{r}}_{k,p,p_2}\|_{s_1}&\lesssim_{s_1,q,d,M}\varepsilon\gamma^{-1}\|\Delta_{12}i\|_{s_1+\widetilde{\sigma}_{0}(M)}.\label{estim tildefrakrkp12}
\end{align}
In addition, by virtue of \eqref{rev rm diag final}-\eqref{trav rm diag final}, Lemma \ref{lem:revmomHE} and Proposition \ref{prop trans red}-$(iii)$, for any $p_2\in\llbracket 0,M-1-p\rrbracket,$ the function $\widetilde{\mathfrak{r}
}_{\kappa,p,p_2}^{(d)}\triangleq\widetilde{\mathfrak{r}}_{\kappa,p,p_2}^{(d)}(a,\omega;\varphi,x)\in\mathbb{R}$ satisfying
\begin{align}
	\widetilde{\mathfrak{r}}_{\kappa,p,p_2}(a,\omega;-\varphi,-x)&=(-1)^{p+p_2+1}\widetilde{\mathfrak{r}}_{\kappa,p,p_2}(a,\omega;\varphi,x), \label{rev trm diag} \\
	\forall \, y\in\mathbb{T},\quad\widetilde{\mathfrak{r}}_{\kappa,p,p_2}(a,\omega;\varphi-\vec{\jmath}y,x)&=\widetilde{\mathfrak{r}}_{\kappa,p,p_2}(a,\omega;\varphi,x+y), \label{trav trm diag}
\end{align}
and the operator $\mathfrak{R}_{M,p,\kappa}^{[2]}$ is real, reversible and momentum preserving.
Moreover, from \eqref{Egorov Dq tame}, \eqref{Egorov 12R}, \eqref{estim frakrkp}, \eqref{estim frakrkp12}, \eqref{estimazione beta} and \eqref{estimazione beta12} for any $l\in \mathbb{N}^d$ with $|l|\leqslant l_0$ and $n_1,n_2\in\mathbb{R}$ with $n_1+n_2\leqslant M-(q+l_0)+p$,
\begin{align}
		\mathfrak{M}_{\langle D\rangle^{n_1}\partial_{\varphi}^{l}\mathfrak{R}_{M,p,\kappa}^{[2]}\langle D\rangle^{n_2}}(s)\lesssim_{s,q,d,M}\varepsilon\gamma^{-1}\big(1+\|\mathfrak{I}_0\|_{s+\widetilde{\sigma}_0(M,l_0)}^{q,\gamma,\mathcal{O}}\big),\label{eRR2}\\
		\|\langle D\rangle^{n_1}\partial_{\varphi}^{l}\Delta_{12}\mathfrak{R}_{M,p,\kappa}^{[2]}\langle D\rangle^{n_2}\|_{\mathcal{L}(\mathbf{H}^{s_1})}\lesssim_{s_1,q,d,M}\varepsilon\gamma^{-1}\|\Delta_{12}i\|_{s_1+\widetilde{\sigma}_0(M,l_0)}.\label{e12RR2}
	\end{align}
$\blacktriangleright$ {\sc Conclusion:}
The decomposition \eqref{frakL0} is obtained by putting together \eqref{conj:LM}, \eqref{conj:iomglow}, \eqref{conj.pmbRd} and \eqref{conj:rdx}, setting, for $\kappa\in\{-,+\},$
\begin{equation}\label{def:HOMdelta}
    \begin{aligned}
    \sum_{p=1}^{M-1}\mathfrak{d}_{\kappa,p}^{(d)}\partial_{x}^{-p}&\triangleq\sum_{p=1}^{\mathfrak{N}_{M+1}}\alpha_pa^{1-2p}\left(\mathfrak{p}_{1-2p}^{[p,\kappa]}-1\right)\partial_{x}^{1-2p}+\sum_{p=1}^{\mathfrak{N}_{M+1}}\sum_{k=1}^{M-2p}\alpha_pa^{1-2p}\mathfrak{p}_{1-2p-k}^{[p,\kappa]}\partial_{x}^{1-2p-k}\\
        &\quad+\sum_{p_2=0}^{M-1-p}\widetilde{\mathfrak{r}}_{\kappa,p,p_2}\partial_{x}^{-p-p_2}
\end{aligned}
\end{equation}
and 
\begin{align}
    \mathcal{S}_{M}&\triangleq\pmb{\mathscr{B}}^{-1}\mathbf{S}_{M}\pmb{\mathscr{B}}+\begin{pmatrix}
    \mathfrak{R}_{M,+} & 0\\
    0 & \mathfrak{R}_{M,-}
\end{pmatrix},\label{def:calS0M00}\\
\mathfrak{R}_{M,\kappa}&\triangleq\mathscr{B}_{\kappa}^{-1}S_{M+1}(a,D)\mathscr{B}_{\kappa}-S_{M+1}(a,D)+\sum_{p=1}^{\mathfrak{N}_{M+1}}\alpha_pa^{1-2p}\mathfrak{R}_{M,p,\kappa}^{[1]}+\sum_{p=1}^{M-1}\mathfrak{R}_{M,p,\kappa}^{[2]}.\label{def:frakRMkappa00}
\end{align}
Putting together \eqref{def:HOMdelta}, \eqref{rev fpk}, \eqref{trav fpk}, \eqref{rev trm diag} and \eqref{trav trm diag}, we obtain \eqref{rev delta}-\eqref{trav delta}. The operator $\mathcal{S}_M$ is real, reversible and momentum preserving by virtue of \eqref{def:calS0M00}-\eqref{def:frakRMkappa00}, Lemma \ref{lem prop eig VP}-(iv), Proposition \ref{prop trans red}-$(iii)$ and because $\mathfrak{R}_{M,p,\kappa}^{[1]},\mathfrak{R}_{M,p,\kappa}^{[2]}$ are real reversible and momentum preserving.
According to the definition \eqref{def:HOMdelta}, the estimates \eqref{estimazione frakd}-\eqref{estimazione frakd12} follow from \eqref{efrp1}-\eqref{e12frp1} and \eqref{estim tildefrakrkp}-\eqref{estim tildefrakrkp12}.
Now, we shall obtain the estimates \eqref{estimazione calSM}-\eqref{estimazione calSM12}. First observe that
$$\mathscr{B}_{\kappa}^{-1}S_{M+1}(a,D)\mathscr{B}_{\kappa}-S_{M+1}(a,D)=(\mathscr{B}_{\kappa}^{-1}-\textnormal{Id})S_{M+1}(a,D)\mathscr{B}_{\kappa}+S_{M+1}(a,D)(\mathscr{B}_{\kappa}-\textnormal{Id}).$$
Moreover, for any $l\in\N^d$ with $|l|\leqslant l_0$ and any $n_1+n_2+l_0+q\leqslant M$, we have by Leibniz's rule
\begin{align*}
    \langle D\rangle^{n_1}\partial_{\varphi}^l\pmb{\mathscr{B}}^{-1}\boldsymbol{S}_M\pmb{\mathscr{B}}\langle D\rangle^{n_2}=\sum_{(l_1,l_2,l_3)\in(\mathbb{N}^{d})^3\atop |l_1|+|l_2|+|l_3|=|l|}c_{l_1,l_2,l_3}\big(\langle D\rangle^{n_1}\partial_{\varphi}^{l_1}\pmb{\mathscr{B}}^{-1}\langle D\rangle^{-n_1}\big)\big(\langle D\rangle^{n_1}\partial_{\varphi}^{l_2}\mathbf{S}_M\langle D\rangle^{n_2}\big)\big(\langle D\rangle^{-n_2}\partial_{\varphi}^{l_3}\pmb{\mathscr{B}}\langle D\rangle^{n_2}\big).
\end{align*}
The $c_{l_1,l_2,l_3}$ are real constants coming from Leibniz's formula. Note that the estimate \eqref{B-Id} loses one derivative in the test function and that is the reason why we had to push the expansion of the equilibrium symbol up to order $M+1$ in \eqref{decomp iboldOmg}. With this in hand, it suffices to combine \eqref{def:calS0M00}, \eqref{def:frakRMkappa00}, \eqref{eB}, \eqref{est.bSM}, Proposition \ref{properties OPS}-$(v)$, \eqref{B-Id}, \eqref{estimazione beta}, \eqref{eRR1} and \eqref{eRR2} to get \eqref{estimazione calSM}. Similarly, the estimate \eqref{B12} also loses one derivative in the test function. Therefore, for $n_1+n_2+l_0+q+1\leqslant M$, the estimate \eqref{estimazione calSM12} follows by an analogous argument from \eqref{def:calS0M00}, \eqref{def:frakRMkappa00}, \eqref{est12.bSM}, \eqref{eB}, \eqref{estimazione beta}, \eqref{B12}, \eqref{eRR1}, \eqref{e12RR1}, \eqref{eRR2} and \eqref{e12RR2}. This concludes the proof of Proposition \ref{prop trans Ego}.
\end{proof}

\subsection{Reduction of the operator $\mathscr{L}_{\omega}$}
The purpose of this section is to study the action of the previous transformations on of the linear operator $\mathscr{L}_{\omega}$ in \eqref{sL_omega}, acting only on the normal directions.
Recall that in Sections \ref{sect.block.anti}, \ref{sect.transport}, the linear operator $\mathcal{L}^{(0)}$ in Proposition \ref{prop L0} has been conjugated, under the real, reversibility and momentum preserving map
\begin{equation}\label{map.M}
   \mathscr{M}\triangleq \boldsymbol{\Phi}_{M-1}\pmb{\mathscr{B}},
\end{equation}
for any $(a,\omega)\in \mathcal{O}_{\textnormal{\tiny{T}},\infty}^{\gamma,\tau_1}(i_0)$, $n\in\mathbb{N}$ as in \eqref{transport.nonres}, into the real, reversible and momentum preserving operator
\begin{equation}\label{conju.1}
	\mathfrak{L}^{(0)} \triangleq \mathscr{M}^{-1} \mathcal{L}^{(0)} \mathscr{M}=\pmb{\mathscr{B}}^{-1}  \boldsymbol{\Phi}_{M-1}^{-1} \mathcal{L}^{(0)} \boldsymbol{\Phi}_{M-1}\pmb{\mathscr{B}},
\end{equation}
where the maps $ \boldsymbol{\Phi}_{M-1}$, $\pmb{\mathscr{B}}$ are defined respectively in Proposition \ref{prop.blockanti.final} and Proposition \ref{prop trans Ego}. The operator $\mathfrak{L}^{(0)}$ is defined for any $(a,\omega) \in \mathcal{O}$. As we shall discuss now, a similar conjugation result holds for the projected operator $\mathscr{L}_\omega$ in \eqref{sL_omega}, acting in the normal subspace  $\mathbf{H}_{\overline{\mathbb{S}}_0}^{\perp}$. For this aim, we introduce the map
\begin{equation}\label{map.M.perp}
	\mathscr{M}_{\perp}\triangleq\Pi_{\overline{\mathbb{S}}_0}^\perp\mathscr{M}\Pi_{\overline{\mathbb{S}}_0}^{\perp}.
\end{equation}

\begin{lem}\label{lem.inverse.Mperp}
	Let $M\in\mathbb{N}^*$ and $s_1$ as in \eqref{parametres:transport}. There exists $\widetilde{\sigma}_{1}(M)\triangleq\widetilde{\sigma}_{1}(q,d,M)>0$ such that, assuming \eqref{ansatz.fI.small} with $s_0+\sigma_{\rm max}\geqslant s_1+\widetilde{\sigma}_{1}(M)$, the following holds. 
    \begin{enumerate}[label=(\roman*)]
        \item The map $\mathscr{M}$ in \eqref{map.M} is a real, reversibility and momentum preserving invertible transformation satisfying the following estimates, for all $s\in [s_0,S]$,
	\begin{align}
    \|\mathscr{M}^{\pm1}\rho \|_{s}^{q,\gamma,\mathcal{O}} &\lesssim_{s,q,d,M}\|\rho\|_{s}^{q,\gamma,\mathcal{O}}+\varepsilon \gamma^{-1}\|\mathfrak{I}_{0}\|_{s+\widetilde{\sigma}_{1}(M)}^{q,\gamma,\mathcal{O}}   \|\rho\|_{s_0}^{q,\gamma,\mathcal{O}},\label{e:Mcont}\\
		\|(\mathscr{M}^{\pm1}-\mathbb{I}_2)\rho \|_{s}^{q,\gamma,\mathcal{O}} &\lesssim_{s,q,d,M} \varepsilon \gamma^{-1} \big( \|\rho\|_{s+1}^{q,\gamma,\mathcal{O}}  +  \| \mathfrak{I}_{0}\|_{s+\widetilde{\sigma}_{1}(M)}^{q,\gamma,\mathcal{O}}   \|\rho\|_{s_0+1}^{q,\gamma,\mathcal{O}}  \big),\label{e:scrReps}\\
        \|\Delta_{12}\mathscr{M}^{\pm1}\rho\|_{s_1}&\lesssim_{s_1,q,d,M} \varepsilon\gamma^{-1}\|\Delta_{12} i\|_{s_1+\widetilde{\sigma}_{1}(M)}\|\rho\|_{s_1+1};\label{e12:scrReps}
        \end{align}
        \item The map $\mathscr{M}_{\perp}$ in \eqref{map.M.perp} is a real, reversibility and momentum preserving invertible transformation satisfying the following estimates, for all $s\in[s_0,S]$,
		\begin{align}
				\|\mathscr{M}_\perp^{\pm 1}\rho \|_{s}^{q,\gamma,\mathcal{O}} & \lesssim_{s,q,d,M}\|\rho\|_{s}^{q,\gamma,\mathcal{O}}+\varepsilon\gamma^{-1}\|\mathfrak{I}_{0}\|_{s+\widetilde{\sigma}_{1}(M)}^{q,\gamma,\mathcal{O}}\|\rho\|_{s_0}^{q,\gamma,\mathcal{O}},\label{e:Mperp} \\
			\|\Delta_{12} \mathscr{M}_\perp^{\pm 1}\rho\|_{s_1}&\lesssim_{s_1,q,d,M} \varepsilon\gamma^{-1}\|\Delta_{12} i\|_{s_1+\widetilde{\sigma}_{1}(M)}\|\rho\|_{s_1+1}.  \label{e12:Mperp}
		\end{align}
    \end{enumerate}
\end{lem}
\begin{proof}
$(i)$ The map $\mathscr{M}$ is real, reversibility and momentum preserving because $\boldsymbol{\Phi}_{M-1}$ and $\pmb{\mathscr{B}}$ are, see \eqref{map.M}, Proposition \ref{prop.blockanti.final} and Proposition \ref{prop trans red}-$(iii).$ The estimates \eqref{e:Mcont} follow from \eqref{map.M}, \eqref{e-PhiM}, \eqref{eB}, \eqref{estimazione beta} and \eqref{ansatz.fI.small}. By virtue of \eqref{map.M}, \eqref{compoPhim} and \eqref{def:matB}, we write
    \begin{equation}\label{scrM-I}
        \mathscr{M}-\mathbb{I}_2=(\boldsymbol{\Phi}_{M-1}-\mathbb{I}_2)\pmb{\mathscr{B}}+(\pmb{\mathscr{B}}-\mathbb{I}_2),\qquad\pmb{\mathscr{B}}-\mathbb{I}_2=\begin{pmatrix}
            \mathscr{B}_+-\textnormal{Id} & 0\\
            0 & \mathscr{B}_--\textnormal{Id}
        \end{pmatrix}.
    \end{equation}
	Therefore, the estimate \eqref{e:scrReps} for $\mathscr{M}$ follows from \eqref{scrM-I}, \eqref{e-PhiM}, Proposition \ref{properties OPS}-$(v)$, \eqref{B-Id} and \eqref{estimazione beta}. The inverse is given by
    \begin{equation}\label{scrM-1-I}
        \mathscr{M}^{-1}=\pmb{\mathscr{B}}^{-1}\boldsymbol{\Phi}_{M-1}^{-1}=\mathbb{I}_2+(\pmb{\mathscr{B}}^{-1}-\mathbb{I}_2)\boldsymbol{\Phi}_{M-1}^{-1}+(\boldsymbol{\Phi}_{M-1}^{-1}-\mathbb{I}_2),\qquad\pmb{\mathscr{B}}^{-1}-\mathbb{I}_2=\begin{pmatrix}
            \mathscr{B}_+^{-1}-\textnormal{Id} & 0\\
            0 & \mathscr{B}_-^{-1}-\textnormal{Id}
        \end{pmatrix}.
    \end{equation}
        Therefore, the estimate \eqref{e:scrReps} for $\mathscr{M}^{-1}$ follows from \eqref{scrM-1-I}, \eqref{e-PhiM}, Proposition \ref{properties OPS}-$(v)$, \eqref{B-Id} and \eqref{estimazione beta}. Besides, from \eqref{scrM-I}, we have
        \begin{align}
            \Delta_{12}\mathscr{M}&=(\Delta_{12}\boldsymbol{\Phi}_{M-1})\pmb{\mathscr{B}}(i_2)+\boldsymbol{\Phi}_{M-1}(i_1)(\Delta_{12}\pmb{\mathscr{B}}),\label{12scrM}\\
            \Delta_{12}\mathscr{M}^{-1}&=(\Delta_{12}\pmb{\mathscr{B}}^{-1})\boldsymbol{\Phi}_{M-1}^{-1}(i_2)+\pmb{\mathscr{B}}(i_1)(\Delta_{12}\boldsymbol{\Phi}_{M-1}^{-1}).\label{12scrM-1}
        \end{align}
        Thus, the estimates \eqref{e12:scrReps} follow from \eqref{12scrM}, \eqref{12scrM-1}, \eqref{e-PhiM}, \eqref{e12-PhiM}, Proposition \ref{properties OPS}-$(v)$, \eqref{eB}, \eqref{estimazione beta}, \eqref{B12} and \eqref{ansatz.fI.small}.\\
    $(ii)$ By definition, \eqref{map.M.perp} the map $\mathscr{M}_{\perp}$ is a real, reversibility and momentum preserving because $\mathscr{M}$ is. The estimate \eqref{e:Mperp} for $\mathscr{M}_{\perp}$ is a consequence of \eqref{e-PhiM}, Proposition \ref{properties OPS}-$(v)$ and \eqref{eB}. Then, writing
    $$\Delta_{12}\mathscr{M}_{\perp}=\Pi_{\overline{\mathbb{S}}_0}^\perp(\Delta_{12}\mathscr{M})\Pi_{\overline{\mathbb{S}}_0}^{\perp},$$
    the estimate \eqref{e12:Mperp} for $\mathscr{M}_{\perp}$ is obtained from \eqref{e12:scrReps}. Regarding the same estimates for $\mathscr{M}_{\perp}^{-1}$, it is a bit more involved. For this, we easily adapt  \cite[Lemma 9.6]{HHM21} to our matrix case, allowing us to prove that $\mathscr{M}_{\perp}$ is indeed invertible with inverse explicitly given by
    \begin{equation}\label{M-1perp:formula}
        \mathscr{M}_{\perp}^{-1}\rho=\mathscr{M}^{-1}\rho-\sum_{m\in\overline{\mathbb{S}}_0}\big\langle\rho\,,\,\big((\mathscr{M}^{\star})^{-1}-\mathbb{I}_{2}\big)g_m\big\rangle_{\mathbf{L}_{x}^2(\mathbb{T})}\mathscr{M}^{-1}\vec{e}_{m},
    \end{equation}
    where $\mathscr{M}^{\star}$ is the $\mathbf{L}^2(\mathbb{T})$-adjoint of $\mathscr{M},$ for any $m\in\overline{\mathbb{S}}_0,$
    \begin{align*}
        \forall\, x\in\mathbb{T},\quad\vec{e}_m(x)\triangleq\begin{cases}
            \big( \begin{smallmatrix}
                1\\
                0
            \end{smallmatrix}\big)e^{\ii m x}, & \textnormal{if }m\in\overline{\mathbb{S}}_+,\vspace{0.2cm}\\
            \big(\begin{smallmatrix}
                0\\
                1
            \end{smallmatrix}\big)e^{\ii m x}, & \textnormal{if }m\in\overline{\mathbb{S}}_-
        \end{cases}
    \end{align*}
    and under the notation
    $$\mathbf{M}\triangleq\Big(\big\langle\vec{e}_m,(\mathscr{M}^{\star})^{-1}\vec{e}_k\big\rangle_{\mathbf{L}_{x}^2(\mathbb{T})}\Big)_{(m,k)\in\overline{\mathbb{S}}_0^2},\qquad\mathbf{M}^{-1}\triangleq\big(\alpha_{m,k}\big)_{(m,k)\in\overline{\mathbb{S}}_0^2},$$
    we have
    $$g_m(a,\omega;\varphi,x)\triangleq\sum_{k\in\overline{\mathbb{S}}_0}\overline{\alpha_{m,k}(a,\omega;\varphi)}\vec{e}_k(x).$$
    In view of \eqref{map.M}, we have
    \begin{align*}
    (\mathscr{M}^{\star})^{-1}&=(\pmb{\mathscr{B}}^{\star}\boldsymbol{\Phi}_{M-1}^{\star})^{-1}=(\boldsymbol{\Phi}_{M-1}^{\star})^{-1}(\pmb{\mathscr{B}}^{\star})^{-1}=(\boldsymbol{\Phi}_{M-1}^{\star})^{-1}\pmb{\mathcal{B}}.
\end{align*}
We have used Lemma \ref{prop:Ego}-$(iii)$ to write that
$$(\pmb{\mathscr{B}}^{\star})^{-1}=\begin{pmatrix}
        \mathscr{B}_+^{\star} & 0\\
        0 & \mathscr{B}_-^{\star}
    \end{pmatrix}^{-1}=\begin{pmatrix}
        \mathcal{B}_+^{-1} & 0\\
        0 & \mathcal{B}_-^{-1}
    \end{pmatrix}^{-1}=\begin{pmatrix}
        \mathcal{B}_+ & 0\\
        0 & \mathcal{B}_-
    \end{pmatrix}\triangleq\pmb{\mathcal{B}}.$$
    Therefore, we can write
    \begin{align}
        \mathscr{M}^{\star}=\mathbb{I}_2+(\pmb{\mathcal{B}}^{-1}-\mathbb{I}_2)\boldsymbol{\Phi}_{M-1}^{\star}+(\boldsymbol{\Phi}_{M-1}^{\star}-\mathbb{I}_2),\\
        (\mathscr{M}^{\star})^{-1}=\mathbb{I}_2+\big((\boldsymbol{\Phi}_{M-1}^{\star})^{-1}-\mathbb{I}_2\big)\pmb{\mathcal{B}}+(\pmb{\mathcal{B}}-\mathbb{I}_2)
    \end{align}
    and
    \begin{align}
        \Delta_{12}\mathscr{M}^{\star}&=(\Delta_{12}\pmb{\mathcal{B}}^{-1})\boldsymbol{\Phi}_{M-1}^{\star}(i_2)+\pmb{\mathcal{B}}^{-1}(i_1)(\Delta_{12}\boldsymbol{\Phi}_{M-1}^{\star}),\\
        \Delta_{12}(\mathscr{M}^{\star})^{-1}&=\big(\Delta_{12}(\boldsymbol{\Phi}_{M-1}^{\star})^{-1}\big)\pmb{\mathcal{B}}^{-1}(i_2)+(\boldsymbol{\Phi}_{M-1}^{\star})^{-1}(i_1)(\Delta_{12}\pmb{\mathcal{B}}).
    \end{align}
    Similarly to the point $(i),$ using \eqref{e-PhiMstar}, \eqref{e12-PhiMstar}, Proposition \ref{properties OPS}-$(v)$, \eqref{eB}, \eqref{B-Id}, \eqref{B12} and \eqref{ansatz.fI.small}, we find for any $s\in[s_0,S]$,
    \begin{align}
        \|((\mathscr{M}^{\star})^{\pm1}-\mathbb{I}_2)\rho \|_{s}^{q,\gamma,\mathcal{O}}&\lesssim_{s,q,d,M}\varepsilon\gamma^{-1}\big(\|\rho\|_{s+1}^{q,\gamma,\mathcal{O}}+\|\mathfrak{I}_{0}\|_{s+\widetilde{\sigma}_{1}(M)}^{q,\gamma,\mathcal{O}}   \|\rho\|_{s_0+1}^{q,\gamma,\mathcal{O}}  \big),\label{e:Mstar-Id}\\
        \|\Delta_{12}(\mathscr{M}^{\star})^{\pm1}\rho\|_{s_1}&\lesssim_{s_1,q,d,M} \varepsilon\gamma^{-1}\|\Delta_{12} i\|_{s_1+\widetilde{\sigma}_{1}(M)}\|\rho\|_{s_1+1}\label{e12:Mstar}
    \end{align}
    for some $\widetilde{\sigma}_1(M)>0.$ As a consequence, the matrices $\mathbf{M}$ and $\mathbf{M}^{-1}$ are perturbations of the identity matrix. Quantitatively, the $\alpha_{m,k}$ satisfy the following estimates for any $s\in[s_0,S]$
    \begin{align}
        \|\alpha_{m,k}-\delta_{m,k}\|_{s}^{q,\gamma,\mathcal{O}}&\lesssim_{s,q,d,M}\varepsilon\gamma^{-1}\big(1+\|\mathfrak{I}_0\|_{s+\widetilde{\sigma}_1(M)}^{q,\gamma,\mathcal{O}}\big),\label{e:alphamk}\\
        \|\Delta_{12}\alpha_{m,k}\|_{s_1}&\lesssim_{s_1,q,d,M}\varepsilon\gamma^{-1}\|\Delta_{12}i\|_{s_1+\widetilde{\sigma}_1(M)}.\label{e12:alphamk}
    \end{align}
    Consequently, the estimate \eqref{e:Mperp} for $\mathscr{M}_{\perp}^{-1}$ follows from \eqref{M-1perp:formula}, \eqref{e:Mcont}, \eqref{e:Mstar-Id} and \eqref{e:alphamk}. Also, from \eqref{M-1perp:formula}, we can write
    \begin{equation}\label{12Mperp-1:algebra}
        \begin{aligned}
        \Delta_{12}\mathscr{M}_{\perp}^{-1}\rho&=\Delta_{12}\mathscr{M}^{-1}\rho-\sum_{m\in\overline{\mathbb{S}}_0}\big\langle\rho\,,\,\Delta_{12}(\mathscr{M}^{\star})^{-1}g_m\big\rangle_{\mathbf{L}_{x}^2(\mathbb{T})}\mathscr{M}^{-1}(i_2)\vec{e}_{m}\\
        &\quad-\sum_{m\in\overline{\mathbb{S}}_0}\big\langle\rho\,,\,\big((\mathscr{M}^{\star})^{-1}(i_1)-\mathbb{I}_{2}\big)g_m\big\rangle_{\mathbf{L}_{x}^2(\mathbb{T})}\Delta_{12}\mathscr{M}^{-1}\vec{e}_{m}.
    \end{aligned}
    \end{equation}
    Hence, the estimate \eqref{e12:Mperp} for the inverse is obtained combining \eqref{12Mperp-1:algebra}, \eqref{e:Mcont}, \eqref{e12:scrReps}, \eqref{e:Mstar-Id}, \eqref{e12:Mstar} and \eqref{e12:alphamk}. This ends the proof of Lemma \ref{lem.inverse.Mperp}.
\end{proof}
Now, let us consider the operator $\mathscr{L}_{\perp}$ defined by
\begin{equation}\label{def:sL}
    \mathscr{L}_{\perp}\triangleq\mathscr{M}_{\perp}^{-1} \mathscr{L}_{\omega} \mathscr{M}_{\perp},
\end{equation}
with $\mathscr{M}_{\perp}$ as in \eqref{map.M.perp}. The next proposition provides the structure of the operator $\mathscr{L}_{\perp}$ in \eqref{def:sL}.

\begin{prop}\label{lemma.sL.beforeKAM}
{\bf (Reduction of $\mathscr{L}_\omega$ up to smoothing operators)}. Let $M\in\mathbb{N}^*$ and $s_1$ as in \eqref{parametres:transport}. The operator $\mathscr{L}_{\perp}$ in \eqref{def:sL} is well-defined for the whole set of parameters $\mathcal{O},$ real, reversible and momentum preserving. Moreover, in restriction to the set $\mathcal{O}_{\textnormal{\tiny{T}},\infty}^{\gamma,\tau_1}(i_0)$ in \eqref{transport.nonres}, the operator $\mathscr{L}_\omega$ has the form
\begin{equation}\label{sL.perp}
\mathscr{L}_\perp \triangleq \omega\cdot\partial_{\varphi}\Pi_{\overline{\mathbb{S}}_0}^{\perp}+\im \,\mathscr{D}^{[0]}+\mathscr{R}^{[0]},
\end{equation}
	with the following properties:
	\begin{itemize}
		\item[$(i)$] The operator $\mathscr{D}^{[0]}$ is diagonal, real, reversible and momentum preserving of the form
        \begin{align}  \mathscr{D}^{[0]}=\Pi_{\overline{\mathbb{S}}_0}^{\perp}\mathscr{D}^{[0]}\Pi_{\overline{\mathbb{S}}_0}^{\perp}=\begin{pmatrix}
			\mathscr{D}_{+}^{(0)} & 0\\
			0 & \mathscr{D}_{-}^{(0)}
		\end{pmatrix},\label{sD.0}
        \end{align}
	where, for any $\kappa\in\{-,+\}$, the operator $\mathscr{D}_{\kappa}^{[0]}=\Pi_{\kappa}^{\perp}\mathscr{D}_{\kappa}^{[0]}\Pi_{\kappa}^{\perp}$ is the diagonal operator
	\begin{equation} \label{mu.j.0}
		\mathscr{D}_{\kappa}^{[0]}\triangleq\underset{j\in\mathbb{Z}^*\setminus\overline{\mathbb{S}}_{\kappa}}{{\rm diag}}\big\{\mu_{j,\kappa}^{[0]}\big\}, \qquad \mu_{j,\kappa}^{[0]}(a,\omega;i_0)\triangleq \kappa\Omega_{j}(a)+\big(\mathtt{c}_{\kappa}(a,\omega;i_0)-\kappa a\big)j\in \R,
	\end{equation}
	with $\tc_{\kappa}(a,\omega;i_0)$ satisfying the estimates \eqref{estimazione ttc}-\eqref{estimazione ttc12} in Proposition \ref{prop trans Ego};
\item[$(ii)$] The operator $\mathscr{R}^{[0]}=\Pi_{\overline{\mathbb{S}}_0}^{\perp}\mathscr{R}^{[0]}\Pi_{\overline{\mathbb{S}}_0}^{\perp}$ is real, reversible and momentum preserving. Moreover, for any $l_0\in \N$, $M\geqslant q+l_0+2$, there exists a constant $\aleph(M,l_0)\triangleq\aleph(q,d,M,l_0)>0$ such that, assuming \eqref{linear at general state} with $s_0+\sigma_{\rm max}\geqslant s_1+ \aleph(M,l_0)$, 
and for any $l\in \N^d$, $|l|\leqslant l_0$, the operators $\langle D\rangle^{\frac12}\partial_{\varphi}^{l}\mathscr{R}^{[0]}\langle D\rangle^{\frac12}$ and $\langle D\rangle^{\frac12}[\pa_{\vf}^{l}\mathscr{R}^{[0]},\pa_{x}]\langle D\rangle^{\frac12}$ are $\mathcal{D}^{q}$-tame with tame constants satisfying, for any $s\in[s_0,S]$,
\begin{align} 
\mathfrak{M}_{\langle D\rangle^{\frac12}\partial_{\varphi}^{l}\mathscr{R}^{[0]}\langle D\rangle^{\frac12}}(s), \, \mathfrak{M}_{\langle D\rangle^{\frac12}[\pa_{\vf}^{l}\mathscr{R}^{[0]},\pa_{x}]\langle D\rangle^{\frac12}}(s) & \lesssim_{s,q,d,M}\varepsilon\gamma^{-1}\big(1+\|\mathfrak{I}_0\|_{s+\aleph(M,l_0)}^{q,\gamma,\mathcal{O}}\big) . \label{est.sR0} 
\end{align}
Furthermore, for any $l\in \N^d$, $|l|\leqslant l_0$, 
\begin{equation}\label{est.sR0.12}
    \begin{aligned}
     \|\langle D\rangle^{\frac12}\partial_{\varphi}^{l}\Delta_{12}\mathscr{R}^{[0]}\langle D\rangle^{\frac12}\|_{\mathcal{L}(\mathbf{H}_{\perp}^{s_1})} & \lesssim_{s_1,q,d,M}\varepsilon\gamma^{-1}\|\Delta_{12}i\|_{s_1+\aleph(M,l_0)} , \\
     \|\langle D\rangle^{\frac12} [\partial_{\varphi}^{l}\Delta_{12}\mathscr{R}^{[0]},\pa_{x}]\langle D\rangle^{\frac12}\|_{\mathcal{L}(\mathbf{H}_{\perp}^{s_1})} & \lesssim_{s_1,q,d,M}\varepsilon\gamma^{-1}\|\Delta_{12}i\|_{s_1+\aleph(M,l_0)}.
\end{aligned}
\end{equation}
	\end{itemize}
\end{prop}
\begin{proof}
In view of \eqref{def:sL}, \eqref{sL_omega}. \eqref{map.M.perp}, using the facts that 
$$\mathbb{I}_2=\Pi_{\overline{\mathbb{S}}_0}+\Pi_{\overline{\mathbb{S}}_0}^\perp,\qquad (\Pi_{\overline{\mathbb{S}}_0}^\perp)^2=\Pi_{\overline{\mathbb{S}}_0}^\perp,\qquad\Pi_{\overline{\mathbb{S}}_0}\Pi_{\overline{\mathbb{S}}_0}^\perp=\Pi_{\overline{\mathbb{S}}_0}^\perp\Pi_{\overline{\mathbb{S}}_0}=0$$
and that from Propositions \ref{prop.blockanti.final} and \ref{prop trans Ego}, in restriction to the set $\mathcal{O}_{\textnormal{\tiny{T}},\infty}^{\gamma,\tau_1}(i_0)$, the following holds
$$\mathcal{L}^{(0)}\mathscr{M}=\mathscr{M}\mathfrak{L}^{(0)},$$
we find
\begin{align}
	\mathscr{L}_{\perp} & = \mathscr{M}_{\perp}^{-1} \Pi_{\overline{\mathbb{S}}_0}^\perp \big(  \cL^{(0)} - \varepsilon  \pa_{x} \cR \big)\Pi_{\overline{\mathbb{S}}_0}^\perp \mathscr{M}\Pi_{\overline{\mathbb{S}}_0}^\perp \\
    & = \mathscr{M}_{\perp}^{-1}\Pi_{\overline{\mathbb{S}}_0}^\perp \cL^{(0)} \mathscr{M} \Pi_{\overline{\mathbb{S}}_0}^\perp - \mathscr{M}_{\perp}^{-1} \cL^{(0)} \Pi_{\overline{\mathbb{S}}_0} \mathscr{M} \Pi_{\overline{\mathbb{S}}_0}^\perp - \varepsilon \mathscr{M}_{\perp}^{-1} \pa_{x}\cR \mathscr{M}_{\perp}  \\
    & = \mathscr{M}_{\perp}^{-1} \Pi_{\overline{\mathbb{S}}_0}^\perp\mathscr{M} \fL^{(0)} \Pi_{\overline{\mathbb{S}}_0}^\perp - \mathscr{M}_{\perp}^{-1} \cL^{(0)} \Pi_{\overline{\mathbb{S}}_0} (\mathscr{M}-\mathbb{I}_2) \Pi_{\overline{\mathbb{S}}_0}^\perp - \varepsilon \mathscr{M}_{\perp}^{-1} \pa_{x}\cR \mathscr{M}_{\perp}  \\
     & = \mathscr{M}_{\perp}^{-1} \Pi_{\overline{\mathbb{S}}_0}^\perp\mathscr{M}(\Pi_{\overline{\mathbb{S}}_0}^\perp+\Pi_{\overline{\mathbb{S}}_0}) \fL^{(0)} \Pi_{\overline{\mathbb{S}}_0}^\perp - \mathscr{M}_{\perp}^{-1} \cL^{(0)} \Pi_{\overline{\mathbb{S}}_0} (\mathscr{M}-\mathbb{I}_2)\Pi_{\overline{\mathbb{S}}_0}^\perp - \varepsilon \mathscr{M}_{\perp}^{-1} \pa_{x}\cR \mathscr{M}_{\perp}  \\
    &= \Pi_{\overline{\mathbb{S}}_0}^\perp \mathfrak{L}^{(0)} \Pi_{\overline{\mathbb{S}}_0}^\perp + \mathscr{R}_{\rm f} , \label{sL.perp.conj}
\end{align}
where $\mathscr{R}_{\rm f}$ is the infinitely smooth operator
\begin{equation}\label{new.finiterank}
	\begin{aligned}
		\mathscr{R}_{\rm f} & \triangleq \mathscr{M}_\perp^{-1} \Pi_{\overline{\mathbb{S}}_0}^\perp (\mathscr{M}-\mathbb{I}_2)\Pi_{\overline{\mathbb{S}}_0}\mathfrak{L}^{(0)} \Pi_{\overline{\mathbb{S}}_0}^\perp - \mathscr{M}_\perp^{-1} \Pi_{\overline{\mathbb{S}}_0}^\perp \mathcal{L}^{(0)}  \Pi_{\overline{\mathbb{S}}_0} (\mathscr{M}-\mathbb{I}_2)\Pi_{\overline{\mathbb{S}}_0}^\perp - \varepsilon \mathscr{M}_\perp^{-1} \Pi_{\overline{\mathbb{S}}_0}^\perp \partial_{x} \mathcal{R} \mathscr{M}_\perp .
	\end{aligned}
\end{equation}
The operator $\sR_{\rm f}$ in \eqref{new.finiterank} is
    infinitely smoothing due to the presence of the finite dimensional projection operator $\Pi_{\overline{\mathbb{S}}_0}$ in the first two terms (finite rank contributions, therefore infinitely smoothing) and since $\cR$ in Proposition \ref{prop Lnormal}-$(ii)$ is a matrix integral operator with smooth kernels.
    Let $l_0\in\mathbb{N}$ and $M \geqslant q+l_0+1$. There exists $\aleph(M,l_0)\triangleq \aleph(q,d,M,l_0)>0$ such that, for any $n_1,n_2 \in \mathbb{R}$, with $n_1+n_2 \leqslant M - (q+l_0)-1$ and any $l\in\mathbb{N}^d$, with $|l|\leqslant l_0$, the operator $\langle D \rangle^{n_1} \partial_{\varphi}^{l} \mathscr{R}_{\rm f}\langle D \rangle^{n_2} $ is $\mathcal{D}^{q}$-tame, with a tame constant satisfying, for any $s\in[s_0,S]$ and any $s_1$ as in \eqref{parametres:transport},
    \begin{align}
       & \mathfrak{M}_{\langle D \rangle^{n_1} \partial_{\varphi}^{l} \mathscr{R}_{\rm f}\langle D \rangle^{n_2}}(s) \lesssim_{s,q,d,M,l_0} \varepsilon \gamma^{-1} \big( 1 + \|\mathfrak{I}_{0} \|_{s+\aleph(M,l_0)}^{q,\gamma,\mathcal{O}} \big), \label{stima.Rf}\\
       & \| \langle D \rangle^{m_1} \partial_{\varphi}^{l} \Delta_{12}\mathscr{R}_{\rm f}\langle D \rangle^{m_2}  \|_{\mathcal{L}(\mathbf{H}_{\perp}^{s_1})} \lesssim_{s_1,q,d,M,l_0} \varepsilon \gamma^{-1} \| \Delta_{12}i \|_{s_1+\aleph(M,l_0)} .\label{stima.Rf.12}
    \end{align}
    The estimate \eqref{stima.Rf} follows by \eqref{new.finiterank}, \eqref{map.M.perp}, \eqref{map.M}, \eqref{frakL0},  Lemma \ref{lem.inverse.Mperp} and Proposition \ref{prop Lnormal}-$(ii)$. The estimate \eqref{stima.Rf.12} follows similarly by using \eqref{e:scrReps}, \eqref{e12:scrReps}, \eqref{e:Mperp}, \eqref{e12:Mperp}, Propositions \ref{prop L0} and \ref{prop trans Ego}. Finally, the fact that $\sR_{\rm f}$ is real, reversible and momentum preserving follows by Lemma \ref{lem.inverse.Mperp}, Proposition \ref{prop trans Ego}, Proposition \ref{prop L0} and Proposition \ref{prop Lnormal}-$(ii)$.
    Combining \eqref{frakL0} and \eqref{sL.perp.conj} we get, in restriction to the Cantor set $\mathcal{O}_{\textnormal{\tiny{T}},\infty}^{\gamma,\tau_1}(i_0)$ in \eqref{transport.nonres},
    $$\Pi_{\overline{\mathbb{S}}_0}^\perp \mathfrak{L}^{(0)} \Pi_{\overline{\mathbb{S}}_0}^\perp+\mathscr{R}_{\rm f}=\omega\cdot\partial_{\varphi}\Pi_{\overline{\mathbb{S}}_0}^\perp+\ii\,\mathscr{D}^{[0]}+\mathscr{R}^{[0]},$$
    where
    $$\ii\,\mathscr{D}^{[0]}\triangleq\Pi_{\overline{\mathbb{S}}_0}^{\perp}\bigg(\ii\boldsymbol{\Omega}(a,D)+\begin{pmatrix}
		(\mathtt{c}_+-a)\partial_{x} & 0\\
		0 & (\mathtt{c}_-+a)\partial_{x}
	\end{pmatrix}\bigg)\Pi_{\overline{\mathbb{S}}_0}^{\perp}$$
    and
    \begin{equation}\label{def:init-remainder}
        \mathscr{R}^{[0]}\triangleq\Pi_{\overline{\mathbb{S}}_0}^{\perp} \Big( \cR_{[-1]}^{(d)} + \cS_{M}\Big)\Pi_{\overline{\mathbb{S}}_0}^{\perp} + \mathscr{R}_{\rm f}.
    \end{equation}
    The estimates \eqref{est.sR0} follow from \eqref{def:init-remainder}, Proposition \ref{properties OPS}-$(i)$-$(iii)$-$(iv)$-$(v)$, \eqref{estimazione frakd}, \eqref{estimazione calSM} and \eqref{stima.Rf} with $(n_1,n_2)=(\tfrac{1}{2},\tfrac{1}{2})$. The estimates \eqref{est.sR0.12} follow from \eqref{def:init-remainder}, \eqref{estimazione frakd12}, \eqref{estimazione calSM12} and \eqref{stima.Rf.12}. 
\end{proof}

\subsection{Diagonalization of the linearized operator}\label{sect.KAM}

In this section we diagonalize the operator 
\begin{equation}\label{sL.0.start}
    \mathscr{L}^{[0]}\triangleq\omega\cdot\partial_{\varphi}\Pi_{\overline{\mathbb{S}}_0}^{\perp}+\im \,\mathscr{D}^{[0]}+\mathscr{R}^{[0]}.
\end{equation}
The diagonalization is performed as a KAM iterative scheme. Recall that the operator $\mathscr{L}^{[0]}$ is defined for all $(a,\omega)\in\cO$ and acting on $\bH_{\overline{\S}_{0}}^\perp$ (see \eqref{H.normal.space}), where $\sD^{[0]}$ is as in \eqref{sD.0}, with $\sD_{\kappa}^{[0]}$ acting on $H_{\overline{\S}_{0,\kappa}}^\perp$, $\kappa\in\{-,+\}$, and 
\begin{equation}
    \sR^{[0]}\triangleq \begin{pmatrix}
        \sR_{+}^{[0,d]} & \sR_{+}^{[0,o]} \\
        \sR_{-}^{[0,o]} & \sR_{-}^{[0,d]}
    \end{pmatrix}, \qquad \begin{aligned}
       & \sR_{\kappa}^{[0,d]} : H_{\overline{\S}_{0,\kappa}}^\perp \to H_{\overline{\S}_{0,\kappa}}^\perp, \\
       &  \sR_{\kappa}^{[0,o]} : H_{\overline{\S}_{0,-\kappa}}^\perp \to H_{\overline{\S}_{0,\kappa}}^\perp,
    \end{aligned}
\end{equation}
which satisfies \eqref{est.sR0}, \eqref{est.sR0.12} in Proposition \ref{lemma.sL.beforeKAM}: here we denoted
\begin{equation}
    H_{\overline{\S}_{0,\kappa}}^\perp \triangleq \Big\{\sum_{j \in \Z^*\setminus \overline{\S}_{\kappa}} h_j \be_{j} \in L_0^2(\T) \Big\}, \quad \kappa\in\{-,+\}.
\end{equation}
From now on, we fix $s_1\triangleq s_0$ and we define the constants 
\begin{align}\label{some.constants}
   & \tb \triangleq \left[\tfrac{3}{2}\ta\right] + 1 \in \N^* , \qquad \ta \triangleq {\rm max} \{ 2\tau_{2}+1 ,\mu_2\}, \qquad\tau>\tau_1,\qquad  \tau_{2} \triangleq \tau (q+1) +  q,
\end{align}
where $[\cdot]$ denotes the integer part function. We now fix the constant $M$ large enough, namely
\begin{equation}\label{choice M}
    M\triangleq q+s_0 + \tb+2\in\mathbb{N}.
\end{equation}
We also set
\begin{equation}\label{mu.tb}
    \mu(\tb)\triangleq \aleph(M,s_0+\tb),
\end{equation}
where the constant $\aleph(M,l_0)$ is given in Proposition \ref{lemma.sL.beforeKAM}, with $l_0 = s_0+\tb$. 
\begin{lem}[Smallness of $\sR^{[0]}$]\label{lem:initKAM}
    Assume \eqref{ansatz.fI.small} with $\sigma_{\rm max}\geqslant \mu(\tb)$. Then the operators 
  \begin{equation}
      \sR^{[0]}, \quad \pa_{\vf_{k}^{s_0}} \sR^{[0]}, \quad [\pa_{\vf_{k}^{s_0}}
  \sR^{[0]},\pa_{x}], \quad \pa_{\vf_{k}^{s_0+\tb}}
  \sR^{[0]}, \quad [\pa_{\vf_{k}^{s_0+\tb}}
  \sR^{[0]},\pa_{x}], \quad k\in\llbracket1,d\rrbracket, 
  \end{equation}
  are $\cD^{q}$-tame and, defining
  \begin{align}
    \mathbb{M}^{[0]}(s) & \triangleq \max_{\kappa\in\{-,+\}\atop \delta\in\{o,d\}} \max_{k\in\llbracket1,d\rrbracket} \bigg\{ \fM_{\braket{D}^{\frac12}\sR_{\kappa}^{[0,\delta]}\braket{D}^{\frac12}}(s) 
 ,  \fM_{\braket{D}^{\frac12}[\sR_{\kappa}^{[0,\delta]},\pa_{x}]\braket{D}^{\frac12}}(s),  \\
 & \qquad \qquad \qquad \qquad \fM_{\braket{D}^{\frac12} \pa_{\vf_{k}^{s_0}}
  \sR_{\kappa}^{[0,\delta]}\braket{D}^{\frac12}}(s) , \fM_{\braket{D}^{\frac12} [\pa_{\vf_{k}^{s_0}}
  \sR_{\kappa}^{[0,\delta]},\pa_{x}]\braket{D}^{\frac12}}(s) 
    \bigg\}, \label{bbM.s} \\
    \mathbb{M}^{[0]}(s,\tb) & \triangleq \max_{\kappa\in\{-,+\}\atop \delta\in\{o,d\}} \max_{k\in\llbracket1,d\rrbracket} \bigg\{  \fM_{\braket{D}^{\frac12} \pa_{\vf_{k}^{s_0+\tb}}
  \sR_{\kappa}^{[0,\delta]}\braket{D}^{\frac12}}(s) , \fM_{\braket{D}^{\frac12} [\pa_{\vf_{k}^{s_0+\tb}}
  \sR_{\kappa}^{[0,\delta]},\pa_{x}]\braket{D}^{\frac12}}(s) 
    \bigg\},\label{bbM.s.b}
\end{align}
we have, for all $s\in[s_0,S]$,
\begin{align}
    \fM_{0}(s,\tb) & \triangleq \max \big\{ \mathbb{M}^{[0]}(s), \mathbb{M}^{[0]}(s,\tb) \big\} \leqslant C(S) \varepsilon \gamma^{-1} \big(  1 + \| \fI_{0} \|_{s+\mu(\tb)}^{q,\gamma,\cO} \big), \label{def:fM0}\\
    \fM_{0}(s_0,\tb) & \leqslant C(S) \varepsilon \gamma^{-1}. \label{sR0.initial.KAM}
\end{align}
Moreover, for all $l\in \N^d$, with $|l|\leqslant s_0 + \tb$,
\begin{equation}
     \label{sR0.initial.12}
     \begin{aligned}
         \|\braket{D}^\frac12 \pa_{\vf}^{l} \Delta_{12} \sR^{[0]} \braket{D}^\frac12 \|_{\cL(\mathbf{H}_{\perp}^{s_0})},   \ \|\braket{D}^\frac12 \pa_{\vf}^{l} \Delta_{12} [\sR^{[0]} ,\pa_{x}]\braket{D}^\frac12 \|_{\cL(\mathbf{H}_{\perp}^{s_0})}\leqslant C(S) \varepsilon \gamma^{-1} \|\Delta_{12}i\|_{s_0+\mu(\tb)},
     \end{aligned}
\end{equation}
for some constant $C(S)\triangleq C(q,d,S)>0$.
\end{lem}
\begin{proof}
    Recalling \eqref{bbM.s}, \eqref{bbM.s.b}, the bounds \eqref{sR0.initial.KAM}, \eqref{sR0.initial.12} follow by \eqref{choice M} and \eqref{mu.tb}, and by \eqref{est.sR0}, \eqref{est.sR0.12} in Proposition \ref{lemma.sL.beforeKAM}. 
\end{proof}


The purpose of this subsection is to prove the reducibility of $\sL^{[0]}$ by using a KAM iteration along the scale $(N_n)_{n\in\N\cup\{-1\}}$ given by
	\begin{equation}\label{scale.KAM}
		N_{-1}\triangleq 1,\qquad \forall \, n\in\mathbb{N},\quad N_{n}\triangleq N_{0}^{\left(\frac{3}{2}\right)^{n}}, \qquad N_0\geqslant2.
	\end{equation}
We have the following proposition.
\begin{prop}[Diagonalization of $\sL^{[0]}$: the KAM iteration]\label{prop.KAM.iter}
	 Let $(\ta,\tb,\tau_{2},\mu(\tb))$ as in \eqref{some.constants}, \eqref{mu.tb}. There exist $\tau_{3}>\tau_{2}+1+\ta$ and $\delta=\delta(s_0,\tau,d)\in (0,1)$ such that,  for all $S>s_0$, there is $N_{0}\triangleq N_0(S,\tb) \in \N$ such that, if
     \begin{equation}\label{small.cond.KAM}
         N_0^{\tau_3}\fM_{0}(s_0,\tb) \gamma^{-1} \leqslant \delta \,,
     \end{equation}
     then, for any $ n\in\mathbb{N}$, the following properties hold true:
	\begin{enumerate}
		\item [$(S1)_n$] There exists a real, reversible and momentum preserving operator in the form
		\begin{equation}
\mathscr{L}^{[n]}=\omega\cdot\partial_{\varphi}\Pi_{\overline{\mathbb{S}}_0}^{\perp}+\im\,\mathscr{D}^{[n]}+\mathscr{R}^{[n]},\label{sL.n.form}
		\end{equation}
		with
        \begin{align}
            \mathscr{D}^{[n]}\triangleq\begin{pmatrix}
			\mathscr{D}_{+}^{[n]} & 0\\
			0 & \mathscr{D}_{-}^{[n]}
		\end{pmatrix},\qquad\forall\,\kappa\in\{-,+\},\quad\mathscr{D}_{\kappa}^{[n]}\triangleq\mathtt{diag}\big(\mu_{j,\kappa}^{[n]}\big)_{j\in\mathbb{Z}^*\setminus\overline{\mathbb{S}}_{\kappa}}, \label{sD.n.form}
        \end{align}
	defined for all  $(a,\omega)\in \cO$, where $\mu_{j,\kappa}^{[n]}$ are $q$-times differentiable  real functions
    \begin{equation}
        \mu_{j,\kappa}^{[n]}(a,\omega;i_0)\triangleq\mu_{j,\kappa}^{[0]}(a,\omega;i_0)+\mathtt{r}_{j,\kappa}^{[n]}(a,\omega;i_0) \in \R, \label{mu.j.n}
    \end{equation}
    with $\mu_{j,\kappa}^{[0]}(a,\omega;i_0)$ as in Proposition \ref{lemma.sL.beforeKAM}-$(i)$, satisfying $\tr_{j,\kappa}^{[0]}\triangleq 0$ and, for $n\geqslant 1$, $\kappa\in\{-,+\}$ and any $j\in\Z^* \setminus\overline{\S}_{\kappa} $,
    \begin{equation}\label{rev:trjk}
        \tr_{-j,\kappa}^{[n]}=-\tr_{j,\kappa}^{[n]}
    \end{equation}
    and
    \begin{align}
        |j| \big\|\tr_{j,\kappa}^{[n]}\big\|^{q,\gamma,\cO} \leqslant C(S,\tb) \varepsilon \gamma^{-1}, \qquad |j| \big\|\mu_{j,\kappa}^{[n]} - \mu_{j,\kappa}^{[n-1]}\big\|^{q,\gamma,\cO} \leqslant C(S,\tb) \varepsilon \gamma^{-1} N_{n-2}^{-\ta}, \label{mu.j.n.est}
    \end{align}
    for some constant $C(S,\tb)\triangleq C(q,d,S,\tb)>0$.
    \\[1mm]
    The remainder $\sR^{[n]}$ is a real, reversible and momentum preserving operator, with
  \begin{equation}
    \sR^{[n]}\triangleq \begin{pmatrix}
        \sR_{+}^{[n,d]} & \sR_{+}^{[n,o]} \vspace{0.1cm}\\
        \sR_{-}^{[n,o]} & \sR_{-}^{[n,d]}
    \end{pmatrix}, \qquad \begin{aligned}
       & \sR_{\kappa}^{[n,d]} : H_{\overline{\S}_{0,\kappa}}^\perp \to H_{\overline{\S}_{0,\kappa}}^\perp, \\
       &  \sR_{\kappa}^{[n,o]} : H_{\overline{\S}_{0,-\kappa}}^\perp \to H_{\overline{\S}_{0,\kappa}}^\perp,
    \end{aligned} \label{sR.n.form}
\end{equation}
and the operator $\braket{\pa_{\vf}}^{\tb} \sR^{[n]}$ are $\cD^{q}$-$(-1)$-modulo-tame, with modulo-tame constants
\begin{align}
    \fM_{n}^\sharp(s) \triangleq \fM_{\braket{D}^\frac12 \sR^{[n]}\braket{D}^\frac12 }^\sharp (s), \qquad \fM_{n}^\sharp(s,\tb) \triangleq \fM_{\braket{D}^\frac12 \braket{\pa_{\vf}}^{\tb}\sR^{[n]}\braket{D}^\frac12 }^\sharp (s),\label{Mn.sharp.def}
\end{align}
which satisfy, for some constant $C_*(s_0,\tb)\triangleq C_*(q,d,s_0,\tb)>0$, for all $s\in[s_0,S]$,
\begin{align}
     \fM_{n}^\sharp(s)&  \leqslant C_*(s_0,\tb) \fM_{0}(s,\tb) N_{0}^{\ta} N_{n}^{-\ta}, \label{Mn.lowest}\\
      \fM_{n}^\sharp(s,\tb)&  \leqslant C_*(s_0,\tb)\big( 2- \tfrac{1}{n+1} \big)\fM_{0}(s,\tb)  .  \label{Mn.highest}
\end{align}
	We define the sets $\cO_{\mathtt{KAM},0}^{\gamma,\tau}(i_0)\triangleq \mathcal{O}$ and, for $n\geqslant 1$, 
	\begin{align}
		&\mathcal{O}_{\mathtt{KAM},n}^{\gamma,\tau}(i_0)\triangleq
        \bigcap_{\kappa\in\{-,+\}}\bigcap_{(\ell,j,j')\in\mathbb{Z}^d\times(\mathbb{Z}\setminus\overline{\mathbb{S}}_{\kappa})^2\atop\underset{|\ell|\leqslant N_{n-1}}{\vec{\jmath}\cdot \ell+j-j'=0}}  \bigg\lbrace(a,\omega)\in\mathcal{O} \,\textnormal{ s.t. }\big|\omega\cdot \ell+\mu_{j,\kappa}^{[n-1]}(a,\omega;i_0)-\mu_{j',\kappa}^{[n-1]}(a,\omega;i_0)\big|>\frac{\gamma\langle j-j'\rangle}{\langle \ell\rangle^{\tau}}\bigg\rbrace\\
		& \cap \bigcap_{(\ell,j,j')\in\mathbb{Z}^d\times(\mathbb{Z}\setminus\overline{\mathbb{S}}_{+})\times(\mathbb{Z}\setminus\overline{\mathbb{S}}_{-})\atop\underset{|\ell|\leqslant N_{n-1}}{\vec{\jmath}\cdot \ell+j-j'=0}}  \bigg\lbrace(a,\omega)\in\mathcal{O}\,\textnormal{ s.t. }\big|\omega\cdot \ell+\mu_{j,+}^{[n-1]}(a,\omega;i_0)-\mu_{j',-}^{[n-1]}(a,\omega;i_0)\big|>\frac{\gamma \braket{j,j'}}{\langle \ell\rangle^{\tau}}\bigg\rbrace. \label{second.meln.KAM}
	\end{align}
    There exists a real, reversibility preserving and momentum preserving map, defined for all $(a,\omega)\in \cO$, of the form $\boldsymbol{\Psi}_{n}= \Pi_{\overline{\mathbb{S}}_0}^{\perp} + \bX^{[n]}$, where $\bX^{[0]}\triangleq 0$ and, for $n\in\mathbb{N}^{*}$,
      \begin{equation}
    \bX^{[n]}\triangleq \begin{pmatrix}
        X_{+}^{[n,d]} & X_{+}^{[n,o]}\vspace{0.1cm} \\
        X_{-}^{[n,o]} & X_{-}^{[n,d]}
    \end{pmatrix}, \quad \begin{aligned}
       & X_{\kappa}^{[n,d]} : H_{\overline{\S}_{0,\kappa}}^\perp \to H_{\overline{\S}_{0,\kappa}}^\perp, \\
       &  X_{\kappa}^{[n,o]} : H_{\overline{\S}_{0,-\kappa}}^\perp \to H_{\overline{\S}_{0,\kappa}}^\perp,
    \end{aligned}
\end{equation}
such that, for all $(a,\omega) \in \mathcal{O}_{\mathtt{KAM},n}^{\gamma,\tau}$
the following identity holds
\begin{equation}
\mathscr{L}^{[n]}=\boldsymbol{\Psi}_{n}^{-1}\mathscr{L}^{[0]}\boldsymbol{\Psi}_{n}. \label{sL.n.conj.to.0}
\end{equation}
The operator $\bX^{[n]}$ is $\cD^{q}$-$(-1)$-modulo-tame satisfying, for all $s\in[s_0,S]$, 
\begin{align}
     \fM_{\braket{D}^\frac12 \bX^{[n]}\braket{D}^\frac12 }^\sharp (s) & \leqslant C(q)C_{*}(s_0,\tb) \gamma^{-1} N_0^{\tau_{2}} \prod_{j=1}^{n-1} \big( 1 + 2 C(q)  N_{0}^{\bar\ta} N_{j}^{-\bar\ta} \big)\fM_{0}(s,\tb),\label{Xn.est}
     \end{align}
and, for $n\geqslant 1$,
     \begin{align}
      \fM_{\braket{D}^\frac12 (\bX^{[n]}-\bX^{[n-1]})\braket{D}^\frac12 }^\sharp (s) &\leqslant 2 C^2(q)C_{*}(s_0,\tb) \gamma^{-1} N_0^{\ta} N_{n-1}^{-\bar\ta}\prod_{j=1}^{n-2} \big( 1 + 2 C(q)  N_{0}^{\bar\ta} N_{j}^{-\bar\ta} \big)\fM_{0}(s,\tb),\label{Xn-n-1.est}
\end{align}
where $\bar\ta\triangleq \ta-\tau_{2}>0$ (see \eqref{some.constants}) and $\prod_{j=1}^{-1}=\prod_{j=1}^{0}\triangleq 1$.
\item[$(S2)_n$] Let $i_1(a,\omega)$ and $i_2(a,\omega)$ are two reversible and momentum preserving tori such that $\sR^{[n]}(i_1)$, $\sR^{[n]}(i_2)$ satisfy \eqref{sR0.initial.KAM}, \eqref{sR0.initial.12}. Then, for all $(a,\omega) \in \mathcal{O}$,
\begin{align}
    \big\| \braket{D}^{\frac12} |\Delta_{12}\sR^{[n]}| \braket{D}^{\frac12} \big\|_{\cL(\mathbf{H}_{\perp}^{s_0})} & \lesssim_{S,\tb} \varepsilon \gamma^{-1} N_{n-1}^{-\ta}\| \Delta_{12} i \|_{s_0+\mu(\tb)}, \label{Rn.12.lowest} \\
     \big\| \braket{D}^{\frac12} |\braket{\pa_{\vf}}^{\tb}\Delta_{12}\sR^{[n]}| \braket{D}^{\frac12} \big\|_{\cL(\mathbf{H}_{\perp}^{s_0})} & \lesssim_{S,\tb} \varepsilon \gamma^{-1} N_{n-1}\| \Delta_{12} i \|_{s_0+\mu(\tb)}.\label{Rn.12.highest}
\end{align}
Furthermore, for $n\geqslant 1$, $k\in\{-,+\}$ and any $j\in\Z^* \setminus\overline{\S}_{\kappa}$,
\begin{align}
|j|\big\|\Delta_{12}(\tr_{j,\kappa}^{[n]}-\tr_{j,\kappa}^{[n-1]})\big\| & \lesssim C(S,\tb)\varepsilon \gamma^{-1} N_{n-1}^{-\ta}\| \Delta_{12} i \|_{s_0+\mu(\tb)},  \label{tr.n.n-1.12} \\
|j|\big\|\Delta_{12}\tr_{j,\kappa}^{[n]}\big\| & \lesssim C(S,\tb)\varepsilon \gamma^{-1}\| \Delta_{12} i \|_{s_0+\mu(\tb)}. \label{tr.n.12}
\end{align}
	\end{enumerate}
\end{prop}
\begin{proof} We proceed by induction on $n\in\mathbb{N}.$ First, we prove that $(S1)_{n}$, $(S2)_{n}$
hold when $n=0$. 
\\
$\blacktriangleright$ {\sc Proof of $(S1)_{0}$-$(S2)_{0}$.} The property \eqref{mu.j.n} holds at $n=0$ by Proposition \ref{lemma.sL.beforeKAM}-$(i)$ with $\tr_{j,\kappa}^{[0]}\triangleq 0$.
The estimates \eqref{Mn.lowest}, \eqref{Mn.highest} at $n=0$ follow by the following claim (with the same argument as in \eqref{Mn.sharp.def} and Lemma \ref{lem:initKAM}, arguing as in Lemma 8.4 in \cite{BFM21}.
Take $R\in\big\{\sR_{+}^{[0,d]},\sR_{+}^{[0,o]},\sR_{-}^{[0,o]},\sR_{-}^{[0,d]}\big\}.$ From the Cauchy-Schwarz inequality, we infer
\begin{align}
    &\||\langle D\rangle^{\frac{1}{2}}\langle\partial_{\varphi}\rangle^{\tb}\partial_{\lambda}^k R\langle D\rangle^{\frac{1}{2}}|\rho\|_{s}^2\\
    &\leqslant\sum_{(\ell,j)\in\mathbb{Z}^{d+1}}\langle\ell,j\rangle^{2s}\bigg(\sum_{(\ell',j')\in\mathbb{Z}^{d+1}}\langle j\rangle^{\frac{1}{2}}\langle\ell-\ell'\rangle^{\tb}|(\partial_{\lambda}^kR)(\ell-\ell')|\langle j'\rangle^{\frac{1}{2}}\,\,|\rho_{\ell',j'}|\bigg)^2\\
    &\leqslant\sum_{(\ell,j)\in\mathbb{Z}^{d+1}}\langle\ell,j\rangle^{2s}\bigg(\sum_{(\ell',j')\in\mathbb{Z}^{d+1}}\langle j\rangle^{\frac{1}{2}}\langle\ell-\ell'\rangle^{s_0+\tb}\langle j-j'\rangle|(\partial_{\lambda}^kR)(\ell-\ell')|\langle j'\rangle^{\frac{1}{2}}\,\,|\rho_{\ell',j'}|\frac{1}{\langle\ell-\ell'\rangle^{s_0}\langle j-j'\rangle}\bigg)^2\\
    &\lesssim_{s_0}\sum_{(\ell,j)\in\mathbb{Z}^{d+1}}\langle j\rangle^{\frac{1}{2}}\langle\ell,j\rangle^{2s}\sum_{(\ell',j')\in\mathbb{Z}^{d+1}}\langle\ell-\ell'\rangle^{2(s_0+\tb)}\langle j-j'\rangle^2|(\partial_{\lambda}^kR)(\ell-\ell')|^2\langle j'\rangle^{\frac{1}{2}}|\rho_{\ell',j'}|^2.
\end{align} 
Let us mention that the matrix representation of the commutator $[R,\partial_x]$ is given by $\big(\ii(j-j')R_j^{j'}(\ell-\ell')\big)_{(\ell,\ell',j,j')\in(\mathbb{Z})^d\times\mathbb{Z}^2}.$ Therefore, by \eqref{useful:e-tame} and \eqref{def:fM0}, we find
\begin{align*}
    &\gamma^{2|k|}\sum_{(\ell,j)\in\mathbb{Z}^{d+1}}\langle\ell,j\rangle^{2s}\langle j\rangle^{\frac{1}{2}}\langle\ell-\ell'\rangle^{2(s_0+\tb)}\langle j-j'\rangle^2|(\partial_{\lambda}^kR)(\ell-\ell')|^2\langle j'\rangle^{\frac{1}{2}}\\
    &\lesssim_{s_0,\tb}\sum_{(\ell',j')\in\mathbb{Z}^{d+1}}|\rho_{\ell',j'}|^2\big(\fM_0^2(s_0,\tb)\langle\ell',j'\rangle^{2s}+\fM_0^2(s,\tb)\langle\ell',j'\rangle^{2s_0}\big).
\end{align*}
Combining the last two estimates, we deduce that
$$\||\langle D\rangle^{\frac{1}{2}}\langle\partial_{\varphi}\rangle^{\tb}\partial_{\lambda}^kR\langle D\rangle^{\frac{1}{2}}|\rho\|_{s}^2\lesssim_{s_0,\tb}\gamma^{-2|k|}\sum_{(\ell',j')\in\mathbb{Z}^{d+1}}|\rho_{\ell',j'}|^2\big(\fM_0^2(s_0,\tb)\langle\ell',j'\rangle^{2s}+\fM_0^2(s,\tb)\langle\ell',j'\rangle^{2s_0}\big).$$
We thus have proved \eqref{Mn.lowest}-\eqref{Mn.highest} for $n=0.$
The conjugation in \eqref{sL.n.conj.to.0} and the estimate \eqref{Xn.est} with $n=0$ hold trivially with $\bX^{[0]}\triangleq 0$. The estimates \eqref{Rn.12.lowest}-\eqref{Rn.12.highest} similarly hold at $n=0$ by Lemma \ref{lem:initKAM}.\\
$\blacktriangleright$ {\sc The reducibility step.} Take now $n\in\mathbb{N}^*$ and we assume that $(S1)_{n-1}$, $(S2)_{n-1}$
hold. We have to prove $(S_1)_{n}$ and $(S_2)_{n}$.
Let
   \begin{equation}
   \wt{\boldsymbol{\Psi}}_{n}:=\Pi_{\overline{\mathbb{S}}_0}^{\perp}+ \bY^{[n]}, \qquad \bY^{[n]}\triangleq \begin{pmatrix}
        Y_{+}^{[n,d]} & Y_{+}^{[n,o]} \vspace{0.1cm}\\
        Y_{-}^{[n,o]} & Y_{-}^{[n,d]}
    \end{pmatrix}, \qquad \begin{aligned}
       & Y_{\kappa}^{[n,d]} : H_{\overline{\S}_{0,\kappa}}^\perp \to H_{\overline{\S}_{0,\kappa}}^\perp, \\
       &  Y_{\kappa}^{[n,o]} : H_{\overline{\S}_{0,-\kappa}}^\perp \to H_{\overline{\S}_{0,\kappa}}^\perp,
    \end{aligned}\label{gen.tilde}
\end{equation}
with $Y_{\kappa}^{[n,d]}$, $Y_{\kappa}^{[n,o]}$ to determine. We transform $\sL^{[n-1]}$, given by \eqref{sL.n.form}, \eqref{sD.n.form}, \eqref{sR.n.form} at the step $n-1$, into
\begin{align}
    \sL^{[n]} & \triangleq \wt{\boldsymbol{\Psi}}_{n}^{-1} \sL^{[n-1]} \wt{\boldsymbol{\Psi}}_{n} \label{sL.n.redustep} \\
    & = \omega\cdot\partial_{\varphi}\Pi_{\overline{\mathbb{S}}_0}^{\perp}+\im\,\mathscr{D}^{[n-1]}+\wt{\boldsymbol{\Psi}}_{n}^{-1}\Big(\big[\omega\cdot\partial_{\varphi}\Pi_{\overline{\mathbb{S}}_0}^{\perp}+\im\,\mathscr{D}^{[n-1]},\bY^{[n]}\big]+\Pi_{N_{n-1}}\mathscr{R}^{[n-1]}+\Pi_{N_{n-1}}^\perp\mathscr{R}^{[n-1]}+\mathscr{R}^{[n-1]} \bY^{[n]} \Big),
\end{align}
with $\Pi_{N_{n-1}}$ defined as in \eqref{def:projectors}, \eqref{def:projectors-matrix}, and $\Pi_{N_{n-1}}^\perp \triangleq {\rm Id} - \Pi_{N_{n-1}}$.
 We shall solve the homological equation
 \begin{equation}     \big[\omega\cdot\partial_{\varphi}\Pi_{\overline{\mathbb{S}}_0}^{\perp}+\im\,\mathscr{D}^{[n-1]},\bY^{[n]}\big]+\Pi_{N_{n-1}}\mathscr{R}^{[n-1]}=\big\lfloor \Pi_{N_{n-1}}\mathscr{R}^{[n-1]}\big\rfloor, \label{homol.compact}
 \end{equation}
 with $\big\lfloor \Pi_{N_n}\mathscr{R}^{[n-1]}\big\rfloor$ defined as in \eqref{def:diagcomp}, \eqref{def:diagcomp-matrix}.
    This latter corresponds to a system of four equations
   \begin{equation}
        \begin{cases}
\big[\omega\cdot\partial_{\varphi}\Pi_{\kappa}^{\perp}+\im\,\mathscr{D}_{\kappa}^{[n-1]},Y_{\kappa}^{[n,d]}\big]=\big\lfloor \Pi_{N_{n-1}}\mathscr{R}_{\kappa}^{[n-1,d]}\big\rfloor-\Pi_{N_{n-1}}\mathscr{R}_{\kappa}^{[n-1,d]}, \vspace{0.1cm}\\   \big(\omega\cdot\partial_{\varphi}\Pi_{\kappa}^{\perp}+\im\,\mathscr{D}_{\kappa}^{[n-1]}\big)Y_{\kappa}^{[n,o]}-Y_{\kappa}^{[n,o]}\big(\omega\cdot\partial_{\varphi}\Pi_{-\kappa}^{\perp}+\im\,\mathscr{D}_{-\kappa}^{[n-1]}\big)=-\Pi_{N_{n-1}}\mathscr{R}_{\kappa}^{[n-1,o]},
        \end{cases} \kappa\in \{ -,+ \} . \label{hom.eq.system}
    \end{equation}
We now define the components of $\bY^{[n]}$ in \eqref{gen.tilde}, for any $(a,\omega)\in \cO$ and $\kappa\in\{-,+\}$, by
    \begin{equation}\label{choice psin diag}
        \big( Y_{\kappa}^{[n,d]} \big)_{j}^{j'}(\ell)\triangleq\begin{cases}
        \im \,\varrho_{\ell,j,j',k}^{[n-1,d]}\big( \sR_{\kappa}^{[n-1,d]} \big)_{j}^{j'}(\ell)  & \textnormal{ if } \begin{cases}
            (\ell,j,j')\neq(0,j,j) , \ \ell \in \Z^d,  \ j,j'\in \Z\setminus \overline{\S}_{\kappa}, \\
            \vec{\jmath}\cdot \ell +j-j' = 0 , \  \braket{\ell} \leqslant N_{n-1},
        \end{cases} \\
        0  & \textnormal{ otherwise },
    \end{cases}
    \end{equation}
    and 
      \begin{equation}\label{choice psin antidiag}
        \big( Y_{\kappa}^{[n,o]} \big)_{j}^{j'}(\ell)\triangleq\begin{cases}
        \im\,\varrho_{\ell,j,j',k}^{[n-1,o]}\big( \sR_{\kappa}^{[n-1,o]} \big)_{j}^{j'}(\ell)  & \textnormal{ if } \begin{cases}
           \ell \in \Z^d,  \ j\in \Z\setminus \overline{\S}_{\kappa}, \ j'\in \Z\setminus \overline{\S}_{-\kappa}, \\
            \vec{\jmath}\cdot \ell +j-j' = 0 , \  \braket{\ell} \leqslant N_{n-1},
        \end{cases} \\
        0  & \textnormal{ otherwise },
    \end{cases}
    \end{equation}
    where
    \begin{align}
\varrho_{\ell,j,j',\kappa}^{[n-1,d]}(a,\omega;i_0) &  \triangleq\frac{\chi\Big(\big(\omega\cdot\ell+\mu_{j,\kappa}^{[n-1]}(a,\omega;i_0)-\mu_{j',\kappa}^{[n-1]}(a,\omega;i_0)\big)(\gamma\langle j-j'\rangle)^{-1}\langle\ell\rangle^{\tau}\Big)}{\omega\cdot\ell+\mu_{j,\kappa}^{[n-1]}(a,\omega;i_0)-\mu_{j',\kappa}^{[n-1]}(a,\omega;i_0)}, \label{rhodiag} \\
\varrho_{\ell,j,j',\kappa}^{[n-1,o]}(a,\omega;i_0) & \triangleq\frac{\chi\Big(\big(\omega\cdot\ell+\mu_{j,\kappa}^{[n-1]}(a,\omega;i_0)-\mu_{j',-\kappa}^{[n-1]}(a,\omega;i_0)\big)\big(\gamma\braket{j,j'}\big)^{-1}\langle\ell\rangle^{\tau}\Big)}{\omega\cdot\ell+\mu_{j,\kappa}^{[n-1]}(a,\omega;i_0)-\mu_{j',-\kappa}^{[n-1]}(a,\omega;i_0)}, \label{rho antidiag}
    \end{align}
and $\chi \in C^{\infty}(\R,\R)$ is an even, positive $C^\infty$ cut-off function as in \eqref{def chi}. From the definition of the components in \eqref{choice psin diag}-\eqref{choice psin antidiag}, we deduce that the operators $\bY^{[n]}$ and $\wt{\boldsymbol{\Psi}}_{n}$ are reversibility and momentum preserving, using Lemma \ref{lem Fourier coeff op} and the induction hypothesis on $\sR^{[n-1]}$.
With this choice, by \eqref{homol.compact}, \eqref{hom.eq.system} \eqref{choice psin diag}, \eqref{choice psin antidiag}, for any $(a,\omega)\in \cO_{\mathtt{KAM},n}^{\gamma,\tau}(i_0)$ as in \eqref{second.meln.KAM}, the conjugation in \eqref{sL.n.redustep} takes the form in \eqref{sL.n.form}, with
\begin{align}
    \mathscr{D}^{[n]} & \triangleq\mathscr{D}^{[n-1]} -\im \,\big\lfloor \Pi_{N_{n-1}}\mathscr{R}^{[n-1]}\big\rfloor, \label{sD.update}\\
    \mathscr{R}^{[n]} & \triangleq \wt{\boldsymbol{\Psi}}_{n}^{-1} \Big(-\bY^{[n]} \big\lfloor \Pi_{N_{n-1}}\mathscr{R}^{[n-1]}\big\rfloor+\Pi_{N_{n-1}}^{\perp}\mathscr{R}^{[n-1]}+\mathscr{R}^{[n-1]} \bY^{[n]} \Big). \label{sR.update}
\end{align}
From \eqref{sD.update}-\eqref{sR.update}, we deduce that the operator $\sL^{[n]}$ in \eqref{sL.n.form} and $\sR^{[n]}$ are real, reversible and momentum preserving, using Remark \ref{compo:RevMom}, the fact that $\bY^{[n]}$ and $\wt{\boldsymbol{\Psi}}_{n}$ are reversibility and momentum preserving, and the induction hypothesis on $\sR^{[n-1]}$.
Moreover, the conjugation in \eqref{sL.n.conj.to.0} holds at the step $n$ by defining
\begin{equation}\label{bX.n.update}
    \boldsymbol{\Psi}_{n} \triangleq  \boldsymbol{\Psi}_{n-1} \wt{\boldsymbol{\Psi}}_{n} = \Pi_{\overline{\mathbb{S}}_0}^{\perp} + \bX^{[n]}, \quad \bX^{[n]} \triangleq \bX^{[n-1]} + \big( \Pi_{\overline{\mathbb{S}}_0}^{\perp} + \bX^{[n-1]} \big) \bY^{[n]}, 
\end{equation}
where $\boldsymbol{\Psi}_{n-1} = \Pi_{\overline{\mathbb{S}}_0}^{\perp} + \bX^{[n-1]}$ is given by induction and $\wt{\boldsymbol{\Psi}}_{n}$ is as in \eqref{gen.tilde}, \eqref{choice psin diag}, \eqref{choice psin diag}. From \eqref{bX.n.update}, we deduce that $\boldsymbol{\Psi}_{n}$ is real, reversibility and momentum preserving, using Remark \ref{compo:RevMom}, the fact that $\bY^{[n]}$ is reversibility and momentum preserving, and the induction hypothesis on $\bX^{[n-1]}$.

Before continuing proving the remaining claimed estimate, we first show the estimates for the solutions of the homological equations.
\\
$\blacktriangleright$ {\sc Estimates for the operator $\bY^{[n]}$.} We claim that the operator $\bY^{[n]}$, defined in \eqref{gen.tilde}, \eqref{choice psin diag}, \eqref{choice psin antidiag}, and the operator $\braket{\pa_{\vf}}^{\tb}\bY^{[n]}$ are $\cD^{q}$-$(-1)$-modulo-tame satisfying, for all $s \in [s_0,S-\mu(\tb)]$,
\begin{align}
     \fM_{\braket{D}^\frac12 \bY^{[n]}\braket{D}^\frac12 }^\sharp (s)& \lesssim_{q} \gamma^{-1} N_{n-1}^{\tau_2} \fM_{n-1}^\sharp(s), \label{bY.n.lowest}\\ 
    \fM_{\braket{D}^\frac12 \braket{\pa_{\vf}}^{\tb}\bY^{[n]}\braket{D}^\frac12 }^\sharp (s) &  \lesssim_{q} \gamma^{-1} N_{n-1}^{\tau_2}  \fM_{n-1}^\sharp(s,\tb), \label{bY.n.highest}
\end{align}
where $\tau_{2}\triangleq \tau(q+1)+q$ (see also \eqref{some.constants}), and
\begin{align}
    \big\| \braket{D}^{\frac12} |\Delta_{12}\bY^{[n]}| \braket{D}^{\frac12} \big\|_{\cL(\mathbf{H}_{\perp}^{s_0})} & \lesssim \gamma^{-1} N_{n-1}^{2\tau+1}\Big(  \big\| \braket{D}^{\frac12} |\Delta_{12}\sR^{[n-1]}| \braket{D}^{\frac12} \big\|_{\cL(\mathbf{H}_{\perp}^{s_0})}  \label{bY.n.12.lowest} \\
    & \quad + \big\| \braket{D}^{\frac12} |\sR^{[n-1]}(i_2)| \braket{D}^{\frac12} \big\|_{\cL(\mathbf{H}_{\perp}^{s_0})}  \| \Delta_{12} i \|_{s_0+\mu(\tb)} \Big), \\
     \big\| \braket{D}^{\frac12} |\braket{\pa_{\vf}}^{\tb}\Delta_{12}\sR^{[n]}| \braket{D}^{\frac12} \big\|_{\cL(\mathbf{H}_{\perp}^{s_0})} &\lesssim \gamma^{-1} N_{n-1}^{2\tau+1}\Big(  \big\| \braket{D}^{\frac12} |\braket{\pa_{\vf}}^{\tb}\Delta_{12}\sR^{[n-1]}| \braket{D}^{\frac12} \big\|_{\cL(\mathbf{H}_{\perp}^{s_0})}  \label{bY.n.12.highest} \\
    & \quad + \big\| \braket{D}^{\frac12} |\braket{\pa_{\vf}}^{\tb}\sR^{[n-1]}(i_2)| \braket{D}^{\frac12} \big\|_{\cL(\mathbf{H}_{\perp}^{s_0})}  \| \Delta_{12} i \|_{s_0+\mu(\tb)} \Big)  .
\end{align}
We prove that the estimate \eqref{bY.n.highest} holds for $\braket{\pa_{\vf}}^{\tb}Y_{\kappa}^{[n,d]}$. The analogous estimates \eqref{bY.n.lowest}, \eqref{bY.n.highest} for  $Y_{\kappa}^{[n,d]}$, $Y_{\kappa}^{[n,o]}$, $\braket{\pa_{\vf}}^{\tb}Y_{\kappa}^{[n,o]}$ and the estimates \eqref{bY.n.12.lowest}, \eqref{bY.n.12.highest} follow similarly and therefore we omit them.  \\
First, setting $\lambda\triangleq (a,\omega) \in \mathcal{O}$, we rewrite $\varrho_{\ell,j,j',\kappa}^{[n-1,d]}$ in \eqref{rhodiag} as
\begin{equation}
    \varrho_{\ell,j,j',\kappa}^{[n-1,d]}(\lambda) = \frac{\chi(f(\lambda)\rho^{-1})}{f(\lambda)} , \qquad f(\lambda) \triangleq \omega\cdot\ell + \mu_{j,\kappa}^{[n-1]}(\lambda) - \mu_{j',\kappa}^{[n-1]}(\lambda), \qquad \rho\triangleq \gamma \braket{j-j'} \braket{\ell}^{-\tau} . \label{rho.diag.simple}
\end{equation}
By Lemma \ref{lem prop eig VP}-$(iv)$, \eqref{estimazione ttc} in Proposition \ref{prop trans Ego}-$(i)$ and the induction assumption on \eqref{mu.j.n.est}, we have that
\begin{equation}
    \forall \, k_0 \in \N^{d+1},\quad |k_0|\leqslant q,\quad |\pa_{\lambda}^{k_0} f(\lambda) | \lesssim_{k_0} \braket{\ell} \braket{j-j'}.
\end{equation}
Therefore, together with \eqref{rho.diag.simple}, \eqref{second.meln.KAM}, \eqref{def chi} and the Fa\'a di Bruno formula, we deduce that, for any $k_1 \in \N^{d+1}$, $|k_1|\leqslant q$,
\begin{equation}
    |\pa_{\lambda}^{k_1}\varrho_{\ell,j,j',\kappa}^{[n-1,d]}(\lambda) | \lesssim_{k_1} \gamma^{-1-|k_1|} \braket{\ell}^{\tau(|k_1|+1)+|k_1|} \braket{j-j'}^{-1} \lesssim_{k_1} \gamma^{-1-|k_1|} \braket{\ell}^{\tau(|k_1|+1)+|k_1|} .
\end{equation}
Consequently, by \eqref{choice psin diag} and Leibniz's rule, for any $k \in \N^{d+1}$, $|k|\leqslant q$,
\begin{equation}
    \big| \pa_{\lambda}^{k}\big( Y_{\kappa}^{[n,d]} \big)_{j}^{j'}(\ell) \big| \lesssim_{q} \gamma^{-1-|k|} \braket{\ell}^{\tau_{2}} \sum_{k_2\in \N^{d+1} \atop |k_2|\leqslant |k|}  \gamma^{|k_2|}  \big| \pa_{\lambda}^{k_2}\big( \sR_{\kappa}^{[n,d]} \big)_{j}^{j'}(\ell) \big|,
\end{equation}
with $\tau_{2}$ as in \eqref{some.constants}.
By \eqref{choice psin diag}, we have $\big( \sR_{\kappa}^{[n,d]} \big)_{j}^{j'}(\ell)= 0 $ for all $\braket{\ell}> N_{n-1}$. For any $k\in \N^{d+1}$, $|k|\leqslant q$, we get
\begin{align}
     \big\| \braket{D}^{\frac12} |\braket{\pa_{\vf}}^{\tb}\pa_{\lambda}^{k}Y_{\kappa}^{[n,d]}| \braket{D}^{\frac12} h \big\|_{s}^2 & \lesssim_{q} \gamma^{-2(1+|k|)} N_{n-1}^{2\tau_{2}}  \sum_{k_2\in \N^{d+1} \atop |k_2|\leqslant |k|}  \gamma^{|k_2|}  \big\| \braket{D}^{\frac12} |\braket{\pa_{\vf}}^{\tb}\pa_{\lambda}^{k}\sR_{\kappa}^{[n-1,d]}| \braket{D}^{\frac12} h \big\|_{s}^2 \\
     & \lesssim_{q}\gamma^{-2(1+|k|)} N_{n-1}^{2\tau_{2}}  \big( \fM_{n-1}^\sharp(s,\tb)^2 \| h \|_{s_0}^{2} + \fM_{n-1}^\sharp(s_0,\tb)^2 \| h \|_{s}^{2}  \big),
\end{align}
where we used Definition \ref{modulo.tame.def} and \eqref{Mn.sharp.def}, from which we conclude that $$\fM_{\braket{D}^\frac12 \braket{\pa_{\vf}}^{\tb}Y_{\kappa}^{[n,d]}\braket{D}^\frac12 }^\sharp (s)   \lesssim_{q} \gamma^{-1} N_{n-1}^{\tau_2}  \fM_{n-1}^\sharp(s,\tb)$$
as claimed.
\\
$\blacktriangleright$ {\sc Estimates on the new diagonal $\sD^{[n]}$.}
The diagonal operator $\sD^{[n]}$ in \eqref{sD.update} has the form in \eqref{sD.n.form}, with eigenvalues $\mu_{j,\kappa}^{[n]}(a,\omega,i_0)$ of the form \eqref{mu.j.n}, where we define, using \eqref{sD.update}, Lemma \ref{lem Fourier coeff op} and the fact that $\sR^{[n-1]}$ is reversible by induction assumption,
\begin{equation}\label{tr.update}
    \tr_{j,\kappa}^{[n]} \triangleq \tr_{j,\kappa}^{[n-1]} -\im \big( \sR_{\kappa}^{[n-1,d]} \big)_{j}^{j}(0) \in \R.
\end{equation}
The property \eqref{rev:trjk} follows from \eqref{tr.update} and the reversibility of $\mathscr{R}^{[n-1]}$ by virtue of Lemma \ref{lem Fourier coeff op}.
We now prove that the estimates in \eqref{mu.j.n.est} hold. The proof of \eqref{tr.n.12}, \eqref{tr.n.n-1.12} follows analogously and therefore we omit it. Recalling the definition of $\fM_{n-1}^\sharp(s_0)$ in \eqref{Mn.sharp.def} (with $s=s_0$) and Definition \ref{modulo.tame.def}, we deduce that for all $k\in\N^{d+1}$ with $|k|\leqslant q$,
\begin{equation}\label{tr.update.est}
    |j| \big| \pa_{\lambda}^{k}\big(\tr_{j,\kappa}^{[n]} -\tr_{j,\kappa}^{[n-1]}\big)\big| = |j|^{\frac12} \big| \pa_{\lambda}^{k}\big( \sR_{\kappa}^{[n-1,d]} \big)_{j}^{j}(0)\big| |j|^{\frac12} \lesssim \gamma^{-|k|} \fM_{n-1}^\sharp(s_0).
\end{equation}
Therefore, the estimates in \eqref{mu.j.n.est} at the step $n$ follow by \eqref{tr.update}, \eqref{tr.update.est} and the induction assumption on \eqref{mu.j.n.est}, \eqref{Mn.lowest} at  the step $n-1$.
\\
$\blacktriangleright$ {\sc Estimates for the new remainder $\sR^{[n]}$.} We now prove that $\sR^{[n]}$ in \eqref{sR.update} and $\braket{\pa_{\vf}}^{\tb} \sR^{[n]}$ are $\cD^{q}$-$(-1)$-modulo-tame, and that \eqref{Mn.lowest}, \eqref{Mn.highest} hold. The proof of \eqref{Rn.12.lowest}, \eqref{Rn.12.highest} follows analogously and therefore we omit it.
The estimates in  \eqref{Mn.lowest}, \eqref{Mn.highest} follow by the following iterative estimates, for some positive constants $C(q), C(q,\tb)>0$,
\begin{align}
    \fM_{n}^\sharp(s) & \leqslant N_{n-1}^{-\tb}\fM_{n-1}^\sharp(s,\tb)+C(q)\gamma^{-1} N_{n-1}^{\tau_2}\fM_{n-1}^\sharp(s)\fM_{n-1}^\sharp(s_0), \label{iter.lowest} \\
    \fM_{n}^\sharp(s,\tb) & \leqslant \fM_{n-1}^\sharp(s,\tb)+C(q,\tb) \gamma^{-1} N_{n-1}^{\tau_{2}}\big(\fM_{n-1}^\sharp(s,\tb)\fM_{n-1}^\sharp(s_0)+\fM_{n-1}^\sharp(s_0,\tb)\fM_{n-1}^\sharp(s)\big) . \label{iter.highest}
\end{align}
The above estimates \eqref{iter.lowest}-\eqref{iter.highest} follow by \eqref{sR.update}, \eqref{Mn.sharp.def}, Proposition \ref{prop.modulo.tame} and estimates \eqref{bY.n.lowest}-\eqref{bY.n.highest}. The claimed estimates are then deduced by an induction argument on \eqref{iter.lowest} and \eqref{iter.highest}, imposing the smallness conditions
\begin{equation}
    \big(2-\tfrac{1}{n}\big) N_{n-1}^{-\tb} \leqslant \tfrac12 N_{0}^{\ta} N_{n}^{-\ta}, \qquad C(q)C_{*}(s_0,\tb) \gamma^{-1} \fM_{0}(s_0,\tb) N_{0}^{\ta} N_{n-1}^{\tau_{2}-2\ta} \leqslant \tfrac12 N_{n}^{-\ta},
\end{equation}
to obtain \eqref{Mn.lowest}, and
\begin{equation}
    2\big(2-\tfrac{1}{n}\big) C(q,\tb) C_{*}(s_0,\tb) \gamma^{-1}  \fM_{0}(s_0,\tb) N_{0}^{\ta} N_{n-1}^{\tau_{2}-\ta} \leqslant \tfrac{1}{n}-\tfrac{1}{n+1}
\end{equation}
to obtain \eqref{Mn.highest}: both the above conditions hold by the choice of parameters in \eqref{some.constants}, \eqref{scale.KAM}, and the smallness condition \eqref{small.cond.KAM}, with $N_{0}=N_{0}(S,\tb)\gg 1$ sufficiently large.
\\
$\blacktriangleright$ {\sc Estimates on the transformation $\boldsymbol{\Psi}_{n}$.} We now prove the estimates \eqref{Xn.est}, \eqref{Xn-n-1.est} for $\bX^{[n]}$ defined in \eqref{bX.n.update}. First, we show \eqref{Xn-n-1.est}. By \eqref{bX.n.update} and Proposition \ref{prop.modulo.tame}-$(i)$-$(ii)$, we have
\begin{align}
    &\fM_{\braket{D}^\frac12 (\bX^{[n]}-\bX^{[n-1]}) \braket{D}^\frac12}^\sharp(s)   \leqslant  \fM_{\braket{D}^\frac12 \bY^{[n]}\braket{D}^\frac12}^\sharp(s) \\
    & \quad \quad + C(q) \Big( \fM_{\braket{D}^\frac12 \bX^{[n-1]}\braket{D}^\frac12}^\sharp(s) \fM_{\braket{D}^\frac12 \bY^{[n]}\braket{D}^\frac12}^\sharp(s_0) + \fM_{\braket{D}^\frac12 \bX^{[n-1]}\braket{D}^\frac12}^\sharp(s_0) \fM_{\braket{D}^\frac12 \bY^{[n]}\braket{D}^\frac12}^\sharp(s)\Big).
\end{align}
By \eqref{bY.n.lowest}, \eqref{Mn.lowest} at the step $n-1$, \eqref{scale.KAM} and the induction assumption on \eqref{Xn.est} at the step $n-1$, we have
\begin{align}
    &\fM_{\braket{D}^\frac12 (\bX^{[n]}-\bX^{[n-1]}) \braket{D}^\frac12}^\sharp(s)   \leqslant C(q)C_{*}(s_0,\tb) \gamma^{-1} N_0^{\ta} N_{n-1}^{\tau_{2}-\ta}\fM_{0}(s,\tb) \\
    & \quad  +2 C(q)^3C_{*}(s_0,\tb) \gamma^{-2} N_0^{\tau_2+\ta} N_{n-1}^{\tau_{2}-\ta} \prod_{j=1}^{n-2} \big( 1 + 2 C(q) N_0^{\bar\ta} N_{j}^{-\bar\ta} \big) \fM_{0}(s,\tb)\fM_{0}(s_0,\tb) \\
    & =  C(q)C_{*}(s_0,\tb) \gamma^{-1} N_0^{\ta} N_{n-1}^{-\bar\ta} \fM_{0}(s,\tb)\Big( 1 + 2 C(q)^2 C_{*}(s_0,\tb)\gamma^{-1} N_0^{\tau_{2}}  \prod_{j=1}^{n-2} \big( 1 + 2 C(q) N_0^{\bar\ta} N_{j}^{-\bar\ta} \big) \fM_{0}(s_0,\tb) \Big)\fM_{0}(s,\tb) \\
    & \leqslant 2 C(q)^2 C_*(s_0,\tb) \gamma^{-1} N_0^{\ta} N_{n-1}^{-\bar\ta}  \fM_{0}(s,\tb) \prod_{j=1}^{n-2} \big( 1 + 2 C(q) N_0^{\bar\ta} N_{j}^{-\bar\ta} \big) \fM_{0}(s,\tb),
\end{align}
where in the last inequality we used the smallness condition \eqref{small.cond.KAM}. This proves \eqref{Xn-n-1.est}. We now prove \eqref{Xn.est}. By Proposition \ref{prop.modulo.tame}-$(i)$ and \eqref{Xn-n-1.est}, we get
\begin{align}
    \fM_{\braket{D}^\frac12 \bX^{[n]}\braket{D}^\frac12}^\sharp(s)   & \leqslant \fM_{\braket{D}^\frac12 \bX^{[n-1]} \braket{D}^\frac12}^\sharp(s) + \fM_{\braket{D}^\frac12 (\bX^{[n]}-\bX^{[n-1]}) \braket{D}^\frac12}^\sharp(s)  \\
    & \leqslant C(q) C_{*}(s_0,\tb) \gamma^{-1} N_{0}^{\tau_{2}}  \prod_{j=1}^{n-2} \big( 1 + 2 C(q) N_0^{\bar\ta} N_{j}^{-\bar\ta} \big) \fM_{0}(s,\tb) \\
    &\quad +2 C(q)^2 C_*(s_0,\tb) \gamma^{-1} N_0^{\ta} N_{n-1}^{-\bar\ta}  \fM_{0}(s,\tb) \prod_{j=1}^{n-2} \big( 1 + 2 C(q) N_0^{\bar\ta} N_{j}^{-\bar\ta} \big) \fM_{0}(s,\tb) \\
    & = C(q) C_{*}(s_0,\tb) \gamma^{-1} N_{0}^{\tau_{2}}  \big( 1 + 2C(q) N_0^{\ta-\tau_{2}} N_{n-1}^{\bar\ta} \big)\prod_{j=1}^{n-2} \big( 1 + 2 C(q) N_0^{\bar\ta} N_{j}^{-\bar\ta} \big) \fM_{0}(s,\tb),
\end{align}
which proves \eqref{Xn.est}. The proof of Proposition \ref{prop.KAM.iter} is now complete.
\end{proof}

We are now in position to fully diagonalize the operator $\sL^{[0]}$ in \eqref{sL.0.start}.
\begin{prop}[Diagonalization of $\sL^{[0]}$: convergence]\label{prop.KAM.conv}
     Under the assumptions of Proposition \ref{prop.KAM.iter},
     the following properties hold true:
	\begin{itemize}
    \item [$(i)$]
    For any $\kappa\in\{-,+\}$, $j\in\Z^*\setminus\overline{\S}_\kappa$, the real eigenvalues $\big(\mu_{j,\kappa}^{[n]}(a,\omega;i_0)\big)_{n\in\N}$ form a Cauchy sequence in $\| \,\cdot\, \|^{q,\gamma,\cO}$, defined for all $(a,\omega)\in\mathcal{O}$, that converges to
    \begin{equation}
        \mu_{j,\kappa}^{[\infty]}(a,\omega;i_0)\triangleq\mu_{j,\kappa}^{[0]}(a,\omega;i_0)+\mathtt{r}_{j,\kappa}^{[\infty]}(a,\omega;i_0) \in \R, \label{mu.j.infty}
    \end{equation}
    with $\mu_{j,\kappa}^{[0]}(a,\omega,i_0)$ as in Proposition \ref{lemma.sL.beforeKAM}-$(i)$, satisfying 
    \begin{equation}\label{rev:trjinfty}
        \tr_{-j,\kappa}^{[\infty]}=-\tr_{j,\kappa}^{[\infty]},
    \end{equation}
    \begin{align}
        |j| \big\|\tr_{j,\kappa}^{[\infty]}\big\|^{q,\gamma,\cO} \leqslant C(S,\tb) \varepsilon \gamma^{-1}, \qquad  |j| \big\|\tr_{j,\kappa}^{[\infty]} - \tr_{j,\kappa}^{[n]} \big\|^{q,\gamma,\cO} \leqslant C(S,\tb) \varepsilon \gamma^{-1} N_{n-1}^{-\ta}, \quad n\in\N, \label{mu.j.infty.est}
    \end{align}
    and, for any $i_1(a,\omega)$ and $i_2(a,\omega)$ two reversible and momentum preserving tori such that $\sR^{[n]}(i_1)$, $\sR^{[n]}(i_2)$ satisfy \eqref{sR0.initial.KAM}, \eqref{sR0.initial.12}, for all $(a,\omega) \in \mathcal{O}$,
\begin{align}
|j|\big\|\Delta_{12}\tr_{j,\kappa}^{[\infty]}\big\| & \lesssim C(S,\tb)\varepsilon \gamma^{-1}\| \Delta_{12} i \|_{s_0+\mu(\tb)}  . \label{tr.inf.12}
\end{align}
\item[$(ii)$]
 Define the set 
\begin{align}
		&\mathcal{O}_{\mathtt{KAM},\infty}^{2\gamma,\tau}(i_0)\triangleq
        \bigcap_{\kappa\in\{-,+\}}\bigcap_{(\ell,j,j')\in\mathbb{Z}^d\times(\mathbb{Z}\setminus\overline{\mathbb{S}}_{\kappa})^2\atop{\vec{\jmath}\cdot \ell+j-j'=0}}  \bigg\lbrace(a,\omega)\in\mathcal{O} \,\textnormal{ s.t. }\big|\omega\cdot \ell+\mu_{j,\kappa}^{[\infty]}(a,\omega;i_0)-\mu_{j',\kappa}^{[\infty]}(a,\omega;i_0)\big|>\frac{2\gamma\langle j-j'\rangle}{\langle \ell\rangle^{\tau}}\bigg\rbrace\\
		& \cap \bigcap_{(\ell,j,j')\in\mathbb{Z}^d\times(\mathbb{Z}\setminus\overline{\mathbb{S}}_{+})\times(\mathbb{Z}\setminus\overline{\mathbb{S}}_{-})\atop{\vec{\jmath}\cdot \ell+j-j'=0}}  \bigg\lbrace(a,\omega)\in\mathcal{O}\,\textnormal{ s.t. }\big|\omega\cdot \ell+\mu_{j,+}^{[\infty]}(a,\omega;i_0)-\mu_{j',-}^{[\infty]}(a,\omega;i_0)\big|>\frac{2\gamma \braket{j,j'}}{\langle \ell\rangle^{\tau}}\bigg\rbrace. \label{second.meln.KAM.infty}
	\end{align}
    Then, it holds that
    \begin{equation}
        \mathcal{O}_{\mathtt{KAM},\infty}^{2\gamma,\tau}(i_0) \subseteq 
        \bigcap_{n\in\N} \mathcal{O}_{\mathtt{KAM},n}^{\gamma,\tau}(i_0),
    \end{equation}
    where $\mathcal{O}_{\mathtt{KAM},n}^{\gamma,\tau}(i_0)$ is given in \eqref{second.meln.KAM}.
		\item [$(iii)$]
         There exists a real, invertible, reversibility preserving and momentum preserving map, defined for all $(a,\omega)\in \mathcal{O}$, of the form $\boldsymbol{\Psi}_{\infty}= \Pi_{\overline{\mathbb{S}}_0}^{\perp} + \bX^{[\infty]}$, where $\bX^{[\infty]}$ is a $\cD^q$-$(-1)$-modulo-tame satisfying, for all $s\in[s_0,S]$,
         \begin{equation}
\fM_{\braket{D}^\frac12\bX^{[\infty]}\braket{D}^\frac12}^\sharp(s) \lesssim_{S,\tb} \varepsilon \gamma^{-2}N_{0}^{\tau_{2}}\big( 1 + \| \fI_{0} \|_{s+\mu(\tb)}^{q,\gamma,\cO} \big),  
         \end{equation}
such that, for all $(a,\omega) \in \mathcal{O}_{\mathtt{KAM},\infty}^{2\gamma,\tau}(i_0)$,
the operator $\sL^{[0]}$ in \eqref{sL.0.start}
is conjugated to the real, reversible and momentum preserving operator
		\begin{equation}
\mathscr{L}^{[\infty]}\triangleq \boldsymbol{\Psi}_{\infty}^{-1}\mathscr{L}^{[0]}\boldsymbol{\Psi}_{\infty}=\omega\cdot\partial_{\varphi}\Pi_{\overline{\mathbb{S}}_0}^{\perp}+\im\,\mathscr{D}^{[\infty]},\label{sL.infty.conj}
		\end{equation}
		with
        \begin{align}
            \mathscr{D}^{[\infty]}\triangleq\begin{pmatrix}
			\mathscr{D}_{+}^{[\infty]} & 0\\
			0 & \mathscr{D}_{-}^{[\infty]}
		\end{pmatrix},\qquad\forall\,\kappa\in\{-,+\},\quad\mathscr{D}_{\kappa}^{[\infty]}\triangleq\underset{j\in\mathbb{Z}^*\setminus\overline{\mathbb{S}}_{\kappa}}{\mathtt{diag}}\Big(\mu_{j,\kappa}^{[\infty]}\Big), \label{sD.infty.form}
        \end{align}
	defined for all  $(a,\omega)\in \mathcal{O}$, where $\mu_{j,\kappa}^{[\infty]}$ are as in \eqref{mu.j.infty}.
	\end{itemize}
\end{prop}
\begin{proof}
     $(i)$ Follows by \eqref{mu.j.n}, \eqref{mu.j.infty.est}, \eqref{tr.n.12}, \eqref{tr.n.n-1.12} in Proposition \ref{prop.KAM.iter}, together with \eqref{scale.KAM}, \eqref{some.constants} and a telescopic series argument. The reversibility property \eqref{rev:trjinfty} is obtained at the limit from \eqref{rev:trjk}.\\
     $(ii)$ Given $n\in\N$, we aim to show that $ \mathcal{O}_{\mathtt{KAM},\infty}^{2\gamma,\tau}(i_0) \subseteq \mathcal{O}_{\mathtt{KAM},n}^{\gamma,\tau}(i_0)$. When $n=0$, the claim is trivial because $\mathcal{O}_{\mathtt{KAM},0}^{\gamma,\tau}(i_0)= \mathcal{O}$. Take $n\in\mathbb{N}^*$ and $(a,\omega)\in\mathcal{O}_{\mathtt{KAM},\infty}^{2\gamma,\tau}(i_0)$. Then, for any $\kappa\in\{-,+\}$, $(\ell,j,j')\in\mathbb{Z}^d\times(\mathbb{Z}\setminus\overline{\mathbb{S}}_{\kappa})^2$, $\vec{\jmath}\cdot \ell+j-j'=0$ and $|\ell|\leqslant N_{n-1}$, we have, by \eqref{mu.j.n}, \eqref{mu.j.infty}, \eqref{mu.j.infty.est},
    \begin{align}
        \big|\omega\cdot \ell+\mu_{j,\kappa}^{[n-1]}(a,\omega;i_0)-\mu_{j',\kappa}^{[n-1]}(a,\omega;i_0)\big| & \geqslant \big|\omega\cdot \ell+\mu_{j,\kappa}^{[\infty]}(a,\omega;i_0)-\mu_{j',\kappa}^{[\infty]}(a,\omega;i_0)\big| \\
        & \quad - \big| \tr_{j,\kappa}^{[\infty]}(a,\omega;i_0)-\tr_{j,\kappa}^{[n-1]}(a,\omega;i_0) \big| - \big| \tr_{j',\kappa}^{[\infty]}(a,\omega;i_0)-\tr_{j',\kappa}^{[n-1]}(a,\omega;i_0) \big| \\
        & \geqslant \frac{2\gamma\braket{j-j'}}{\braket{\ell}^\tau}  - C\varepsilon\gamma^{-1} |j|^{-1} N_{n-2}^{-\ta}  - C\varepsilon\gamma^{-1} |j'|^{-1} N_{n-2}^{-\ta} \\
        & \geqslant \frac{2\gamma\braket{j-j'}}{\braket{\ell}^\tau}  - 2C\varepsilon\gamma^{-1} N_{n-2}^{-\ta} \braket{j-j'} \geqslant  \frac{\gamma\braket{j-j'}}{\braket{\ell}^\tau} 
    \end{align}
    as soon as
    \begin{equation}\label{cond.incl.KAM}
        2C\varepsilon\gamma^{-1} N_{n-2}^{-\ta} \leqslant \frac{\gamma}{\braket{\ell}^\tau}, \qquad |\ell|\leqslant N_{n-1}
    \end{equation}
    which holds by \eqref{some.constants}, \eqref{scale.KAM}, for $N_0\gg 1$ large enough, independently of $n\in\N$. Similarly, for all $(\ell,j,j')\in\mathbb{Z}^d\times(\mathbb{Z}\setminus\overline{\mathbb{S}}_{+})\times(\mathbb{Z}\setminus\overline{\mathbb{S}}_{-})$, ${\vec{\jmath}\cdot \ell+j-j'=0}$, $|\ell|\leqslant N_{n-1}$, we get
    \begin{align}
        \big|\omega\cdot \ell+\mu_{j,+}^{[n-1]}(a,\omega;i_0)-\mu_{j',-}^{[n-1]}(a,\omega;i_0)\big| & \geqslant  \frac{\gamma\braket{j-j'}}{\braket{\ell}^\tau} 
    \end{align}
    as soon as \eqref{cond.incl.KAM} is satisfied. This concludes the proof of the claim.\\
    $(iii)$ We note that, for any $n\in\mathbb{N}^*$, by a convexity argument and \eqref{scale.KAM}, we have
    \begin{align}
        \prod_{j=1}^{n-1}\underbrace{\big( 1 + 2C(q) N_{0}^{\bar\ta} N_{j}^{-\bar\ta} \big)}_{\geqslant1} \leqslant \prod_{j=1}^{n-1} {\rm exp}\big(  2C(q) N_{0}^{\bar\ta} N_{j}^{-\bar\ta} \big) \leqslant {\rm exp}\Big(  2C(q) N_{0}^{\bar\ta} \sum_{j=1}^{\infty}N_{j}^{-\bar\ta}  \Big)<\infty.
    \end{align}
    This implies the convergence of the product. By \eqref{scale.KAM}, \eqref{some.constants} (recall that $\bar\ta=\ta-\tau_{2}>0$) and \eqref{Xn.est}, \eqref{Xn-n-1.est} in Proposition \ref{prop.KAM.iter}, we deduce the existence of the $\cD^{q}$-$(-1)$-modulo-tame operator $\bX^{[\infty]}$, defined for all $(a,\omega)\in\mathcal{O}$ by
    \begin{equation}
        \boldsymbol{\Psi}_{\infty} \triangleq \Pi_{\overline{\mathbb{S}}_0}^{\perp} + \bX^{[\infty]} \triangleq \lim_{n\to\infty}\big( \Pi_{\overline{\mathbb{S}}_0}^{\perp} + \bX^{[n]}  \big) = \lim_{n\to\infty} \boldsymbol{\Psi}_{n},
    \end{equation}
    such that \eqref{sL.infty.conj}, \eqref{sD.infty.form} hold for any $(a,\omega)\in\cO_{\mathtt{KAM},\infty}^{2\gamma,\tau}(i_0)$ by taking the limit in \eqref{sL.n.conj.to.0}, \eqref{sL.n.form}, \eqref{sD.n.form} in Proposition \ref{prop.KAM.iter}, using also items $(i)$, $(ii)$. This concludes the proof of item $(iii)$ and of Proposition \ref{prop.KAM.conv}.
\end{proof}

\subsection{Inversion of the linearized operator}\label{sect.almost.inverse}
We conclude this section by explicitly constructing a full right inverse for the linearized operator in the normal directions. Owing to the prior introduction of the diagonal operator $\sL^{[\infty]}$, whose entries are given by Fourier multipliers, the problem reduces to the analysis of two independent diagonal scalar components. We therefore appeal to the argument developed in \cite[Proposition 6.6]{HR21} and \cite[Lemma 12.10]{FMT25}. The result reads as follows.

\begin{prop}\label{prop.almost.inv}
    Let $(\gamma,q,d,\tau_1,\tau,s_0,\mu_{2})$ as in \eqref{d gm S s0}, \eqref{setDC}, \eqref{parametres:transport} and \eqref{some.constants}. There exists  $\sigma \triangleq \sigma(\tau,d,q) \gg 1$ large enough such that, if \eqref{ansatz.fI.small} holds with $\sigma_{\rm max}\geqslant\sigma,$
    then the following assertions are true:
    \begin{itemize}
        \item [$(i)$] Let $\sL^{[\infty]}$ be the operator defined in Proposition \ref{prop.KAM.conv}-$(iii)$. There exists a  real, reversible and momentum preserving operator $\tT$, defined for all $(a,\omega)\in\mathcal{O}$ and satisfying, for all $s\in[s_0,S]$,
        \begin{equation}\label{Est-T.full}
           \|\tT \rho \|_{s}^{q,\gamma,\cO} \lesssim \gamma^{-1} \|\rho \|_{s+\tau(q+1)+q}^{q,\gamma,\cO},
        \end{equation}
        such that for any $(a,\omega)\in \mathcal{O}_{\mathrm{\tiny{Inv}},\infty}^{\gamma,\tau}(i_0)$, where
        \begin{align}
            \mathcal{O}_{\mathrm{\tiny{Inv}},\infty}^{\gamma,\tau}(i_0)\triangleq\bigcap_{\kappa\in\{-,+\}}\bigcap_{(\ell,j)\in\mathbb{Z}^{d}\times(\mathbb{Z}^*\setminus\overline{\mathbb{S}}_{\kappa})\atop{\vec{\jmath}\cdot \ell+j=0}}\bigg\{(a,\omega)\in\mathcal{O}\quad\textnormal{s.t.}\quad\big|\omega\cdot \ell+\mu_{j,\kappa}^{[\infty]}(a,\omega;i_0)\big|>\frac{\gamma\braket{j}}{\braket{\ell}^{\tau}}\bigg\} ,\label{first.meln.cond.n}
        \end{align}
        we have
        \begin{equation}\label{LT}
            \sL^{[\infty]} \tT = \Pi_{\overline{\mathbb{S}}_0}^{\perp}.
        \end{equation}
        \item [$(ii)$] There exists a real, reversible and momentum preserving operator $\wh\tT$, defined for all $(a,\omega)\in\mathcal{O}$ and satisfying, for all $s\in[s_0,S]$,
        \begin{equation}\label{tT.est.full}
           \|\wh\tT \rho \|_{s}^{q,\gamma,\cO} \lesssim \gamma^{-1}\big( \|\rho \|_{s+\sigma}^{q,\gamma,\cO} + \| \fI_{0} \|_{s+\sigma}^{q,\gamma,\cO} \|\rho \|_{s+\sigma}^{q,\gamma,\cO}\big),
        \end{equation}
        such that, for any $n\in\N$ and for any $(a,\omega)\in \tG(\gamma,\tau_{1},\tau,i_0)$, where (see \eqref{transport.nonres}, \eqref{second.meln.KAM.infty}, \eqref{first.meln.cond.n})
        \begin{align}\label{Gn.set}
           \tG(\gamma,\tau_{1},\tau,i_0) \triangleq \cO_{{\rm T},\infty}^{\gamma,\tau_{1}} (i_0) \cap \cO_{{\tt KAM},\infty}^{2\gamma,\tau}(i_0) \cap \mathcal{O}_{\mathrm{\tiny{Inv}},\infty}^{\gamma,\tau}(i_0),
        \end{align}
        we have
        \begin{equation}\label{LT.full}
            \mathscr{L}_{\omega} \wh\tT = \Pi_{\overline{\mathbb{S}}_0}^{\perp},
        \end{equation}
        with $\mathscr{L}_{\omega}$ as in \eqref{Lperp}.
    \end{itemize}
\end{prop}
\begin{proof}
    $(i)$ By the definition of $\sL^{[\infty]}$ in \eqref{sL.infty.conj}, \eqref{sD.infty.form} of Proposition \ref{prop.KAM.conv}, we have
\begin{equation}
    \sL^{[\infty]} 
    = \begin{pmatrix}
        \omega\cdot \pa_{\vf} \Pi_{\overline{\S}_+}^\perp + \im \sD_{+}^{[\infty]} & 0 \\ 0 &   \omega\cdot \pa_{\vf} \Pi_{\overline{\S}_-}^\perp + \im \sD_{-}^{[\infty]}
    \end{pmatrix}  \triangleq \begin{pmatrix}
             \sL_{+}^{[\infty]} & 0 \\ 0 &  \sL_{-}^{[\infty]}
        \end{pmatrix},
\end{equation}
with
\begin{equation}
\forall\,\kappa\in \{-,+\}, \ \  \forall\, (\ell,j)\in \Z^d \times \Z^*\setminus \overline{\S}_{\kappa}, \quad  \mathscr{L}_{\kappa}^{[\infty]}{\bf e}_{\ell,j}=
 \im\big(\omega\cdot\ell+\mu_{j,\kappa}^{[\infty]}\big) {\bf e}_{\ell,j}.
\end{equation}
We define the  operator  $\mathtt{T}$ by 
\begin{equation}
    \tT \triangleq \begin{pmatrix}
        \tT_{+} & 0 \\ 0 & \tT_{-}
    \end{pmatrix},
\end{equation}
where, for any $\kappa\in \{-,+\}$, the diagonal operator $\tT_{\kappa}$ is defined by
\begin{equation}\label{T.in.Fourier}
h=\sum_{(\ell,j)\in \Z^{d} \times (\Z^*\setminus\overline{\mathbb{S}}_{\kappa}) \atop \vec{\jmath}\cdot\ell+j=0}h_{\ell,j}{\bf e}_{\ell,j},\qquad\mathtt{T}_{\kappa}h \triangleq
 \sum_{ (\ell,j)\in \Z^d \times (\Z^*\setminus\overline{\mathbb{S}}_{\kappa})\atop \vec{\jmath}\cdot\ell+j=0}\frac{\chi\big((\omega\cdot\ell+\mu_{j,\kappa}^{[\infty]})\upsilon^{-1}\langle\ell\rangle^{\tau}\langle j\rangle^{-1}\big)}{\im\big(\omega\cdot\ell+\mu_{j,\kappa}^{[\infty]}\big)}h_{\ell,j}\,{\bf e}_{\ell,j},
\end{equation}
where $\chi$ is the cut-off function defined in \eqref{def chi}. The operator $\tT_{\kappa}$ in \eqref{T.in.Fourier} is real, reversible and momentum preserving by \eqref{mu.j.infty}, \eqref{rev:trjinfty}, \eqref{mu.j.0}, Lemma \ref{lem prop eig VP}-$(i)$ and Lemma \ref{lem Fourier coeff op}.
Thus, for $(a,\omega)\in \cO_{{\rm Inv},\infty}^{\gamma,\tau}(i_0)$, one deduces \eqref{LT}.
Since $\|\mu_{j,\kappa}^{[\infty]}\|^{q,\gamma,\mathcal{O}}\underset{j\to\infty}{\sim}|j|$, then one easily checks the estimate \eqref{Est-T.full} directly from \eqref{T.in.Fourier} and \eqref{first.meln.cond.n}.
\\[1mm]
$(ii)$ The relation between $\sL_{\omega}$ and $\sL^{[\infty]}$ is given by 
   \begin{align}
\mathscr{W}^{-1}\sL_{\omega} \mathscr{W}=\sL^{[\infty]}, \qquad  \mathscr{W} \triangleq \mathscr{M}_{\perp} \boldsymbol{\Psi}_{\infty},
        \end{align}
        with $\mathscr{M}_{\perp}$ as in \eqref{map.M.perp} and $\boldsymbol{\Psi}_{\infty}$ as in Proposition \ref{prop.KAM.conv}-$(iii)$. Note that the map $\mathscr{W}$ is invertible by Lemma \ref{lem.inverse.Mperp}-$(ii)$ and Proposition \ref{prop.KAM.conv}-$(iii)$.
        In view of item $(i)$, it follows that 
        \begin{align}
		\mathscr{W}^{-1}\sL_{\omega} \mathscr{W} \tT =\sL^{[\infty]} \tT = \Pi_{\overline{\S}_0}^\perp.
        \end{align}
        Conjugating back, we obtain the identity
        \begin{align}
		\sL_{\omega} \mathscr{W} \tT \mathscr{W}^{-1}
        & = \mathscr{W} \mathscr{W}^{-1} = \Pi_{\overline{\S}_0}^\perp,
        \end{align}
        which proves \eqref{LT.full}  with $\wh\tT$ defined as
        \begin{equation}\label{ARI.cT.n}
            \wh\tT \triangleq  \mathscr{W} \tT \mathscr{W}^{-1}.
        \end{equation}
        The corresponding estimate \eqref{tT.est.full} for $\wh\tT$  and the fact that $\wh\tT$ is real, reversible and momentum preserving follow from the properties  of  $\tT$ in item $(i)$, together with those of $\mathscr{M}_{\perp}$ and $\boldsymbol{\Psi}_{\infty}$ established in Lemma \ref{lem.inverse.Mperp}-$(ii)$ and Proposition \ref{prop.KAM.conv}-$(iii)$, respectively, and Proposition \ref{prop.modulo.tame}-$(v)$, Proposition \ref{prop.tameconst}-$(ii)$. This ends the proof.
\end{proof}

\section{Construction of quasi-periodic solutions}\label{sect.nonlin.sol}
We provide, in this last section, a construction of a non-trivial solution to the equation 
\begin{equation}
    \mathscr{F}_{\varepsilon}(i,\alpha)\triangleq\mathscr{F}_{\varepsilon}(i,\alpha,a,\omega) = 0, \label{solve.equa.NM}
\end{equation}
with $\mathscr{F}_{\varepsilon}(i,\alpha,a,\omega)$ as in \eqref{sF.vare.def} This is done in two steps. 
First, we implement a Nash–Moser iteration, where we find a solution provided that the parameters $(a,\omega)$ belong to a suitable Borel set. The latter is constructed as the intersection of the Cantor sets required to invert the linearized operator in the normal modes for all the steps of the procedure. Then we rigidified the frequencies in order to get a solution for the original problem where $\alpha=-\tJ \omega_{\rm Eq}(a)$.
This gives rise to a final set described in terms of $a$ that we should estimate its Lebesgue measure. Actually, we prove that it has asymptotically full measure as the parameter $\varepsilon$ vanishes.

\subsection{Nash-Moser construction of a priori solutions}\label{subsect.NASH}
Here, we perform the Nash-Moser scheme which allows to find a solution of \eqref{solve.equa.NM}. This method is nowadays classic in the context of quasi-periodic solutions for quasilinear PDEs, see for instance \cite{BBHM18, BFM21, BM18} and recent developments \cite{BHM23, HHM21, HR21}. The iterative construction of the approximate solutions is summarized in the following proposition: it follows \cite[Prop. 8.1]{HHR23} (with \cite[Prop. 9.2]{BFM21} to prove the tori are traveling reversible) and therefore we only present the general scheme of the proof, omitting all the explicit estimations. 
\begin{prop}\label{prop.NASH}
	\textbf{(Nash-Moser scheme)}
	Let $(\tau_{1},\tau,q,d,S,s_0)$ as in \eqref{d gm S s0} and \eqref{some.constants}. We consider the following parameters 
    \begin{equation}\label{param.NASH.1}
        \begin{cases}
            \tilde{\ta} = \tau+3,\\
            \mu_1=3q(\tau+3)+6\overline{\sigma}+6,\\
            \nu_1=6q(\tau+3)+12\overline{\sigma}+15,\\
            \nu_2=3q(\tau+3)+6\overline{\sigma}+9,\\
            \mathtt{a}=2q(\tau+3)+5\overline{\sigma}+7,\\
            s_h=s_0+4q(\tau+3)+9\overline{\sigma}+11,\\
            \kappa_1=2s_h-s_0,
        \end{cases}
    \end{equation}
    where $\bar\sigma\triangleq \bar\sigma(\tau_{1},\tau,d)>0$ is the total loss of regularity given by Theorem \ref{theo appr inv}. 
    For $n\in\N$, we introduce the finite dimensional subspace $E_{n}$ defined by
    \begin{align}
         E_{n} \triangleq \Big\{ \fI(\vf)=\big( \Theta(\vf),I(\vf),z(\vf) \big) \ \textnormal{s.t. \eqref{rev trav torus} holds} \ : \, \Theta = \Pi_{n}\Theta , \ I = \Pi_{n}I, \ z=\Pi_{n}z  \Big\},
    \end{align}
    where $\Pi_{n} z \triangleq \Pi_{N_n}z$ as in \eqref{proj.traveling} with $N_n$ in \eqref{scale.KAM}, and $\Pi_{n} g(\vf) \triangleq \sum_{|\ell|\leqslant N_{n}} g_{\ell}e^{\im \ell\cdot \vf}$.
    \\[1mm]
    There exist $C_*>0$ and $\varepsilon_{0}>0$ such that, for any $\varepsilon\in [0,\varepsilon_{0}]$, imposing the constraint relating $\gamma$ and $N_0$ to $\varepsilon$,
    \begin{equation}\label{param.NASH.2}
        0< \varrho <  (\mathtt{a}+q+1)^{-1}, \qquad \gamma  \triangleq \varepsilon^{\varrho}, \qquad N_0 \triangleq \gamma^{-1} ,
    \end{equation}
     then following properties hold true, for any $n\in\N$:
    \begin{itemize}
        \item [$(\cP 1)_{n}$] There exists a $q$-times differentiable function
        \begin{equation}
            \tW_{n}: \cO \to E_{n-1} \times \R^d , \qquad \lambda\triangleq (a,\omega) \mapsto \big(\fI_{n}(a,\omega),\alpha_{n}(a,\omega)-\tJ \omega\big),
        \end{equation}
        satisfying $\tW_{0}\triangleq 0$ and, for $n\in\mathbb{N}^*$,
        \begin{equation}
            \| \tW_{n} \|_{s_0+\bar\sigma}^{q,\gamma,\cO} \leqslant C_{*}\varepsilon \gamma^{-1} N_0^{q\tilde{\ta}} .
        \end{equation}
        We set $\tU_{0} \triangleq \big( (\vf,0,0), -\tJ\omega \big)$ and, for $n\in\mathbb{N}^*$,
        \begin{equation}\label{Un}
            \tU_{n} \triangleq \tU_{0} + \tW_{n}=(i_n,\alpha_n), \qquad \tH_{n}\triangleq \tU_{n}- \tU_{n-1}.
        \end{equation}
        Then, for any $s\in [s_0,S]$,
        \begin{align}
            \| \tH_{1} \|_{s}^{q,\gamma,\cO} &  \leqslant \tfrac12 C_{*}\varepsilon \gamma^{-1} N_0^{q\tilde{\ta}}, \\
            \| \tH_{m} \|_{s_0+\overline{\sigma}}^{q,\gamma,\cO} & \leqslant C_{*}\varepsilon \gamma^{-1} N_{m-1}^{-\nu_2}, \quad \forall \, 2\leqslant m \leqslant n; \label{H.m.per.Cauchy}
        \end{align}
        \item [$(\cP 2)_{n}$] We set
        \begin{align}
            i_{n}(\vf) \triangleq (\vf,0,0) + \fI_{n}(\vf), \qquad \gamma_{n} \triangleq \gamma(1+2^{-n}) \in [\gamma,2\gamma] . \label{gamma.n}
        \end{align}
        Then $i_{n}$ is a traveling reversible torus, that is, according to Definition \ref{rev trav torus}, for any $\varphi\in\mathbb{T}^d$ and any $y\in\mathbb{T},$
        \begin{equation}\label{rev-trav-NM}
            \mathfrak{S}i_n(\varphi)=i_n(-\varphi)\qquad\textnormal{and}\qquad\mathfrak{T}_yi_n(\varphi)=i_n(\varphi-\vec{\jmath}y).
        \end{equation}
        We also define
        \begin{equation}\label{An.set}
            \tA_{0}^{\gamma} \triangleq \cO, \qquad \tA_{n+1}^{\gamma} \triangleq \tA_{n}^{\gamma} \cap \tG(\gamma_{n+1},\tau_{1},\tau,i_{n}) \cap \tD\tC(\gamma_{n+1},\tau),
        \end{equation}
        where $\tG(\gamma_{n+1},\tau_{1},\tau,i_{n})$ and $\tD\tC(\gamma_{n+1},\tau)$ are given respectively through \eqref{Gn.set}, \eqref{setDC}, and for any $\tv >0$,
        \begin{align}
            {\rm O}_{n}^{\tv} \triangleq \Big\{ \lambda=(a,\omega) \in \cO \ \textnormal{ s.t. } \ {\rm dist} \big( \lambda , \tA_{n}^{2\gamma} \big) < \tv N_{n}^{-\nu_1} \Big\}, \qquad {\rm dist}(x,\tA) \triangleq \inf_{y\in\tA}\| x-y \| .
        \end{align}
        Then we have the following estimate
        \begin{equation}\label{decayF}
            \| {\mathscr{F}}_{\varepsilon}(\tU_{n}) \|_{s_0}^{q,\gamma,{\rm O}_{n}^{2\gamma}} \triangleq \sum_{k \in \N^{d+1} \atop |k|\leqslant q} \gamma^{|k|} \sup_{\lambda \in {\rm O}_{n}^{2\gamma}} \| \pa_{\lambda}^{k} \mathscr{F}_{\varepsilon}(\tU_{n},\lambda) \|_{s_{0}-|k|} \leqslant C_{*} \varepsilon N_{n-1}^{-\nu_1} ;
        \end{equation}
        \item [$(\cP 3)_{n}$] We have the following growth in high regularity norm
        \begin{equation}\label{ttW:high}
            \| \tW_{n} \|_{\kappa_{1}+\bar\sigma}^{q,\gamma,\cO} \leqslant C_{*} \varepsilon \gamma^{-1} N_{n-1}^{\mu_{1}} .
        \end{equation}
    \end{itemize}
\end{prop}

\begin{proof}
We briefly sketch the iterative construction of the approximate solution, which is built over an induction argument. The starting point is $\tU_{0} \triangleq \big( (\vf,0,0), -\tJ\omega \big)$, which is an approximate zero of $\mathscr{F}_{\varepsilon}$ in \eqref{sF.vare.def}, that is up to $\varepsilon$-small error (recall also $i_{\rm triv}=(\vf,0,0)$ in \eqref{itriv}).
Let now $n\in\N$ and define
\begin{equation}
    L_{n} \triangleq \di_{i,\alpha} \mathscr{F}_{\varepsilon}(\tU_{n}),
\end{equation}
where $\tU_{n}$ is the approximate solution at the step $n$.
By Theorem \ref{theo appr inv}, we have the existence of the real, reversible and momentum preserving operator ${\rm T}_{n}$, which is an approximate right inverse of $L_{n}$ on the Cantor set $\tA_{n+1}^{\gamma}$, that is, for all $(a,\omega)\in \tA_{n+1}^{\gamma}$, one has
\begin{equation}\label{splitting.proof.NASH}
			L_{n}{\rm T}_n-{\rm Id}= \mathfrak{E}^{(n)},
		\end{equation}
where ${\rm T}_{n}$
satisfies the estimates 
\begin{align}
    	\forall\, s\in [s_0,S],\quad \| {\rm T}_{n} g\|_{s}^{q,\gamma,\mathcal{O}}\lesssim_{s,q,d}\gamma^{-1}\big(\|g\|_{s+{\overline\sigma}}^{q,\gamma,\mathcal{O}}+\|\tU_{n}\|_{s+{\overline\sigma}}^{q,\gamma,\mathcal{O}}\|g\|_{s_{0}+\overline{\sigma}}^{q,\gamma,\mathcal{O}}\big),
\end{align}
and $\mathfrak{E}^{(n)}$ satisfies the estimate \eqref{calE1}. At the step $n+1$, we construct the new approximate solution by defining
\begin{equation}\label{H-n+1}
    \tU_{n+1}\triangleq\tU_n + H_{n+1}, \quad  H_{n+1}\triangleq(\mathfrak{I}_{n+1},\alpha_{n+1})\triangleq-\mathbf{\Pi}_n{\rm T}_n \Pi_{n}\mathcal{F}_{\varepsilon}(\tU_n),
\end{equation}
where
$$\mathbf{\Pi}_n(\mathfrak{I},\alpha)\triangleq(\Pi_n\mathfrak{I},\alpha),\qquad\mathbf{\Pi}_n^{\perp}(\mathfrak{I},\alpha)\triangleq(\Pi_n^{\perp}\mathfrak{I},0).$$ Since ${\rm T}_{n}$ is reversible and momentum preserving and since $\mathscr{F}_{\varepsilon}$ is a reversible, space translation invariant vector field, one deduces that $\tU_{n+1}$ is a reversible traveling wave, as is $\tU_{n}$ by induction. 
By Taylor formula and the choice of $H_{n+1}$ in \eqref{H-n+1}, we have
\begin{align}\label{TaylorNM}
    \mathscr{F}_{\varepsilon}(\tU_{n+1}) & =  \mathscr{F}_{\varepsilon}(\tU_{n}) + L_{n} H_{n+1} + Q_{n} \\
    & = \mathscr{F}_{\varepsilon}(\tU_{n}) - L_{n} \mathbf{\Pi}_n{\rm T}_n \Pi_{n}\mathcal{F}_{\varepsilon}(\tU_n) + Q_{n} \\
    &=\Pi_{n}^{\perp}\mathscr{F}(\tU_n)-\Pi_n(L_n{\rm T}_{n}-{\rm Id})\Pi_n\mathscr{F}(\tU_n)+ (L_n\mathbf{\Pi}_n^{\perp}-\Pi_n^{\perp}L_n){\rm T}_{n}\Pi_n\mathscr{F}_{\varepsilon}(\tU_n)+Q_n,
\end{align}
 where
    $$Q_n\triangleq\mathscr{F}_{\varepsilon}(\tU_n+H_{n+1})-\mathscr{F}(\tU_n)-L_nH_{n+1}.$$
The estimate \eqref{decayF} is obtained by induction using \eqref{TaylorNM}, \eqref{calE1}, Lemma \ref{lemma.Xp.est} Lemma \ref{lem:projn} and the choice of parameters \eqref{param.NASH.1}-\eqref{param.NASH.2}. the estimate \eqref{ttW:high} is obtained by a direct estimation using \eqref{H-n+1}, the induction hypothesis and the choice of parameters \eqref{param.NASH.1}. For more details, we refer to \cite{HR21}.
    \end{proof}

The previous iteration procedure converges and allows to find a non trivial reversible
quasi-periodic solution of our problem provided some restriction on 
$a\in [a_0,a_1]$. More precisely, we have the following result.
\begin{cor}\label{cor.limit.NM}
    Under the assumption of Theorem \ref{prop.NASH}, there exists $\varepsilon_{0}>0$ such that, for all $\varepsilon\in (0,\varepsilon_{0})$, the following assertions are true. There exists a $q$-times differentiable function
        \begin{equation}
            \tU_{\infty}: \cO \to \big( \T^d \times \R^d \times \bH_{\perp}^{s_0} \big) \times \R^d, \quad  (a,\omega) \mapsto \big(i_{\infty}(a,\omega),\alpha_{\infty}(a,\omega) \big) ,
        \end{equation}
    such that, for any  $(a,\omega)\in\tG_{\infty}^{\gamma}$, where $\tG_{\infty}^{\gamma}$ is defined by
    \begin{equation}\label{tG.infty.set}
        \tG_{\infty}^{\gamma} \triangleq \bigcap_{n\in\N} \tA_{n}^{\gamma},
    \end{equation}
    with $\tA_{n}^{\gamma}$ as in \eqref{An.set}, we have
    \begin{equation}
        \mathscr{F}_{\varepsilon}\big( \tU_{\infty}(a,\omega) \big) = 0 .
    \end{equation}
    The torus $i_{\infty}$ is traveling and reversible, that is \eqref{rev trav torus} holds. The vector $\alpha_{\infty} \in W^{q,\infty,\gamma}(\cO,\R^{d})$ satisfies
    \begin{equation}
        \alpha_{\infty}(a,\omega) = - \tJ \omega + {\rm r}_{\varepsilon}(a,\omega), \qquad \| {\rm r}_{\varepsilon} \|^{q,\gamma,\cO} \lesssim \varepsilon \gamma^{-1} N_{0}^{q\bar\ta} .
    \end{equation}
    In addition, there exists a $q$-times differentiable function $[a_0,a_1]\ni a \mapsto \vec\omega_{\varepsilon}(a)$ implicitly defined by
    \begin{equation}
        \alpha_{\infty}\big( a, \vec\omega_{\varepsilon}(a) \big) = -\tJ \omega_{\rm Eq} (a),
    \end{equation}
    with
    \begin{equation}\label{final.omega}
       \vec\omega_{\varepsilon}(a) =  \omega_{\rm Eq}(a) + \bar r_{\varepsilon}(a), \qquad \| \bar r_{\varepsilon} \|^{q,\gamma,\cO} \lesssim \varepsilon \gamma^{-1} N_{0}^{q\bar\ta},
    \end{equation}
    such that, for any $a \in \cC_{\infty}^\varepsilon$,
    \begin{equation}
        \mathscr{F}_{\varepsilon}\big( \tU_{\infty}(a,\vec\omega_{\varepsilon}(a)) \big) =0,
    \end{equation}
    where
    \begin{equation}
        \mathscr{C}(\varepsilon) \triangleq \Big\{  a \in [a_0,a_1] \ \textnormal{ s.t. } \ \big( a,\vec\omega_{\varepsilon}(a)\big) \in \tG_{\infty}^{\gamma} \Big\} . \label{cC.infty.set}
    \end{equation}
\end{cor}
\begin{proof}
    It follows analogously as in \cite[Cor. 8.1]{HHR23}. In particular, $i_{\infty}$ is the traveling reversible torus obtained as the limit of the Cauchy sequence of traveling reversible tori $(i_{n})_{n\in\N}$ in \eqref{Un}, \eqref{H.m.per.Cauchy} (noting that \eqref{rev-trav-NM} holds in the limit), which is a solution of the system by passing in the limit $n\to \infty$ in \eqref{decayF}.
\end{proof}

\subsection{Measure of the final Cantor set}\label{subsect.conclusion}

In Corollary \ref{cor.limit.NM},  we proved the existence of solutions to the equation \eqref{solve.equa.NM}. In this last section, in order to finally prove Theorem \ref{thm QP traveling electron layers}, we check that the final Cantor set $\mathscr{C}(\varepsilon)$ in the variable $a$ given by \eqref{cC.infty.set} is massive set, which proves the existence of non-trivial quasi-periodic solution to our problem. Actually, we prove that the measure of $\mathscr{C}(\varepsilon)$ is $\varepsilon$-close to $\big|[a_0,a_1] \big| = a_1-a_0 >0$. One of the main technical ingredient is R\"ussmann Theorem \cite[Thm. 17.1]{R01} (see also its slight reformulation in \cite[Lemma 8.1]{HHR23}) that will allows to estimate the measure of the sublevel sets for non-degenerate functions. 

First, we prove the following proposition, which is the main result of this section. The proof relies of several lemmata that, for clarity of the presentation, we state and prove right after.
\begin{prop}\label{prop.measure}
    Let $q_0$ be as in Lemma \ref{Russmann equilibre}, $\varrho$ as in \eqref{param.NASH.2} and assume that \eqref{param.NASH.1}, \eqref{param.NASH.2} hold with $q\triangleq q_0+2$. Assume the additional conditions
    \begin{equation}\label{fix.tau}
        \tau_{1}> d q_0, \qquad \tau > (d + 2\tau_{1}-1)q_0, \qquad \upsilon \triangleq (4q_0+1)^{-1}.
    \end{equation}
    Then there exists $C>0$ such that
    \begin{equation}
        a_1 - a_0 - C \varepsilon^{\frac{\varrho\upsilon}{q_0}} \leqslant |\mathscr{C}(\varepsilon)| \leqslant a_1 - a_0 .
    \end{equation}
\end{prop}
\begin{proof}
We split the presentation of the proof in several steps.
\\[1mm]
$\blacktriangleright$ {\sc Decomposition in ``nearly-resonant sets''.}
    The definitions in \eqref{tG.infty.set}, \eqref{cC.infty.set} provide the decomposition of the final Cantor set
    \begin{equation}\label{cC.n.set}
        \mathscr{C}(\varepsilon)=\bigcap_{n\in\N} \cC_{n}^{\varepsilon}, \quad \textnormal{where} \quad \cC_{n}^{\varepsilon} \triangleq \Big\{  a \in [a_0,a_1] \ \textnormal{ s.t. } \ \big( a,\vec\omega_{\varepsilon}(a) \big) \in \tA_{n}^{\gamma} \Big\},
    \end{equation}
    with $\tA_{n}^{\gamma}$ as in \eqref{An.set} and $\omega(a,\varepsilon)$ as in \eqref{final.omega}. We can write
    \begin{equation}\label{decomp.complement}
        [a_0,a_1] \setminus \mathscr{C}(\varepsilon) = \big( [a_0,a_1]\setminus \cC_{0}^{\varepsilon} \big) \sqcup  \bigsqcup_{n\in\N} \big( \cC_{n}^{\varepsilon} \setminus \cC_{n+1}^{\varepsilon} \big) .
    \end{equation}
    First, we claim that
    \begin{equation}
        [a_0,a_1] \setminus \cC_{0}^{\varepsilon} = \varnothing, \quad \textnormal{that is} \quad \cC_{0}^\varepsilon = [a_0,a_1] .
    \end{equation}
    By \eqref{cC.n.set}, \eqref{An.set}, \eqref{cO.ref.set}, this is equivalent to prove that
    \begin{equation}\label{omega.eps.inball}
        \vec\omega_{\varepsilon}\big( [a_0,a_1] \big) \subset B_{R_0}(0),
    \end{equation}
    with $R_0>0$ fixed such that \eqref{omega.Eq.inball} holds. By \eqref{final.omega} and \eqref{param.NASH.2}, we have that
    \begin{align}
        \sup_{a\in[a_0,a_1]} |\vec\omega_{\varepsilon}(a)| & \leqslant \frac{R_0}{2} + C \varepsilon \gamma^{-1} N_0^{q\bar\ta} \\
        & \leqslant \frac{R_0}{2} + C \varepsilon^{1- \varrho(1+q\bar\ta)} \leqslant R_0,
    \end{align}
    which holds for $0<\varepsilon\ll 1$ small enough, with $0<\varrho<(1+q\bar\ta)^{-1}$ as in \eqref{param.NASH.2}. This shows \eqref{omega.eps.inball} and the claim is proved. Back to \eqref{decomp.complement}, we have
    \begin{equation}
        \big| [a_0,a_1] \setminus \cC_{\infty}^\varepsilon \big| \leqslant \sum_{n\in\N} \big| \cC_{n}^\varepsilon \setminus \cC_{n+1}^\varepsilon \big| \triangleq \sum_{n\in\N} \cS_{n} . \label{cSn.compl.set}
    \end{equation}
    With a slight abuse of notation, by \eqref{mu.j.0}, \eqref{mu.j.infty}, we denote the perturbed frequencies
associated with the reduced linearized operator at state $i_n$ in the following way
    \begin{align}
        \mathtt{c}_{\kappa,n}(a) & \triangleq \mathtt{c}_{\kappa}(a,\vec{\omega}_{\varepsilon}(a);i_{n}), \\
        \mu_{j,\kappa,n}^{[\infty]}(a) & \triangleq \mu_{j,\kappa}^{[\infty]}(a,\vec{\omega}_{\varepsilon}(a);i_{n}) = \kappa\Omega_{j}(a) + (\mathtt{c}_{\kappa,n}(a) - \kappa a ) j + \mathtt{r}_{j,\kappa,n}^{[\infty]}(a), \label{mu.j.infty:Cant}\\
         \mathtt{r}_{j,\kappa,n}^{[\infty]}(a) & \triangleq \mathtt{r}_{j,\kappa}^{[\infty]}(a,\vec{\omega}_{\varepsilon}(a);i_{n}).
    \end{align}
Moreover, by Corollary \ref{cor.limit.NM}, Proposition \ref{prop.KAM.conv} and Proposition \ref{prop trans Ego}, we have the estimates
\begin{align}
    |\partial_{a}^{k}\vec{r}_{\varepsilon}(a)| & \leqslant C \varepsilon \gamma^{-1-k} \,, \quad \forall \, k \in \llbracket 0,q_0 \rrbracket, \quad \textnormal{uniformly in } a \in [a_0,a_1]\,, \label{natale 1} \\
    |\partial_{a}^{k}\big( \mathtt{c}_{\kappa,n}(a)-\kappa a \big)| &\leqslant C \varepsilon\gamma^{-k}, \quad \forall \, k \in \llbracket 0,q_0 \rrbracket, \quad \textnormal{uniformly in } a \in [a_0,a_1], \label{natale 2} \\
    \sup_{j\in \mathbb{Z}\setminus \overline{\mathbb{S}}_{\kappa}}|\partial_{a}^k \mathtt{r}_{j,\kappa,n}^{[n]}(a)|& \leqslant C\varepsilon\gamma^{-1-k}, \quad \forall \, k \in \llbracket 0,q_0 \rrbracket, \quad \textnormal{uniformly in } a \in [a_0,a_1]. \label{natale 3}
\end{align}
According to \eqref{cC.n.set}, \eqref{An.set} and Propositions \ref{prop trans Ego}, \ref{prop.KAM.conv}, \ref{prop.almost.inv}, we have
\begin{align}
 \cC_{n}^\varepsilon \setminus\cC_{n+1}^{\varepsilon} &= \Bigg[ \bigcup_{\ell \in \mathbb{Z}^d\setminus\{0\} } R_{\ell}^{(0,n)}(i_n) \Bigg] \cup \bigcup_{\kappa\in\{-,+\}}\Bigg[
    \bigcup_{{(\ell,j)\in\mathbb{Z}^d\times\mathbb{Z}\setminus\{(0,0)\}\atop\vec{\jmath}\cdot\ell+j=0}} R_{\ell,j,\kappa}^{({\rm Tr},n)}(i_n)\Bigg]  \cup  \bigcup_{\kappa\in\{-,+\}} \Bigg[ \bigcup_{(\ell,j)\in\mathbb{Z}^{d}\times(\mathbb{Z}^*\setminus\overline{\mathbb{S}}_{\kappa})\atop{\vec{\jmath}\cdot \ell+j=0}} R_{\ell,j,\kappa}^{(I,n)}(i_n) \Bigg]  \\
    &
    \cup 
    \bigcup_{\kappa\in\{-,+\}} \Bigg[ \bigcup_{(\ell,j,j')\in\mathbb{Z}^d\times(\mathbb{Z}\setminus\overline{\mathbb{S}}_{\kappa})^2\atop{\vec{\jmath}\cdot \ell+j-j'=0 ,\, j\neq j' }} R_{\ell,j,j',\kappa}^{(II,n)}(i_n) \Bigg]	\cup \Bigg[ \bigcup_{(\ell,j,j')\in\mathbb{Z}^d\times(\mathbb{Z}\setminus\overline{\mathbb{S}}_{+})\times(\mathbb{Z}\setminus\overline{\mathbb{S}}_{-})\atop{\vec{\jmath}\cdot \ell+j-j'=0}} Q_{\ell,j,j'}^{(II,n)}(i_n) \Bigg],\label{compl.set.C.eps}
\end{align}
where the ‘‘nearly-resonant sets'' are
\begin{align}
    R_{\ell}^{(0,n)}(i_n) & \triangleq \bigg\{ a\in  \cC_{n}^\varepsilon \quad \textnormal{s.t.} \quad |\vec{\omega}_{\varepsilon}(a)\cdot \ell| \leqslant \frac{4\gamma_{n+1}}{\langle\ell\rangle^{\tau}} \bigg\}, \label{R.Dioph.set} \\
    R_{\ell,j,\kappa}^{({\rm Tr},n)}(i_n) & \triangleq \bigg\{
    a \in \cC_{n}^\varepsilon  \quad\textnormal{s.t.}\quad\big|\vec{\omega}_{\varepsilon}(a)\cdot\ell+ j\tc_{\kappa,n}(a)\big|\leqslant \frac{4\gamma_{n+1}^{\upsilon}}{\langle\ell\rangle^{\tau_1}}\bigg\}, \label{R.TR.set}\\
    R_{\ell,j,k}^{(I,n)}(i_n) & \triangleq \bigg\{ a \in \cC_{n}^\varepsilon  \quad\textnormal{s.t.}\quad\big|\vec{\omega}_{\varepsilon}(a)\cdot \ell+\mu_{j,\kappa,n}^{[\infty]}(a)\big| \leqslant \frac{2\gamma_{n+1}|j|}{\langle \ell \rangle^{\tau_1}}\bigg\}, \label{R.1st.set} \\
    R_{\ell,j,j',\kappa}^{(II,n)}(i_n) & \triangleq \bigg\{ a \in \cC_{n}^\varepsilon  \quad\textnormal{s.t.}\quad\big|\vec{\omega}_{\varepsilon}(a)\cdot \ell+\mu_{j,\kappa,n}^{[\infty]}(a)-\mu_{j',\kappa,n}^{[\infty]}(a)\big| \leqslant \frac{2\gamma_{n+1}\langle j-j'\rangle}{\langle \ell \rangle^{\tau}} \bigg\},\label{R.2nd.set} \\
    Q_{\ell,j,j}^{(II,n)}(i_n) & \triangleq \bigg\{ a \in \cC_{n}^\varepsilon \quad\textnormal{s.t.}\quad\big|\vec{\omega}_{\varepsilon}(a)\cdot \ell+\mu_{j,+,n}^{[\infty]}(a)-\mu_{j',-,n}^{[\infty]}(a)\big|\leqslant \frac{2\gamma_{n+1}\braket{j,j'}}{\langle \ell \rangle^{\tau}} \bigg\} . \label{Q.2nd.set}
\end{align}
Note that, in \eqref{compl.set.C.eps}, we require the union for $R_{\ell,j,j',\kappa}^{(II,n)}$ only for $j\neq j'$ because we have $R_{\ell,j,j,\kappa}^{(II,n)}\subset R_{\ell}^{(0,n)}$. 

In the sequel, we will always suppose the momentum conditions on the indexes $\ell,j,j'$ in \eqref{compl.set.C.eps}. 
\\[1mm]
$\blacktriangleright$ {\sc Application of R\"ussmann's Theorem.}
Since $W^{q,\infty,\gamma}(\cO,\C) \hookrightarrow C^{q-1}(\cO,\C)$ and $q=q_0+2$, one obtains, for any $n\in\N$ and any $k\in \{-,+\}$, the $C^{q_0+1}$-regularity of the curves
\begin{align}
    & a\mapsto \vec\omega_{\varepsilon}(a)\cdot \ell, & \ell \in \Z^{d}\setminus\{0\}, \\
    & a\mapsto \vec\omega_{\varepsilon}(a)\cdot \ell +  j \big(\tc_{\kappa,n}(a)-\kappa a\big), & (\ell,j) \in \Z^{d}\times \Z \setminus\{(0,0)\},  \\
    & a\mapsto \vec\omega_{\varepsilon}(a)\cdot \ell + \mu_{j,\kappa,n}^{[\infty]}(a), & (\ell,j) \in \Z^{d}\times (\Z^* \setminus \overline{\S}_\kappa), \\
    & a\mapsto \vec\omega_{\varepsilon}(a)\cdot \ell + \mu_{j,\kappa,n}^{[\infty]}(a) - \mu_{j',\kappa,n}^{[\infty]}(a), & (\ell,j,j') \in \Z^{d}\times (\Z^* \setminus \overline{\S}_\kappa)^2, \\
    & a\mapsto \vec\omega_{\varepsilon}(a)\cdot \ell + \mu_{j,+,n}^{[\infty]}(a) - \mu_{j',-,n}^{[\infty]}(a), & (\ell,j,j') \in \Z^{d}\times (\Z^* \setminus \overline{\S}_{+})\times (\Z^* \setminus \overline{S}_{-}).
\end{align}
Therefore, applying R\"ussmann Theorem \cite[Thm. 17.1]{R01} together with the perturbed transversality estimates in Lemma \ref{pert.transversality} and the momentum conditions, we obtain that the measure of the sets \eqref{R.Dioph.set}-\eqref{Q.2nd.set} satisfy
\begin{align}
    &|R_{\ell}^{(0,n)}(i_n)| \lesssim \gamma_{n+1}^{\frac{1}{q_0}} \braket{\ell}^{-\frac{\tau+1}{q_0}-1}, \quad \quad  \quad \   |R_{\ell,j}^{({\rm Tr},n)}(i_n)| \lesssim \gamma_{n+1}^\frac{\upsilon}{q_0} \braket{\ell}^{-\frac{\tau_{1}+1}{q_0}-1}, \\
    & |R_{\ell,j}^{(I,n)}(i_n)| \lesssim \gamma_{n+1}^{\frac{1}{q_0}} \braket{\ell}^{-\frac{\tau}{q_0}-1},   \quad  \quad \quad \ \ \,  |R_{\ell,j,j'}^{(II,n)}(i_n)| \lesssim \gamma_{n+1}^{\frac{1}{q_0}} \braket{\ell}^{-\frac{\tau}{q_0}-1}, \label{russ.est.pert}  \\
     &  |Q_{\ell,j,j'}^{(II,n)}(i_n)| \lesssim \gamma_{n+1}^{\frac{1}{q_0}} \braket{j,j'}^{\frac{1}{q_0}}\braket{\ell}^{-\frac{\tau+1}{q_0}-1} .
\end{align}
$\blacktriangleright$ {\sc Estimates of the leading orders $\cS_{0}$, $\cS_{1}$.}
We first estimate the size of $\cS_{0}$ and $\cS_{1}$ defined in \eqref{compl.set.C.eps} and \eqref{cSn.compl.set}. By  \eqref{R.Dioph.set}, \eqref{R.TR.set}, \eqref{R.1st.set} and \eqref{russ.est.pert}, we have, for $n=0,1$, $\kappa=\{-,+\}$
\begin{align}
    &\Big|\bigcup_{\ell\in\mathbb{Z}^d\setminus\{0\}} R_{\ell}^{(0,n)}(i_n) \Big| \leqslant \sum_{\ell\in\mathbb{Z}^d\setminus\{0\}} |R_{\ell}^{(0,n)}(i_n)| \lesssim \gamma_{n+1}^{\frac{1}{q_0}} \sum_{\ell\in\mathbb{Z}^d\setminus\{0\}} \braket{\ell}^{-\frac{\tau+1}{q_0}-1}\lesssim \gamma^{\frac{1}{q_0}}, \\
    &\Big|\bigcup_{(\ell,j) \neq (0,0) \atop \vec \jmath \cdot \ell + j = 0} R_{\ell,j,\kappa}^{({\rm Tr},n)}(i_n) \Big| \leqslant \sum_{\ell\in\mathbb{Z}^d\setminus\{0\}} |R_{\ell,-\vec\jmath\cdot\ell,\kappa}^{({\rm Tr},n)}(i_n)| \lesssim \gamma_{n+1}^{\frac{\upsilon}{q_0}} \sum_{\ell\in\mathbb{Z}^d\setminus\{0\}} \braket{\ell}^{-\frac{\tau_{1}+1}{q_0}-1} \lesssim \gamma^{\frac{\upsilon}{q_0}},  \\
    & \Big| \bigcup_{(\ell,j)\in\mathbb{Z}^{d}\times(\mathbb{Z}^*\setminus\overline{\mathbb{S}}_{\kappa})\atop{\vec{\jmath}\cdot \ell+j=0}} R_{\ell,j,\kappa}^{(I,n)}(i_n) \Big| \leqslant \sum_{|j|\leqslant C\braket{\ell}} |R_{\ell,j,\kappa}^{(I,n)}| \lesssim \gamma_{n+1}^{\frac{1}{q_0}} \sum_{\ell\in\mathbb{Z}^d\setminus\{0\}} \braket{\ell}^{-\frac{\tau}{q_0} } \lesssim \gamma^{\frac{1}{q_0}}, 
\end{align}
where we used the choice of $\tau$ in \eqref{fix.tau} to ensure the convergence of the series.
We now estimate the measure, for any $n=0,1$ $\kappa\in\{-,+\}$,  of
\begin{align}
\bigcup_{(\ell,j,j')\in\mathbb{Z}^d\times(\mathbb{Z}\setminus\overline{\mathbb{S}}_{\kappa})^2\atop{\vec{\jmath}\cdot \ell+j-j'=0 ,\, j\neq j' }} R_{\ell,j,j',\kappa}^{(II,n)}(i_n) = \Bigg[ \bigcup_{{\vec{\jmath}\cdot \ell+j-j'=0 ,\, j\neq j' } \atop j\cdot j' <0 } R_{\ell,j,j',\kappa}^{(II,n)}(i_n)\Bigg] \cup \Bigg[\bigcup_{{\vec{\jmath}\cdot \ell+j-j'=0 ,\, j\neq j' } \atop j\cdot j' >0 } R_{\ell,j,j',\kappa}^{(II,n)}(i_n)\Bigg] \triangleq \tI_{1,n} \cup \tI_{2,n} .
\end{align}
We first estimate the measure of $\tI_{1,n}$. By Lemma \ref{neg.moment.lemma} and \eqref{russ.est.pert}, we deduce that, for any $n=0,1$ $\kappa\in\{-,+\}$,
\begin{align}
    |\tI_{1,n}| \leqslant \sum_{\ell \in \Z^{d}\atop |j|,|j'|\leqslant C \braket{\ell}} |R_{\ell,j,j',\kappa}^{(II,n)}| \lesssim \gamma_{n+1}^{\frac{1}{q_0}} \sum_{\ell\in\Z^{d} } \braket{\ell}^{-\frac{\tau}{q_0}+1} \lesssim \gamma^{\frac{1}{q_0}},
\end{align}
where we used the choice of $\tau$ in \eqref{fix.tau} to ensure the convergence of the series. To estimate the measure of $\tI_{2,n}$, we need to use Lemma \ref{Q.empty.lemma}-$(i)$. We note that the set of indexes $(\ell,j,j')$ such that $\vec\jmath\cdot \ell + j-j'=0$ and $\min\{|j|,|j'|\} < C_{1} \gamma_{n+1}^{-\upsilon} \braket{\ell}^{\tau_{1}}$ is included, for $\gamma=\varepsilon^{\varrho}\in (0,1)$ small enough, into the set
\begin{equation}
    \cI_{\ell} \triangleq \big\{ (\ell,j,j') \,:\, \vec\jmath\cdot \ell + j-j'=0, \  |j|,|j'|< \gamma_{n+1}^{-2\upsilon} \braket{\ell}^{\tau_{1}} \big\}
\end{equation}
because 
\begin{equation}
    \max\{ |j|,|j'| \} \leqslant \min\{|j|,|j'|\}+ |j-j'| <  C_{1} \gamma_{n+1}^{-\upsilon} \braket{\ell}^{\tau_{1}} + C\braket{\ell} \leqslant \gamma_{n+1}^{-2\upsilon} \braket{\ell}^{\tau_{1}} .
\end{equation}
As a consequence, by Lemma \ref{Q.empty.lemma}-$(i)$, we deduce that
\begin{equation}
    \tI_{2,n} = \bigcup_{{\vec{\jmath}\cdot \ell+j-j'=0 ,\, j\neq j' } \atop j\cdot j' >0 } R_{\ell,j,j',\kappa}^{(II,n)}(i_n) \subset \Bigg[ \bigcup_{\ell\in\mathbb{Z}^d\setminus\{0\}} R_{\ell,-\vec\jmath\cdot\ell,\kappa}^{({\rm Tr},n)}(i_n) \Bigg] \cup \Bigg[ \bigcup_{(\ell,j,j')\in \cI_{\ell}} R_{\ell,j,j',\kappa}^{(II,n)}(i_n) \Bigg].
\end{equation}
We compute
\begin{align}
    & \Bigg| \bigcup_{\ell\in\mathbb{Z}^d\setminus\{0\}} R_{\ell,-\vec\jmath\cdot\ell,\kappa}^{({\rm Tr},n)}(i_n) \Bigg| \leqslant \sum_{\ell\in\mathbb{Z}^d\setminus\{0\}}|R_{\ell,-\vec\jmath\cdot\ell,\kappa}^{({\rm Tr},n)}(i_n) |\lesssim \gamma_{n+1}^{\frac{\upsilon}{q_0}} \sum_{\ell\in\mathbb{Z}^d\setminus\{0\}} \braket{\ell}^{-\frac{\tau_{1}+1}{q_0}-1} \lesssim \gamma^{\frac{\upsilon}{q_0}},  \\
    & \Bigg| \bigcup_{(\ell,j,j')\in \cI_{\ell}} R_{\ell,j,j',\kappa}^{(II,n)}(i_n)\Bigg| \leqslant \sum_{\ell \in \Z^d \atop |j|,|j'|< \gamma_{n+1}^{-2\upsilon}\braket{\ell}^{\tau_{1}}} |R_{\ell,j,j',\kappa}^{(II,n)}(i_n)| \lesssim \gamma_{n+1}^{\frac{1}{q_0}-4\upsilon} \sum_{\ell\in\Z^{d}} \braket{\ell}^{-\frac{\tau}{q_0}-1 + 2\tau_{1}} \lesssim \gamma^{\frac{\upsilon}{q_0}},
\end{align}
where we used the choice of $\tau$ and $\upsilon$ in \eqref{fix.tau} to ensure the convergence of the series, and we obtaing that
\begin{equation}
    |\tI_{2,n}| \lesssim \gamma^{\frac{\upsilon}{q_0}} .
\end{equation}
Finally, we estimate the measure, for any $n=0,1$, $\kappa\in\{-,+\}$, of
\begin{align} \bigcup_{(\ell,j,j')\in\mathbb{Z}^d\times(\mathbb{Z}\setminus\overline{\mathbb{S}}_{+})\times(\mathbb{Z}\setminus\overline{\mathbb{S}}_{-})\atop{\vec{\jmath}\cdot \ell+j-j'=0}} Q_{\ell,j,j'}^{(II,n)}(i_n) = \Bigg[\bigcup_{\vec{\jmath}\cdot \ell+j-j'=0 \atop j\cdot j'< 0} Q_{\ell,j,j'}^{(II,n)}(i_n) \Bigg]  \cup \Bigg[\bigcup_{\vec{\jmath}\cdot \ell+j-j'=0 \atop j\cdot j'> 0 } Q_{\ell,j,j'}^{(II,n)}(i_n) \Bigg]. 
\end{align}
By Lemma \ref{Q.empty.lemma}-$(ii)$, Lemma \ref{neg.moment.lemma} and \eqref{russ.est.pert}, we deduce that
\begin{align}
    \Bigg| \bigcup_{(\ell,j,j')\in\mathbb{Z}^d\times(\mathbb{Z}\setminus\overline{\mathbb{S}}_{+})\times(\mathbb{Z}\setminus\overline{\mathbb{S}}_{-})\atop{\vec{\jmath}\cdot \ell+j-j'=0}} Q_{\ell,j,j'}^{(II,n)}(i_n) \Bigg| \leqslant \sum_{\ell \in \Z^{d} \atop \braket{j,j'}\leqslant C\braket{\ell}} |Q_{\ell,j,j',\kappa}^{(II,n)}| \lesssim \gamma_{n+1}^{\frac{1}{q_0}} \sum_{\ell \in \Z^{d}} \braket{\ell}^{-\frac{\tau+1}{q_0}-1 +\frac{1}{q_0}+1} \lesssim \gamma^{\frac{1}{q_0}},
\end{align}
where we used the choice of $\tau$  in \eqref{fix.tau} to ensure the convergence of the series. Summing up all the above contributions, we conclude that
\begin{equation}
    |\cS_{0}|, |\cS_{1}| \leqslant C \gamma^{\frac{\upsilon}{q_0}} .
\end{equation}
$\blacktriangleright$ {\sc Estimates of the lower orders $\cS_{n}$, $n\geqslant 2$, and conclusion.} It remains now to estimate the size of $\cS_{n}$ in \eqref{compl.set.C.eps} and \eqref{cSn.compl.set} for $n\geqslant 2$. In particular, we claim that they can be estimated with decreasing terms of a convergent geometric series: by Lemma \ref{lem.emptysets.low}-$(vi)$ one can show that, for $n\geqslant 2$
\begin{align}
    \cS_{n} \lesssim \gamma^{\frac{\upsilon}{q_0}} N_{n-1}^{-\tp}, \quad \textnormal{for some} \quad \tp\triangleq \tp(\tau,d)\ll 1 .
\end{align}
For instance, we have to estimate, for any $\kappa\in\{-,+\}$ and $n\geqslant 2$, using \eqref{russ.est.pert} and \eqref{scale.KAM},
\begin{align}
    \Bigg| \bigcup_{(\ell,j)\in\mathbb{Z}^{d}\times(\mathbb{Z}^*\setminus\overline{\mathbb{S}}_{\kappa})\atop\underset{ |\ell|\geqslant N_{n-1} }{\vec{\jmath}\cdot \ell+j=0}} R_{\ell,j,\kappa}^{(I,n)}(i_n) \Bigg| & \lesssim \gamma_{n+1}^{\frac{1}{q_0}}\sum_{|j|\leqslant C\braket{\ell} \atop |\ell|\geqslant N_{n-1}}  \braket{\ell}^{-\frac{\tau}{q_0}-1} \\
    & \lesssim \gamma^{\frac{1}{q_0}} \sum_{|\ell|\geqslant N_{n-1}}\braket{\ell}^{-\frac{\tau}{q_0}} \lesssim \gamma^{\frac{1}{q_0}} N_{n-1}^{-\frac{\tau}{q_0}+1}
\end{align}
having $\frac{\tau}{q_0}>d$ by \eqref{fix.tau}. The other terms are estimated analogously with similar arguments as in the previous step, and therefore we omit the details.

Finally, we conclude that the measure in \eqref{cSn.compl.set} is given by
\begin{align}
    |[a_0,a_1]\setminus \mathscr{C}(\varepsilon)| & \leqslant \cS_{0} + \cS_{1} + \sum_{n \geqslant 2} \cS_{n}  \lesssim \gamma^{\frac{\upsilon}{q_0}}\Big( 1 + \sum_{n\geqslant 2} N_{n-1}^{-\tp} \Big) \lesssim \gamma^{\frac{\upsilon}{q_0}}
\end{align}
and we deduce that, together with \eqref{param.NASH.2},
\begin{equation}
   |[a_0,a_1]| \geqslant |\mathscr{C}(\varepsilon)| \geqslant |[a_0,a_1]| - |[a_0,a_1]\setminus \mathscr{C}(\varepsilon)|  \geqslant a_1-a_0 - C\gamma^{\frac{\upsilon}{q_0}} = a_1-a_0 - C\varepsilon^{\frac{\varrho\upsilon}{q_0}}.
\end{equation}
The proof of the main claim is concluded.
\end{proof}

We now state and prove the lemmata that we used in the above proof for Proposition \ref{prop.measure}.

\begin{lem}\label{Q.empty.lemma}
    For any $n\in\N$ and for $\varepsilon \in (0,\varepsilon_{0})$ small enough, the following assertions are true:
    \begin{itemize}
        \item [$(i)$] Let $(\ell,j,j')\in \Z^{d}\times (\Z\setminus\overline{\S}_{\kappa})^2$, $\kappa\in\{-,+\}$. There exists a constant $C_1>0$, independent of $n\in\N$ and $\varepsilon>0$, such that, if $\vec\jmath\cdot \ell + j-j' = 0$, $j\cdot j'>0$ and  $\min\{ |j|,|j'| \} \geqslant C_1 \gamma_{n+1}^{-\upsilon} \braket{\ell}^{\tau_1}$, then $R_{\ell,j,j',\kappa}^{(II,n)} (i_n) \subset R_{\ell,j-j',\kappa}^{({\rm Tr},n)}(i_n)$;
        \item [$(ii)$] Let $(\ell,j,j')\in \Z^{d}\times (\Z\setminus\overline{\S}_{+})\times (\Z\setminus\overline{\S}_{-})$.  If $Q_{\ell,j,j'}^{(II,n)}\neq \varnothing$ and $j\cdot j'>0$, then $\braket{j,j'}\leqslant C \braket{\ell}$.
    \end{itemize}
\end{lem}
\begin{proof}
$(i)$ We observe that, by \eqref{eigenvalues VP}, for any $a\in[a_0,a_1]$,
\begin{align}
    |\Omega_{j}(a)-ja| = a|j| \bigg| \sqrt{1+ \frac{1}{a^2j^2}}-1 \bigg| = \frac{1}{a|j|} \bigg( \sqrt{1+ \frac{1}{a^2j^2}}+1 \bigg)^{-1} \leqslant \frac{1}{2a_0 |j|} \cdot\label{stima.easy}
\end{align}
Let $a\in R_{\ell,j,j',\kappa}^{(II,n)}(i_n)$. By \eqref{R.2nd.set}, \eqref{mu.j.infty.est}, \eqref{stima.easy} and \eqref{natale 3} (recalling also the choice of $\upsilon$ in Proposition \ref{prop trans Ego} and the choice of $\tau > \tau_{1}$ in \eqref{fix.tau}), we have, for $0<\varepsilon\ll 1$ small enough,
\begin{align}
    \big|\vec{\omega}_{\varepsilon}(a)\cdot \ell  + (j-j')\tc_{\kappa,n}(a) \big| &\leqslant \big|\vec{\omega}_{\varepsilon}(a)\cdot \ell+\mu_{j,\kappa,n}^{[\infty]}(a)-\mu_{j',\kappa,n}^{[\infty]}(a)\big| \\
    & + \big| \kappa(ja - \Omega_{j}(a)) - \kappa(j'a-\Omega_{j'}(a)) - \tr_{j,\kappa,n}^{[\infty]}(a) + \tr_{j',\kappa,n}^{[\infty]}(a)  \big| \\
    & \leqslant \frac{2\gamma_{n+1}\braket{j-j'}}{\braket{\ell}^{\tau}} +\Big( \frac{1}{2a_0} + C\varepsilon\gamma^{-1} \Big)\Big( \frac{1}{|j|} + \frac{1}{|j'|} \Big) \\
    & \leqslant \frac{2\gamma_{n+1}^\upsilon\braket{j-j'}}{\braket{\ell}^{\tau_{1}}} + \frac14 \Big( \frac{1}{|j|} + \frac{1}{|j'|} \Big) \braket{j-j'} \leqslant  \frac{4\gamma_{n+1}^\upsilon\braket{j-j'}}{\braket{\ell}^{\tau_{1}}} 
\end{align}
if we require that
\begin{equation}
    |j|, |j'|\geqslant C_{1} \gamma_{n+1}^{\upsilon} \braket{\ell}^{-\tau_{1}}, \qquad C_{1}\triangleq 4a_0,
\end{equation}
and the claim is proved.
\\[1mm]
$(ii)$ Since $Q_{\ell,j,j'}^{(II,n)}\neq \varnothing$, there exists $a \in [a_0,a_1]$ such that, using also \eqref{natale 1}, \eqref{gamma.n},
    \begin{equation}\label{majoration diff muinftypm}
        \begin{aligned}
            \big| \mu_{j,+,n}^{[\infty]}(a)-\mu_{j',-,n}^{[\infty]}(a) \big| &\leqslant \frac{2\gamma_{n+1} \braket{j,j'}}{\braket{\ell}^{\tau}} + C |\ell| \leqslant 4\gamma \braket{j,j'}+C\braket{\ell}.
        \end{aligned}
    \end{equation}
    By \eqref{mu.j.infty:Cant}, we have
    \begin{align}
         \mu_{j,+,n}^{[\infty]}(a)-\mu_{j',-,n}^{[\infty]}(a) & = \Omega_{j}(a) + \Omega_{j'}(a) + j \big( \tc_{+,n}(a)- a \big) -  j' \big( \tc_{-,n}(a) + a \big) + \fr_{j,+,n}^{[\infty]}(a) - \fr_{j',-,n}^{[\infty]}(a).
    \end{align}
     Using that $j\cdot j'>0$ and recalling \eqref{eigenvalues VP} (see also Lemma \ref{lem prop eig VP}-$(iii)$), we get
 \begin{align}
     \big|  \Omega_{j}(a) + \Omega_{j'}(a)  \big| = \Omega_{|j|}(a) + \Omega_{|j'|}(a) 
     \geqslant a_0\big( |j| + |j'| \big) =a_0\braket{j,j'}.
 \end{align}
    Therefore, making appeal to the estimates \eqref{natale 2}, \eqref{natale 3}, we get that
    \begin{equation}\label{minoration diff muinftypm}
        \begin{aligned}
         \big| \mu_{j,+,n}^{[\infty]}(a)-\mu_{j',-,n}^{[\infty]}(a) \big| &\geqslant  \big|  \Omega_{j}(a) + \Omega_{j'}(a)  \big| -C\varepsilon\gamma^{-1}\\
         &\geqslant(a_0-C\varepsilon \gamma^{-1})\braket{j,j'}.
    \end{aligned}
    \end{equation}
    Combining \eqref{majoration diff muinftypm} and \eqref{minoration diff muinftypm} and recalling \eqref{param.NASH.2}, we conclude the desired result for $0<\varepsilon \ll 1$ sufficiently small.
\end{proof}

Under the momentum condition, we have some immediate restrictions of the Fourier indexes.
\begin{lem}\label{neg.moment.lemma}
    If $(\ell,j,j') \in \Z^d \times \Z^2$ are such that $\vec\jmath\cdot \ell +j-j'= 0$ and $j\cdot j'< 0$, then there exists $C>0$ such that $|j|,|j'|\leqslant C\braket{\ell}$.
\end{lem}
\begin{proof}
    For $j\cdot j'<0$,  the momentum condition reads
    \begin{equation}
        j-j' = {\rm sgn}(j) |j| - {\rm sgn}(j') |j'| = {\rm sgn}(j) (|j|+|j'|) = - \vec\jmath\cdot \ell,
    \end{equation}
    from which we deduce the claim.
\end{proof}

The perturbed frequencies satisfy estimates similar to the ones in Lemma \ref{Russmann equilibre}.

\begin{lem}[Perturbed transversality]\label{pert.transversality}
    There exist $\rho_0>0$ and $q_0\in\mathbb{N}$ such that the following hold true.
	\begin{enumerate}[label=(\roman*)]
		\item For all $\ell\in\mathbb{Z}^d\setminus\{0\}$,
		$$\inf_{a\in[a_0,a_1]}\max_{k\in\llbracket 0,q_0\rrbracket}\big|\partial_{a}^{k}\vec\omega_{\varepsilon}(a)\cdot\ell\big|\geqslant \tfrac{\rho_{0}}{2} \langle\ell\rangle.$$
		\item  Let $\kappa \in \{ -,+ \}$, $n\in\N$. For all $(\ell,j)\in\mathbb{Z}^{d}\times\mathbb{Z}^*$, with $\vec{\jmath}\cdot\ell + j = 0$,
		$$\inf_{a\in[a_0,a_1]}\max_{k\in\llbracket 0,q_0\rrbracket}\big|\partial_{a}^{k}\big(\vec\omega_{\varepsilon}(a)\cdot\ell+ j \tc_{\kappa,n}(a) \big)\big|\geqslant \tfrac{\rho_{0}}{2} \langle\ell\rangle.$$
		\item Let $\kappa\in\{-,+\}$, $n\in\N$. For all $(\ell,j)\in\mathbb{Z}^{d}\times(\mathbb{Z}^*\setminus\overline{\mathbb{S}}_{\kappa})$ with $\vec{\jmath}\cdot\ell+j=0,$
		$$\inf_{a\in[a_0,a_1]}\max_{k\in\llbracket 0,q_0\rrbracket}\big|\partial_{a}^{k}\big(\vec\omega_{\varepsilon}(a)\cdot\ell+\mu_{j,\kappa,n}^{[\infty]}(a)\big)\big|\geqslant \tfrac{\rho_{0}}{2} \langle\ell\rangle.$$
		\item Let $\kappa\in\{-,+\}$, $n\in\N$. For all $(\ell,j,j')\in\mathbb{Z}^{d}\times(\mathbb{Z}^*\setminus\overline{\mathbb{S}}_{\kappa})^2$ with $(\ell,j,j')\neq(0,j,j)$ and $\vec{\jmath}\cdot\ell+j-j'=0,$
		$$\inf_{a\in[a_0,a_1]}\max_{k\in\llbracket 0,q_0\rrbracket}\big|\partial_{a}^{k}\big(\vec\omega_{\varepsilon}(a)\cdot\ell+\mu_{j,\kappa,n}^{[\infty]}(a)-\mu_{j',\kappa,n}^{[\infty]}(a)\big)\big|\geqslant \tfrac{\rho_{0}}{2} \langle\ell\rangle.$$
		\item Let $n\in\N$. For all $(\ell,j,j')\in\mathbb{Z}^{d}\times(\mathbb{Z}^*\setminus\overline{\mathbb{S}}_+)\times(\mathbb{Z}^*\setminus\overline{\mathbb{S}}_-)$ with $\vec{\jmath}\cdot\ell+j-j'=0$,
    $$\inf_{a\in[a_0,a_1]}\max_{k\in\llbracket 0,q_0\rrbracket}\big|\partial_{a}^{k}\big(\vec\omega_{\varepsilon}(a)\cdot\ell+\mu_{j,+,n}^{[\infty]}(a)-\mu_{j,-,n}^{[\infty]}(a)\big)\big|\geqslant \tfrac{\rho_{0}}{2} \langle\ell\rangle.$$
	\end{enumerate}
\end{lem}

\begin{proof}
    We prove items $(iv)$, $(v)$. The proofs of items $(i)$, $(ii)$ and $(iii)$ are similar and therefore omitted. We start with item $(iv)$. By \eqref{final.omega}, \eqref{mu.j.infty:Cant} we have
    \begin{align}
        \vec\omega_{\varepsilon}(a)\cdot\ell  +\mu_{j,\kappa,n}^{[\infty]}(a)-\mu_{j',\kappa,n}^{[\infty]}(a) & = \omega_{\rm Eq}(a) \cdot \ell + \vec{r}_{\varepsilon}(a) \cdot \ell + \kappa\,\Omega_{j}(a) - \kappa\,\Omega_{j'}(a)  \\
        & +  (j-j') \big( \tc_{\kappa,n}(a) - \kappa a \big) + \tr_{j,\kappa,n}^{[\infty]}(a) - \tr_{j',\kappa,n}^{[\infty]}(a) . \label{smalldiv.II.diag}
    \end{align}
    By \eqref{smalldiv.II.diag}, \eqref{natale 1}, \eqref{natale 2}, \eqref{natale 3} and the momentum condition $\vec\jmath\cdot \ell + j - j' = 0$, we have that, for any $q\in \llbracket0,q_{0}\rrbracket $,
    \begin{align}
        \big|\pa_{a}^{q}\big(\vec\omega_{\varepsilon}(a)\cdot\ell+\mu_{j,\kappa,n}^{[\infty]}(a)-\mu_{j',\kappa,n}^{[\infty]}(a)\big)\big| \geqslant \big|\pa_{a}^{q}\big( \omega_{\rm Eq}(a) \cdot \ell  + \kappa\,\Omega_{j}(a) - \kappa\,\Omega_{j'}(a)\big)\big| - C \varepsilon \gamma^{-1-q_0}\braket{\ell}  .
    \end{align}
    Therefore, by Lemma \ref{Russmann equilibre}-$(iv)$, we deduce that
    \begin{align}
        \inf_{a\in[a_0,a_1]}\max_{q\in\llbracket 0,q_0\rrbracket}\big|\pa_{a}^{q}\big(\vec\omega_{\varepsilon}(a)\cdot\ell+\mu_{j,\kappa,n}^{[\infty]}(a)-\mu_{j,\kappa,n}^{[\infty]}(a)\big)\big|\geqslant \rho_{0} \braket{\ell} - C \varepsilon \gamma^{-1-q_0}\braket{\ell}  \geqslant \tfrac{\rho_{0}}{2}\braket{\ell}
    \end{align}
    for $\varepsilon>0$ small enough. This proves item $(iv)$. \\
    We now prove item $(v)$. By \eqref{final.omega}, \eqref{mu.j.infty:Cant} we have
    \begin{align}
        \vec\omega_{\varepsilon}(a)\cdot\ell  +\mu_{j,+,n}^{[\infty]}(a)-\mu_{j',-,n}^{[\infty]}(a) & = \omega_{\rm Eq}(a) \cdot \ell + \vec{r}_{\varepsilon}(a) \cdot \ell + \Omega_{j}(a) + \Omega_{j'}(a)  \\
        & \quad+ j \big( \tc_{+,n}(a) -  a \big) - j' \big( \tc_{-,n}(a) +  a \big) + \tr_{j,+,n}^{[\infty]}(a) - \tr_{j',-,n}^{[\infty]}(a) . \label{smalldiv.II.antidiag}
    \end{align}
    Note that we can restrict to indexes $\ell,j,j'$ such that $ |j|,|j'|\leqslant C\braket{\ell}$
    (when $j\cdot j'>0$ by Lemma \ref{Q.empty.lemma}, otherwise $Q_{\ell,j,j'}^{(II,n)}=\varnothing$; when $j\cdot j'<0$ by Lemma \ref{neg.moment.lemma}, using the momentum condition). By \eqref{smalldiv.II.antidiag}, \eqref{natale 1}, \eqref{natale 2}, \eqref{natale 3} and the restriction $ |j|,|j'|\leqslant C\braket{\ell}$, we have that, for any $q\in \llbracket0,q_{0}\rrbracket $,
    \begin{align}
        \big|\pa_{a}^{q}\big(\vec\omega_{\varepsilon}(a)\cdot\ell+\mu_{j,+,n}^{[\infty]}(a)-\mu_{j',-,n}^{[\infty]}(a)\big)\big| \geqslant \big|\pa_{a}^{q}\big( \omega_{\rm Eq}(a) \cdot \ell  + \Omega_{j}(a) + \Omega_{j'}(a)\big)\big| - C \varepsilon \gamma^{-1-q_0}\braket{\ell}  .
    \end{align}
    Therefore, by Lemma \ref{Russmann equilibre}-$(v)$, we deduce that
    \begin{align}
        \inf_{a\in[a_0,a_1]}\max_{q\in\llbracket 0,q_0\rrbracket}\big|\pa_{a}^{q}\big(\vec\omega_{\varepsilon}(a)\cdot\ell+\mu_{j,+}^{[\infty]}(a)-\mu_{j,-}^{[\infty]}(a)\big)\big|\geqslant \rho_{0} \braket{\ell} - C \varepsilon \gamma^{-1-q_0}\braket{\ell}  \geqslant \tfrac{\rho_{0}}{2}\braket{\ell}
    \end{align}
    for $\varepsilon>0$ small enough. This proves item $(v)$ and concludes the proof of the lemma.
\end{proof}

In the following lemma, we show that some of the ``nearly-resonant set'' are empty for lower Fourier modes. 
\begin{lem}\label{lem.emptysets.low}
    Let $n\in\mathbb{N}\setminus\{0,1\}$ and $\kappa\in\{-,+\}$. The following assertions hold:
    \begin{itemize}
        \item [$(i)$] For $\ell \in \Z^{d}\setminus\{0\}$, with $|\ell| \leqslant N_{n-1}$, we have $R_{\ell}^{(0,n)}(i_n)= \varnothing$;
        \item [$(ii)$] For $(\ell,j) \in \Z^{d}\times \Z \setminus\{(0,0)\}$, with $|\ell| \leqslant N_{n-1}$, 
        we have $R_{\ell,j,\kappa}^{({\rm Tr},n)}(i_n)= \varnothing$;
        \item [$(iii)$] For $(\ell,j) \in \Z^{d}\times (\Z\setminus\overline{\S}_{\kappa}) $, with $|\ell| \leqslant N_{n-1}$,
        we have $R_{\ell,j,\kappa}^{(I,n)}(i_n)= \varnothing$;
        \item [$(iv)$] For $(\ell,j,j') \in \Z^{d}\times (\Z\setminus\overline{\S}_{\kappa})^2 $, with $|\ell| \leqslant N_{n-1}$
        and $(\ell,j,j')\neq (0,j,j)$,  we have $R_{\ell,j,j',\kappa}^{(II,n)}(i_n)= \varnothing$;
        \item [$(v)$] For $(\ell,j,j') \in \Z^{d}\times (\Z\setminus\overline{\S}_{+})\times (\Z\setminus\overline{\S}_{-}) $, with $|\ell| \leqslant N_{n-1}$,
        we have $Q_{\ell,j,j'}^{(II,n)}(i_n)= \varnothing$;
        \item [$(vi)$] The splitting of $\cC_{n}^\varepsilon \setminus \cC_{n+1}^{\varepsilon}$ in \eqref{compl.set.C.eps} reduces to
    \begin{footnotesize}
    \begin{align}
    \cC_{n}^\varepsilon \setminus\cC_{n+1}^{\varepsilon} &= \Bigg[ \bigcup_{\ell \in \mathbb{Z}^d\setminus\{0\} \atop |\ell|\geqslant N_{n-1}} R_{\ell}^{(0,n)}(i_n) \Bigg] \cup \bigcup_{\kappa\in\{-,+\}}\Bigg[
    \bigcup_{\underset{|\ell|\geqslant N_{n-1}}{(\ell,j)\in\mathbb{Z}^d\times\mathbb{Z}\setminus\{(0,0)\}\atop\vec{\jmath}\cdot\ell+j=0}} R_{\ell,j,\kappa}^{({\rm Tr},n)}(i_n)\Bigg]  \cup  \bigcup_{\kappa\in\{-,+\}} \Bigg[ \bigcup_{(\ell,j)\in\mathbb{Z}^{d}\times(\mathbb{Z}^*\setminus\overline{\mathbb{S}}_{\kappa})\atop\underset{|\ell|\geqslant N_{n-1}}{\vec{\jmath}\cdot \ell+j=0}} R_{\ell,j,\kappa}^{(I,n)}(i_n) \Bigg]  \\
    &
    \cup 
    \bigcup_{\kappa\in\{-,+\}} \Bigg[ \bigcup_{(\ell,j,j')\in\mathbb{Z}^d\times(\mathbb{Z}\setminus\overline{\mathbb{S}}_{\kappa})^2\atop\underset{|\ell|\geqslant N_{n-1}}{\vec{\jmath}\cdot \ell+j-j'=0 ,\, j\neq j' }} R_{\ell,j,j',\kappa}^{(II,n)}(i_n) \Bigg]	\cup \Bigg[ \bigcup_{(\ell,j,j')\in\mathbb{Z}^d\times(\mathbb{Z}\setminus\overline{\mathbb{S}}_{+})\times(\mathbb{Z}\setminus\overline{\mathbb{S}}_{-})\atop\underset{|\ell|\geqslant N_{n-1}}{\vec{\jmath}\cdot \ell+j-j'=0}} Q_{\ell,j,j'}^{(II,n)}(i_n) \Bigg] .
\end{align}
         \end{footnotesize}
    \end{itemize}
\end{lem}
\begin{proof}
    Item $(vi)$ is a direct consequence of \eqref{compl.set.C.eps} and items $(i)$-$(v)$. We prove now item $(v)$. The proofs of items $(i)$-$(iv)$ are similar and therefore omitted. First, by Proposition \ref{prop.NASH}-$(\cP 2)_{n}$, we have, for any $n\geqslant 2$,
    \begin{equation}\label{esti.tori.nn-1}
        \| i_{n} - i_{n-1} \|_{s_0+\overline{\sigma}}^{q,\gamma,\cO} \leqslant \| \tU_{n} - \tU_{n-1} \|_{s_0+\overline{\sigma}}^{q,\gamma,\cO} = \| \tH_{n}  \|_{s_0+\overline{\sigma}}^{q,\gamma,\cO} \leqslant C_* \varepsilon \gamma^{-1} N_{n-1}^{-\nu_2} .
    \end{equation}
    To conclude the claim it is sufficient to prove that $Q_{\ell,j,j'}^{(II,n)} (i_n) \subset Q_{\ell,j,j'}^{(II,n-1)} (i_{n-1})$. Indeed, if this holds, then
    \begin{equation}
       Q_{\ell,j,j'}^{(II,n)} (i_n) \subset Q_{\ell,j,j'}^{(II,n)} (i_n) \cap  Q_{\ell,j,j'}^{(II,n-1)} (i_{n-1}) \subset (\cC_{n}^{\varepsilon} \setminus \cC_{n+1}^{\varepsilon}) \cap (\cC_{n-1}^{\varepsilon} \setminus \cC_{n}^\varepsilon) = \varnothing . 
    \end{equation}
    Let $a \in Q_{\ell,j,j'}^{(II,n)}(i_{n})$. By \eqref{mu.j.infty:Cant} and triangle inequality, we have that
    \begin{equation}
        \big| \vec\omega_{\varepsilon}(a) + \mu_{j,+,n-1}^{[\infty]}(a) - \mu_{j',-,n-1}^{[\infty]}(a) \big| \leqslant \frac{2\gamma_{n+1}\braket{j,j'}}{\braket{\ell}^\tau} + \varrho_{j,j',n}(a,\varepsilon), \label{lunedi.1}   
    \end{equation}
    where, recalling \eqref{mu.j.infty:Cant}, \eqref{estimazione ttc12} in Proposition \ref{prop trans Ego}-$(i)$,  \eqref{tr.inf.12} in Proposition \ref{prop.KAM.conv}-$(i)$ and \eqref{esti.tori.nn-1},
    \begin{align}
        \varrho_{j,j',n}(a,\varepsilon) & \triangleq \big| \mu_{j,+,n}^{[\infty]}(a) - \mu_{j,+,n-1}^{[\infty]}(a) - \big( \mu_{j',-,n}^{[\infty]}(a)-\mu_{j',-,n-1}^{[\infty]}(a) \big) \big| \\
        & \leqslant |j| \big| \tc_{+,n}(a)- \tc_{+,n-1}(a) \big| + |j'| \big| \tc_{-,n}(a)- \tc_{-,n-1}(a) \big| \\
        & \quad + \big| \tr_{j,+,n}^{[\infty]}(a) - \tr_{j,+,n-1}^{[\infty]}(a)\big| + \big| \tr_{j',-,n}^{[\infty]}(a) - \tr_{j',-,n-1}^{[\infty]}(a)\big| \\
        & \lesssim  \varepsilon \braket{j,j'} \| i_n - i_{n-1} \|_{s_0 +\overline{\sigma}} \lesssim \varepsilon^{2} \gamma^{-1} N_{n-1}^{-\nu_2} \braket{j,j'} . \label{lunedi.2}
    \end{align}
    Therefore, by \eqref{lunedi.1} and \eqref{lunedi.2}, we deduce that $a \in Q_{\ell,j,j'}^{(II,n-1)}(i_{n-1})$ if it holds, for any $|\ell| \leqslant N_{n-1}$,
    \begin{align}
      2\gamma_{n+1}\braket{\ell}^{-\tau} + C \varepsilon^{2} \gamma^{-1} N_{n-1}^{-\nu_2} \leqslant 2\gamma_{n}\braket{\ell}^{-\tau},
    \end{align}
    which is equivalent to ask, by \eqref{gamma.n} and for any $|\ell| \leqslant N_{n-1}$, that
    \begin{align}
        \tfrac12 C (\varepsilon \gamma^{-1})^2 N_{n-1}^{\tau-\nu_2} \leqslant 2^{-n}-2^{-(n+1)} = 2^{-(n+1)} . \label{small.cond.emptysetQ}
    \end{align}
    Recalling the choice of parameters in \eqref{param.NASH.1}, \eqref{param.NASH.2} and \eqref{scale.KAM}, the condition in \eqref{small.cond.emptysetQ} is equivalent, by taking the logarithm, to
    \begin{equation}
        \frac{(\frac{3}{2})^{n-1}}{n+1}(a_2-\tau) \log(N_0) \geqslant \log(2),
    \end{equation}
    which is satisfied for $N_0> 1$ large enough, uniformly in $n\geqslant 2$. This concludes the proof of Lemma \ref{lem.emptysets.low}.
\end{proof}

\appendix
\section{Functional tools}\label{appendix-gene}
In this appendix, we gather all the technical material used throughout the document. In particular, we introduce the functions spaces, the pseudo-differential calculus and we discuss a useful Egorov Theorem for homogeneous operators.\\

Along the manuscript, we use the notations
$$\mathbb{N}\triangleq\{0,1,2,\ldots\},\qquad\mathbb{N}\triangleq\mathbb{N}\setminus\{0\},\qquad\mathbb{Z}\triangleq\mathbb{N}\cup(-\mathbb{N}),\qquad\mathbb{Z}\triangleq\mathbb{Z}\setminus\{0\}$$
and for $N\in\mathbb{N}^*,$
$$\llbracket1,N\rrbracket\triangleq\{1,2,\ldots,N\}.$$
Throughout the whole article, the following parameters are fixed.
\begin{equation}\label{d gm S s0}
	d,q\in\mathbb{N}^*, \qquad \gamma\in(0,1), \qquad S\gg s_0\geqslant\frac{d+1}{2}+q+2 .
\end{equation}
Let us consider $\mathcal{O}$ an open bounded subset of $\mathbb{R}^{d+1}.$ The set $\mathcal{O}$ is precised in  \eqref{cO.ref.set}.
Along the manuscript, we use the notation $A\lesssim_{\alpha,\beta,...}B$ to designate $A\leqslant C B$, with $C$ a constant depending only on the parameters $\alpha,\beta,...$

\subsection{Function spaces}
We consider the space $L^{2}(\mathbb{T}^{d+1},\mathbb{C})$ of square integrable functions, which is a Hilbert space endowed with the Hermitian scalar product
$$\forall u_1,u_2\in L^{2}(\mathbb{T}^{d+1},\mathbb{C}),\quad\langle u_1,u_2\rangle_{L^2(\mathbb{T}^{d+1},\mathbb{C})}\triangleq\int_{\mathbb{T}^{d+1}}u_1(\varphi,x)\overline{u_2(\varphi,x)}{\rm d}\varphi{\rm d}x.$$
This space admits a Hilbert basis $(\mathbf{e}_{\ell,j})_{(\ell,j)\in\mathbb{Z}^{d+1}}$ given by
$$\mathbf{e}_{\ell,j}(\varphi,x)\triangleq e^{\ii(\ell\cdot\varphi+jx)}.$$
A function $u\in L^{2}(\mathbb{T}^{d+1},\mathbb{C})$ can be decomposed into Fourier series
\begin{equation}\label{fourier.u}
    u=\sum_{(\ell,j)\in\mathbb{Z}^{d+1}}u_{\ell,j}\mathbf{e}_{\ell,j},\qquad u_{\ell,j}\triangleq\langle u,\mathbf{e}_{\ell,j}\rangle_{L^2(\mathbb{T}^{d+1},\mathbb{C})}.
\end{equation}
For a given $s\geqslant0,$ we define the Sobolove space of regularity $s$ by
\begin{equation}\label{sobolev.space}
    H^s(\mathbb{T}^{d+1},\mathbb{C})\triangleq\Big\{u\in L^{2}(\mathbb{T}^{d+1},\mathbb{C})\quad\textnormal{s.t.}\quad\|u\|_{s}<\infty\Big\},\qquad\|u\|_{s}^2\triangleq\sum_{(\ell,j)\in\mathbb{Z}^{d+1}}\langle\ell,j\rangle^{2s}|u_{\ell,j}|^2 .
\end{equation}
Here for $\ell=(\ell_1,...,\ell_d)\in\mathbb{Z}^d$ and $j\in\mathbb{Z},$ we denote 
$$|\ell|\triangleq|\ell_1|+...+|\ell_d|\qquad\textnormal{and}\qquad\langle\ell,j\rangle\triangleq\max(1,|l|,|j|).$$
The closed sub-vector space of real valued functions is denoted
\begin{align*}
	\nonumber H^{s}\triangleq H^{s}(\mathbb{T}^{d+1},\mathbb{R})&\triangleq\left\lbrace\rho\in H^{s}(\mathbb{T}^{d +1},\mathbb{C})\quad\textnormal{s.t.}\quad\forall\, (\varphi,x)\in\mathbb{T}^{d+1 },\,u(\varphi,x)=\overline{u(\varphi,x)}\right\rbrace\\
	&=\Big\{u\in H^{s}(\mathbb{T}^{d +1},\mathbb{C})\quad\textnormal{s.t.}\quad\forall \,(l,j)\in\mathbb{Z}^{d+1},\,u_{-\ell,-j}=\overline{u_{\ell,j}}\Big\}.
\end{align*}
We shall also consider functions depending on a parameter $\lambda\in\mathcal{O}\subset\mathbb{R}^{d+1}$ and taken among the following function spaces
\begin{align*}
W^{q,\gamma,\mathcal{O}}\big(\mathcal{O},H^{s}(\mathbb{T}^{d+1},\mathbb{C})\big)&\triangleq\big\{u:\mathcal{O}\to H^{s}(\mathbb{T}^{d+1},\mathbb{C})\quad\textnormal{s.t.}\quad\|u\|_{s}^{q,\gamma,\mathcal{O}}<\infty\big\},\\
W^{q,\gamma,\mathcal{O}}(\mathcal{O},\mathbb{C})&\triangleq\big\{u:\mathcal{O}\to\mathbb{C}\quad\textnormal{s.t.}\quad\|u\|^{q,\gamma,\mathcal{O}}<\infty\big\},
\end{align*}
associated with the norms
\begin{align*}
\|u\|_{s}^{q,\gamma,\mathcal{O}}&\triangleq\max_{\alpha\in\mathbb{N}^{d+1}\atop|\alpha|\leqslant q}\gamma^{|\alpha|}\sup_{\lambda\in\mathcal{O}}\|(\partial_{\lambda}^{\alpha}u)(\lambda,\cdot,\cdot)\|_{s-|\alpha|},\quad \|u\|^{q,\gamma,\mathcal{O}}\triangleq\max_{\alpha\in\mathbb{N}^{d+1}\atop|\alpha|\leqslant q}\gamma^{|\alpha|}\sup_{\lambda\in\mathcal{O}}|(\partial_{\lambda}^{\alpha}u)(\lambda)|.
\end{align*}
The following lemma gathers the classical estimates.
\begin{lem}\label{lem functions}
	Let $(d,\gamma,q,s_{0})$ as in \eqref{d gm S s0} and $s\geqslant s_0$. Take $u_1,u_2\in W^{q,\infty,\gamma}(\mathcal{O},H^s).$ Then the following properties hold true:
	\begin{enumerate}[label=(\roman*)]
		\item {(Product law)}: We have $u_{1}u_{2}\in W^{q,\infty,\gamma}(\mathcal{O},H^{s})$ and 
		$$\|u_{1}u_{2}\|_{s}^{q,\gamma,\mathcal{O}}\lesssim\|u_{1}\|_{s_{0}}^{q,\gamma,\mathcal{O}}\| u_{2}\|_{s}^{q,\gamma,\mathcal{O}}+\| u_{1}\|_{s}^{q,\gamma,\mathcal{O}}\| u_{2}\|_{s_{0}}^{q,\gamma,\mathcal{O}};$$
		\item {(Composition law)}: For $f\in C^{\infty}(\mathcal{O}\times\mathbb{R},\mathbb{R})$, if there exists $\mathtt{M}>0$ such that
		$$\|u_{1}\|_{s}^{q,\gamma,\mathcal{O}},\|u_{2}\|_{s}^{q,\gamma,\mathcal{O}}\leqslant \mathtt{M},$$ 
		then $f(u_{1})-f(u_{2})\in W^{q,\infty,\gamma}(\mathcal{O},H^{s})$ with 
		$$\|f(u_{1})-f(u_{2})\|_{s}^{q,\gamma,\mathcal{O}}\leqslant C(s,d,q,f,\mathtt{M})\|u_{1}-u_{2}\|_{s}^{q,\gamma,\mathcal{O}},$$
		where we used the notation
		$$\forall \,(\lambda,\varphi,x)\in \mathcal{O}\times\mathbb{T}^{d+1},\quad f(u)\equiv f(u)(\lambda,\varphi,x)\triangleq  f(\lambda,u(\lambda,\varphi,x)).$$
	\end{enumerate}
\end{lem}

\subsection{Reversible and traveling quasi-periodic waves}\label{sec travQP}
Let $\mathscr{S}$ by the involution defined in \eqref{inv trans} acting on the real variable $r=(r_{+},r_{-})\in\R^2$.
\begin{defin}\label{revanti.r}
    A function $r:\mathbb{T}^{d}\times\mathbb{T}\to\mathbb{R}^2$ is called:
    \begin{enumerate}[label=\textbullet]
        \item \textit{reversible} if, for any $\vf\in\T^d$, $\mathscr{S}r(\vf,\,\cdot\,)=r(-\vf,\,\cdot\,)$;
        \item \textit{anti-reversible} if, for any $\vf\in\T^d$, $-\mathscr{S}r(\vf,\,\cdot\,)=r(-\vf,\,\cdot\,)$.
    \end{enumerate} 
\end{defin}

With slight modifications, we recall the definition of traveling quasi-periodic waves given in Definition \ref{def.travQP.intro}
\begin{defin}{\bf (Quasi-periodic traveling waves)}\label{def.travQP.app}
        Let $d\in\N^*$ and $\vec\jmath\triangleq (\bar\jmath_{1},...,\bar\jmath_{d})\in\Z^{d}\setminus\{0\}$. A function $r:\T^d\times\mathbb{T}\to\C^K$, $K\in\N^*$, is called \textit{quasi-periodic traveling} with velocity vector $\vec{\jmath}$ if there exists a $(2\pi)^d$-periodic function $R:\mathbb{T}^{d}\to \C^K $ such that $r(\vf,x)= R(\vf-\vec\jmath x)$.
    \end{defin}
Comparing with Definition \ref{def.travQP.intro}, we find convenient to call {\it traveling quasi-periodic} wave both the function $r(\vf,x)= R(\vf-\vec\jmath x)$ and the function of time $r(t,x)= R(\omega t-\vec\jmath x)$. Quasi-periodic traveling waves are characterized by the relation
\begin{equation}\label{traveling.w.app}
    \forall\,y\in\R,\quad r(\vf-\vec{\jmath} y,\,\cdot\,) = \mathscr{T}_{y} r,
\end{equation}
where $\mathscr{T}_{y}$ is the translation operator in \eqref{inv trans}. Product and composition of quasi-periodic traveling waves are quasi-periodic traveling waves. 
Expanded in Fourier series as in \eqref{fourier.u}, a quasi-periodic traveling wave has the form
\begin{equation}
    r(\vf,x) = \sum_{(\ell,j)\in\Z^{d+1} \atop \vec\jmath\cdot \ell +j=0} r_{\ell,j} \be_{\ell,j}(\vf,x),
\end{equation}
namely, comparing with Definition \ref{def.travQP.app},
\begin{equation}\label{fourier.cond.T}
    r(\vf,x)= R(\vf-\vec\jmath x), \qquad R(\psi) = \sum_{\ell \in \Z^d} R_\ell
e^{\im\ell\cdot\psi}, \qquad R_{\ell} = r_{\ell,-\vec\jmath\cdot \ell} .
\end{equation}
The traveling waves $r(\vf,x)=R(\vf-\vec{\jmath} x)$, where $R(\,\cdot\,)$ belongs to the Sobolev space $H^s(\T^d,\C^K)$ as in \eqref{sob.space.d.intro} (with values in $\C^K$, $K\in\N^*$)
form a subspace of the Sobolev space $H^s(\T^{d+1},\C^K)$ as in \eqref{sobolev.space} (with values in $\C^K$, $K\in\N^*$). Note the equivalence of the norms (use \eqref{fourier.cond.T})
\begin{equation}
    \|r\|_{H^s(\T_{\vf}^d\times \T_{x})} \simeq_{s} \|R\|_{H^s(\T^d)}.
\end{equation}
For $N\geqslant 1$, we define the smoothing operator $\Pi_{N}$ on the traveling waves
\begin{equation}\label{proj.traveling}
    \Pi_{N}: r = \sum_{(\ell,j)\in\in \Z^{d+1} \atop \vec\jmath\cdot\ell+j=0} r_{\ell,j} \be_{\ell,j} \mapsto \Pi_{N} r \triangleq \sum_{(\ell,j)\in\in \Z^{d+1}, \ \braket{\ell}\leqslant N\atop \vec\jmath\cdot\ell +j= 0 } r_{\ell,j} \be_{\ell,j}, \qquad \Pi_{N}^\perp \triangleq {\rm Id} - \Pi_{N}.
\end{equation}
Note that, writing a traveling wave as in \eqref{fourier.cond.T}, the projector $\Pi_{N}$ in \eqref{proj.traveling} is equal to
\begin{equation}
    (\Pi_{N}r)(\vf,x) = R_N(\vf-\vec\jmath x), \qquad R_{N}(\psi) \triangleq \sum_{\ell\in\Z^{d} \atop \braket{\ell}\leqslant N} R_{\ell} e^{\im \ell\cdot \psi} .
\end{equation}
The following smoothing estimates are classical.
\begin{lem}{\bf (Smoothing).}\label{lem:projn}
    Let $(d,\gamma,q)$ as in \eqref{d gm S s0}. For any $N\geqslant 1$ and any traveling wave $r$, we have
    \begin{align}
        \| \Pi_{N}r\|_{s}^{q,\gamma,\cO} & \leqslant N^{\alpha} \|r\|_{s-\alpha}^{q,\gamma,\cO}, \quad 0\leqslant \alpha \leqslant s, \\
         \| \Pi_{N}^\perp r\|_{s}^{q,\gamma,\cO} & \leqslant N^{-\alpha} \|r\|_{s+\alpha}^{q,\gamma,\cO}, \quad  \alpha \geqslant 0 .
    \end{align}
\end{lem}

\subsection{Linear operators}
Here we give the basic notations for our functional framework together. Then, we introduce the notion of $\mathcal{D}^{q}$-$\sigma$-tame and $\mathcal{D}^{q}$-modulo-tame operators useful for the reduction of the linearized operator.
\subsubsection{Basic notations for functions and operators}
Throughout the document, we consider $\varphi$-dependent families of operators acting on $L^2(\mathbb{T}^{d+1},\mathbb{C})$ functions as follows
$$(Au)(\varphi,x)\triangleq\big(A(\varphi)u(\varphi,\cdot)\big)(x).$$
In Fourier, we can write this action
\begin{equation}\label{action.A.fou}
    (Au)(\varphi,x)=\sum_{(\ell,\ell')\in(\mathbb{Z}^d)^2\atop(j,j')\in\mathbb{Z}^2}A_{j}^{j'}(\ell-\ell')u_{\ell',j'}e^{\ii(\ell\cdot\varphi+jx)} .
\end{equation}
This allows to identify the operator $A$ with the matrix $\big(A_{j}^{j'}(\ell-\ell')\big)_{(\ell,\ell',j,j')\in(\mathbb{Z}^d)^2\times\mathbb{Z}^2}$ which is T\"oplitz with respect to the index $\ell.$ Given $A$ and $B$ two operators, we denote:
\begin{itemize}
    \item $[A,B]\triangleq AB-BA$ the commutator and $[A,B]_{\rm a}\triangleq AB+BA$ the anti-commutator;
    \item $|A|$ the majorant operator, defined by
    $$(|A|u)(\varphi,x)\triangleq\sum_{(\ell,\ell')\in(\mathbb{Z}^d)^2\atop(j,j')\in\mathbb{Z}^2}\big|A_{j}^{j'}(\ell-\ell')\big|u_{\ell',j'}e^{\ii(\ell\cdot\varphi+jx)};$$
    \item $\lfloor A\rfloor$ the diagonal component, defined by
    \begin{equation}\label{def:diagcomp}
        \left(\lfloor A\rfloor u\right)(\varphi,x)\triangleq\sum_{(\ell,j)\in\mathbb{Z}^d\times\mathbb{Z}}A_j^j(0)u_{\ell,j}e^{\ii(\ell\cdot\varphi+jx)};
    \end{equation}
    \item Given $N\in\mathbb{N}$, the projected operator $\Pi_{N}A$, defined by
    \begin{equation}\label{def:projectors}
        (\Pi_NAu)(\varphi,x)\triangleq\sum_{(\ell,\ell',j,j')\in(\mathbb{Z}^d)^2\times\mathbb{Z}^2\atop|\ell-\ell'|\leqslant N}A_{j}^{j'}(\ell-\ell')u_{\ell',j'}e^{\ii(\ell\cdot\varphi+jx)};
    \end{equation}
	\item Given $l\in\mathbb{N}$ and $k\in\llbracket 0,d\rrbracket,$ the operators $\partial_{\varphi_k}^lA$ and $\langle\partial_{\varphi}\rangle^{l}A$, defined respectively by
    \begin{align}
        \big(\partial_{\varphi_k}^lAu\big)(\varphi,x)& \triangleq\sum_{(\ell,\ell')\in(\mathbb{Z}^d)^2\atop(j,j')\in\mathbb{Z}^2}(\ell_k-\ell_k')^{l}A_{j}^{j'}(\ell-\ell')u_{\ell',j'}e^{\ii(\ell\cdot\varphi+jx)}, \\
        \big(\langle\partial_{\varphi}\rangle^{l}Au\big)(\varphi,x)&\triangleq\sum_{(\ell,\ell')\in(\mathbb{Z}^d)^2\atop(j,j')\in\mathbb{Z}^2}\langle \ell-\ell'\rangle^{l}A_{j}^{j'}(\ell-\ell')u_{\ell',j'}e^{\ii(\ell\cdot\varphi+jx)};
    \end{align}
    
    \item  Given $n_1,n_2\in\mathbb{R},$ the operator $\langle D\rangle^{n_1}A\langle D\rangle^{n_2}$, defined by
    $$\left(\langle D\rangle^{n_1}A\langle D\rangle^{n_2}u\right)(\varphi,x)\triangleq\sum_{(\ell,\ell')\in(\mathbb{Z}^d)^2\atop(j,j')\in\mathbb{Z}^2}\langle j\rangle^{n_1}A_{j}^{j'}(\ell-\ell')\langle j'\rangle^{n_2}u_{\ell',j'}e^{\ii(\ell\cdot\varphi+jx)}.$$
\end{itemize}

We distinguish symmetry properties of linear operators based on reality properties, and on their behavious with the involution $\mathscr{S}$ and the translation operator $\mathscr{T}_{y}$, $y\in\R$, (see their definition in \eqref{inv trans}).
\begin{defin}\label{defin:OP-sym}
	A $\varphi$-dependent family of operators $\varphi\mapsto A(\varphi)$ is said to be:
	\begin{itemize}
	    \item \textit{real} if, for any $u\in L^{2}(\mathbb{T}^{d+1},\mathbb{C})$ such that $\overline u = u$, it holds $\overline{Au}=Au$;
		\item \textit{reversible} if, for any $\vf\in\T$, it holds $A(-\varphi)\circ\mathscr{S}=-\mathscr{S}\circ A(\varphi)$;
		\item \textit{reversibility preserving} if, for any $\vf\in\T$, it holds $A(-\varphi)\circ\mathscr{S}=\mathscr{S}\circ A(\varphi)$;
		\item \textit{momentum preserving} if, for any $\vf \in \T^d$ and $y\in\R$, it holds $A(\varphi-\vec{\jmath}y)\circ\mathscr{T}_{y}=\mathscr{T}_{y}\circ A(\varphi)$.
	\end{itemize}
\end{defin}

\begin{remark}\label{compo:RevMom} 
Let us make the following remarks:
\begin{enumerate}[label=\textbullet]
    \item Note that a reversible operator sends reversible function onto anti-reversible functions and vice-versa, in the sense of Definition \ref{revanti.r}. Conversely, a reversibility preserving operator let invariant such structures. Besides, a momentum preserving operator sends quasi-periodic traveling waves with velocity vector $\vec{\jmath}$ onto quasi-periodic traveling waves with velocity vector $\vec{\jmath}$.
    \item The composition of two reversible or reversibility preserving operators is reversibility preserving. The composition of a reversible operator with a reversibility preserving one is reversible. The composition of two momentum preserving operators is momentum preserving.  
\end{enumerate}
    
\end{remark}

Operators with symmetries as in Definition \ref{defin:OP-sym} are characterized with in Fourier by means of the expansion in \eqref{action.A.fou}.
\begin{lem}\label{lem Fourier coeff op} 
A $\varphi$-dependent family of operators $\varphi\mapsto A(\varphi)$ is:
	\begin{itemize}
    \item real if and only if, for all $(\ell,j,j')\in\mathbb{Z}^d\times\mathbb{Z}^2$, it holds $A_{-j}^{-j'}(-\ell)=\overline{A_{j}^{j'}(\ell)}$;
		\item reversible if and only if, for all $(\ell,j,j')\in\mathbb{Z}^d\times\mathbb{Z}^2$, it holds $A_{j}^{j'}(\ell)=-\overline{A_{j}^{j'}(\ell)}$;
		\item reversibility preserving if and only if, for all $(\ell,j,j')\in\mathbb{Z}^d\times\mathbb{Z}^2$, it holds $A_{j}^{j'}(\ell)=\overline{A_{j}^{j'}(\ell)}$;
		\item momentum preserving if and only if
        \begin{equation}
            \forall\,(\ell,j,j')\in\mathbb{Z}^d\times\mathbb{Z}^2,\qquad A_{j}^{j'}(\ell)\neq 0\quad\Rightarrow\quad\vec{\jmath}\cdot\ell+j-j'=0.
        \end{equation}
	\end{itemize}
\end{lem}
\noindent In our analysis, we shall also consider matrices of operators acting on
\begin{equation*}
    \mathbf{L}^{2}(\mathbb{T}^{d+1},\mathbb{C}^2)\triangleq L^{2}(\mathbb{T}^{d+1},\mathbb{C})\times L^{2}(\mathbb{T}^{d+1},\mathbb{C})
\end{equation*}
and taking the form
\begin{equation}\label{matrix op}
    \mathbf{A}=\begin{pmatrix}
    A_+^{(d)} & A_+^{(o)}\vspace{0.1cm}\\
    A_-^{(o)} & A_-^{(d)}
\end{pmatrix}.
\end{equation}
All the previous definitions naturally extend to matrix operators $\bA$ as in \eqref{matrix op}, in particular:
\begin{itemize}
    \item A matrix operator $\mathbf{A}$ in the form \eqref{matrix op} is said to be {\it real} (resp. {\it reversible}, {\it reversibility preserving}, {\it momentum preserving}) if each entry $A_+^{(d)},$  $A_+^{(o)},$ $A_-^{(o)},$ $A_-^{(d)}$ is real (resp. reversible, reversibility preserving, momentum preserving);
    \item We denote\begin{equation}\label{def:diagcomp-matrix}
            \lfloor\mathbf{A}\rfloor\triangleq\begin{pmatrix}
            \lfloor A_+^{(d)}\rfloor & 0\\
            0 & \lfloor A_-^{(d)} \rfloor
        \end{pmatrix},
        \end{equation}
        and, for $N \in \N^*$, 
        \begin{equation}\label{def:projectors-matrix}
            \Pi_N\mathbf{A}=\begin{pmatrix}
    \Pi_NA_+^{(d)} & \Pi_NA_+^{(o)}\vspace{0.1cm}\\
    \Pi_NA_-^{(o)} & \Pi_NA_-^{(d)}
\end{pmatrix}.
        \end{equation}
\end{itemize}

\subsubsection{$\mathcal{D}^{q}$(-modulo) tame operators}\label{sec:tameOP}
\noindent In our analysis, the operators might depend on an external parameter $\lambda\in\mathcal{O}.$ We shall emphasize this dependence by denoting $A=A(\lambda).$ In this subsection, we recall the formalism of (modulo)-tame operator with $q$ derivatives with respect to the external parameter $\lambda$. This formalism was introduced in \cite[Section 2.2]{BM18} and is used for the KAM reduction of Section \ref{sect.KAM}. It is a rather general framework capturing the tame estimates for operators, essential property for running KAM reduction schemes.
\begin{defin}
	\textbf{($\mathcal{D}^{q}$-$\sigma$-tame operator)}\\
	Let $\sigma\geqslant0.$ We say that a linear operator $A$ is \textit{$\mathcal{D}^{q}$-$\sigma$-tame} if there exists a non-decreasing function
	$$\begin{array}{rcl}
		[s_0,S] & \rightarrow & \mathbb{R}_+\\
		s & \mapsto & \mathfrak{M}_{A}(q,\sigma,s)\triangleq\mathfrak{M}_{A}(s)
	\end{array}$$
	such that for any $s\in[s_0,S]$ and any $u\in H^{s+\sigma}$, we have
	$$\sup_{\alpha\in\mathbb{N}^{d+1}\atop|\alpha|\leqslant q}\sup_{\lambda\in\mathcal{O}}\gamma^{|\alpha|}\big\|\big(\partial_{\lambda}^{\alpha}A(\lambda)\big)u\big\|_{s}\leqslant\mathfrak{M}_{A}(s_0)\|u\|_{s+\sigma}+\mathfrak{M}_{A}(s)\|u\|_{s_0+\sigma}.$$
	In this case, the function $\mathfrak{M}_{A}$ is called \textit{tame constant of the operator $A$.} The $\mathcal{D}^{q}$-$0$-tame operators are simply called $\mathcal{D}^{q}$-tame operators.
\end{defin}
\begin{remark}
    Representing $A$ by its
matrix elements $\big(A_j^{j'}(\ell-\ell')\big)_{(\ell,\ell',j,j')\in(\mathbb{Z}^d)^2\times\mathbb{Z}^2}$ as in \eqref{action.A.fou}, we have for all $|k|\leqslant q,$ $j'\in\mathbb{Z}$ and $\ell'\in\mathbb{Z}^d,$
\begin{equation}\label{useful:e-tame}
    \gamma^{2|k|}\sum_{(\ell,j)\in\mathbb{Z}^{d+1}}\langle\ell,j\rangle^{2s}|\partial_{\lambda}^kA_j^{j'}(\ell-\ell')|^2\leqslant2\fM_A^2(s_0)\langle\ell',j'\rangle^{2(s+\sigma)}+2\fM_A^2(s)\langle\ell',j'\rangle^{2(s_0+\sigma)}.
\end{equation}
\end{remark}
\begin{prop}\label{prop.tameconst}
	\textbf{(Properties of $\mathcal{D}^{q}$-$\sigma$-tame operators)}\\
	The following properties hold true:
	\begin{itemize}
		\item[$(i)$] (Composition): Let $A$ (resp. $B$) be a $\mathcal{D}^{q}$-$\sigma_A$-tame (resp. $\mathcal{D}^{q}$-$\sigma_B$-tame) operator. Then $AB$ is a $\mathcal{D}^{q}$-$\sigma_A+\sigma_B$-tame operator with
		$$\mathfrak{M}_{AB}(s)\lesssim_{q}\mathfrak{M}_A(s_0)\mathfrak{M}_{B}(s+\sigma_A)+\mathfrak{M}_A(s)\mathfrak{M}_{B}(s_0+\sigma_A).$$
		\item[$(ii)$] (Link with operator norm): For any $u\in W^{q,\infty,\gamma}(\mathcal{O},H^{s+\sigma})$, we have $$\|Au\|_{s}^{q,\gamma,\mathcal{O}}\lesssim_q\mathfrak{M}_A(s_0)\|u\|_{s+\sigma}^{q,\gamma,\mathcal{O}}+\mathfrak{M}_A(s)\|u\|_{s_0+\sigma}^{q,\gamma,\mathcal{O}}.$$
		This means
		$$\|A\|_{\mathcal{L}\big(W^{q,\gamma,\infty}(\mathcal{O},H^{s+\sigma}),W^{q,\gamma,\infty}(\mathcal{O},H^{s})\big)}\lesssim\mathfrak{M}_{A}(s).$$
	\end{itemize}
\end{prop}

\begin{defin}\label{modulo.tame.def}
	\textbf{($\mathcal{D}^q$-$(-1)$-modulo-tame operator)}\\
	We say that a linear operator $A=A(\lambda)$ is {\it $\cD^{q}$-$(-1)$-modulo-tame} if there exists a non-decreasing function
    \begin{equation}
        \begin{array}{rcl}
             [s_0,S] &  \rightarrow  & \mathbb{R}_+,\\
             s & \mapsto &  \fM_{\braket{D}^\frac12 A \braket{D}^\frac12}^{\sharp}(s),
        \end{array}
    \end{equation}
	such that, for all $k\in \N^{d+1}$, $|k|\leqslant q$, the majorant operator $\braket{D}^\frac12 \big|\pa_{\lambda}^k A\big|\braket{D}^\frac12$ satisfies, for any $s\in[s_0,S]$ and any $u\in H^{s}$, 
    \begin{align}
        \sup_{\alpha\in\mathbb{N}^{d+1}\atop|\alpha|\leqslant q}\sup_{\lambda\in\mathcal{O}}\gamma^{|\alpha|}\big\|  \braket{D}^\frac12 \big|\pa_{\lambda}^k A\big|\braket{D}^\frac12u\big\|_{s} \leqslant \fM_{\braket{D}^\frac12 A \braket{D}^\frac12}^{\sharp}(s_0)\|u\|_{s}+\fM_{\braket{D}^\frac12 A \braket{D}^\frac12}^{\sharp}(s)\|u\|_{s_0}.
    \end{align}
	In this case, the function $\mathfrak{M}_{\langle D\rangle^{\frac{1}{2}}A\langle D\rangle^{\frac{1}{2}}}^{\sharp}$ is called \textit{modulo-tame constant of the operator $A$.} For a matrix $\bA$ as in \eqref{matrix op}, we denote
    \begin{equation}
        \fM_{\braket{D}^\frac12 \bA \braket{D}^\frac12}^{\sharp}(s) \triangleq \max_{\kappa\in\{-,+\}} \Big\{ \fM_{\braket{D}^\frac12 A_{\kappa}^{(d)} \braket{D}^\frac12}^{\sharp}(s), \fM_{\braket{D}^\frac12 A_{\kappa}^{(o)} \braket{D}^\frac12}^{\sharp}(s) \Big\} .
    \end{equation}
\end{defin}
\begin{prop}\label{prop.modulo.tame}
	\textbf{(Properties of $\mathcal{D}^{q}$-modulo-tame operators)}\\
	Let $A$ and $B$ be two $\mathcal{D}^{q}$-modulo-tame operators. Then, the following properties hold true:
	\begin{itemize}
		\item[$(i)$] (Sum) The operator $A+B$ is $\mathcal{D}^q$-($-1$)-modulo-tame with, for any $s\in[s_0,S],$
		$$\mathfrak{M}_{\langle D\rangle^{\frac{1}{2}}(A+B)\langle D\rangle^{\frac{1}{2}}}^{\sharp}(s)\leqslant\mathfrak{M}_{\langle D\rangle^{\frac{1}{2}}A\langle D\rangle^{\frac{1}{2}}}^{\sharp}(s)+\mathfrak{M}_{\langle D\rangle^{\frac{1}{2}}B\langle D\rangle^{\frac{1}{2}}}^{\sharp}(s);$$
		\item[$(ii)$] (Composition) For any $\mathtt{b}\in\mathbb{N}$, the operator $\langle\partial_{\varphi}\rangle^{\mathtt{b}}(AB)$ is $\mathcal{D}^q$-($-1$)-modulo-tame with, for any $s\in[s_0,S],$
        \begin{align}
        \mathfrak{M}_{\braket{D}^\frac12\langle\partial_{\varphi}\rangle^{\mathtt{b}}(AB)\braket{D}^\frac12}^{\sharp}(s)& \lesssim_{q,\mathtt{b}}\mathfrak{M}_{\braket{D}^\frac12\langle\partial_{\varphi}\rangle^{\mathtt{b}}A \braket{D}^\frac12}^{\sharp}(s)\mathfrak{M}_{\langle D\rangle^{\frac{1}{2}}B\langle D\rangle^{\frac{1}{2}}}^{\sharp}(s_0)+\mathfrak{M}_{ \braket{D}^\frac12\langle\partial_{\varphi}\rangle^{\mathtt{b}}A\braket{D}^\frac12}^{\sharp}(s_0)\mathfrak{M}_{\braket{D}^\frac12 B\braket{D}^\frac12}^{\sharp}(s) \\
        &+\mathfrak{M}_{\braket{D}^\frac12 A\braket{D}^\frac12}^{\sharp}(s)\mathfrak{M}_{\braket{D}^\frac12\langle\partial_{\varphi}\rangle^{\mathtt{b}}B\braket{D}^\frac12}^{\sharp}(s_0)+\mathfrak{M}_{\braket{D}^\frac12 A\braket{D}^\frac12}^{\sharp}(s_0)\mathfrak{M}_{\braket{D}^\frac12\langle\partial_{\varphi}\rangle^{\mathtt{b}}B\braket{D}^\frac12}^{\sharp}(s);
        \end{align}
        \item[$(iii)$] (Projection) For any $N\in\mathbb{N},$ the operator $\Pi_N^{\perp}A$ is $\mathcal{D}^q$-($-1$)-modulo-tame with, for any $\mathtt{b}\in\mathbb{N}$ and any $s\in[s_0,S],$
        $$\mathfrak{M}_{\langle D\rangle^{\frac{1}{2}}\Pi_N^{\perp}A\langle D\rangle^{\frac{1}{2}}}^{\sharp}(s)\leqslant N^{-\mathtt{b}}\mathfrak{M}_{\langle D\rangle^{\frac{1}{2}}\langle\partial_{\varphi}\rangle^{\mathtt{b}}A\langle D\rangle^{\frac{1}{2}}}^{\sharp}(s);$$
        \item[$(iv)$] $A$ is also a $\cD^{q}$-tame operator, with, for any $s\in[s_0,S],$ $$\fM_{A}(s)\leqslant\fM_{\langle D\rangle^{\frac{1}{2}}A\langle D\rangle^{\frac{1}{2}}}^{\sharp}(s).$$
	\end{itemize}
\end{prop}
\subsection{Pseudo-differential calculus}\label{app sect pseudo}
In this last subsection, we introduce the basic definitions and properties for pseudo-differential operators. We also focus on the particular class of pseudo-differential operators, useful in our study, given by homogeneous expansions. Finally, we recall the Egorov theorem for homogeneous operators.
\subsubsection{Pseudo-differential operators}\label{appendix pseudo}
\begin{defin}\label{defin OPS}
	A pseudo-differential symbol $\mathtt{a}(\varphi,x,j)$ of order $m$ is the restriction to $\mathbb{T}^d\times\mathbb{R}\times\mathbb{Z}$ of a function $\mathtt{a}(\varphi,x,\xi)$ which is of class $C^{\infty}$ on $\mathbb{T}^d\times\mathbb{R}\times\mathbb{R}$, $2\pi$-periodic in $x$, and satisfying
	$$\forall\,(\alpha,\beta)\in\mathbb{N}^2,\qquad\exists \,C_{\alpha,\beta}>0,\qquad\Big|\partial_{x}^{\alpha}\partial_{\xi}^{\beta}\mathtt{a}(\varphi,x,\xi)\Big|\leqslant C_{\alpha,\beta}\langle\xi\rangle^{m-\beta}.$$
	We denote by $S^m$ the class of symbols of order $m$ and $S^{-\infty}=\displaystyle\bigcap_{m\in\mathbb{N}}S^m.$ To a symbol $\mathtt{a}(x,\xi)$ in $S^m$ we associate its quantization $\textnormal{Op}(\mathtt{a})$ defined by
	$$u(\varphi,x)=\sum_{j\in\mathbb{Z}}u_j(\varphi)e^{\ii jx}\qquad\Rightarrow\qquad\big[\textnormal{Op}(\mathtt{a})u\big](\varphi,x)=\sum_{j\in\mathbb{Z}}\mathtt{a}(\varphi,x,j)u_j(\varphi)e^{\ii jx}.$$
	Throughout this paper, we may denote either $\textnormal{Op}(\mathtt{a})$ or $\mathtt{a}(x,D)$ the pseudo-differential operator associated with the symbol $\mathtt{a}(x,\xi).$
	We denote by $OPS^m$ the set of pseudo-differential operators of order $m$ and $OPS^{-\infty}=\displaystyle\bigcap_{m\in\mathbb{R}}OPS^m.$
	\begin{equation}\label{def pseudo-norm}
		\|A\|_{m,s,\alpha}^{q,\gamma,\mathcal{O}}\triangleq\sum_{k\in\mathbb{N}^d\atop|k|\leqslant q}\gamma^{|k|}\sup_{\lambda\in\mathcal{O}}\|\partial_{\lambda}^{k}A(\lambda)\|_{m,s-|k|,\alpha},\qquad\|A(\lambda)\|_{m,s,\alpha}\triangleq\max_{0\leqslant\beta\leqslant\alpha}\sup_{\xi\in\mathbb{R}}\,\langle\xi\rangle^{\beta-m}\|\partial_{\xi}^{\beta}\mathtt{a}(\lambda;\cdot,\cdot,\xi)\|_{H^{s}}.
	\end{equation}
   We also say that a $2\times2$-matrix operator $\bA$ in the form \eqref{matrix op}
    is a pseudo-differential operator in the $OPS^{m}$ if $A_{+}^{(d)}, A_{+}^{(o)}, A_{-}^{(o)}, A_{-}^{(d)}\in OPS^{m}$, and its norm is given by
    \begin{equation}\label{normpseudo:matrix}
        \|\bA\|_{m,s,\alpha}^{q,\gamma\cO}\triangleq \max\big\{\|A_{\kappa}^{(\nu)}\|_{m,s,\alpha}^{q,\gamma,\cO},\,\,\kappa\in\{-,+\},\,\,\nu\in\{o,d\}\big\}.
    \end{equation}
\end{defin}
In the next proposition, we gather all the important algebraic and analytic properties of pseudo-differential operators which are  applied all along the Section \ref{sec QPS VP}. For their proofs, we refer for instance to \cite[Section 2]{BM18} and \cite[Section 3]{BFM21}.
\begin{prop}\label{properties OPS}
	Let $m,m'\in \mathbb{R}.$ Consider two pseudo-differential operators $A(\lambda)=\textnormal{Op}\big(\mathtt{a}(\lambda;\varphi,x,\xi)\big)\in OPS^{m}$ and $B(\lambda)=\textnormal{Op}\big(\mathtt{b}(\lambda;\varphi,x,\xi)\big)\in OPS^{m'}.$ The following properties hold:
	\begin{enumerate}[label=(\roman*)]
		\item (Composition) The operator $AB\triangleq A\circ B$ is pseudo-differential with symbol $\mathtt{a}\#\mathtt{b}\in S^{m+m'}$ in the form : for any $N\in\mathbb{N}^*,$
		$$\mathtt{a}\#\mathtt{b}(\lambda;\varphi,x,\xi)\triangleq\sum_{k=0}^{N-1}\frac{1}{k!\ii^{k}}\partial_{\xi}^{k}\mathtt{a}(\lambda;\varphi,x,\xi)\partial_{x}^{k}\mathtt{b}(\lambda;\varphi,x,\xi)+r_{N}(\lambda;\varphi,x,\xi),\qquad r_{N}(\lambda;\varphi,x,\xi)\in S^{m+m'-N}.$$
		In addition, denoting $R_{N}\triangleq\textnormal{Op}(r_{N}),$ we have for any $s\in[s_0,S]$ and any $\alpha\in\mathbb{N},$ 
		\begin{align}
			\|AB\|_{m+m',s,\alpha}^{q,\gamma,\mathcal{O}}&\lesssim\|A\|_{m,s,\alpha}^{q,\gamma,\mathcal{O}}\|B\|_{m',s_0+\alpha+|m|,\alpha}^{q,\gamma,\mathcal{O}}+\|A\|_{m,s,\alpha}^{q,\gamma,\mathcal{O}}\|B\|_{m',s+\alpha+|m|,\alpha}^{q,\gamma,\mathcal{O}},\label{e-pseudo-comp}\\
			\|R_{N}\|_{m+m'-N,s,\alpha}^{q,\gamma,\mathcal{O}}&\lesssim\frac{1}{N!}\Big(\|A\|_{m,s,N+\alpha}^{q,\gamma,\mathcal{O}}\|B\|_{m',s_0+2N+\alpha+|m|,N+\alpha}^{q,\gamma,\mathcal{O}}+\|A\|_{m,s,N+\alpha}^{q,\gamma,\mathcal{O}}\|B\|_{m',s+2N+\alpha+|m|,N+\alpha}^{q,\gamma,\mathcal{O}}\Big).\nonumber
		\end{align}
	Finally, the remainder symbol expresses as	$$r_{N}(\lambda;\varphi,x,\xi)=\frac{1}{\ii^{N}(N-1)!}\int_{0}^{1}(1-\tau)^{N-1}\sum_{j\in\mathbb{Z}}(\partial_{\xi}^{N}\mathtt{a})(\lambda;\varphi,x,\xi+\tau j)(\widehat{\partial_{x}^{N}\mathtt{b}})(\lambda;\varphi,j,\xi)e^{\ii jx}{\rm d}\tau ;$$
		\item (Commutator) The operator $[A,B]\triangleq AB-BA$ is pseudo-differential with symbol $\mathtt{a}\star\mathtt{b}\in S^{m+m'-1}$ in the form : for any $N\in\mathbb{N}^*,$
		$$\mathtt{a}\star\mathtt{b}(\lambda;\varphi,x,\xi)\triangleq -\ii\lbrace\mathtt{a},\mathtt{b}\rbrace(\lambda;\varphi,x,\xi)+\widetilde{r}_2(\lambda;\varphi,x,\xi),\qquad\lbrace\mathtt{a},\mathtt{b}\rbrace\triangleq\partial_{\xi}\mathtt{a}\partial_{x}\mathtt{b}-\partial_{x}\mathtt{a}\partial_{\xi}\mathtt{b},\qquad\widetilde{r}_{2}\in S^{m+m'-2}.$$
		In addition, we have for any $s\in[s_0,S]$ and any $\alpha\in\mathbb{N},$ 
		\begin{align*}
			\|[A,B]\|_{m+m'-1,s,\alpha}^{q,\gamma,\mathcal{O}}&\lesssim\|A\|_{m,s+|m'|+\alpha+2,\alpha+1}^{q,\gamma,\mathcal{O}}\|B\|_{m',s_0+|m|+\alpha+2,\alpha+1}^{q,\gamma,\mathcal{O}}\\
			&\quad+\|A\|_{m,s+|m'|+\alpha+2,\alpha+1}^{q,\gamma,\mathcal{O}}\|B\|_{m',s+|m|+\alpha+2,\alpha+1}^{q,\gamma,\mathcal{O}};
		\end{align*}
	\item (Multiplication operator) If $\mathtt{a}(\lambda;\varphi,x,D)=\mathtt{a}(\lambda;\varphi,x),$ then for any $s\in[s_0,S]$ and any $\alpha\in\mathbb{N},$
	\begin{equation}\label{pseudo-mul}
		\|A\|_{0,s,\alpha}^{q,\gamma,\mathcal{O}}=\|\mathtt{a}\|_{s}^{q,\gamma,\mathcal{O}}.
	\end{equation}
	\item (Fourier multiplier) If $\mathtt{a}(\lambda;\varphi,x,D)=\mathtt{a}(\lambda;D),$ then for any $s\in[s_0,S]$ and any $\alpha\in\mathbb{N},$
	\begin{equation}\label{pseudo FM}
		\|A\|_{m,s,\alpha}^{q,\gamma,\mathcal{O}}=\|A\|_{m,0,\alpha}^{q,\gamma,\mathcal{O}};
	\end{equation}
\item (Link with tame constants) For any $s\in[s_0,S],$
$$\mathfrak{M}_{A}(s)\lesssim_s\|A\|_{0,s,0}^{q,\gamma,\mathcal{O}};$$
\item (Reversibility and momentum preserving properties) 
\begin{enumerate}[label=\textbullet]
	\item The operator $A$ is reversible if and only if
	$$\mathtt{a}(\lambda;-\varphi,-x,\xi)=-\overline{\mathtt{a}(\lambda;\varphi,x,\xi)};$$
	\item The operator $A$ is reversibility preserving if and only if
	$$\mathtt{a}(\lambda;-\varphi,-x,\xi)=\overline{\mathtt{a}(\lambda;\varphi,x,\xi)};$$
	\item The operator $A$ is momentum preserving if and only if
	$$\forall\, y\in\mathbb{T},\quad\mathtt{a}(\lambda;\varphi-\vec{\jmath}y,x,\xi)=\mathtt{a}(\lambda;\varphi,x+y,\xi).$$
\end{enumerate}
	\end{enumerate}
\end{prop}
We now introduce the notion of homogeneous symbols/operators used throughout the article.
\begin{defin}\label{def.HomExp}
	Let $m\in\mathbb{Z}$ and $N\in\mathbb{N}.$ Consider a pseudo-differential operator $A(\lambda)=\textnormal{Op}\big(\mathtt{a}(\lambda;\varphi,x,\xi)\big)\in OPS^{m}$. We say that:
	\begin{enumerate}[label=(\roman*)]
		\item $A$ is \textit{homogeneous of degree $m$} if it can be written as
		$$A=\mathfrak{p}\partial_{x}^{m},\qquad \textnormal{i.e.}\qquad\mathtt{a}(\lambda;\varphi,x,\xi)=\mathfrak{p}(\lambda;\varphi,x)\chi(\xi)(\ii\xi)^{m},$$
		with $\mathfrak{p}$ a real valued function;
		\item $A$ has an \textit{homogeneous expansion of degree $m$ up to order $m-N$} if it can be written as
		$$A=\sum_{k=0}^{N}\mathfrak{p}_{m-k}\partial_{x}^{m-k}\qquad\textnormal{i.e.}\qquad\mathtt{a}(\lambda;\varphi,x,\xi)=\sum_{k=0}^{N}\mathfrak{p}_{m-k}(\lambda;\varphi,x)\chi(\xi)(\ii\xi)^{m-k},$$
		with $(\mathfrak{p}_{m-k})_{0\leqslant k\leqslant N}$ real valued functions.
	\end{enumerate}
\end{defin}
One has the following result.
\begin{lem}\label{lem:revmomHE}
	Let $m\in\mathbb{Z}$ and $N\in\mathbb{N}.$ Consider a pseudo-differential operator $A$ 
    in the form
    \begin{equation}\label{A.op.app}
        A=\sum_{k=0}^{N}\mathfrak{p}_{m-k}\partial_{x}^{m-k} + R_{N},
    \end{equation}
    with $(\mathfrak{p}_{m-k})_{0\leqslant k\leqslant N}$ real valued functions and  $R_{N}={\rm Op}(\tr_n(\lambda;\vf,x,\xi))\in OPS^{m-N-1}$.
	Then:
	\begin{enumerate}[label=\textbullet]
		\item The operator $A$ is reversible if and only if
        \begin{align}
            \forall \, k\in\llbracket 0,N\rrbracket,\quad & \mathfrak{p}_{m-k}(\lambda;-\varphi,-x)=(-1)^{m-k+1}\mathfrak{p}_{m-k}(\lambda;\varphi,x), \\
            \textnormal{and} \quad & \tr_{N}(\lambda;-\vf,-x,\xi)= -\overline{\tr_{N}(\lambda;\vf,x,\xi)} ;
        \end{align}
		\item The operator $A$ is reversibility preserving if and only if
        \begin{align}
            \forall \, k\in\llbracket 0,N\rrbracket,\quad & \mathfrak{p}_{m-k}(\lambda;-\varphi,-x)=(-1)^{m-k}\mathfrak{p}_{m-k}(\lambda;\varphi,x) , \\
            \textnormal{and} \quad & \tr_{N}(\lambda;-\vf,-x,\xi)= \overline{\tr_{N}(\lambda;\vf,x,\xi)} ;
        \end{align}
		\item The operator $A$ is momentum preserving if and only if
        \begin{align}
            \forall \, y\in\mathbb{T},\quad\forall \,k\in\llbracket 0,N\rrbracket,\quad & \mathfrak{p}_{m-k}(\lambda;\varphi-\vec{\jmath}y,x)=\mathfrak{p}_{m-k}(\lambda;\varphi,x+y), \\
            \textnormal{and} \quad & \tr_{N}(\lambda;\vf,-\vec{\jmath}y,x,\xi) = \tr_{N}(\lambda;\vf,x+y,\xi).
        \end{align}
	\end{enumerate}
\end{lem}
\begin{proof}
    We only prove item $(i)$, as the proof for items $(ii)$ and $(iii)$ follow similarly. Moreover, for sake of simplicity, in what follows we omit to write the dependence on $\lambda$ that plays no role here. Moreover, we prove only the direct implication, the converse one being trivial. The operator $A$ in \eqref{A.op.app} is a pseudo-differential operator $A={\rm Op}(\ta(\vf,x,\xi))\in OPS^{m}$ with symbol
    \begin{equation}\label{simbolo.a}
        \ta(\vf,x,\xi) \triangleq \sum_{k=0}^{N} \fp_{m-k}(\vf,x) (\im \xi)^{m-k} + \tr_{N}(\vf,x,\xi).
    \end{equation}
    Since we are assuming that $A$ is reversible, then, by Proposition \ref{properties OPS}-$(vi)$, we have that
    \begin{equation}
        \ta(-\vf,x,\xi) = -\overline{\ta(\vf,x,\xi)}.
    \end{equation}
    In particular, it holds that
    \begin{equation}\label{limite.zero}
        \lim_{\xi\to \infty} (\im\xi)^{-m}\big( \ta(-\vf,x,\xi) + \overline{\ta(\vf,x,\xi)} \big) = 0.
    \end{equation}
    By \eqref{simbolo.a}, it follows that
    \begin{align}
      (\im\xi)^{-m} \big( \ta(-\vf,x,\xi) + \overline{\ta(\vf,x,\xi)} \big)  & = \fp_{m}(-\vf,-x) + (-1)^{m} \fp_{m}(\vf,x) \\
        &\quad+ \sum_{k=1}^{N}\big(  \fp_{m-k}(-\vf,-x) + (-1)^{m-k} \fp_{m-k}(\vf,x) \big) (\im\xi)^{-k} \\
        &\quad+(\im\xi)^{-m}\big( \tr_{N}(-\vf,-x,\xi) + \overline{\tr_{N}(\vf,x,\xi)}\big).
    \end{align}
    From \eqref{limite.zero} and Definition \ref{defin OPS}, we deduce that
    \begin{equation}
        \fp_{m}(-\vf,-x) = (-1)^{m+1}  \fp_{m}(\vf,x).
    \end{equation}
    Repeating this process with multiplicative factor $(\ii\xi)^{k-m}$ for $k\in\llbracket1,N\rrbracket$, we get
    $$\fp_{m-k}(-\vf,-x) = (-1)^{m-k+1}  \fp_{m-k}(\vf,x).$$
    With this in hand, we come back to \eqref{simbolo.a} to obtain
    $$\tr_{N}(-\vf,-x,\xi) = -\overline{\tr_{N}(\vf,x,\xi)}.$$
    This ends the proof of Lemma \ref{lem:revmomHE}.
\end{proof}

The following result is taken from \cite[Prop. 2.11]{BKM21} (up to slight modifications to fit with our functional framework and purposes). It states that the class of homogeneous expansions is stable under composition and (anti-)commutation up to smoothing operators. The proof is based on Proposition \ref{properties OPS}-1.
\begin{prop}\label{lem compo commu hom exp}
	Consider two homogeneous operators $\mathfrak{p}_1\partial_{x}^{m}$ and $\mathfrak{p}_2\partial_{x}^{m'}$ of degree $m\in\mathbb{Z}$ and $m'\in\mathbb{Z}$, respectively. Then, there exist combinatorial constants given by $(C_{k,m})_{k\in\mathbb{N}}$ with
	$$C_{0,m}\triangleq1,\qquad \forall k\in\mathbb{N}^*,\quad C_{k,m}\triangleq\frac{1}{k!}\prod_{p=1}^{k}(m-p+1),$$
	such that:
	\begin{enumerate}[label=(\roman*)]
		\item for any $N\in\mathbb{N},$ the composition admits an homogeneous expansion of degree $m+m'$ up to order $m+m'-N$ up to smoothing pseudo-differential oprators in $OPS^{m+m'-N-1}$. More precisely,
		$$\mathfrak{p}_1\partial_{x}^{m}\circ\mathfrak{p}_2\partial_{x}^{m'}=\sum_{k=0}^{N}C_{k,m}\,\mathfrak{p}_1(\partial_{x}^{k}\mathfrak{p}_2)\partial_{x}^{m+m'-k}+\mathfrak{R}_{N},$$
		where $\mathfrak{R}_{N}$ is in $OPS^{m+m'-N-1}$ with the bound, for any $s\geqslant s_0$ and $\alpha\in\mathbb{N},$
		\begin{equation}\label{e-rem-hom-comp}
			\|\mathfrak{R}_{N}\|_{m+m'-N-1,s,\alpha}^{q,\gamma,\mathcal{O}}\lesssim_{m,m',N,\alpha,s}\|\mathfrak{p}_1\|_{s+\sigma(\alpha,m)}^{q,\gamma,\mathcal{O}}\|\mathfrak{p}_2\|_{s_0+\sigma(\alpha,m)}^{q,\gamma,\mathcal{O}}+\|\mathfrak{p}_1\|_{s_0+\sigma(\alpha,m)}^{q,\gamma,\mathcal{O}}\|\mathfrak{p}_2\|_{s+\sigma(\alpha,m)}^{q,\gamma,\mathcal{O}}
		\end{equation}
		for some $\sigma(\alpha,m)>0;$
		\item for any $N\in\mathbb{N}^*,$ the commutator admits an homogeneous expansion of degree $m+m'-1$ up to order $m+m'-N$ up to smoothing pseudo-differential oprators in $OPS^{m+m'-N-1}$. More precisely,
		$$[\mathfrak{p}_1\partial_{x}^{m},\mathfrak{p}_2\partial_{x}^{m'}]=\sum_{k=1}^{N}\big(C_{k,m}\,\mathfrak{p}_1(\partial_{x}^{k}\mathfrak{p}_2)-C_{k,m'}\,(\partial_{x}^{k}\mathfrak{p}_1)\mathfrak{p}_2\big)\partial_{x}^{m+m'-k}+\mathfrak{R}_{N},$$
		where $\mathfrak{R}_{N}$ is in $OPS^{m+m'-N-1}$ with the bound, for any $s\geqslant s_0$ and $\alpha\in\mathbb{N},$
		$$\|\mathfrak{R}_{N}\|_{m+m'-N-1,s,\alpha}^{q,\gamma,\mathcal{O}}\lesssim_{m,m',N,\alpha,s}\|\mathfrak{p}_1\|_{s+\sigma(\alpha,m,m')}^{q,\gamma,\mathcal{O}}\|\mathfrak{p}_2\|_{s_0+\sigma(\alpha,m,m')}^{q,\gamma,\mathcal{O}}+\|\mathfrak{p}_1\|_{s_0+\sigma(\alpha,m,m')}^{q,\gamma,\mathcal{O}}\|\mathfrak{p}_2\|_{s+\sigma(\alpha,m,m')}^{q,\gamma,\mathcal{O}}$$
		for some $\sigma(\alpha,m,m')>0;$
		\item for any $N\in\mathbb{N}^*,$ the anti-commutator admits an homogeneous expansion of degree $m+m'$ up to order $m+m'-N$ up to smoothing pseudo-differential oprators in $OPS^{m+m'-N-1}$. More precisely,
		$$[\mathfrak{p}_1\partial_{x}^{m},\mathfrak{p}_2\partial_{x}^{m'}]_{\rm a}=\sum_{k=1}^{N}\big(C_{k,m}\,\mathfrak{p}_1(\partial_{x}^{k}\mathfrak{p}_2)+C_{k,m'}\,(\partial_{x}^{k}\mathfrak{p}_1)\mathfrak{p}_2\big)\partial_{x}^{m+m'-k}+\mathfrak{R}_{N},$$
		where $\mathfrak{R}_{N}$ is in $OPS^{m+m'-N-1}$ with the bound, for any $s\geqslant s_0$ and $\alpha\in\mathbb{N},$
		$$\|\mathfrak{R}_{N}\|_{m+m'-N-1,s,\alpha}^{q,\gamma,\mathcal{O}}\lesssim_{m,m',N,\alpha,s}\|\mathfrak{p}_1\|_{s+\sigma(\alpha,m,m')}^{q,\gamma,\mathcal{O}}\|\mathfrak{p}_2\|_{s_0+\sigma(\alpha,m,m')}^{q,\gamma,\mathcal{O}}+\|\mathfrak{p}_1\|_{s_0+\sigma(\alpha,m,m')}^{q,\gamma,\mathcal{O}}\|\mathfrak{p}_2\|_{s+\sigma(\alpha,m,m')}^{q,\gamma,\mathcal{O}}$$
		for some $\sigma(\alpha,m,m')>0.$
	\end{enumerate}
\end{prop}
\subsubsection{Symplectic diffeomorphisms of the torus and Egorov's theorem}
In this last section, we give the definition and properties of symplectic diffeomorphisms on the torus, used in our analysis for the transport reduction in Section \ref{sect.transport}. Then, we discuss a quantitative Egorov-type theorem adapted to our functional context. This latter deals with the conjugation of pseudo-differential operators with symplectic diffeomorphisms.\\

A symplectic diffeomorphism on the torus $\mathbb{T}$ is a transformation $\mathscr{B}$ in the form
\begin{equation}\label{diffeo.app}
   \begin{aligned}
        (\mathscr{B}u)(\lambda;\varphi,x) &  \triangleq\big(1+\partial_{x}\beta(\lambda;\varphi,x)\big) (\cB u) (\lambda;\vf,x),\qquad(\lambda;\varphi,x)\in\mathcal{O}\times\mathbb{T}^d\times\mathbb{T}, \\
        (\cB u) (\lambda;\vf,x) & \triangleq u \big(\lambda,\varphi,x+\beta(\lambda;\varphi,x)\big) .
   \end{aligned}
\end{equation}
The following proposition is a collection of standard properties of the map $\sB$ in \eqref{diffeo.app}, see for instance \cite[Lemma 6.1, Lemma 6.2]{HR21} and \cite[Lemma 3.30]{BFM21}.
\begin{prop}\label{prop:Ego}
	\textbf{(Basic properties of symplectic diffeomorphisms on the torus)}\\
	There exist $\sigma,\delta>0$ such that if  
    \begin{equation}\label{beta.small.app}
        \|\beta\|_{s_0+\sigma}^{q,\gamma,\mathcal{O}}\leqslant\delta,
    \end{equation}
	then, the following properties hold true:
	\begin{enumerate}[label=(\roman*)]
		\item The operator $\mathscr{B}$ is invertible with inverse
        \begin{align}
            (\mathscr{B}^{-1}u)(\lambda;\varphi,x) &  =\big(1+\partial_{x}\wh\beta(\lambda;\varphi,x)\big) (\cB^{-1} u) (\lambda;\vf,x), \\
        (\cB^{-1} u) (\lambda;\vf,x) & =u \big(\lambda,\varphi,x+\wh\beta(\lambda;\varphi,x)\big) .
        \end{align}
		In addition,
		$$\|\widehat{\beta}\|_{s}^{q,\gamma,\mathcal{O}}\lesssim_{s,q,d}\|\beta\|_{s}^{q,\gamma,\mathcal{O}}$$
		and the following estimates occur
		\begin{align}
			&\|\mathcal{B}^{\pm1}u\|_{s}^{q,\gamma,\mathcal{O}}+\|\mathscr{B}^{\pm1}u\|_{s}^{q,\gamma,\mathcal{O}}\lesssim_{s,q,d}\|u\|_{s}^{q,\gamma,\mathcal{O}}+\|\beta\|_{s+1}^{q,\gamma,\mathcal{O}}\|u\|_{s_0}^{q,\gamma,\mathcal{O}},\\
			&\|(\mathcal{B}^{\pm1}-\textnormal{Id})u\|_{s}^{q,\gamma,\mathcal{O}}+\|(\mathscr{B}^{\pm1}-\textnormal{Id})u\|_{s}^{q,\gamma,\mathcal{O}}\lesssim_{s,q,d}\|\beta\|_{s_0}^{q,\gamma,\mathcal{O}}\|u\|_{s+1}^{q,\gamma,\mathcal{O}}+\|\beta\|_{s+1}^{q,\gamma,\mathcal{O}}\|u\|_{s_0}^{q,\gamma,\mathcal{O}};
		\end{align}
    \item  The operator $\sB$ preserves the average in space, namely
    \begin{equation}
        \int_{\T} (\sB u)(\vf,x) \wrt x = \int_{\T}(1+\pa_x \beta(\vf,x))u(\vf,x+\beta(\vf,x)) \wrt x \stackrel{y=x+\beta(\vf,x)}{=} \int_{\T} u(\vf,y) \wrt y .
    \end{equation}
    In particular, functions with zero average in space are preserved by the action of the operator;
    \item Denoting by $\sB^{\star}$ the $L_{x}^2(\T)$-adjoint of $\sB$, then $\sB^{\star}= \cB^{-1}$;
	\item Let $\beta_1,\beta_2$ satisfying \eqref{beta.small.app}, then for any $s\in[s_0,S],$
	\begin{align}
		\|\Delta_{12}\widehat{\beta}\|_{s}\lesssim_{s,d}\|\Delta_{12}\beta\|_{s}+\|\Delta_{12}\beta\|_{s_0}\max_{k\in\{1,2\}}\|\beta_k\|_{s+1}.
	\end{align}
    \item The operator $\mathscr{B}$ (or $\mathcal{B}$) is:
    \begin{enumerate}
        \item reversibility preserving if and only if
        $$\beta(\lambda;-\varphi,-x)=-\beta(\lambda;\varphi,x);$$
        \item momentum preserving if and only if
        $$\forall \, y\in\mathbb{T},\quad\beta(\lambda;\varphi-\vec{\jmath}y,x)=-\beta(\lambda;\varphi,x+y).$$
    \end{enumerate}
	\end{enumerate}
\end{prop}

We consider
\begin{equation}\label{diffeo tau}
	\mathscr{B}^{\tau}u(\lambda;\varphi,x)\triangleq\big(1+\tau\partial_{x}\beta(\lambda;\varphi,x)\big)u\big(\lambda;\varphi,x+\tau\beta(\lambda;\varphi,x)\big).
\end{equation}
Observe that $\mathscr{B}^{\tau}$ is the flow associated with $\partial_{x}\circ b=\textnormal{Op}(\ii b\xi+\partial_{x}b)$ with
$$b(\tau,\lambda;\varphi,x)\triangleq \frac{\beta(\lambda;\varphi,x)}{1+\tau\partial_{x}\beta(\lambda;\varphi,x)}\cdot$$
Namely, we have
$$\partial_{\tau}\big(\mathscr{B}^{\tau}u\big)=\partial_{x}\big(b\mathscr{B}^{\tau}u\big)\qquad\textnormal{and}\qquad\mathscr{B}^{0}=\textnormal{Id}.$$
The following result states the stability of the class of operator admitting an homogeneous expansion  under conjugation by a diffeomorphism of the torus. The proof is a small variation of \cite[Prop. 2.28]{BKM21}.In particular the point $(iv)$ can be obtained tracking the corresponding symmetries along the constructive proof in the given reference.

\begin{prop}\label{Egorov thm}\textbf{(Egorov Theorem for homogeneous operators)}
	Let $N,l_0\in\mathbb{N},$ $m\in\mathbb{Z}$ and $\mathfrak{p}=\mathfrak{p}(\lambda,\varphi,x)$ be a smooth function. Let us simply denote $\mathscr{B}\triangleq\mathscr{B}^{1}$ the time $1$ of the flow introduced in \eqref{diffeo tau}. There exist $C_0,\sigma_{N},\sigma_{N}(l_0)>0$ and $\delta\triangleq\delta(S,N,l_0,q)\in(0,1)$ such that, if
	\begin{equation}\label{small ego}
		\|\beta\|_{s_0+\sigma_{N}(l_0)}^{q,\gamma,\mathcal{O}}\leqslant\delta,\qquad\|\mathfrak{p}\|_{s_0+\sigma_{N}(l_0)}^{q,\gamma,\mathcal{O}}\leqslant C_0,
	\end{equation}
	then the conjugation of $\mathfrak{p}\partial_{x}^{m}$ by the diffeomorphism $\mathscr{B}$ is a pseudo-differential operator of order $m$ admitting an expansion in the form
	$$\mathscr{B}^{-1}\big(\mathfrak{p}\partial_{x}^{m}\big)\mathscr{B}=\sum_{k=0}^{N}\mathfrak{p}_{m-k}\partial_{x}^{m-k}+R_{N},$$
	with the following properties:
	\begin{enumerate}[label=(\roman*)]
		\item The principal symbol $\mathfrak{p}_{m}$ is
		$$\mathfrak{p}_{m}(\lambda;\varphi,x)=\Big[\big(1+\partial_{\eta}\widehat{\beta}(\lambda;\varphi,y)\big)^{m}\mathfrak{p}(\lambda;\varphi,y)\Big]|_{y=x+\beta(\lambda;\varphi,x)}$$
		and we have the estimates
		\begin{equation}\label{Egorov estim}
		    \forall \, s\geqslant s_0,\quad\forall \, k\in\llbracket1,N\rrbracket,\quad\|\mathfrak{p}_{m}-\mathfrak{p}\|_{s}^{q,\gamma,\mathcal{O}}+\|\mathfrak{p}_{m-k}\|_{s}^{q,\gamma,\mathcal{O}}\lesssim_{s,N}\|\beta\|_{s+\sigma_N}^{q,\gamma,\mathcal{O}}+\|\mathfrak{p}\|_{s+\sigma_N}^{q,\gamma,\mathcal{O}}\|\beta\|_{s_0+\sigma_N}^{q,\gamma,\mathcal{O}};
		\end{equation}
		\item The remainder term $R_N$ is a pseudo-differential operator in $OPS^{m-N-1}.$ Moreover, for any $l\in\mathbb{N}^d,$ with $|l|\leqslant l_0$ and for any $n_1,n_2\in\mathbb{R}$ with $n_1+n_2+l_0\leqslant N-q-m$, the operator $\langle D\rangle^{n_1}\partial_{\varphi}^{l}R_{N}\langle D\rangle^{n_2}$ is $\mathcal{D}^{q}$-tame and satisfies the estimate, for any $s\in[s_0,S]$
		\begin{equation}\label{Egorov Dq tame}
		    \mathfrak{M}_{\langle D\rangle^{n_1}\partial_{\varphi}^{l}R_{N}\langle D\rangle^{n_2}}(s)\lesssim_{s,q,d,N,l_0}\|\beta\|_{s+\sigma_{N}(l_0)}^{q,\gamma,\mathcal{O}}+\|\mathfrak{p}\|_{s+\sigma_N(l_0)}^{q,\gamma,\mathcal{O}}\|\beta\|_{s_0+\sigma_N(l_0)}^{q,\gamma,\mathcal{O}};
		\end{equation}
		\item Let $s_1\geqslant s_0.$ Consider $\beta_1,\beta_2,\mathfrak{p}^{[1]},\mathfrak{p}^{[2]}$ satisfying the additional smallness condition 
		$$\max_{k\in\{1,2\}}\|\beta_k\|_{s_1+\sigma_{N}(l_0)}^{q,\gamma,\mathcal{O}}\leqslant\delta,\qquad \max_{k\in\{1,2\}}\|\mathfrak{p}^{[k]}\|_{s_1+\sigma_{N}(l_0)}^{q,\gamma,\mathcal{O}}\leqslant C_0,$$
		then 
		\begin{equation}\label{Egorov estim12}
		    \forall \, k\in\llbracket0,N\rrbracket,\quad\|\Delta_{12}\mathfrak{p}_{m-k}\|_{s_1}\lesssim_{s_1,d,N}\|\Delta_{12}\mathfrak{p}\|_{s_1+\sigma_N}+\|\Delta_{12}\beta\|_{s_1+\sigma_N}
		\end{equation}
		and for any $l\in\mathbb{N}^d,$ with $|l|\leqslant l_0$ and for any $n_1,n_2\in\mathbb{R}$ with $n_1+n_2+l_0\leqslant N-q-m,$
		\begin{equation}\label{Egorov 12R}
		    \|\langle D\rangle^{n_1}\partial_{\varphi}^{l}\Delta_{12}R_{N}\langle D\rangle^{n_2}\|_{\mathcal{L}(H^{s_1})}\lesssim_{s_1,d,N,n_1,n_2,l_0}\|\Delta_{12}\mathfrak{p}\|_{s_1+\sigma_{N}(l_0)}+\|\Delta_{12}\beta\|_{s_1+\sigma_{N}(l_0)}.
		\end{equation}
	\end{enumerate}
\end{prop}

\end{document}